\documentclass[letterpaper,11 pt]{article}
\usepackage[margin=1in]{geometry}
\usepackage[T1]{fontenc}    
\usepackage{multirow}
\usepackage{wrapfig}
\usepackage{booktabs}       
\usepackage{enumerate}
\usepackage{amsmath}
\usepackage{lipsum}
\newcommand{\bxi}{\boldsymbol{\xi}}
\usepackage{mathtools}
\usepackage{float}
\usepackage[dvipsnames,table,xcdraw]{xcolor}
\usepackage{bbm}
\usepackage{varioref}
\usepackage{hyperref}
\usepackage{float}
\usepackage{placeins}
                                
\usepackage{amssymb} 
\usepackage{rotating}
\usepackage{graphicx}
\usepackage{algpseudocode}
\usepackage{mathrsfs}  
\usepackage{algcompatible}
\usepackage{amsmath}\usepackage{amsmath}
\usepackage{cases}
\usepackage{graphicx}
\usepackage{multirow}
\usepackage{comment}
\usepackage{subcaption}
\usepackage[normalem]{ulem}

\usepackage{amsthm}
\DeclareUnicodeCharacter{00A0}{~}
\usepackage{algorithm}
\usepackage{algpseudocode}
\algnewcommand{\Initialize}[1]{%
	\State \textbf{Initialization:}
	\Statex {\raggedright #1}
}
\newtheorem{assumption}{Assumption}
\newtheorem{theorem}{Theorem}
\newtheorem{lemma}{Lemma}
\newtheorem{proposition}{Proposition}
\newtheorem{definition}{Definition}
\newtheorem{corollary}{Corollary}
\theoremstyle{plain}
\newtheorem{remark}{Remark}
\usepackage{xcolor}  
\newcommand{\mee}[1]{{\color{black}#1}}
\newcommand{\fy}[1]{{\color{black}#1}}
\newcommand{\us}[1]{{\color{black}#1}}
\newcommand{\uvs}[1]{{\color{black}#1}}
\newcommand{\vvs}[1]{{\color{black}#1}}

\newcommand{\mje}[1]{{\color{black}#1}}
\newcommand{\mj}[1]{{\color{black}#1}}

\newcommand{\fyy}[1]{{\color{black}#1}}
\newcommand{\fyr}[1]{{\color{black}#1}}

\newcommand{\RRmath}[1]{\mathcolor{black}{#1}}

\newcommand{\RRme}[1]{{\color{black}#1}}

\newcommand{\Rme}[1]{{\color{black}#1}}

\usepackage[textsize=small]{todonotes}

\usepackage{tcolorbox}
\title{\LARGE \bf
 On the Resolution of Stochastic MPECs over Networks: Distributed Implicit Zeroth-Order Gradient Tracking Methods
}

\author{Mohammadjavad Ebrahimi$^{1}$, Uday V. Shanbhag$^{2}$, and Farzad Yousefian$^{1}$
\thanks{$^{1}$Ebrahimi and Yousefian are with the Department of Industrial and Systems Engineering, Rutgers University, United States. {\tt\small \{mohammadjavad.ebrahimi,farzad.yousefian\}@rutgers.edu}. $^{2}$Shanbhag is with the Department of Industrial and Operations Engineering, University of Michigan, United States. {\tt\small udaybag@umich.edu}.\\
We acknowledge the funding support from the U.S. Department of Energy under grants \#DE-SC0023303 and \#DE-SC0025570, the U.S. Office of Naval Research under grants \#N00014-22-1-2757 and \#N00014-22-1-2589, and AFOSR Grant FA9550-24-1-0259.\\
A very preliminary version of this work \cite{ebrahimi2023distributed} has appeared in the Proceedings of the 2024 American Control Conference.}%
}

\begin{document}
\sloppy
\maketitle
\thispagestyle{empty}
\pagestyle{plain}

\maketitle
 \begin{abstract}
The mathematical program with equilibrium
constraints (MPEC) is a powerful yet challenging class
of constrained optimization problems, where the
constraints are characterized by a parametrized
variational inequality (VI) \Rme{problem}. While efficient algorithms for addressing MPECs and
their stochastic variants (SMPECs) have been recently presented, distributed SMPECs
over networks \us{pose significant challenge\mee{s}}. This work aims to
develop fully \uvs{distributed} iterative methods with complexity
guarantees for resolving distributed SMPECs in two problem settings: (1) distributed
single-stage SMPECs and (2) distributed two-stage
SMPECs. In both cases, the global objective function
is distributed among a network of agents that
communicate cooperatively. In (1), each agent is
constrained by a local VI problem with an
expectation-valued mapping, while in (2), each agent’s
constraint is associated with a second-stage VI problem, where the
mapping depends on a first-stage random variable and a random second-stage parameter. \us{Under the assumption \fyy{that} the parametrized VI is uniquely solvable, the resulting implicit problem in upper-level decisions is generally neither convex nor smooth.}  
Under some
standard assumptions, including the uniqueness of the
solution to the VI problems and the Lipschitz
continuity of the implicit global objective function,
we propose single-stage and two-stage zeroth-order
distributed gradient tracking optimization methods
where the gradient of a smoothed implicit objective
function is approximated using two (possibly inexact)
evaluations of the lower-level VI \Rme{solution map}. \mee{In} the
exact setting of both the single-stage and two-stage
problems, we achieve the best-known complexity bound
for centralized nonsmooth nonconvex stochastic
optimization, that is
\Rme{$\mathcal{O}\left(\eta^{-1} n^{3/2}\epsilon^{-2}\right)$} where
$n$ denotes the upper-level decision dimension, \Rme{$\eta$ is the smoothing parameter,}
$\epsilon>0$ is an arbitrary scalar such that
$\mathbb{E}[\mbox{dist}^2(0,\partial_{\eta}
f(\bullet)) ] \leq \epsilon$, $f$ denotes the
implicit objective function, \us{and} $\partial_{\eta} f\Rme{(x)}$
denotes the $\eta$-Clarke generalized gradient of $f$ \Rme{at $x$}.
\Rme{The complexity bound of $\mathcal{O}\left( n^{3/2}\epsilon^{-2}\right)$ is} achieved (for the first
time) for our method in addressing the inexact setting
of the distributed two-stage SMPEC. \Rme{In addressing the
inexact setting of the single-stage problem, we derive
the overall complexity of
$\mathcal{O}\left(n^{7/2}\epsilon^{-4}\right)$.}
Preliminary numerical experiments demonstrate the
efficiency of our methods across various network
topologies compared to centralized zeroth-order
schemes.  

\end{abstract}

\section{Introduction}
The mathematical program with equilibrium constraints
(MPEC) is an immensely powerful yet challenging class
of constrained optimization problems \RRme{where} the
constraints are captured by a parametrized variational
inequality problem. Consider the canonical MPEC, given
by $$ \min_{x,z} \, \left\{ \, f(x,z) \, \mid \,  z \in
\mbox{SOL}(\mathcal{Z}(x),F(x,\bullet)), \ x \in
\mathcal{X} \, \right\} ,$$ where $f:\mathbb{R}^n\times
\mathbb{R}^p \to \mathbb{R}$ is a real-valued
function, $F:\mathcal{X}\times 	\mathbb{R}^p \to
\mathbb{R}^p$ is a single-valued mapping,
$\mathcal{Z}:\mathcal{X} \rightrightarrows
\mathbb{R}^p$ is a set-valued mapping, and
$\mathcal{X} \subseteq \mathbb{R}^n$ is a closed and
convex set. Recall that given a closed and convex set
$\mathcal{Y} \subseteq \mathbb{R}^p$ and mapping
$G:\mathcal{Y}\to\mathbb{R}^p $, vector $y^* \in
\mathcal{Y}$ solves the variational inequality (VI)
problem, denoted by $\mbox{VI}(\mathcal{Y},G)$ if
$G(y^*)^\top(y-y^*)\geq 0$ for all $y \in
\mathcal{Y}$. We let $\mbox{SOL}(\mathcal{Y},G)$
denote the set of all such solutions. \Rme{MPECs model the class of bilevel programs in which the lower-level problem is a smooth convex optimization problem and can capture the equilibria of Stackelberg games with single or multiple followers when the follower-level game can be captured by an equivalent VI problem. MPECs also find} applicability in structural design, power markets~\cite{hobbs2000strategic},
\Rme{transportation~\cite{lawphongpanich2004mpec},} and finance, among other
disciplines~\cite{sherali1983stackelberg,luo1996mathematical}.
See for instance \cite{hobbs2000strategic} and
\cite{lawphongpanich2004mpec} for applications of
MPECs in electric power systems and transportation
networks, respectively. In the past few decades,
significant progress has been made in addressing some
theoretical challenges of resolving MPECs, including
necessary and sufficient optimality
conditions~\cite{fukushima1998some,outrata2000generalized,jane2005necessary},
constraint
qualifications~\cite{flegel2005abadie,kanzow2010mathematical,guo2015solving}
facilitating the development of schemes including
exact penalization~\cite{scholtes1999exact},
sequential quadratic programming
schemes~\cite{fletcher2006local}, nonlinear
programming approaches~\cite{dirkse2002mathematical},
smoothing methods for implicit
formulations~\cite{facchinei1999smoothing}, and
constraint relaxation
schemes~\cite{steffensen2010new}, among others. \Rme{A comprehensive description of the solution properties, algorithms, and applications of MPECs is provided in the following two seminal monographs~\cite{luo1996mathematical,outrata1998nonsmooth}}. 

\smallskip 

\noindent  {\bf Stochastic MPECs.} Motivated by
hierarchical decision-making under uncertainty and
large-scale bilevel programming, stochastic
generalizations of MPECs (SMPECs) were introduced
in~\cite{patriksson1999stochastic} where the
directional differentiability of the implicit
upper-level objective function is studied. SMPECs \us{can capture a breadth} of applications in the presence
of hierarchy and uncertainty, e.g., generation
capacity expansion in electricity
markets~\cite{wogrin2011generation} and optimal grid
tariff design for power
systems~\cite{askeland2023stochastic}. In resolving
SMPECs, the expected residual minimization
formulations are
considered~\cite{lin2007new} and sample average
approximation (SAA) schemes are
developed~\cite{shapiro2008stochastic} (see
\cite{lin2010stochastic} for a survey on SMPECs).
In~\cite{lin2009solving} a method integrating
smoothing implicit programming and a penalty approach
is developed for addressing SMPECs with linear
complementarity constraints. First-order optimality
conditions for two-stage SMPECs whose second-stage
problem has multiple equilibria/solutions \RRme{are} studied
in~\cite{xu2010necessary}. Motivated by the absence of
solution methods equipped with non-asymptotic rate and overall complexity guarantees
for addressing SMPECs, a class of randomized smoothing
and inexact zeroth-order schemes is developed~\cite{cui2023complexity} for resolving both
single-stage and two-stage SMPECs, \Rme{possibly representing} the first instance of efficient schemes
with complexity and rate guarantees. Extensions
to decentralized settings and federated learning
regimes are studied in~\cite{qiu2023j}. Further,
schemes for computing equilibria for a class of
stochastic multi-leader multi-follower games
where players solve parametrized stochastic \RRme{MPECs}
are considered~\cite{cui2023computation},
where guarantees are derived under
either a potentiality or a monotonicity
requirement on a suitably defined implicit
leader-level objective.  

\smallskip 

\begin{tcolorbox}
 {\bf Gaps \us{and goals}.} Despite recent progress in addressing MPECs in large-scale stochastic settings, distributed variants of SMPECs over networks remain \us{challenging generalizations}. Our goal lies in developing computational methods with complexity guarantees for resolving \us{distributed SMPECs}, where the information of the upper-level (global) objective function is distributed among multiple agents, who cooperatively seek to find a solution to the SMPEC by \us{communicating} over a network. 
\end{tcolorbox}

To address \fyy{these gaps}, we consider two variants of distributed SMPECs described next. 
\fy{

\noindent {\bf (i)  Distributed single-stage SMPECs.} We consider a distributed single-stage SMPEC among $m$ agents \us{that takes} the following form.
\begin{tcolorbox}[colback=blue!10,colframe=blue!75!black,title=]
\vspace{-.2in}
 \begin{align}
&\hbox{minimize}_x\quad f(x)\ \triangleq \frac{1}{m}\sum_{i=1}^m  \mathbb{E}_{\bxi_i(\omega)}\left[\tilde{h}_i(x,z_i(x),\bxi_i(\omega))\right] \label{eqn:prob-1stage}\tag{\footnotesize{\bf dist-SMPEC}$_\texttt{1s}$}\\
& \hbox{subject to} \quad  z_i(x)  \, = \, \mbox{SOL}\left(\mathcal{Z}_i(x),  \mathbb{E}_{\bxi_i(\omega)}\left[\, \tilde{F}_i(x,\bullet,\bxi_i(\omega))\, \right]\, \right) \quad \hbox{for all } i \, \in \, \{1,\ldots,m\}. \notag
\end{align}
\end{tcolorbox}
Here, agent $i$ is associated with a local
stochastic objective $\tilde{h}_i:
\mathbb{R}^n \times \mathbb{R}^{p} \times
\mathbb{R}^d \to \mathbb{R}$, an independent
random vector $\us{\bxi}_i(\omega):\Omega\rightarrow
\mathbb{R}^d$ associated with the probability
space $(\Omega,\mathcal{F},\mathbb{P}_i)$, a
single-valued mapping $\RRmath{\tilde F_i: \mathbb{R}^n
\times \mathbb{R}^{p}\times \mathbb{R}^{d} \to \mathbb{R}^{p}}$, and a
parametrized set-valued mapping $\mathcal{Z}_i\us{(\bullet)}$.
Here, we assume that for any $i$, the mapping
$z_i:\mathbb{R}^n  \to \mathbb{R}^{p}$ returns
the unique solution to \us{a VI, parametrized} for any $x$.

\noindent {\bf (ii)  Distributed two-stage SMPECs.} We consider a distributed two-stage SMPEC among $m$ agents, \us{formulated as follows.} 
 \begin{tcolorbox}[colback=blue!10,colframe=blue!75!black,title=] 
 \vspace{-.2in}
\begin{align}
&\hbox{minimize}_x\quad f(x)\ \triangleq \frac{1}{m}\sum_{i=1}^m  \mathbb{E}_{\bxi_i(\omega)}\left[\, \tilde{h}_i(x,z_i(x,\bxi_i(\omega)),\bxi_i(\omega))\, \right] \label{eqn:prob-2stage}\tag{\footnotesize{\bf dist-SMPEC}$_\texttt{2s}$}\\
& \hbox{subject to} \quad  z_i(x,\xi_i(\omega)) \,=\, \mbox{SOL}\left(\, \mathcal{Z}_i(x,\xi_i(\omega)),\tilde F_i(x,\bullet,\xi_i(\omega))\, \right) \quad \text{for a.e. } \omega \text{ and } i \in \{\, 1,\ldots,m\, \}. \notag
\end{align} 
\end{tcolorbox}
 Here, we \Rme{adopt} the notation used for
\eqref{eqn:prob-1stage} \Rme{but}
emphasize \Rme{the following} distinctions:
(i) The mapping
$z_i(\bullet,\xi_i(\omega))$ and
set $\mathcal{Z}_i$ are afflicted by
randomness, as captured by the
dependence on  the random sample
$\omega$; and (ii) \RRme{each} of the agents is
associated with a lower-level VI problem, \us{parametrized by $x$},
characterized by a random mapping.
We assume that the function
$\tilde{h}_i$, the mapping
$\tilde{F}_i$, and random samples from
$\xi_i(\omega)$ are locally available to
the agent $i$, for all $i$\Rme{, and we denote a realization of $\bxi_i(\omega)$ by $\xi_i(\omega)$}. 
}
\smallskip


\noindent {\bf Main contributions.} We
design and analyze two iterative methods
for the resolution of distributed
single-stage SMPECs in
\eqref{eqn:prob-1stage} and two-stage
SMPECs in \eqref{eqn:prob-2stage} over
networks, where we assume that the
implicit function is possibly nonconvex.
To address each of these problems, we
propose a zeroth-order distributed
gradient tracking optimization scheme
where the gradient of a smoothed
implicit objective function is
approximated through modified and
possibly inexact evaluations of
$z_i(\bullet)$ for the single-stage
setting and $z_i(\bullet,\xi_i(\omega))$
for the two-stage setting. Notably, the
proposed methods \mj{extend} existing
distributed gradient tracking methods
for smooth nonconvex distributed
stochastic optimization
in~\cite{xin2021improved} to address
much broader hierarchical optimization
problem classes in this work.  We
conduct a rigorous theoretical analysis
of the proposed methods, establishing
new complexity guarantees for computing
\RRme{an approximate} stationary point in both
single-stage and two-stage SMPECs. The
main results are summarized in
\fyy{Table~\ref{table:tbl-combined-smpec-complexity-1s and 2s}}\Rme{, where $L_0$ is specified in Assumption~\ref{assum:main1}. Additionally, $\kappa_F:=\tfrac{L_F}{\mu_F}$, where $L_F$ and $\mu_F$ are defined in Assumption~\ref{assum:main2}.}
Importantly, our complexity guarantees
for both the upper-level and lower-level
problems not only align with those
established for centralized stochastic
MPECs in~\cite{cui2023complexity}, but
they also demonstrate improved scaling
with respect to $n$. This is
accomplished by leveraging the use of
central difference zeroth-order gradient
approximation
schemes~\cite{lin2022gradient}. For the
exact setting of both the single-stage
and two-stage cases, we achieve the
best-known complexity bound for
centralized nonsmooth nonconvex
stochastic
optimization~\cite{lin2022gradient},
that is
\Rme{$\mathcal{O}\left(\eta^{-1}n^{3/2}\epsilon^{-2}\right)$}
where $\epsilon>0$ is an arbitrary
scalar such that
$\mathbb{E}[\mbox{dist}^2(0,\partial_{\eta}
f(\bullet)) ] \leq \epsilon$ \us{and} $f$
denotes the implicit objective function. \fyy{Here, $\eta>0$ denotes a randomized smoothing parameter employed in the zeroth-order method and 
$\partial_\eta f$ denotes
$\eta$-Clarke generalized gradient of
$f$ (see Section~\ref{sec:prelim} for details).}
Importantly, this complexity bound is
achieved for our method in addressing
the inexact setting of the distributed
two-stage SMPEC \eqref{eqn:prob-2stage}.
This appears to be the first \us{such 
guarantee}  for this class of
problems over networks. \Rme{In addressing
the inexact setting of the single-stage
problem \eqref{eqn:prob-1stage}, the
overall complexity of
{$\mathcal{O}\left(n^{7/2}\epsilon^{-4}\right)$}
is obtained.} This result is novel and
appears to be new for the distributed
single-stage SMPECs over networks. Next, we present the detailed contributions for each setting in the following.

\begin{table}[]\centering
\fyy{
 \footnotesize{
\caption{Complexity guarantees for solving~\eqref{eqn:prob-1stage} and~\eqref{eqn:prob-2stage}}
\label{table:tbl-combined-smpec-complexity-1s and 2s}
\begin{tabular}{@{}l@{\hspace{2em}}c@{\hspace{10em}}c@{}}
\toprule
\multicolumn{1}{l}{{\bf{Setting}}} & \multicolumn{2}{c}{{\bf{Iteration complexity bound}}} \\ \cmidrule(l){2-3}
& Distributed single-stage SMPEC & Distributed two-stage SMPEC \\ \midrule

{\bf{\underline{Inexact setting}}} & & \\[0.5ex]

\begin{tabular}[c]{@{}l@{}}Upper level\\ problem\end{tabular} & 
\Rme{$\mathcal{O}\left(\tfrac{L_0^3n^{3/2}\epsilon^{-2}}{\eta^{4/3}}\right)$ }& 
\Rme{$\mathcal{O}\left(\tfrac{L_0^3n^{3/2}\epsilon^{-2}}{\eta^{4/3}}\right)$ }\\[2ex]

\begin{tabular}[c]{@{}l@{}}Lower level\\ problem\end{tabular} & 
$\mathcal{O}\left(\RRmath{\kappa_F^2}\epsilon^{-1}\right)$ & 
$\mathcal{O}\left(\ln({\epsilon})/\ln(1- \RRmath{\kappa_F^{-2}/2})\right)$ \\[2ex]

\begin{tabular}[c]{@{}l@{}}Overall\\ complexity\end{tabular} & 
\Rme{$\mathcal{O}\left(\tfrac{L_0^6n^{7/2}\epsilon^{-4}}{\eta^{10/3}}\right)$} & 
\Rme{$\tilde{\mathcal{O}}\left(\tfrac{L_0^3n^{3/2}\epsilon^{-2}}{\eta^{4/3}}\right)$} \\[3ex] \midrule

\textbf{Exact setting} & 
\Rme{$\mathcal{O}\left(\tfrac{L_0^3n^{3/2}\epsilon^{-2}}{\eta}\right)$} & 
\Rme{$\mathcal{O}\left(\tfrac{L_0^3n^{3/2}\epsilon^{-2}}{\eta}\right)$} \\
\bottomrule
\end{tabular}}}
\vspace{0.5ex}

\Rme{\parbox{0.95\linewidth}{\footnotesize
\emph{Note:} The third row reports the overall lower-level sample complexity required to obtain an $\epsilon$-stationary solution.}}
\end{table}

\noindent {\bf{(1)}} We address the distributed single-stage problem \eqref{eqn:prob-1stage} in Section~\ref{sec:1s}. 

\noindent  {\bf{(1-\noindent i)}} {Inexact setting.} We propose (DiZS-GT$^{\text{1s}}$), outlined in Algorithm~\ref{alg:DZGT}, which utilizes a distributed zeroth-order method to minimize the implicit function. In its inexact form, the method addresses the stochastic VI at the lower level \uvs{via} a stochastic approximation approach, presented in Algorithm~\ref{alg:lowerlevel-1stage}. In Theorem~\ref{Theorem:thm 1}, we establish non-asymptotic convergence rates and derive the upper-level and lower-level iteration complexity bounds.  

\noindent{\bf{(1-ii)}} {Exact setting.} The convergence analysis for the exact case of (DiZS-GT$^{\text{1s}}$) is detailed in Corollary~\ref{corollary:corollary 1}.

 \noindent {\bf{(2)}} We examine the distributed two-stage problem \eqref{eqn:prob-2stage} in Section~\ref{sec:2s}. 
 
\noindent  {\bf{(2-i)}} {Inexact setting.} We propose (DiZS-GT$^{\text{2s}}$), outlined in Algorithm~\ref{alg:DZGT-2stage}, to minimize the implicit objective function. In the inexact variant of (DiZS-GT$^{\text{2s}}$), no additional sampling is required for solving the lower-level problem, \RRme{in} contrast with the single-stage case. Instead, Algorithm~\ref{alg:lowerlevel-2stage} is focused on solving a \RRme{parametrized} deterministic VI problem. In Theorem~\ref{Theorem:thm 2}, we derive non-asymptotic convergence rates and provide the iteration complexity for both the upper-level and lower-level problems. While the upper-level iteration complexity is comparable to that of the single-stage model, both the lower-level and overall iteration complexities are significantly lower. This improvement is due to the deterministic nature of the lower-level problem, \uvs{eliminating} the need for extra sampling and \mee{leading} to a more efficient computation.
 
\noindent {\bf{(2-ii)}} {Exact setting.} The convergence analysis for the exact case of \RRme{(DiZS-GT$^{\text{2s}}$)}, similar to that of the single-stage counterpart, is provided in Corollary~\ref{corollary:corollary 1}.
 


%
\noindent  \us{{\bf Outline.} \Rme{After recapping the
notation, in Section~\ref{sec:prelim},
{we provide some remarks on the problems of interest, comment on the associated approach and challenges, and finally review} some preliminaries on
nonsmooth calculus and 
randomized smoothing. Sections~\ref{sec:1s}--\ref{sec:2s} (as discussed next) contain our main results on the single and two-stage problem, respectively, while   
preliminary numerics and some
concluding remarks are presented in
Section~\ref{sec:num} and
Section~\ref{sec:conc}, respectively}. }

\noindent {\bf Notation.} 
Given a vector $x \in \mathbb{R}^{n\times 1}$, we let $\|x\|$ denote the Euclidean norm of $x$ and we let $x^\top$ denote its transpose.  Given a matrix $\mathbf{A} \in \mathbb{R}^{m\times n}$, we let $\|\mathbf{A}\|$ denote the Frobenius norm of $\mathbf{A}$. The spectral \Rme{norm} of $\mathbf{A}$ (also referred to as $\ell_2$-norm of $\mathbf{A}$~\cite{boyd2004convex}) is denoted by $\|\mathbf{A}\|_2$ and is defined as $\|\mathbf{A}\|_2 \triangleq \sqrt{\lambda_{\max}(\mathbf{A}^\top\mathbf{A})}$ where $\lambda_{\max}(\bullet)$ is the largest eigenvalue of a square matrix. We use $[m]$ to denote $\{1,\ldots,m\}$. Given matrices $\mathbf{u},\mathbf{v} \in \mathbb{R}^{m\times n}$, we let $\langle \mathbf{u},\mathbf{v} \rangle $ denote their Frobenius inner product defined as $\langle \mathbf{u},\mathbf{v} \rangle =\mbox{Trace}(\mathbf{u}^\top\mathbf{v})$ where $\mbox{Trace}(\bullet)$ denotes the trace operator. Throughout, for each $i \in [m]$, we denote the implicit stochastic local function by $\tilde{f}_i(x,\xi_i)$ and its expectation-valued variant by ${f}_i(x) \triangleq \mathbb{E}_{\uvs{\bxi}_i}[\tilde{f}_i(x,\uvs{\bxi}_i)]$, \uvs{where $\bxi_i: \Omega \to \mathbb{R}^d$ denotes a $d$-dimensional random variable, taking realizations denoted by $\xi_i(\omega) \in \mathbb{R}^d$. We use $\xi_i$ to abbreviate $\xi_i(\omega)$ for any $i$.} \uvs{In} the single-stage setting, we define $\tilde{f}_i(x,\xi_i) \triangleq \tilde{h}_i(x,z_i(x),\xi_i)$, while in the two-stage setting, we define $\tilde{f}_i(x,\xi_i) \triangleq \tilde{h}_i(x,z_i(x,\xi_i),\xi_i)$. We further define 
\begin{align}\label{eqn:notations}
&\mathbf{x} \triangleq  [x_1,x_2,\ldots,x_m]^{\top}\RRme{\in \mathbb{R}^{m\times n}}, \qquad   \mathbf{y}  \triangleq  [y_1,y_2,\ldots,y_m]^{\top} \in \mathbb{R}^{m \times n}, \notag\\
&\bar{x} \triangleq  \tfrac{1}{m}\mathbf{1}^{\top} \mathbf{x}   \in \mathbb{R}^{1 \times n}, \qquad  \bar{y} \triangleq  \tfrac{1}{m}\mathbf{1}^{\top} \mathbf{y}    \in \mathbb{R}^{1 \times n},\notag\\
& f(x) \triangleq  \tfrac{1}{m}\textstyle\sum_{i=1}^{m} f_i(x), \qquad  \mathbf{f}(\mathbf{x}) \triangleq  \tfrac{1}{m}\textstyle\sum_{i=1}^{m} f_i(x_i),\notag
  \qquad \mathbf{f}^\eta(\mathbf{x}) \triangleq  \tfrac{1}{m}\textstyle\sum_{i=1}^{m} f_i^\eta(x_i),   \\
& \uvs{\xi \triangleq [\xi_1,\xi_2,\ldots,\xi_m]^{\top}\in \mathbb{R}^{m \times d}\notag ,}\ \qquad    \nabla \mathbf{f}^\eta(\mathbf{x}) \triangleq [\nabla f_1^\eta(x_{1}), \ldots,  \nabla f_m^\eta(x_{m})]^{\top} \in \mathbb{R}^{m \times n},
\\   
&\mee{\overline{ \nabla {f}^\eta}(\mathbf{x}) \triangleq\tfrac{1}{m}\mathbf{1}^{\top} \nabla \mathbf{f}^\eta(\mathbf{x}) \in \mathbb{R}^{1 \times n},} \qquad  \mathbf{v} \triangleq [v_1,v_2,\ldots,v_m]^{\top}\in \mathbb{R}^{m \times n}\mee{,}
\end{align}
where \RRme{for all $i\in[m]$, $x_i,y_i \in \mathbb{R}^{n}$ and $\xi_i \in \mathbb{R}^d$ are column vectors, and} given $\eta>0$, $ f_i^\eta$ is
$\eta$-smoothed variant of $f_i$ that
will be introduced formally in
\RRme{Section~\ref{sec:prelim}}. We let
$\mathbf{1} \in \mathbb{R}^{m \times 1}$
denote the vector whose elements are all
one. We use $\mathbb{E}[\bullet]$ to
denote the expectation. We also denote
the Euclidean projection of a vector $x$ \RRme{onto} a set $\mathcal{X}$ by
$\Pi_{\mathcal{X}}(x)$, i.e.
$\|x-\Pi_{\mathcal{X}}(x)\|=\min_{\bar{x}\in
\mathcal{X}}\|x-\bar{x}\| $. We let
``${\it conv}(\bullet)$'' denote the
convex hull of a set.  We let
$\mathbb{B}$ denote the
\us{$n$-}dimensional unit ball given as
$\mathbb{B}\triangleq \{u\in
\mathbb{R}^n\mid \|u\| \le 1\}$.  We let
$\mathbb{S}$ denote the surface of the
ball $\mathbb{B}$, i.e.,
$\mathbb{S}\triangleq \{v\in
\mathbb{R}^n\mid \|v\| = 1\}$.
Throughout, $ {\mathcal{O}}(\cdot)$ denotes
the big-O notation. We also use
$\tilde{\mathcal{O}}(\cdot)$ in some
settings where we ignore logarithmic
factors.  $\hfill \Box$



\section{Preliminaries \us{and background}}\label{sec:prelim}

\us{In this section, we \Rme{provide motivating examples} in \RRme{Section~\ref{sec:2.1}},
comment on the challenges in addressing such
problems in Section~\ref{sec:2.2}, and conclude with a background on nonsmooth analysis and
randomized smoothing in Section~\ref{sec:2.3}.}

\subsection{\Rme{Motivating examples}}\label{sec:2.1}

\Rme{\noindent {\bf {(i)} Personalized distributed learning in heterogeneous systems.}
A standard formulation in distributed learning is
$\min_{x \in \mathbb{R}^n}\frac{1}{m}\sum_{i=1}^m \mathbb{E}_{\bxi_i(\omega)}[\tilde{L}_i(x,\bxi_i(\omega))]$,
where $m$ agents collaboratively train a shared model $x$ using their local data.
This formulation enforces a common solution (consensus) across all agents. However, in heterogeneous settings, agents may have data drawn from different distributions,
{motivating the need for personalized solutions~\Rme{\cite{t2020personalized,li2021ditto}}}. 
To capture this, {one may} consider a personalized {distributed} framework {given as}
\begin{align}
&\hbox{minimize}_x\quad f(x)\ \triangleq \frac{1}{m}\sum_{i=1}^m
\mathbb{E}_{\bxi_i(\omega)}\!\left[\tilde{L}_i(z_i(x),\bxi_i(\omega))\right] \label{eqn:prob-1stage_hetero}\\
& \hbox{subject to} \quad z_i(x)\ \hbox{ solves }\ 
\min_{z_i}\Big\{\mathbb{E}_{\bxi_i(\omega)}[\tilde{L}_i(z_i,\bxi_i(\omega))]
+\frac{\rho_i}{2}\|z_i-x\|^2\Big\}, \quad \hbox{for all } i \in \{1,\ldots,m\}. \notag
\end{align}
Here, $x$ is a \emph{shared reference} {model} and $z_i(x)$ denotes the personalized {model} of agent $i$.
In the lower-level problem, each agent computes $z_i(x)$ by minimizing its expected local loss while
penalizing deviation from the reference {model} $x$ through the proximal term
$\frac{\rho_i}{2}\|z_i-x\|^2$. This coupling discourages excessive deviation from $x$ while still allowing
agents to adapt to heterogeneity. In the {upper level}, the reference {model} $x$ is selected so as
to minimize the average \emph{post-personalization} expected loss attained by the agents. In particular,
large values of $\rho_i$ yield $z_i(x)\approx x$ (near-consensus), whereas smaller values {encourage} personalization. Under suitable convexity assumptions, \eqref{eqn:prob-1stage_hetero} is an instance of \eqref{eqn:prob-1stage}, where the lower-level VI solution sets capture the optimality conditions
of the local optimization problems in \eqref{eqn:prob-1stage_hetero}.
\vspace{5pt}
}

\Rme{\noindent {\bf {(ii)} Resource allocation under heterogeneous demand.} 
A standard formulation in stochastic programming is $\min_{x \in X} \; c^\top x + \mathbb{E}_{\bxi(\omega)}[Q(x,\bxi(\omega))],$
where $x$ is a first-stage (here-and-now) decision made prior to observing uncertainty, and $Q(x,\bxi(\omega))$ denotes the optimal value of a second-stage recourse problem after the realization of $\bxi(\omega)$ \cite{shapiro2021lectures,birge1997introduction,ghazanfariharandi2025value}. In many applications, such as capacity planning and supply-chain management, uncertain demand motivates the use of such two-stage models \cite{shapiro2021lectures,birge1997introduction,santoso2005stochastic}. To capture this in  heterogeneous {uncertain environments, consider a distributed extension of the form}
\begin{align}
&\hbox{minimize}_x\quad f(x)\ \triangleq \frac{1}{m}\sum_{i=1}^m
\mathbb{E}_{\bxi_i(\omega)}\!\left[\tilde{q}_i(x,z_i(x,\bxi_i(\omega)),\bxi_i(\omega))\right] \label{eqn:prob-2stage_hetero2}\\
& \hbox{subject to} \quad z_i(x,\xi_i(\omega))\ \hbox{ solves }\ 
\begin{aligned}[t]
&\min_{u_i,v_i,w_i \ge 0}
\left\{
\tfrac12 y_i^\top H_i y_i
+
c_i^\top u_i + r_i^\top v_i + p_i^\top w_i
\right\} \notag\\
& \RRme{\hbox{subject to}} \quad 
u_i + v_i + w_i \ge \xi_i(\omega), \quad 
u_i \le B_i x, \quad \hbox{for all } i \in [m], \notag
\end{aligned}
\end{align}
where $y_i := (u_i,v_i,w_i)$ and $H_i$ is a symmetric positive definite matrix. Here, $x$ is a \emph{shared first-stage decision} representing, {for instance,} capacity allocation, procurement commitments, or inventory planning, while $z_i(x,\xi_i(\omega)) := (u_i,v_i,w_i)$ denotes the second-stage decision of agent $i$, computed after observing the realization of demand $\xi_i(\omega)$. In the lower-level problem, $u_i$ represents regular supply using the allocated capacity, $v_i$ denotes emergency procurement, and $w_i$ captures unmet demand or backorders. The cost parameters $c_i$, $r_i$, and $p_i$ correspond to regular, emergency, and shortage costs, respectively, while the matrix $H_i$ models convex quadratic recourse costs. Such quadratic costs capture increasing marginal costs associated with high utilization, emergency procurement, and severe shortages/backorders, which are common features in supply-chain and resource allocation systems~\cite{mehrotra2009implementation,liu2020two}. The constraint $u_i \le B_i x$ couples the recourse decisions to the shared first-stage allocation. This formulation allows for stochastic and heterogeneous demand across agents through the distributions of $\xi_i(\omega)$, a standard modeling feature in stochastic programming and supply-chain applications~\cite{shapiro2021lectures,birge1997introduction,santoso2005stochastic}. Since $H_i$ is positive definite, the lower-level problem is strongly convex over a polyhedral feasible set and therefore admits a unique solution for each realization $\xi_i(\omega)$. The upper-level problem selects $x$ to minimize the average \emph{post-recourse} expected cost across all agents. 
\eqref{eqn:prob-2stage_hetero2} constitutes an instance of \eqref{eqn:prob-2stage}, where the second-stage recourse decisions are given by the solutions of scenario-dependent VIs.
}

\subsection{Challenges \us{and their potential resolution}} \label{sec:2.2}
\us{In this subsection, we comment on the challenges arising in the resolution of the single and two-stage problems of interest.} Recall that our approach  is reliant on
addressing the implicit variants (in
terms of variable $x$) of the
distributed problems
\eqref{eqn:prob-1stage} and
\eqref{eqn:prob-2stage}. To elaborate,
consider the single-stage problem
\eqref{eqn:prob-1stage}. Let us \us{consider} 
the \us{random} local implicit objectives \us{$\tilde{f}_i(\bullet,\xi_i)$, defined as } 
$\tilde{f}_i(x,\xi_i)\triangleq
\tilde{h}_i(x,z_i(x),\xi_i) $ for all
$i$. Then, \eqref{eqn:prob-1stage} can
be  written succinctly as $\min_x \
\frac{1}{m}\sum_{i=1}^m
\mathbb{E}_{\bxi_i}
[\tilde{f}_i(x,\bxi_i)]$. In solving this
problem, \us{we discuss the challenges \mee{that} may arise and the potential resolution that our approach \Rme{provides.}}

 \smallskip

\noindent {\bf (i)} \underline{\it
Nondifferentiability and nonconvexity of
implicit objectives}. The local implicit
objectives \us{$f_i$, defined as}  $f_i(x) \triangleq
\mathbb{E}_{{\bxi_i}}
[\tilde{f}_i(x,{\bxi_i})]$ are possibly
nonsmooth and nonconvex. This could be
the case, even when
$\tilde{h}_i(\bullet,z_i,\xi_i)$ is
smooth and convex for all $i$. For
instance, consider the MPEC given by
$$\min_{x\in \mathbb{R}}\ (x+1-z(x))^2, \mbox{ where } z(x) = \hbox{arg}\min_{z \geq 0} \
0.5|z-x|^2.$$ Then, the implicit function
is given by $f(x) =
(x+1-\max\{0,x\})^2$, which is neither
smooth nor convex.  The minimization of
nonsmooth nonconvex functions is
computationally challenging. Recall
that a point $\bar{x}$ is called
$\epsilon$-stationary \us{with respect to} \Rme{the} problem
$\min_{x \in \mathbb{R}^n} \, h(x)$ if
$\mbox{dist}(0_n,\partial h\RRmath{(\bar{x})}) \leq
\epsilon$ where $\partial h$ denotes the
Clarke generalized gradient of
$h$~\cite{clarke2008nonsmooth}. It is
shown in recent
work~\cite{zhang2020complexity} that
computing an $\epsilon$-stationary point
for the minimization of a suitable class
of nonsmooth functions is impossible in
finite time. To \Rme{address} this \Rme{shortcoming}, we
leverage the notion of {\it approximate
stationarity} studied
in~\cite{zhang2020complexity} as a
weakening of stationarity. A point $\bar
x$ is called $(\delta,\epsilon)$
stationary to the problem $\min_{x \in
\mathbb{R}^n} \ h(x)$ if
$\mbox{dist}(0_n,\partial_\delta h(\bar
x)) \leq \epsilon$\RRme{,} where
$\partial_\delta h\RRmath{(\bar
x)}$ denotes
$\delta$-Clarke generalized gradient of
$h$ \RRme{at} $\RRmath{\bar
x}$~\cite{goldstein1977optimization}.
While even computing an approximate
stationary point remains challenging in
general nonsmooth nonconvex
settings~\cite{kornowski2021oracle},
when $h$ is a locally Lipschitz
continuous function, for any $\eta>0$,
if a vector $x$ is a stationary point of
$\min_{x \in \mathbb{R}^n} \ h^\eta(x)$,
then $\RRmath{0_n}  \in \partial_{\eta}
h(x)$~\cite[Thm. 3.1]{lin2022gradient}.
Here, $h^\eta$ is a suitably defined
randomized smoothed approximation of
$h$~\cite{nesterov2017random,cui2023complexity}.
Utilizing this result, our goal in this
work lies in the development of
distributed methods with complexity
guarantees for computing an
$\epsilon$-stationary point \us{corresponding to} the
minimization of \Rme{the $\eta$-}smoothed implicit
problem\Rme{s} in the SMPECs
\eqref{eqn:prob-1stage} and
\eqref{eqn:prob-2stage}.

 \smallskip

\noindent {\bf (ii)} \underline{\it Absence of
first-order and zeroth-order information
of local implicit objectives}. Often a
closed-form expression for $z_i(x)$ is
unavailable; in particular, in the
single-stage setting
\eqref{eqn:prob-1stage} where the
lower-level VIs are characterized by
expectation-valued mappings with respect
to random variables \Rme{$\bxi_i$} with
possibly unknown probability
distributions. Consequently, agents may
have access to neither subgradients nor
function evaluations of their local
implicit objectives $f_i$. As such,
existing distributed first-order (e.g.,
\cite{pu2021distributed}) and
zeroth-order (e.g.,
\cite{marrinan2026zeroth,tang2020distributed,hajinezhad2019zone})
optimization methods cannot be employed
for addressing the proposed formulations
in this work. To resolve this issue, we
employ an inexact zeroth-order scheme in
that $z_i(x)$ is approximated inexactly
at each iteration with a prescribed
inexactness. However, some computational
challenges may arise, as explained next. 

 \smallskip

\noindent {\bf (iii)} \underline{\it Bias and its
propagation caused by inexact stochastic
schemes}. In resolving (ii), we note
that a naive implementation of inexact
zeroth-order schemes may be problematic.
To elaborate, let us consider the
deterministic single-stage setting. Let
$z_i^{\varepsilon}(x)$ denote a
$\varepsilon$-inexact approximation of
$z_i(x)$, such that
$\|z_i^{\varepsilon}(x) - z_i(x) \|^2
\leq \varepsilon$ for any $x$ and  all
$i$.  Consider the  Gaussian smoothing
zeroth-order scheme
in~\cite{nesterov2017random} being
applied to the inexact local implicit
function $f_i^{\varepsilon}(x)\triangleq
h_i(x,z_i^\varepsilon(x))$. Let the
inexact zeroth-order gradient mapping be
given by $g_i^{\eta,\varepsilon}(x)
\triangleq
\frac{f_i^{\varepsilon}(x+\eta u)
-f_i^{\varepsilon}(x\mee{-\eta u})}{\eta}\, B\,u $,
where $u$ is a Gaussian random variable
associated with a correlation operator
$B^{-1}$~\cite[Sec.
3]{nesterov2017random}. One major issue
arising from this approach is that
$g_i^{\eta,\varepsilon}(x)$ is not
necessarily an unbiased estimator of
$\nabla f_i^{\eta}(x)$ \Rme{({since $f_i^\varepsilon$ is used instead of $f_i$})}, where
$f_i^{\eta}$ denotes the $\eta$-smoothed
local function. This issue brings forth a
challenge in employing existing
zeroth-order schemes in addressing the
distributed SMPECs considered in this
work. In particular, given that $x$ is
being updated iteratively in the
zeroth-order scheme, the gradient bias
may undesirably propagate throughout the
implementation, possibly causing the
divergence of the generated sequence of
iterates by the underlying inexact
zeroth-order scheme. Importantly, this
becomes even more challenging when
considering the presence of
expectation-valued objectives and
mappings in the distributed SMPECs
studied in this work. We aim to resolve
such issues arising from bias and
inexactness through a detailed and
rigorous convergence analysis equipped
with iteration and sample complexity
bounds.

 \smallskip

\noindent {\bf (iv)} \mee{\underline{\it The need for
distributed implementation over a
network}}. In the distributed SMPECs
considered in this work, agents only
have partial information about the
upper-level global objective function
and lower-level VIs.  This lack of {\it
perfect} information would only
exacerbate {the difficulties in
contending with} (i), (ii), and (iii).
To \Rme{address} this \Rme{complication}, we employ a
distributed gradient tracking scheme.
Recently, gradient tracking (GT) methods
have been developed for solving standard
distributed (and stochastic)
optimization problems in
convex~\cite{pu2020push,pu2021distributed,sun2022distributed}
and nonconvex
cases~\cite{xin2021improved,ghiasvand2024communication,ghiasvand2025robust}.
In particular, GT schemes bridge the gap
between centralized and distributed
optimization by being equipped with the
same speed of convergence as their
centralized counterparts~
\cite{xin2018linear,pu2021distributed,liu2024distributed,saadatniaki2020decentralized}.
\uvs{While there}
has been some recent progress in
addressing nonsmooth nonconvex
distributed optimization problems using
zeroth-order GT
methods~\cite{mhanna2023single}, \uvs{
existing GT methods} mainly
address unconstrained or
simple-to-project constrained
optimization problems. It appears that
no GT methods exist that can accommodate
the MPECs over networks\mee{,} even in
deterministic settings.  
 
\begin{remark}\em
\Rme{{As noted, GT methods are} well studied {for addressing distributed (single-level) optimization problems. However, in this work,} the {local} objectives are defined implicitly through lower-level equilibrium solution mappings, {and thus,} agents generally lack direct access to {both their local gradients (even when they exist) and function values of}  their local implicit objectives. Moreover, in the single-stage formulation, the lower-level problems are expectation-valued, requiring stochastic approximation and introducing further inexactness in the upper-level gradient surrogates. We {aim to resolve these challenges through designing randomized zeroth-order inexact GT schemes} for these distributed stochastic MPECs, together with a convergence analysis that explicitly accounts for the resulting {biased} inexact-gradient errors in the tracking recursion.} \mee{$\hfill \Box$}
\end{remark}
\subsection{Nonsmooth analysis and randomized smoothing}\label{sec:2.3}
As mentioned earlier, the implicit
function {$f$}, defined in
\eqref{eqn:prob-1stage} and
\eqref{eqn:prob-2stage}, is possibly
nondifferentiable and nonconvex. Recall
that the notion of {\it subgradient} was
formally introduced
in~\cite{rockafellar1963convex,moreau1963proprietes}
for resolving nonsmoothness in convex
functions. In addressing nonsmoothness
in the absence of convexity, the {\it
Clarke generalized
gradient}~\cite{clarke2008nonsmooth}
serves as a key tool, subsuming several
properties of the subgradient. In the
following, we review some key
definitions and properties.
\begin{definition}[{Clarke generalized
gradient~\cite[Ch.
2]{clarke2008nonsmooth}}]\em 
Let $\mathbb{E}$ be a real Banach space and $f:\mathbb{E} \to \mathbb{R} $ be a given function. The {\it Clarke generalized gradient} of $f$ at $x$ is given as
\begin{align*}
\partial f(x) \triangleq \left\{\zeta \in \RRmath{\mathbb{E}^*}\ |\ f^{\circ}(x,v) \ge \langle \zeta , v\rangle, \RRme{\hbox{ for all }
} v \in \mathbb{E}\right\},
\end{align*}
where $f^{\circ}(x,v)$ denotes the {\it generalized directional derivative} of $f$ at $x$ in the direction $v$, defined as $f^{\circ}(x,v)\triangleq \limsup_{y \rightarrow x,\,  y \in \mathbb{E}, \,t \downarrow 0}  \, \frac{f(y+tv)-f(y)}{t}$\RRme{, and $\RRmath{\mathbb{E}^*}$ is the space of continuous linear functionals defined on $\RRmath{\mathbb{E}}$}.\mee{$\hfill \Box$}
\end{definition}
\begin{remark}\em 
Note that if $f$ is continuously differentiable at $x$, then the Clarke generalized gradient \uvs{reduces} to the standard gradient, i.e. $\partial f(x)=\{\nabla_xf(x)\}$. Given $x,v \in \mathbb{E}$, $f^{\circ}(x,v)$ is not necessarily finite. However, if $f$ is locally Lipschitz near $x$, some key properties hold, including the following: (i) $f^{\circ}(x,v)$ is finite; (ii) $\partial f(x)$ is a nonempty and convex set; (iii) $\RRmath{f^{\circ}(x,v)=\max\{\langle \zeta ,v\rangle \mid \zeta \in \partial f(x) \}}$. \mee{$\hfill \Box$}
\end{remark}
In the following result, we turn our attention to the Euclidean space where several useful properties {hold} and  Rademacher's theorem {can be applied}. 
\begin{proposition}[{\cite[Ch. 2 and 3]{clarke2008nonsmooth}}] \em Let $x \in \mathbb{R}^n$ and let $f:\mathbb{R}^n \to \mathbb{R}$ be  Lipschitz continuous on $\mathbb{R}^n$. Then, the following statements hold.

\noindent (i) $\partial f(x)$ is a nonempty, convex, and compact set and $\|g\| \le L$ for any $g \in \partial f(x)$, where $L>0$ denotes the Lipschitz parameter of $f$.

\noindent (ii) $f$ is differentiable almost everywhere.

\noindent (iii) Let $\Rme{\Omega_f \subseteq \mathbb{R}^n}$ be the measure-zero set of points in $\mathbb{R}^n$ at which $f$ fails to be differentiable. Then, $
\partial f(x) = conv\left\{g\, |\, g=\lim_{k\to \infty}\nabla_xf(x_k), \ \Omega_f \not\ni x_k\to x\right\},$ i.e., $\partial f$ is ``blind to sets of measure zero.''
\end{proposition}
To contend with challenges in the computation of a Clarke-stationary point~\cite{zhang2020complexity}, we will utilize the notion of $\delta$-Clarke generalized gradient defined next. 
\begin{definition}[{Approximate Clarke generalized gradient~\cite{goldstein1977optimization}}]\label{def:apprpx_Clarle}\em 
Let $f:\mathbb{R}^n \to \mathbb{R} $ be a function.  Given $\delta>0$, the {\it $\delta$-Clarke generalized gradient} of $f$ at $x$ is given as
\begin{align*}
\partial_\delta f(x) \triangleq  conv\left( \bigcup_{y \in x+\delta \mathbb{B} } \partial f(y) \right).
\end{align*}
Noting that $\partial f(x)  \subset \partial_\delta   f(x)$,  the set $\partial_\delta f$ is an expansion of the Clarke subgradient  $\partial f$. \mee{$\hfill \Box$}
\end{definition}

To contend with the nondifferentiability of the implicit function,  we employ a randomized smoothing technique that {can be traced} back to a class of {\it averaged functions}~\cite{steklov1907expressions} in the early 1900s.  This technique has been utilized in recent decades, in both convex~\cite{bertsekas1972stochastic,lakshmanan2008decentralized,yousefian2012stochastic} and nonconvex optimization~\cite{nesterov2017random,lin2022gradient,cui2023complexity,qiu2023zeroth,marrinan2026zeroth}.\\ 

\noindent {\bf Randomized smoothing scheme.} Given a continuous function $h:\mathbb{R}^n \to \mathbb{R}$ and a smoothing parameter $\eta>0$,  we define the smoothed function $h^{\eta}(x)\triangleq \mathbb{E}_{\mee{\bf u}\in \mathbb{B}}\left[h(x+\eta \mee{\bf u})\right]$,  where $u$ is a random vector in the unit ball $\mathbb{B}$. In the following, we review some of the key properties of this smoothing. \fyy{The proof of the following lemma is presented in the Appendix.}
\begin{lemma}\label{lem:smoothing_props}  \em
Let $h:\mathbb{R}^n \to \mathbb{R}$ be given\Rme{. Suppose that $h$ is Lipschitz continuous with parameter $L_0>0$ and define} $h^{\eta}(x)\triangleq \mathbb{E}_{{\bf u}\in \mathbb{B}}\left[h(x+\eta {\bf u})\right]$ for some $\eta>0$. Then, the following results hold.   

\noindent (i) The smoothed function $h^{\eta}$ is continuously differentiable and for all $x \in \mathbb{R}^n$ we have 
\begin{align}\label{eqn:ZO}
\nabla_x h^{\eta}(x) =\left(\tfrac{n}{\eta}\right)\mathbb{E}_{{\bf v}\in \mathbb{S}}\left[\left(h(x+\eta {\bf v})\right){\bf v}\right]=\left(\tfrac{n}{2\eta}\right)\mathbb{E}_{{\bf v}\in \mathbb{S}}\left[\left(h(x+\eta {\bf v})-h(x-\eta {\bf v})\right){\bf v}\right].
\end{align}
\noindent (ii) $|h^{\eta}(x)-h(x)|\le L_0\eta$.

\noindent (iii) \Rme{T}he mapping $\nabla h^{\eta}$ is Lipschitz continuous, i.e.,  $\|\nabla h^{\eta}(x)-\nabla h^{\eta}(y)\|\le \frac{L_0\mj{\sqrt{n}}}{\eta}\|x-y\|$ for any $x , y \in \mathbb{R}^n$.

\end{lemma}

\begin{remark}\em  The equation \eqref{eqn:ZO} provides a randomized zeroth-order gradient for the smoothed function $h^\eta$.  Indeed,  it implies that $\left(\tfrac{n}{\mj{2}\eta}\right)\left(h(x+\eta v)-h(x {-\eta v})\right) v$ is an unbiased stochastic zeroth-order gradient of $\nabla_x h^{\eta}$ at $x$\Rme{, where $v$ denotes a realization of the random vector $\mathbf{v}$}. This will be utilized in the design of zeroth-order gradient tracking methods in the subsequent sections. {$\hfill \Box$} 
\end{remark}
Consider \uvs{minimizing} a Lipschitz continuous function $h$.  To alleviate the challenges arising from nonsmoothness,  one may consider \uvs{minimizing} the approximate smoothed function $h^\eta$.  In the following, we provide a result that builds a connection between a stationary point of the smoothed problem and an approximate Clarke stationary point of the original nonsmooth optimization problem. 
\begin{proposition}\label{Prop:2eta}\em
Let $h:\mathbb{R}^n \to \mathbb{R}$ be given and $h^{\eta}(x)\triangleq \mathbb{E}_{{\bf u}\in \mathbb{B}}\left[h(x+\eta {\bf u})\right]$ for some arbitrary $\eta>0$.  Then, the following results hold. 

\noindent (i) For any $x \in \mathbb{R}^n$,  $\nabla h^{\eta}(x)\in \partial_{\eta} h(x)$.

\noindent (ii) For  any $x \in \mathbb{R}^n$,  $\mbox{dist}(0,\partial_{\eta} h(x)) \leq \|\nabla h^\eta(x)\|$.

\end{proposition}
\begin{proof}
(i)  See~\cite[Lemma 4]{kornowski2024algorithm}.

\noindent (ii) 
By the definition of the distance function, we have $\mbox{dist}(0, \partial_\eta h(x)) = \inf_{y \in \partial_\eta h(x)} \|0-y\|= \inf_{y \in \partial_\eta h(x)} \|y\|.$ 
From part (i), we know that \(\nabla h^\eta(x) \in \partial_\eta h(x)\). Therefore, it follows that $\inf_{y \in \partial_\eta h(x)} \|y\| \le \|\nabla h^\eta(x)\|$. The preceding two relations imply the result. 
\end{proof}
{We use} the following result to construct a bound on the second moment of zeroth-order gradients. 
\begin{lemma}[L\'{e}vy concentration on $\mathbb{S}$ {\cite[Proposition 3.11 and Example 3.12]{wainwright2019high}}]\em \label{Lemma:Levy concentration} Assume that ${c}:\mathbb{R}^n\to \mathbb{R}$ is $L$-Lipschitz \RRme{continuous} and let ${\bf v}$ be a random vector uniformly distributed on the unit sphere $\mathbb{S}$. Then, we have
\Rme{$\mathbb{P} [\lvert{c}({\bf v})-\mathbb{E}[{c}({\bf v})]\rvert\ge \alpha]\le 2\sqrt{2\pi}\RRme{\exp{\left({-\frac{n\alpha^2 }{8L^2}}\right)}}.$}
\end{lemma}

\Rme{\begin{remark}\em In addressing the single-stage problem, we derive the overall complexity $\mathcal{O}\!\left(n^{7/2}\epsilon^{-4}\right)$, improving the dependence on $n$ compared to the existing results for the centralized SMPECs~\cite{cui2023complexity}. This improvement is \emph{not} due to decentralization or graph topology. In our terminology, the centralized setting is the special case $m=1$, and for $m=1$ our formulation reduces to the centralized SMPEC setting in~\cite{cui2023complexity}. The better $n$-dependence comes from using a central-difference smoothing estimator (instead of the forward-difference estimator in~\cite{cui2023complexity}), and we show that the gradient of the smoothed function is Lipschitz continuous with constant $\frac{L_0\sqrt{n}}{\eta}$ (Lemma~\ref{lem:smoothing_props}(iii)), which improves upon the corresponding bound $\frac{L_0 n}{\eta}$ in~\cite[Lemma~1(iv)]{cui2023complexity}. Hence, the improved $n$-dependence persists even for $m=1$ and should be attributed to the estimator choice rather than decentralization.
\end{remark}}
 
\section{Distributed single-stage stochastic MPECs}\label{sec:1s}
{In this section, we consider the
single-stage SMPEC \Rme{defined}
in~\eqref{eqn:prob-1stage}, where $m$
agents aim to cooperatively minimize
the global implicit objective
function. We assume that agents can
communicate their updated iterates
with their neighboring agents over an
undirected and connected network. We \us{begin \Rme{with} a formal statement of} the network assumptions.
\begin{assumption}\label{assum:mixxx} \em
(i) We assume that the agents
communicate {over} an undirected and
connected graph $\mathcal{G}
=(\mathcal{V},\mathcal{E})$, where
$\mathcal{V}=\{1,\ldots,m\}$ {denotes} the
set of \us{nodes} and $\mathcal{E}
\subseteq \mathcal{V}\times
\mathcal{V}$ is the set of connecting
edges. (ii) We let $\mathbf{W} \in
\mathbb{R}^{m \times m}$ denote the
mixing matrix of agents and assume
that $\mathbf{W}$ is doubly stochastic
and $w_{i,i}>0$ for some $i \in
\mathcal{V}$. 
\end{assumption}
Under Assumption~\ref{assum:mixxx}, we
have $\lambda_{\mathbf{W}}\triangleq
\|\mathbf{W}-\tfrac{1}{m}\mathbf{1}\mathbf{1}^{\top}\|_2<1$
(cf.~\cite{qu2017harnessing,pu2021distributed}).
This is formally presented in the
following result that will be utilized
throughout the analysis. Note that
$\us{(1-\lambda_{\mathbf{W}})} \in (0,1]$ can
be viewed as a measure of the
network's connectivity.  
\begin{lemma}[{\cite[Lemma 1]{pu2021distributed}}]\label{lem:lambda_w}\em
Let Assumption~\ref{assum:mixxx} hold and define $\RRmath{\lambda_{\mathbf{W}}\triangleq \|\mathbf{W}-\tfrac{1}{m}\mathbf{1}\mathbf{1}^{\top}\|_2}$. Then, $0 \leq \lambda_{\mathbf{W}}<1$ and for all $\mathbf{x} \in \mathbb{R}^{m\times n}$, we have $\|\mathbf{W}\mathbf{x} - \mathbf{1}\bar{x}\| \leq \lambda_{\mathbf{W}}\|\mathbf{x}-\mathbf{1}\bar{x}\|$, where $\bar{x} \triangleq \tfrac{1}{m}\mathbf{1}^\top \mathbf{x}$.
\end{lemma}
\begin{assumption}\label{assump:opt_f_bounded_below}\em
Consider the implicit function \us{$f$} given by \eqref{eqn:notations}. Suppose $\min_{x \in \mathbb{R}^n}\,f(x) > -\infty$.
\end{assumption}
\begin{remark}\em
Notably, under the above assumption, $\min_{x \in \mathbb{R}^n}\,f^{\eta}(x) > -\infty$, \us{since $f(x)-L_0\eta \le f^{\eta}(x)$ for all $x \in \mathbb{R}^n$, a consequence of }  Lemma~\ref{lem:smoothing_props}. $\hfill \Box$
\end{remark}
\begin{assumption}\label{assum:main1}\em
Consider~\eqref{eqn:prob-1stage}. Let
the following conditions hold. (i) For
any agent $i \in [m]$, $\tilde
h_i(x,\bullet,\xi_i)$ is $\tilde
L_0(\xi_i)$-\us{Lipschitz} for
any $\xi_i$, \us{where} $\tilde L_0
\triangleq  \max_{i\in
[m]}\sqrt{\mathbb{E}[\tilde
L_0^2(\bxi_i)]}$ is finite. (ii)
$\tilde
h_i(\bullet,z_i(\bullet),{\xi}_i)$ is $
L_0({\xi}_i)$-\us{Lipschitz} for
any $\xi_i$, and $ L_0 \triangleq
\max_{i\in [m]}\sqrt{\mathbb{E}[
L_0^2(\bxi_i)]}$ is finite.
\end{assumption}
\begin{remark}\em
Notably, we do not assume that the implicit objective is
differentiable. The Lipschitz continuity of the implicit
function in Assumption~\ref{assum:main1}, studied in
\cite{patriksson1999stochastic}, holds under some
conditions.~$\hfill \Box$
\end{remark}
\begin{remark}\label{rem:smoothness_of_f_i}\em
\us{By}
\Rme{Assumption~\ref{assum:main1}(ii)} and
the definition of the local implicit
functions $f_i$, $f_i$ is
$\mathbb{E}[L_0(\bxi_i)]$-Lipschitz.
Notably, invoking Jensen's
inequality, we have
$\mathbb{E}[L_0(\bxi_i)] \leq
\sqrt{\mathbb{E}[L_0^2(\bxi_i)]} \leq
\max_{i\in [m]}\sqrt{\mathbb{E}[
L_0^2(\bxi_i)]} = L_0$. Thus, $f_i$ is
$L_0$-Lipschitz for all $i\in [m]$.
Invoking
\RRme{Lemma~\ref{lem:smoothing_props}(iii)},
$\nabla f^{\eta}_i$ is $\frac{L_0
\sqrt{n}}{\eta}$-Lipschitz for all
$i\in [m]$ \us{and} the smoothed
global implicit function $f^\eta$ is
$\frac{L_0 \sqrt{n}}{\eta}$-smooth. $\hfill \Box$
\end{remark}
\begin{assumption}\label{assum:main2}\em
Consider problem
(\ref{eqn:prob-1stage}). Let
$F_i(x,\bullet) \triangleq
\mathbb{E}[\tilde{F}_i(x,\bullet,\bxi_i)]$
be a $\mu_F$-strongly monotone and
$L_F$-Lipschitz continuous mapping
uniformly in $x$. For each $i \in [m]$
and any $x \in \mathbb{R}^n$, assume
that the set
$\mathcal{Z}_i(x)\subseteq
\mathbb{R}^p$ is nonempty, closed, and
convex. Further,
$\sup_{x\in\mathbb{R}^n}\sup_{z_1,z_2
\in \mathcal{Z}_i(x)}\|z_1-z_2\|^2
\leq D_i$ for some $D_i>0$, for all $i
\in [m]$. 
\end{assumption}
To address \eqref{eqn:prob-1stage}, we
propose DiZS-GT$^{\text{1s}}$, outlined in
Algorithm~\ref{alg:DZGT}, which is a distributed implicit
zeroth-order gradient tracking method for addressing
problem~\eqref{eqn:prob-1stage}. Each agent $i$ is associated
with $x_{i,k}$, which \us{represents} a local copy of the
decision variable $x$ \us{at epoch $k$} while $y_{i,k}$ tracks
the gradient of the smoothed implicit objective function.
We use color-coding in Algorithm~\ref{alg:DZGT} to capture
both the inexact and exact setting\mee{s}. In the exact setting, at
iteration $k$, agent $i$ evaluates the optimal solution to its
local lower-level problem twice, once at ${x}_{i,{k}}+\eta
v_{i,{k}}$ and once at ${x}_{i,{k}}-\eta v_{i,{k}}$, denoted
by $z_{i}({x}_{i,{k}}+\eta
v_{i,{k}})$ and
$z_{i}({x}_{i,{k}}-\eta v_{i,{k}})$
for the exact setting, while 
$z_{i,\varepsilon_{k}}({x}_{i,{k}}+\eta
v_{i,{k}})$ and
$z_{i,\varepsilon_{k}}({x}_{i,{k}}-\eta
v_{i,{k}})$ are used \us{in} the inexact setting. Then\mee{,} employing the randomized
smoothing scheme and utilizing
\Rme{Lemma~\ref{lem:smoothing_props}(i)},
the gradient of the $\eta$-smoothed
implicit local objective function
$\tilde{f}_i^\eta(\bullet,\xi_{i,k})$,
denoted by $g_{i,{k}}^{\eta}$ is
approximated by \begin{align}
 & g_{i,{k}}^{\eta}\triangleq \left(\tfrac{n}{\mj{2}\eta}\right)(\tilde h_i({x}_{i,{k}}+\eta v_{i,{k}},z_{i}({x}_{i,{k}}+\eta v_{i,{k}}),\xi_{i,{k}})-\tilde h_i({x}_{i,{k}}\mj{-\eta v_{i,{k}}},z_{i}({x}_{i,{k}}\mj{-\eta v_{i,{k}}}),\xi_{i,{k}}))\, v_{i,{k}}.\label{eqn:g_eta}
\end{align} In the inexact setting, where $z_i(\bullet)$ is not available, agent $i$ employs a standard stochastic approximation scheme, outlined in Algorithm~\ref{alg:lowerlevel-1stage}, to approximate this term. To this end, we let $z_{i,\varepsilon_{k}}({x}_{i,{k}}+\eta v_{i,{k}})$ and $z_{i,\varepsilon_{k}}({x}_{i,{k}}-\eta v_{i,{k}})$ denote the $\varepsilon_{k}$-inexact approximation of $z_{i }({x}_{i,{k}}+\eta v_{i,{k}})$ and $z_{i }({x}_{i,{k}}-\eta v_{i,{k}})$, respectively\Rme{, obtained from Algorithm~\ref{alg:lowerlevel-1stage}}, such that $\mathbb{E}[\|z_{i,\varepsilon_k}(x) -z_i(x)\|^2 \mid x]\leq \varepsilon_k$ \RRme{holds} for any random variable $x \in \mathbb{R}^n$ almost surely and $\varepsilon_k$ is a deterministic scalar (independent \RRme{of} $x$). \Rme{This} leads to an inexact zeroth-order stochastic local gradient estimator given by   
\begin{align}
&  g_{i,{k}}^{\eta,\varepsilon_{k}}\triangleq \left(\tfrac{n}{\mj{2}\eta}\right)(\tilde h_i({x}_{i,{k}}+\eta v_{i,{k}},z_{i,\varepsilon_{k}}({x}_{i,{k}}+\eta v_{i,{k}}),\xi_{i,{k}})-\tilde h_i({x}_{i,{k}}\mj{-\eta v_{i,{k}}},z_{i,\varepsilon_{k}}({x}_{i,{k}}\mj{-\eta v_{i,{k}}}),\xi_{i,{k}}))\, v_{i,{k}}.\label{eqn:g_eta_eps}
\end{align}
In the inexact setting, a key research question lies in deriving a prescribed termination criterion for Algorithm~\ref{alg:lowerlevel-1stage} such that the convergence of the iterates generated by DiZS-GT$^{\text{1s}}$ can be guaranteed. 
\begin{assumption}\label{assum:alg2}\em
Consider
Algorithm~\ref{alg:lowerlevel-1stage}.
For any fixed \Rme{$i,k,$} and $\ell \in
\{1,2\}$, let the samples
$\xi_{i,t}^{k,\ell}$, for $t\geq 0$,
be iid such that for a  given $x,z$,
we have (i)
$\mathbb{E}[\tilde{F}_i(x,z,\bxi_{i,t}^{k,\ell})\mid
\{x,z\}] = {F}_i(x,z)$, \Rme{where ${F}_i(x,z)$ is defined in Assumption~\ref{assum:main2},} and (ii)
$\mathbb{E}[\|\tilde{F}_i(x,z,\bxi_{i,t}^{k,\ell})-{F}_i(x,z)\|^2\mid
\{x,z\}] \leq \nu_F^2$, for some
$\nu_F>0$.
\end{assumption}
\begin{remark}\em
Note that $\xi_{i,k}$ and $\xi_{i,t}^{k,\ell}$ refer to samples used in Algorithm~\ref{alg:DZGT} and Algorithm~\ref{alg:lowerlevel-1stage}, respectively. Both represent instances of the random variable $\bxi_i$, drawn from the same distribution. Despite the overload of notation, the distinction will be clear from context.$\hfill \Box$
\end{remark}
}
\begin{algorithm}
\caption{\fy{DiZS-GT$^{\text{1s}}$} (by agent $i$)} \label{alg:DZGT}
{\begin{algorithmic}[1]
\State {\bf input}  weights $w_{ij}$ for all $j\in[m]$, stepsize $\gamma$ and smoothing parameter $\eta$, local random initial vector $x_{i,0} \in \mathbb{R}^n$, $y_{i,0}:=0_{n}$, and $g_{i,-1}^{\eta,\varepsilon_{-1}}=g_{i,-1}^{\eta}:=0_n$. (\colorbox{blue!10}{Inexact} and \colorbox{yellow!22}{Exact} schemes are highlighted)
%

 \FOR {$k = 0,1,2, \ldots$} 

\State Generate random samples $\xi_{i,k}$ and $v_{i,k} \in   \mathbb{S}$

 \State \colorbox{blue!10}{Call Alg.~\ref{alg:lowerlevel-1stage} twice to \us{generate} $z_{i,\varepsilon_{k}}({x}_{i,{k}}-\eta v_{i,{k}})$ and $z_{i,\varepsilon_{k}}({x}_{i,{k}}+\eta v_{i,{k}})$}  
 
 \hspace{-0.19 in} \colorbox{yellow!22}{Evaluate $z_i({x}_{i,{k}}-\eta v_{i,{k}})$ and $z_i({x}_{i,{k}}+\eta v_{i,{k}})$}
 
\State \colorbox{blue!10}{$g_{i,{k}}^{\eta,\varepsilon_{k}}:=  \tfrac{n}{\mj{2}\eta} (\tilde h_i({x}_{i,{k}}+\eta v_{i,{k}},z_{i,\varepsilon_{k}}({x}_{i,{k}}+\eta v_{i,{k}}),\xi_{i,{k}})-\tilde h_i({x}_{i,{k}}-\eta v_{i,{k}},z_{i,\varepsilon_{k}}({x}_{i,{k}}-\eta v_{i,{k}}),\xi_{i,{k}}))\, v_{i,{k}}$}

\hspace{-0.19 in} \colorbox{yellow!22}{$g_{i,{k}}^{\eta}:=  \tfrac{n}{\mj{2}\eta} (\tilde h_i({x}_{i,{k}}+\eta v_{i,{k}},z_{i}({x}_{i,{k}}+\eta v_{i,{k}}),\xi_{i,{k}})-\tilde h_i({x}_{i,{k}}-\eta v_{i,{k}},z_{i}({x}_{i,{k}}-\eta v_{i,{k}}),\xi_{i,{k}}))\, v_{i,{k}}$}

\State  \colorbox{blue!10}{$y_{i,k+1}:=\sum_{j=1}^mw_{ij}\left(\RRmath{y_{j,k}}+g_{j,k}^{\eta,\varepsilon_{k}}-g_{j,k-1}^{\eta,\varepsilon_{k-1}}\right)$}  \colorbox{yellow!22}{$y_{i,k+1}:=\sum_{j=1}^mw_{ij}\left(\RRmath{y_{j,k}}+g_{j,k}^{\eta}-g_{j,k-1}^{\eta}\right)$}
 
\State $
{x}_{i,k+1}:=\sum_{j=1}^mw_{ij}\left(x_{j,k}-\gamma y_{j,k+1} \right)$
\ENDFOR
\end{algorithmic}}
\end{algorithm}

\noindent \fy{\textbf{History of the method.} In the inexact setting, we define the history of Algorithm~\ref{alg:DZGT} as $\mathcal{F}_k \triangleq \cup_{i=1}^m \mathcal{F}_{i,k} $ for $k\geq 0$, where 
$$\mathcal{F}_{i,k}\, \triangleq \,  \mathcal{F}_{i,k-1}\cup \{\xi_{i,k-1},v_{i,k-1}\} \cup \left(\cup_{\ell=1}^2 \left(\left(\cup_{t=0}^{t_{k-1}-1}\{\xi_{i,t}^{k-1,\ell}\}\right)\cup \{z_{i,0}^{k-1,\ell}\} \right)\right) $$ for any $k\geq 1$, and $\mathcal{F}_{i,0}\triangleq  \{x_{i,0}\}$. Further, we let $\tilde{\mathcal{F}}_{i,t_k}^{k,1}$ and $\tilde{\mathcal{F}}_{i,t_k}^{k,2}$ denote the history associated with Algorithm~\ref{alg:lowerlevel-1stage} to evaluate $z_{i,\varepsilon_{k}}({x}_{i,{k}}-\eta v_{i,{k}})$ and $z_{i,\varepsilon_{k}}({x}_{i,{k}}+\eta v_{i,{k}})$, respectively. We define $\tilde{\mathcal{F}}_{i,j}^{k,\ell}\triangleq \{\xi_{i,j-1}^{k,\ell}\} \cup\tilde{\mathcal{F}}_{i,j-1}^{k,\ell} $ for any $k \geq 1$ for all $j=1,\ldots,t_k$ where $\ell \in \{1,2\}$, and $\tilde{\mathcal{F}}_{i,0}^{k,\ell}\triangleq  \{z_{i,0}^{k,\ell}\}\cup \mathcal{F}_{k} $ for any $k \geq 0$. Further, we define $\tilde{\mathcal{F}}_{j}^{k,\ell}\triangleq \cup_{i=1}^m\tilde{\mathcal{F}}_{i,j}^{k,\ell}$ for any $k\geq 0$,  $j=0,\ldots,t_k$, and $\ell \in \{1,2\}$. In the exact setting, we define the history of the method as follows. $\mathcal{F}_k \triangleq \cup_{i=1}^m \mathcal{F}_{i,k} $ for $k\geq 0$, where $\mathcal{F}_{i,k}\triangleq\left( \cup_{j=0}^{k-1}\{\xi_{i,j},v_{i,j}\} \right)\cup\{x_{i,0}\}$ for any $k\geq 1$, and $\mathcal{F}_{i,0}\triangleq  \{x_{i,0}\}$.
\begin{remark}\em
From the above definitions, in both inexact and exact settings we have that for any $i \in [m]$, $x_{i,k}$ is $\RRmath{\mathcal{F}_k}$-measurable for all $k \geq 0$, {\us{where \fyy{$x_{i,k}$} is } used in Algorithm~\ref{alg:DZGT}}. Further, in the inexact setting, for any $i \in [m]$ and $k\geq 0$, $z_{i,t}^{k,\ell}$ is $\RRmath{\tilde{\mathcal{F}}_{t}^{k,\ell}}$-measurable for all $t=0,\ldots,t_k$, {\us{where \fyy{$\RRmath{z_{i,t}^{k,\ell}}$} is} employed in Algorithm~\ref{alg:lowerlevel-1stage}}. $\hfill \Box$
\end{remark}}
\Rme{\begin{remark}\em
Algorithm~\ref{alg:lowerlevel-1stage} is a stochastic approximation scheme for {approximating a} solution of the lower-level stochastic VI problem associated with agent $i$~\cite{jiang2008stochastic}. In particular, at each iteration of Algorithm~\ref{alg:DZGT}, agent $i$ calls Algorithm~\ref{alg:lowerlevel-1stage} to compute the $\varepsilon_k$-inexact approximations of $z_i(x_{i,k}+\eta v_{i,k})$ and $z_i(x_{i,k}-\eta v_{i,k})$, denoted by $z_{i,\varepsilon_k}(x_{i,k}+\eta v_{i,k})$ and $z_{i,\varepsilon_k}(x_{i,k}-\eta v_{i,k})$, respectively. \mee{$\hfill \Box$}
\end{remark}}
\begin{algorithm}
\caption{{Stochastic approximation method (by agent $i$)}}
\label{alg:lowerlevel-1stage}
{\begin{algorithmic}[1]
\State {\bf input} upper-level iteration index $k$, input vector $\hat{x}_{i,k}^{\ell}$, set $\ell:=1$ if  $\hat{x}_{i,k}^{\ell}=x_{i,k}-\eta v_{i,k}$  and $\ell:=2$ if  $\hat{x}_{i,k}^{\ell}=x_{i,k}+\eta v_{i,k}$, a random $z_{i,0}^{k,\ell} \in \mathcal{Z}_i(\hat x_{i,k}^{\ell})$, scalars $\hat{\gamma} > \frac{1}{\mu_F} $ and $\hat{\Gamma} >\frac{\hat{\gamma}L_F^2}{\mu_F}$, and $t_k=\left\lceil \left(n^{1/2}(k+\hat\Gamma)^a\right)/\left(\eta^{2/3}\right)\right\rceil$ for some $a>0.5$
 \FOR {$t = 0,1, \ldots, t_k-1 $} 
\State  $\hat{\gamma}_{t}:=\frac{\hat\gamma}{t+\hat\Gamma}$ and $z_{i,t+1}^{k,\ell}:= \Pi_{\mathcal{Z}_i(\hat x_{i,k}^{\ell})}\left[z_{i,t}^{k,\ell}-\hat{\gamma}_t\tilde{F}_i(\hat{x}_{i,k}^{\ell},z_{i,t}^{k,\ell},\xi_{i,t}^{k,\ell})\right]$
\ENDFOR
\State \Rme{ {\bf output} $z_{i,t_k}^{k,\ell}$ for $\ell \in \{1,2\}$}
\end{algorithmic}}
\end{algorithm}
{\begin{assumption}\em\label{assum:samples} For each agent $i \in [m]$, random variables ${\bxi}_{i,k}$ are iid, random variables ${\bf v}_{i,k} \in   \mathbb{S}$ are iid, and  $\bxi_{i,k}$ and ${\bf v}_{i,k}$ are independent \RRme{of} each other, for all $k\geq 0$. \RRme{Moreover, the random pairs
$(\xi_{i,k},v_{i,k})$ are mutually independent across agents and iterations.} 
\end{assumption}
\Rme{\begin{remark}[GT error decomposition]\em
Since the gradient-tracking update employs an inexact zeroth-order estimator, we decompose the estimation error as
\begin{equation}\label{eq:decomp_ref_response}
{g_{i,{k}}^{\eta,\varepsilon_{k}} 
=
\underbrace{g_{i,{k}}^{\eta,\varepsilon_{k}}-g_{i,{k}}^{\eta}}_{\mathrm{{inexactness\ } bias}}
\!+\!
\underbrace{g_{i,{k}}^{\eta}-\nabla f_i^\eta(x_{i,k})}_{\mathrm{Monte\text{-}Carlo\ sampling\ error}}
\!+\!
\underbrace{\nabla f_i^\eta(x_{i,k})-\tilde{\nabla}  f_i(x_{i,k})}_{\mathrm{smoothing\ error}}
\ +\ \tilde{\nabla}  f_i(x_{i,k})} ,\notag
\end{equation}
where {$\tilde{\nabla}  f_i(x)\in \partial  f_i(x)$ denotes a Clarke generalized gradient of $f_i$ at $x$}. In the analysis, {all these components are addressed} when establishing the tracking and stationarity guarantees. \mee{$\hfill \Box$}
\end{remark}}
\us{Next}, we provide a compact
representation of the method. For
\us{convenience}, let us define
$g_{i,-1}^{\eta,\varepsilon_{-1}}=g_{i,-1}^{\eta}=\nabla
f_i^\eta(x_{i,-1})=\delta_{i,-1}^\eta={e}_{i,-1}^{\eta,\varepsilon_{-1}}=y_{i,0}=0_n$
for all $i \in [m]$. Consider the
error terms
$\boldsymbol{\delta}_{k}^\eta\triangleq
[\delta_{1,k}^\eta,\ldots,\delta_{m,k}^\eta]^\top$
and
$\mathbf{e}_{k}^{\eta,\varepsilon_k}\triangleq[{e}_{1,k}^{\eta,\varepsilon_k},\ldots,{e}_{m,k}^{\eta,\varepsilon_k}]^\top$,
where for $i \in [m]$ and $k\ge -1$, \Rme{we define the zeroth-order approximation error
$\delta_{i,k}^\eta\triangleq
g_{i,k}^\eta-\nabla f_i^\eta(x_{i,k})$
and the bias
${e}_{i,k}^{\eta,\varepsilon_k}\triangleq{g}_{i,k}^{\eta,\varepsilon_k}-{g}_{i,k}^\eta.$}
We define the averaged terms \us{as} $\bar
\delta_{k}^\eta=\tfrac{1}{m}\mathbf{1}^\top
\boldsymbol{\delta}_{k}^\eta$ and
$\bar{e}_{k}^{\eta,\varepsilon_k}=\tfrac{1}{m}\mathbf{1}^\top\mathbf{e}_{k}^{\eta,\varepsilon_k}.$
Then for $k \geq 0$,
{DiZS-GT$^{\text{1s}}$} can be
\us{represented by an update rule, compactly represented as }  \begin{align}
\mathbf{y}_{k+1} & =\mathbf{W}(\mathbf{y}_{k}+ \nabla \mathbf{f}^\eta(\mathbf{x}_k) - \nabla \mathbf{f}^\eta(\mathbf{x}_{k-1})+\boldsymbol{\delta}_{k}^\eta-\boldsymbol{\delta}_{k-1}^\eta+\mathbf{e}_{k}^{\eta,\varepsilon_k}-\mathbf{e}_{k-1}^{\eta,\varepsilon_{k-1}})\tag{R1}\label{eqn:R1}\\
\mbox{ and } 
\mathbf{x}_{k+1} & =\mathbf{W}(\mathbf{x}_{k}-\RRmath{\gamma}\mathbf{y}_{k+1})\tag{R2}\label{eqn:R2}.
\end{align}
Next, we derive some properties of the first and second moments of the bias and the inexact error.  
\begin{lemma}\em\label{lemma:g_ik_eta_props}
{Let} \RRme{Assumptions~\ref{assum:main1} and \ref{assum:samples}} hold. Then, the following statements hold
for any $i\in [m]$ and all $k\geq 0$ almost surely. (i)
$\mathbb{E}\left[\delta_{i,k}^\eta\mid \mathcal{F}_k\right] =
0.$    (ii)  $ \mathbb{E}\left[\|\delta_{i,k}^\eta\|^2\mid
\mathcal{F}_k\right] \leq
\mathbb{E}\left[\|g_{i,k}^\eta\|^2\mid \mathcal{F}_k\right]
\leq \mj{16\sqrt{2\pi}L_0^2n}.$  (iii)
$\mathbb{E}\left[(\delta_{i,k}^\eta)^\top(\delta_{j,k}^\eta)\mid
\mathcal{F}_k\right]=0$, for all $i,j\in \mathcal{V}$, such
that $i\neq j$. (iv) Let $\mathbb{E}[\|z_{i,\varepsilon_k}(x)
-z_i(x)\|^2 \mid x]\leq \varepsilon_k$ hold for \us{a realization $x$} of any random
variable $\us{\bf x} \in \mathbb{R}^n$ such that $\varepsilon_k$ is a
deterministic scalar (independent \RRme{of} $\us{\bf x}$). Then,
$\mathbb{E}[\|e_{i,k}^{\eta,\varepsilon_k}\|^2
|\mathcal{F}_k]\le \mj{\left(\tfrac{\tilde
L_0^2n^2\varepsilon_k}{\eta^2}\right)}.$
\end{lemma} }
\begin{proof}
\noindent (i) \Rme{The proof follows directly from~\cite[Lemma E.1]{lin2022gradient}.}

\noindent (ii) We have
\begin{align*}
\mathbb{E}\left[\|\delta_{i,k}^\eta\|^2\mid \mathcal{F}_k\right]&=\mathbb{E}\left[\|g_{i,k}^\eta- \nabla f_i^\eta(\RRmath{x_{i,k}})\|^2\mid \mathcal{F}_k\right]  = \mathbb{E}\left[\|g_{i,k}^\eta\|^2 +\|\nabla f_i^\eta(\RRmath{x_{i,k}})\|^2 -2 {g_{i,k}^\eta}^{\top}\nabla f_i^\eta(\RRmath{x_{i,k}})\mid \mathcal{F}_k\right] \\
& \overset{\scriptsize\mbox{(i)}}{=} \mathbb{E}\left[\|g_{i,k}^\eta\|^2\mid \mathcal{F}_k\right]-\|\nabla f_i^\eta(\RRmath{x_{i,k}})\|^2\leq \mathbb{E}\left[\|g_{i,k}^\eta\|^2\mid \mathcal{F}_k\right].
\end{align*}
It suffices to show that $\mathbb{E}\left[\|g_{i,k}^\eta\|^2\mid \mathcal{F}_k\right] \leq {16\sqrt{2\pi}L_0^2n}$. \Rme{This follows from~\cite[Lemma E.1]{lin2022gradient}.} 

 \noindent (iii) By recalling Assumption~\ref{assum:samples}, this result follows from the independence of the random variable pair $(\bxi_{i,k},{\bf v}_{i,k})$ from the pair $(\bxi_{j,k},{\bf v}_{j,k})$ for $i\neq j$. 
 
\noindent (iv) Using the definition of $e_{i,k}^{\eta,\varepsilon_k} $,   equations~\eqref{eqn:g_eta_eps} and \eqref{eqn:g_eta}, and the triangle inequality, we have
\begin{align*}
\|e_{i,k}^{\eta,\varepsilon_k}\| &= \|g_{i,k}^{\eta,\varepsilon_k}-g_{i,k}^{\eta}\| \\
  & \leq (\tfrac{n}{\mj{2}\eta})|\tilde h_i({x}_{i,k}+\eta v_{i,k},z_{{i,\varepsilon}_k}({x}_{i,k}+\eta v_{i,k}),\xi_{i,k})-\tilde h_i({x}_{i,k}+\eta v_{i,k},z_i({x}_{i,k}+\eta v_{i,k}),\xi_{i,k})| \|v_{i,k}\|\\
&+(\tfrac{n}{\mj{2}\eta})|\tilde h_i(x_{i,k}\mj{-\eta v_{i,{k}}},z_{i,\varepsilon_k}({x}_{i,k}\mj{-\eta v_{i,{k}}}),\xi_{i,k})-\tilde h_i(x_{i,k}\mj{-\eta v_{i,{k}}},z_i({x}_{i,k}\mj{-\eta v_{i,{k}}}),\xi_{i,k})|\|v_{i,k}\|.
\end{align*}
Invoking the Lipschitz continuity of $\tilde h_i(x,\bullet,\xi_i)$ in \Rme{Assumption~\ref{assum:main1}(i)} and noting that $\|v_{i,k}\|=1$,  
\begin{align*}
\|e_{i,k}^{\eta,\varepsilon_k}\|    &\le  (\tfrac{n}{{2}\eta} )\tilde{L}_0(\xi_{i,k})\left\|z_{{i,\varepsilon_k}}({x}_{i,k}+\eta v_{i,k})-z_i({x}_{i,k}+\eta v_{i,k})\right\| \\
&+  (\tfrac{n}{{2}\eta} )\tilde{L}_0(\xi_{i,k})\left\|z_{{i,\varepsilon_k}}({x}_{i,k}{-\eta v_{i,{k}}})-z_i({x}_{i,k}{-\eta v_{i,{k}}})\right\|.
\end{align*}
Taking conditional expectations on both sides of the preceding inequality, we obtain
\begin{align*}
\mathbb{E}\left[\|e_{i,k}^{\eta,\varepsilon_k}\| ^2 \mid\mathcal{F}_k\right] &\le 2(\tfrac{n}{{2}\eta} )^2\mathbb{E}\left[\tilde{L}_0({\bxi_{i,k}})^2\left\|z_{{i,\varepsilon_k}}({x}_{i,k}+\eta {\bf v}_{i,k})\right.\left.-z_i({x}_{i,k}+\eta {\bf v}_{i,k})\right\|^2\, \big| \,  \mathcal{F}_k\right]\\
&+ 2(\tfrac{n}{{2}\eta} )^2\mathbb{E}\left[\tilde{L}_0({\bxi}_{i,k})^2\left\|z_{{i,\varepsilon_k}}({x}_{i,k}{-\eta {\bf v}_{i,{k}}})-z_i({x}_{i,k}{-\eta {\bf v}_{i,{k}}})\right\|^2 \, \big|\, \mathcal{F}_k\, \right].
\end{align*}
We now analyze each expectation term in the preceding inequality. For convenience, let us use the notation $\overline{\mathcal{F}}_{i,k}\triangleq \mathcal{F}_k\cup\left(\cup_{\ell=1}^2 \left(\left(\cup_{t=0}^{t_{k}-1}\{\xi_{i,t}^{k,\ell}\}\right)\cup \{z_{i,0}^{k,\ell}\} \right)\right) \cup \{v_{i,k}\}$. {Since $\mathcal{F}_k \subset \overline{\mathcal{F}_k}$}, \us{by the Tower law}, 
\begin{align*}
 &\mathbb{E}[\tilde{L}_0({\bxi}_{i,k})^2\left\|z_{{i,\varepsilon_k}}({x}_{i,k}+\eta \mee{\bf v}_{i,k})-z_i({x}_{i,k}+\eta \mee{\bf v}_{i,k})\right\|^2 \mid\mathcal{F}_k] \\
 &=\mathbb{E}\left[{\mathbb{E}}\left[\tilde{L}_0({\bxi_{i,k}})^2\left\|z_{{i,\varepsilon_k}}({x}_{i,k}+\eta {\bf v}_{i,k})-z_i({x}_{i,k}+\eta {\bf v}_{i,k})\right\|^2 \mid\overline{\mathcal{F}}_{i,k}\right] { \, \bigg| \, \mathcal{F}_k}\, \right]\\
  &=\mathbb{E}\left[\mathbb{E}_{{\bxi_{i,k}}}\left[\tilde{L}_0({\bxi_{i,k}})^2\right]{\mathbb{E}}\left[\left\|z_{{i,\varepsilon_k}}({x}_{i,k}+\eta {\bf v}_{i,k})-z_i({x}_{i,k}+\eta {\bf v}_{i,k})\right\|^2\,  \big| \, \overline{\mathcal{F}}_{i,k}\right]{ \, \big|  \, \mathcal{F}_k}\,\right]\\
  &=\tilde{L}_0^2\,\mathbb{E}\left[{\mathbb{E}}\left[\left\|z_{{i,\varepsilon_k}}({x}_{i,k}+\eta {\bf v}_{i,k})-z_i({x}_{i,k}+\eta {\bf v}_{i,k})\right\|^2 \, \big| \, \overline{\mathcal{F}}_{i,k}\right] {\, \big| \mathcal{F}_k \, } \right]\\
  &=\tilde{L}_0^2 \,\mathbb{E}\left[\left\|z_{{i,\varepsilon_k}}({x}_{i,k}+\eta {\bf v}_{i,k})-z_i({x}_{i,k}+\eta {\bf v}_{i,k})\right\|^2 \mid \mathcal{F}_k\right] \\
  &=\tilde{L}_0^2 \,\mathbb{E}\left[\mathbb{E}\left[\left\|z_{{i,\varepsilon_k}}({x}_{i,k}+\eta {\bf v}_{i,k})-z_i({x}_{i,k}+\eta {\bf v}_{i,k})\right\|^2 \, \mid \, \mathcal{F}_k \cup \{v_{i,k}\}\right] {\, \mid \, \mathcal{F}_k \, }\right]  \\
&  \leq\tilde{L}_0^2 \,\mathbb{E}\left[ \varepsilon_k { \, \mid \,  \mathcal{F}_k  \, } \right]\leq \tilde{L}_0^2 \varepsilon_k,
\end{align*}
where we invoked the independence of random variables, the definition of $\tilde{L}_0$ in \Rme{Assumption~\ref{assum:main1}(i)}, and the inexactness bound. Similarly, we have
\begin{align*}
 \mathbb{E}\left[\, \tilde{L}_0({\bxi}_{i,k})^2\left\|z_{{i,\varepsilon_k}}({x}_{i,k}{-\eta {\bf v}_{i,{k}}})-z_i({x}_{i,k}{-\eta {\bf v}_{i,{k}}})\right\|^2 \, \bigg| \, \mathcal{F}_k\, \right] \leq  \tilde{L}_0^2 \varepsilon_k.
\end{align*}
From the preceding three inequalities, we obtain the result. 
\end{proof} 
\fy{Next, we provide additional relations that will be utilized in the analysis in this section. 
\begin{lemma}\label{lemma:multiple parts}\em Suppose \RRme{Assumptions~\ref{assum:mixxx}, \ref{assum:main1}, and \ref{assum:samples}} hold. Then for all $k\ge 0$, the following \us{hold}.

\noindent (i) $\bar y_{k+1}=\overline{\nabla{f}^\eta}(\mathbf{x}_k)+\bar \delta_{k}^\eta+\bar{e}_{k}^{\eta,\varepsilon_k}$.

\noindent (ii) $\bar x_{k+1}=\bar x_k-\RRmath{\gamma}\bar y_{k+1}.$

 \noindent (iii) $ \|\overline{\nabla{f}^\eta}(\mathbf{x}_k)-\nabla f^\eta(\bar{x}_k)\|^2 \le \mj{\left(\tfrac{L_0^2n}{m\eta^2}\right)}\|\mathbf{x}_k-\mathbf{1}\bar{x}_k\|^2.$

 \noindent (iv) $\mathbb{E}[\|\bar \delta_{k}^\eta\|^2\mid \mathcal{F}_k]\le \mj{\frac{\mj{16\sqrt{2\pi}L_0^2n}}{m}}$.
 
\Rme{\noindent (v) {If} $\mathbb{E}[\|z_{i,\varepsilon_k}(x) -z_i(x)\|^2 \mid {x}]\leq \varepsilon_k$ for {a realization $x$ of any random variable} ${\bf x} \in\mathbb{R}^n$ almost surely, {then,} $\mathbb{E}[\|\bar{e}_{k}^{\eta,\varepsilon_k}\|^2 |\mathcal{F}_k]\le { \tfrac{\tilde L_0^2n^2\varepsilon_k}{\eta^2} }.$}
\end{lemma} 
\begin{proof}
(i) We use mathematical induction on $k\geq 0$. Consider \eqref{eqn:R1}. Multiplying {both} sides by $\tfrac{1}{m}\mathbf{1}^\top$ and invoking \Rme{Assumption~\ref{assum:mixxx}(ii)}, we obtain, for $k\geq 0$,
\begin{align}
&\bar{y}_{k+1}= \bar{y}_{k}+ \overline{\nabla{f}^\eta}(\mathbf{x}_k) - \overline{\nabla{f}^\eta}(\mathbf{x}_{k-1})+\bar{\delta}_{k}^\eta-\bar{\delta}_{k-1}^\eta+\bar{e}_{k}^{\eta,\varepsilon_k}-\bar{e}_{k-1}^{\eta,\varepsilon_{k-1}}.\label{eqn:R1_avg}
\end{align}
Note that following the notation introduced earlier, we have $\bar{\delta}_{-1}^\eta=\bar{e}_{-1}^{\eta,\varepsilon_{-1}}=\overline{\nabla{f}^\eta}(\mathbf{x}_{-1})=\bar{y}_0=0$. Thus, substituting $k=0$ in \eqref{eqn:R1_avg}, the equation in (i) holds for $k=0$. Now suppose $\bar y_{k+1}=\overline{\nabla{f}^\eta}(\mathbf{x}_k)+\bar \delta_{k}^\eta+\bar{e}_{k}^{\eta,\varepsilon_k}$ for some $k\geq 0$. From \eqref{eqn:R1_avg}, we have 
\begin{align*}
&\bar{y}_{k+2}= \bar{y}_{k+1}+ \overline{\nabla{f}^\eta}(\mathbf{x}_{k+1}) - \overline{\nabla{f}^\eta}(\mathbf{x}_{k})+\bar{\delta}_{k+1}^\eta-\bar{\delta}_{k}^\eta+\bar{e}_{k+1}^{\eta,\varepsilon_{k+1}}-\bar{e}_{k}^{\eta,\varepsilon_{k}}. 
\end{align*}
Invoking $\bar y_{k+1}=\overline{\nabla{f}^\eta}(\mathbf{x}_k)+\bar \delta_{k}^\eta+\bar{e}_{k}^{\eta,\varepsilon_k}$, we obtain $\bar y_{k+2}=\overline{\nabla{f}^\eta}(\mathbf{x}_{k+1})+\bar \delta_{k+1}^\eta+\bar{e}_{k+1}^{\eta,\varepsilon_{k+1}}$. This concludes the proof of part (i). 

\noindent (ii) Consider \eqref{eqn:R2}. Multiplying both sides by $\tfrac{1}{m}\mathbf{1}^\top$, invoking \Rme{Assumption~\ref{assum:mixxx}(ii)}, and using \eqref{eqn:notations}, we obtain the result.

\noindent (iii) Recall that in view of Remark~\ref{rem:smoothness_of_f_i}, $\nabla f^{\eta}_i$ is $\frac{L_0 \mj{\sqrt{n}}}{\eta}$-Lipschitz for all $i\in [m]$. Utilizing this, we have
\begin{align*}
 \|\overline{\nabla{f}^\eta}(\mathbf{x}_k)-\nabla f^\eta(\bar{x}_k)\|^2  & =  \left\|\tfrac{1}{m}\textstyle\sum_{i=1}^m \left( \nabla f^\eta_i(x_{i,k})- \nabla f^\eta_i(\bar x_{k})\right)\right\|^2 
 \leq \tfrac{1}{m}\textstyle\sum_{i=1}^m \left\|    \nabla f^\eta_i(x_{i,k})- \nabla f^\eta_i(\bar x_{k}) \right\|^2 \\
& \le \mj{\tfrac{L_0^2n}{m\eta^2}}\textstyle\sum_{i=1}^m\|x_{i,k}-\bar x_{k}\|^2=\mj{\tfrac{L_0^2n}{m\eta^2}}\|\mathbf{x}_k-\mathbf{1}\bar{x}_k\|^2.
\end{align*}

\noindent (iv) Invoking \Rme{Lemma~\ref{lemma:g_ik_eta_props}(ii) and~(iii)}, we have 
\begin{align*}
\mathbb{E}\left[\|\bar \delta_{k}^\eta\|^2\mid \mathcal{F}_k\right]&=\mathbb{E}\left[\|\tfrac{1}{m}\textstyle\sum_{i=1}^m \delta_{i,k}^\eta\|^2\mid \mathcal{F}_k\right] \\
& = \tfrac{1}{m^2}\left(\textstyle\sum_{i=1}^m\mathbb{E}[\| \delta_{i,k}^\eta\|^2\mid \mathcal{F}_k]+2\textstyle\sum_{i=1}^{m-1}\textstyle\sum_{j=i+1}^m \mathbb{E}[ {\delta_{i,k}^\eta}^\top { \delta_{j,k}^\eta}  \mid \mathcal{F}_k]\right)  \\
&= \tfrac{1}{m^2} \textstyle\sum_{i=1}^m\mathbb{E}[\| \delta_{i,k}^\eta\|^2\mid \mathcal{F}_k]  \overset{\Rme{\text{Lemma~\ref{lemma:g_ik_eta_props}(ii)}}}{\le} \mj{\frac{\mj{16\sqrt{2\pi}L_0^2n}}{m}}.
\end{align*}

\noindent (v)  Invoking \Rme{Lemma~\ref{lemma:g_ik_eta_props}(iv)}, we have 
\begin{align*}
\mathbb{E}[\|\bar e_{k}^\eta\|^2\mid \mathcal{F}_k]&=\mathbb{E}[\|\tfrac{1}{m}\textstyle\sum_{i=1}^m e_{i,k}^\eta\|^2\mid \mathcal{F}_k] 
\leq \tfrac{1}{m^2}\left(m\textstyle\sum_{i=1}^m\mathbb{E}[\| e_{i,k}^\eta\|^2\mid \mathcal{F}_k]\right)  \leq \mj{\tfrac{\tilde L_0^2n^2\varepsilon_k}{\eta^2}} .
\end{align*}
\end{proof}}
{To analyze the convergence of the proposed method, motivated by the analysis in distributed smooth nonconvex optimization~\cite{xin2021improved}, we consider three error metrics, {defined} as follows.\\ 

\noindent (i) {$\mathbb{E}[\|{\nabla{f}^\eta}(\bar{x}_k)\|^2]$} quantifying the {stationarity residual} of {the} generated iterates; \\

\noindent  (ii) $\mathbb{E}[\|{ \mathbf{x}}_{k}-\mathbf{1}{\bar{ {x}}}_{k}\|^2]$  measuring {the} mean-squared consensus error for local variables; and \\

\noindent (iii) $\mathbb{E}[\|{ \mathbf{y}}_{k}-\mathbf{1}{\bar{ {y}}}_{k}\|^2]$ measuring {the} mean-squared consensus error for the gradient {tracking} variables.\\

To this end, in the following three lemmas, we will derive three recursive inequalities characterized by these error metrics. These results will be utilized later in this section to establish the convergence and derive rate statements. 
\begin{lemma}\em \label{lemma:descent lemma inexact} 
Consider Algorithm~\ref{alg:DZGT}. Let \RRme{Assumptions~\ref{assum:mixxx}, \ref{assum:main1}, and \ref{assum:samples}} hold. Suppose \mj{$\RRmath{\gamma}\le \frac{\eta}{6L_0\sqrt{n}}$} for all $k\geq 0$. Then for any $k \geq 0$, we have
\begin{align*}
\mathbb{E}[f^\eta(\bar{{x}}_{k+1})\mid \mathcal{F}_k]&\le f^\eta(\bar{{x}}_{k})-\tfrac{\RRmath{\gamma}}{4}\|\nabla f^\eta(\bar{{x}}_k)\|^2-\tfrac{\RRmath{\gamma}}{4}\|\overline{\nabla f^\eta}(\mathbf{x}_k)\|^2+\mj{\left(\tfrac{L_0^2n \RRmath{\gamma}}{2 m\eta^2}\right)}\|\mathbf{x}_k-\mathbf{1}\bar{x}_k\|^2\\
& +\mj{\left(\tfrac{24\sqrt{2\pi}L_0^3n^{3/2}}{m\eta}\right)}\RRmath{\gamma}^2+\mj{\left(1+\mj{\tfrac{3\RRmath{\gamma}L_0\sqrt{n}}{2\eta}}\right)\RRmath{\gamma}\left(\tfrac{\tilde L_0^2n^2\varepsilon_k}{\eta^2}\right) }.
\end{align*} 
\end{lemma}
\begin{proof}
Recall that from Remark~\ref{rem:smoothness_of_f_i}, $f^{\eta}$ is \mj{$\tfrac{L_0\sqrt{n}}{\eta}$}-smooth. Utilizing \Rme{Lemma~\ref{lemma:multiple parts}}, we can write
\begin{align*}
f^\eta(\bar{{x}}_{k+1})&\le f^\eta(\bar{{x}}_{k})+\nabla f^\eta(\bar{{x}}_k)^\top(\bar{{x}}_{k+1}-\bar{{x}}_{k})+\mj{\tfrac{L_0\sqrt{n}}{2\eta}}\|\bar{{x}}_{k+1}-\bar{{x}}_{k}\|^2\\
&=f^\eta(\bar{{x}}_{k})-\RRmath{\gamma}\nabla f^\eta(\bar{{x}}_k)^{\top}(\overline{\nabla{f}^\eta}(\mathbf{x}_k)+\bar \delta_{k}^\eta+\bar e_{k}^{\eta,\varepsilon_k})+\mj{\tfrac{\RRmath{\gamma}^2L_0\sqrt{n}}{2\eta}}\|\overline{\nabla{f}^\eta}(\mathbf{x}_k)+\bar \delta_{k}^\eta+\bar e_{k}^{\eta,\varepsilon_k}\|^2.
\end{align*}
Taking conditional expectations on both sides and invoking the fact that $\mathbb{E}[\bar\delta_{k}^\eta \mid \mathcal{F}_k]=0$, we obtain
\begin{align*}
\mathbb{E}[f^\eta(\bar{{x}}_{k+1})\mid \mathcal{F}_k]&\le f^\eta(\bar{{x}}_{k})-\RRmath{\gamma}\nabla f^\eta(\bar{{x}}_k)^\top(\overline{\nabla f^\eta}(\mathbf{x}_k)+\mathbb{E}[\bar e_{k}^{\eta,\varepsilon_k}\mid \mathcal{F}_k])\\
&+\mj{\tfrac{\RRmath{\gamma}^2L_0\sqrt{n}}{2\eta}}\mathbb{E}[\|\overline{\nabla{f}^\eta}(\mathbf{x}_k)+\bar \delta_{k}^\eta+\bar e_{k}^{\eta,\varepsilon_k}\|^2\mid \mathcal{F}_k]\\
&=f^\eta(\bar{{x}}_{k})-\tfrac{\RRmath{\gamma}}{2}\|\nabla f^\eta(\bar{{x}}_k)\|^2-\tfrac{\RRmath{\gamma}}{2}\|\overline{\nabla f^\eta}(\mathbf{x}_k)\|^2+\tfrac{\RRmath{\gamma}}{2}\|\nabla f^\eta(\bar{{x}}_k)-\overline{\nabla f^\eta}(\mathbf{x}_k)\|^2\\
&-\RRmath{\gamma}\nabla f^\eta(\bar{{x}}_k)^\top \mathbb{E}[\bar e_{k}^{\eta,\varepsilon_k}\mid \mathcal{F}_k] +\mj{\tfrac{3\RRmath{\gamma}^2L_0\sqrt{n}}{2\eta}}\left(\|\overline{\nabla{f}^\eta}(\mathbf{x}_k)\|^2+\mathbb{E}[\|\bar \delta_{k}^\eta\|^2+\|\bar e_{k}^{\eta,\varepsilon_k}\|^2\mid \mathcal{F}_k]\right).
\end{align*}
Note that we have
\begin{align*}
-\RRmath{\gamma}\nabla f^\eta(\bar{{x}}_k)^\top \mathbb{E}[\bar e_{k}^{\eta,\varepsilon_k}\mid \mathcal{F}_k] &\le\RRmath{\gamma}\left|\nabla f^\eta(\bar{{x}}_k)^\top \mathbb{E}[\bar e_{k}^{\eta,\varepsilon_k}\mid \mathcal{F}_k] \right|\\
&\le\tfrac{\RRmath{\gamma}}{2\alpha}\|\nabla f^\eta(\bar{{x}}_k)\|^2+\tfrac{\alpha\RRmath{\gamma}}{2}\|\mathbb{E}[\bar e_{k}^{\eta,\varepsilon_k}\mid \mathcal{F}_k]\|^2\\
&\le\tfrac{\RRmath{\gamma}}{2\alpha}\|\nabla f^\eta(\bar{{x}}_k)\|^2+\tfrac{\alpha\RRmath{\gamma}}{2}\mathbb{E}[\|\bar e_{k}^{\eta,\varepsilon_k}\|^2\mid \mathcal{F}_k],
\end{align*}
where $\alpha>0$ is a scalar. Invoking \Rme{Lemma~\ref{lemma:multiple parts}(iii)} and setting $\alpha = 2$ in the preceding two relations, 
\begin{align*}
\mathbb{E}[f^\eta(\bar{{x}}_{k+1})\mid \mathcal{F}_k]&\le f^\eta(\bar{{x}}_{k})-\tfrac{\RRmath{\gamma}}{4}\|\nabla f^\eta(\bar{{x}}_k)\|^2-\tfrac{\RRmath{\gamma}}{2}\left(1-\mj{\tfrac{3\RRmath{\gamma}L_0 \sqrt{n}}{\eta}}\right)\|\overline{\nabla f^\eta}(\mathbf{x}_k)\|^2\\
&+\mj{\left(\tfrac{L_0^2n \RRmath{\gamma}}{2 m\eta^2}\right)}\|\mathbf{x}_k-\mathbf{1}\bar{x}_k\|^2 +\mj{\tfrac{3\RRmath{\gamma}^2L_0\sqrt{n}}{2\eta}}\mathbb{E}[\|\bar \delta_{k}^\eta\|^2\mid \mathcal{F}_k]+\left(\RRmath{\gamma}+\mj{\tfrac{3\RRmath{\gamma}^2L_0\sqrt{n}}{2\eta}}\right)\mathbb{E}[\|\bar e_{k}^{\eta,\varepsilon_k}\|^2\mid \mathcal{F}_k]  .
\end{align*}
Invoking \Rme{Lemmas~\ref{lemma:multiple parts}(iv) and~(v)}, we obtain 
\begin{align*}
\mathbb{E}[f^\eta(\bar{{x}}_{k+1})\mid \mathcal{F}_k]&\le f^\eta(\bar{{x}}_{k})-\tfrac{\RRmath{\gamma}}{4}\|\nabla f^\eta(\bar{{x}}_k)\|^2-\tfrac{\RRmath{\gamma}}{2}\left(1-\mj{\tfrac{3\RRmath{\gamma}L_0 \sqrt{n}}{\eta}}\right)\|\overline{\nabla f^\eta}(\mathbf{x}_k)\|^2\\
&+\mj{\left(\tfrac{L_0^2n \RRmath{\gamma}}{2 m\eta^2}\right)}\|\mathbf{x}_k-\mathbf{1}\bar{x}_k\|^2 +\mj{\left(\tfrac{24\sqrt{2\pi}L_0^3n^{3/2}}{m\eta}\right)}\RRmath{\gamma}^2+\mj{\left(1+\mj{\tfrac{3\RRmath{\gamma}L_0\sqrt{n}}{2\eta}}\right)\RRmath{\gamma}\left(\tfrac{\tilde L_0^2n^2\varepsilon_k}{\eta^2}\right)} .
\end{align*}
 The result is obtained by noting that {$\RRmath{\gamma}\le \frac{\eta}{6L_0\sqrt{n}}$} {implying} that $1-\mj{\tfrac{3\RRmath{\gamma}L_0 \sqrt{n}}{\eta}} \geq \tfrac{1}{2}$.
\end{proof}}
\begin{lemma}\em \label{lemma: multiple inequalities for the metrics} Consider Algorithm~\ref{alg:DZGT}. Let Assumptions~\ref{assum:mixxx} \RRme{hold}. Then, the following inequalities hold for all $k\ge 0$.

 \noindent (i) For any $\theta>0$, we have $\|\mathbf{x}_{k+1}-\mathbf{1}\bar x_{k+1}\|^2\le  (1+\theta)\lambda_{\mathbf{W}}^2\|\mathbf{x}_{k}-\mathbf{1}\bar x_{k}\|^2+(1+\tfrac{1}{\theta})\RRmath{\gamma}^2\lambda_{\mathbf{W}}^2\|\mathbf{y}_{k+1}-\mathbf{1}\bar y_{k+1}\|^2$.
 

 
 \noindent (ii) $\|\mathbf{x}_{k+1}-\mathbf{1}\bar x_{k+1}\|\le \lambda_{\mathbf{W}}\|\mathbf{x}_{k}-\mathbf{1}\bar x_{k}\|+\RRmath{\gamma}\lambda_{\mathbf{W}}\|\mathbf{y}_{k+1}-\mathbf{1}\bar y_{k+1}\|.$
\end{lemma}
\begin{proof}
\noindent (i) Utilizing \eqref{eqn:R2} and \Rme{Lemma~\ref{lemma:multiple parts}(ii)}, we obtain \us{the following for any $\theta > 0$.} 
\begin{align*}
\|\mathbf{x}_{k+1}-\mathbf{1}\bar x_{k+1}\|^2&=\|\mathbf{W}(\mathbf{x}_{k}-\RRmath{\gamma}\mathbf{y}_{k+1})-\mathbf{1}(\bar x_k-\RRmath{\gamma}\bar y_{k+1})\|^2\\
&\le (1+\theta)\|\mathbf{W}\mathbf{x}_{k}-\mathbf{1}\bar x_k\|^2+ (1+\tfrac{1}{\theta})\RRmath{\gamma}^2\|\mathbf{W}\mathbf{y}_{k+1}-\mathbf{1}\bar y_{k+1}\|^2.
\end{align*}
Invoking Lemma~\ref{lem:lambda_w} twice, we obtain the inequality in (i).

\noindent (ii) Utilizing \eqref{eqn:R2} and \Rme{Lemma~\ref{lemma:multiple parts}(ii)} once again, we obtain 
\begin{align*}
\|\mathbf{x}_{k+1}-\mathbf{1}\bar x_{k+1}\|&=\|\mathbf{W}(\mathbf{x}_{k}-\RRmath{\gamma}\mathbf{y}_{k+1})-\mathbf{1}(\bar x_k-\RRmath{\gamma}\bar y_{k+1})\| \le \|\mathbf{W}\mathbf{x}_{k}-\mathbf{1}\bar x_k\|+\RRmath{\gamma}\|\mathbf{W}\mathbf{y}_{k+1}-\mathbf{1}\bar y_{k+1}\|.
\end{align*}
Invoking {Lemma~\ref{lem:lambda_w}}, we obtain the inequality in (ii).
\end{proof}
\fyy{In the following lemma, we derive a recursive bound for the error term \uvs{corresponding to the consensus error in gradient tracking, given by} $\mathbb{E}[\|{ \mathbf{y}}_{k+1}-\mathbf{1}{\bar{ {y}}}_{k+1}\|^2]$. The proof is relegated to the Appendix due to its length.}
\begin{lemma}\em \label{Lemma:main bound for the gradient tracker in terms of other main terms-inexact}
Consider Algorithm~\ref{alg:DZGT}. \RRme{Let Assumptions~\ref{assum:mixxx}, \ref{assum:main1}, and \ref{assum:samples} hold}. Suppose $\RRmath{\lambda_{\mathbf{W}}>0}$ \RRme{and} $\RRme{\gamma} \leq \min \Big\{
 \tfrac{ \sqrt{1-\lambda_{\mathbf{W}}^2}}{10\sqrt{3}\lambda_{\mathbf{W}}^2 }
 ,\tfrac{(1-\lambda_{\mathbf{W}}^2) }{20\lambda_{\mathbf{W}}^3 }\Big\} \left(\tfrac{\eta}{\mj{\sqrt{n}} L_0 }\right)$. Then, for all $k\ge 0$, we have
 \begin{align*}
\mathbb{E}[\|{ \mathbf{y}}_{k+2}-\mathbf{1}{\bar{ {y}}}_{k+2}\|^2]
&\le  \left(\tfrac{1+\lambda_{\mathbf{W}}^2}{2} \right)\mathbb{E}[\|{ \mathbf{y}}_{k+1}-\mathbf{1}{ \bar{y}}_{k+1}\|^2]   +\mj{\left(\tfrac{2m\|\mathbf{W}\|^2\tilde L_0^2n^2\varepsilon_k}{\eta^2}\right)} \\
& +\mj{32m\sqrt{2\pi}\left(\|\mathbf{W}\|^2+\mee{\tfrac{11}{2}}\lambda_{\mathbf{W}}^2\right)\mj{nL_0^2}  }  
 +\left(\tfrac{\mj{\sqrt{n}}L_0}{\eta}\right)^2\left( \tfrac{30 m\RRme{\gamma}^2\lambda_{\mathbf{W}}^2}{1-\lambda_{\mathbf{W}}^2} \right)\mathbb{E}[\|\overline{\nabla{f}^\eta}(\mathbf{x}_k)\|^2] \\
 & +
\left(\tfrac{\mj{\sqrt{n}}L_0}{\eta}\right)^2\left( \tfrac{5\lambda_{\mathbf{W}}^2(9-\lambda_{\mathbf{W}}^2)}{1-\lambda_{\mathbf{W}}^2} \right)\mathbb{E}[\|\mathbf{x}_{k}-\mathbf{1}\bar x_{k}\|^2]  +\mj{\tfrac{160\sqrt{2\pi}n^2L_0^4\RRme{\gamma}^2\lambda_{\mathbf{W}}^2(3+\lambda_{\mathbf{W}}^2)}{\eta^2(1-\lambda_{\mathbf{W}}^2)}
}\\
&+\mj{\RRme{\left(\tfrac{30\lambda_{\mathbf{W}}^2\RRme{\gamma}^2L_0^2\tilde L_0^2 {n^3} m      
\varepsilon_k}{\eta^4(1-\lambda_{\mathbf{W}}^2)}\right)} }+\mj{5m\lambda_{\mathbf{W}}^2 \left(\tfrac{n^2\tilde L_0^2}{\eta^2}\right)\left(\tfrac{1+3\lambda_{\mathbf{W}}^2}{1-\lambda_{\mathbf{W}}^2}\right)(\varepsilon_k+\varepsilon_{k+1})  }.
\end{align*}
\RRme{Furthermore, if } $\RRmath{\lambda_{\mathbf{W}}=0}$\RRme{, then} $\RRmath{
\mathbb{E}\left[
\|\mathbf{y}_{k+1}-\mathbf{1}\bar{y}_{k+1}\|^2
\right]=0}$, \RRme{for all} $\RRmath{k\geq0}$.
\end{lemma}

\fy{The inequalities derived in the previous three lemmas are characterized by \us{recursions involving} three \us{distinct} error metrics. To analyze the convergence of the \us{resulting sequence}, we will utilize the following result to obtain a non-recursive error bound.   \begin{lemma}\label{lem:recursive_matrix_ineq} \em 
Let $\{u_k\}$ and $\{b_k\}$ denote sequences of nonnegative column vectors in $\mathbb{R}^d$, satisfying 
$
{u}_{k+1}\le\mathbf{A}{u}_{k}+{b}_{k}$, for $k\geq 0$, where $\mathbf{A}\in \mathbb{R}^{d\times d}$ denotes a nonnegative square matrix such that $\rho(\mathbf{A}) <1$.  Then, for any $K\geq 1$, we have $\sum_{k=1}^{K}   {u}_{k} \leq (\mathbf{I}_d-\mathbf{A})^{-1}\left({u}_{0}+\sum_{k=0}^{K-1}{b}_{k}\right)$.
\end{lemma}
\begin{proof}
Unrolling $
{u}_{k+1}\le\mathbf{A}{u}_{k}+{b}_{k}$ recursively, we obtain for any $k\geq 0$,  $
{u}_{k+1}\le\mathbf{A}^{k+1}{u}_{0}+\textstyle\sum_{j=0}^{k}\mathbf{A}^j{b}_{k-j}.$ Summing both sides of the preceding inequality over $k=0,\ldots,K-1$, where $K\ge1$, 
\begin{align}\label{ineq:uk+1A}
\textstyle\sum_{k=0}^{K-1}{u}_{k+1}& \le\textstyle\sum_{k=0}^{K-1}\mathbf{A}^{k+1}{u}_{0}+\textstyle\sum_{k=0}^{K-1}\textstyle\sum_{j=0}^{k}\mathbf{A}^j{b}_{k-j}\notag \\
&=\left(\textstyle\sum_{k=0}^{K-1}\mathbf{A}^{k+1}\right){u}_{0}+\textstyle\sum_{k=0}^{K-1}\left(\textstyle\sum_{t=0}^{ K-1-k}\mathbf{A}^{t}\right)b_k .
\end{align}
{{For a given $K \geq 1$, let us define the nonnegative matrix \us{$\mathbf{B}_K$} as $\mathbf{B}_K \triangleq \left(\mathbf{I}_d+\textstyle\sum_{k=K}^{\infty}\mathbf{A}^{k+1}\right)$. \us{Furthermore}, for  given $k$ and {$K\geq 1$} such that $0\leq k \leq  K-1$, define the nonnegative matrix \us{$\mathbf{C}_{k,K}$ as} $\mathbf{C}_{k,K} \triangleq \textstyle\sum_{t=K-k}^{ \infty}\mathbf{A}^{t}$. Note that in view of  $\rho(\mathbf{A}) <1$, we have {that} $\textstyle\sum_{k=0}^{\infty}\mathbf{A}^k=(\mathbf{I}_d-\mathbf{A})^{-1}$. As a result, matrices $\mathbf{B}_K$ and $\mathbf{C}_{k,K}$ exist and are well-defined, for any $K\geq 1$ and $0\leq k \leq K-1$.}
We note that both {$\mathbf{B}_K$ and $\mathbf{C}_{k,K}$} are nonnegative matrices and  ${u}_{0}$ and $\textstyle\sum_{k=0}^{K-1}{b}_{k}$ are \us{both} nonnegative column vectors. Therefore, ${\mathbf{B}_K}u_0$, and $\textstyle\sum_{k=0}^{K-1}{\mathbf{C}_{k,K}}{b}_{k}$ are nonnegative column vectors as well. Adding these two nonnegative vectors to the right-hand side of {\eqref{ineq:uk+1A}}, we obtain
\begin{align*}
\textstyle\sum_{k=0}^{K-1}{u}_{k+1}& \le \left(\textstyle\sum_{k=0}^{K-1}\mathbf{A}^{k+1}\right)u_0+\left(\mathbf{I}_d+\textstyle\sum_{k=K}^{\infty}\mathbf{A}^{k+1} \right){u}_{0}+\textstyle\sum_{k=0}^{K-1}\left(\textstyle\sum_{t=0}^{ K-1-k}\mathbf{A}^{t}\right)b_k\\
&+\textstyle\sum_{k=0}^{K-1}\left(\textstyle\sum_{t=K-k}^{ \infty}\mathbf{A}^{t}\right)b_k  = \left(\textstyle\sum_{k=0}^{\infty}\mathbf{A}^k\right){u}_{0}+\left(\textstyle\sum_{k=0}^{\infty}\mathbf{A}^k\right)\textstyle\sum_{k=0}^{K-1}{b}_{k} \\
&=(\mathbf{I}_d-\mathbf{A})^{-1}{u}_{0}+(\mathbf{I}_d-\mathbf{A})^{-1}\textstyle\sum_{k=0}^{K-1}{b}_{k}.
\end{align*}} 
\end{proof}}
\fy{In the next result, we invoke Lemma~\ref{lem:recursive_matrix_ineq} and build a non-recursive error bound for our method. \uvs{Observe that the next result shows that the aggregated consensus error $\sum_{k=0}^K\mathbb{E}\left[\tfrac{\|\mathbf{x}_{k}- \mathbf{1} \bar{x}_{k}\|^2}{m} \right]$ can be shown to be suitably bounded under suitable assumptions. This assumes relevance when deriving a rate guarantee. } 

\begin{lemma}\em\label{lemma:bound for x-bar x inexact} Consider Algorithm~\ref{alg:DZGT}. \RRme{Let Assumptions~\ref{assum:mixxx}, \ref{assum:main1}, and \ref{assum:samples} hold}. Suppose $\RRmath{\lambda_{\mathbf{W}}>0}$ \RRme{and} the stepsize $\gamma$ is a constant satisfying $\gamma \leq \min \Big\{
 \tfrac{ \sqrt{1-\lambda_{\mathbf{W}}^2}}{10\sqrt{3}\lambda_{\mathbf{W}}^2 }
 ,\tfrac{(1-\lambda_{\mathbf{W}}^2) }{20\lambda_{\mathbf{W}}^3 },\tfrac{ (1-\lambda_{\mathbf{W}}^2)^2}{20\lambda_{\mathbf{W}}^2  }\Big\} \left(\tfrac{\eta}{\mj{\sqrt{n}} L_0 }\right)$.  Then\RRme{,} for any $K \geq 1$, we have
\Rme{\begin{align*}
\textstyle \sum_{k=0}^{\fyr{K-1}}\mathbb{E}\left[\tfrac{\|\mathbf{x}_{k}- \mathbf{1} \bar{x}_{k}\|^2}{m} \right]&\le \left(1+\tfrac{20}{1-\lambda_{\mathbf{W}}^2}\right)\mathbb{E}\left[\tfrac{\|\mathbf{x}_{0}- \mathbf{1} \bar{x}_{0}\|^2}{m} \right]+\left(\tfrac{80\gamma^2\lambda_{\mathbf{W}}^2L_0^2{n^2}}{\eta^2(1-\lambda_{\mathbf{W}}^2)^3}\right) { \left(\tfrac{5\lambda_{\mathbf{W}}^2 \tilde L_0^2}{L_0^2}\right)\left(\tfrac{1+3\lambda_{\mathbf{W}}^2}{1-\lambda_{\mathbf{W}}^2}\right)} \textstyle\sum_{k=0}^{K-1}\varepsilon_{k+1}\\
&+\left(\tfrac{80\gamma^2\lambda_{\mathbf{W}}^2L_0^2{n}}{\eta^2(1-\lambda_{\mathbf{W}}^2)^3}\right) 
\left( \tfrac{30  \gamma^2\lambda_{\mathbf{W}}^2}{1-\lambda_{\mathbf{W}}^2} \right)\textstyle\sum_{k=0}^{K-1}\mathbb{E}[\|\overline{\nabla{f}^\eta}(\mathbf{x}_k)\|^2]  
\\
&+\tfrac{160\gamma^2\lambda_{\mathbf{W}}^2}{(1-\lambda_{\mathbf{W}}^2)^3}  \left(\tfrac{\lambda_{\mathbf{W}}^2}{m}\|\nabla \mathbf f^\eta(\mathbf{x}_0)\|^2+ {16\sqrt{2\pi}\lambda_{\mathbf{W}}^2L_0^2n}+{\lambda_{\mathbf{W}}^2\left(\tfrac{\tilde L_0^2n^2\varepsilon_0}{\eta^2}\right)}\right)\notag  \\
& + \left(\tfrac{80\gamma^2\lambda_{\mathbf{W}}^2L_0^2{nK}}{\eta^2(1-\lambda_{\mathbf{W}}^2)^3}\right){32\eta^2\sqrt{2\pi}\left(\|\mathbf{W}\|^2+{\tfrac{11}{2}}\lambda_{\mathbf{W}}^2\right){}  }+{\tfrac{12800K\sqrt{2\pi}n^2L_0^4\gamma^4\lambda_{\mathbf{W}}^4(3+\lambda_{\mathbf{W}}^2)}{m\eta^2\RRme{(1-\lambda_{\mathbf{W}}^2)^4}}
} \\
& + \left(\tfrac{80\gamma^2\lambda_{\mathbf{W}}^2L_0^2{n^2}}{\eta^2(1-\lambda_{\mathbf{W}}^2)^3}\right)\left( { \Rme{\tfrac{2\|\mathbf{W}\|^2\tilde L_0^2}{L_0^2}} }+{ \left(\tfrac{5\lambda_{\mathbf{W}}^2\tilde L_0^2}{L_0^2}\right)\left(\tfrac{1+3\lambda_{\mathbf{W}}^2}{1-\lambda_{\mathbf{W}}^2}\right)}+{ \RRme{\tfrac{30\lambda_{\mathbf{W}}^2\gamma^2\tilde L_0^2 {n}      
}{\eta^2(1-\lambda_{\mathbf{W}}^2)} }}\right)\textstyle\sum_{k=0}^{K-1}\varepsilon_k .
\end{align*}}
\RRme{Furthermore, if } $\RRmath{\lambda_{\mathbf{W}}=0}$\RRme{, then, for any } $\RRmath{K\geq1}$\RRme{, we have}
$\RRmath{
\sum_{k=0}^{K-1}
\mathbb{E}\left[
\tfrac{\|\mathbf{x}_{k}-\mathbf{1}\bar{x}_{k}\|^2}{m}
\right]
=
\mathbb{E}\left[
\tfrac{\|\mathbf{x}_{0}-\mathbf{1}\bar{x}_{0}\|^2}{m}
\right]
}$\RRme{.}
\end{lemma}
\begin{proof} \RRme{We first consider the case } $\RRmath{\lambda_{\mathbf{W}}>0}$\RRme{.} \us{By invoking} \Rme{Lemma~\ref{lemma: multiple inequalities for the metrics}(i)} \us{and choosing} $\theta=\tfrac{1-\lambda_{\mathbf{W}}^2}{2\lambda_{\mathbf{W}}^2}$ and for any $k\ge 0$, we have 
$$\|\mathbf{x}_{k+1}-\mathbf{1}\bar x_{k+1}\|^2\le  \left(\tfrac{1+\lambda_{\mathbf{W}}^2}{2}\right)\|\mathbf{x}_{k}-\mathbf{1}\bar x_{k}\|^2+\left(\tfrac{1+\lambda_{\mathbf{W}}^2}{1-\lambda_{\mathbf{W}}^2}\right)\gamma^2\lambda_{\mathbf{W}}^2\|\mathbf{y}_{k+1}-\mathbf{1}\bar y_{k+1}\|^2.$$
Taking expectation\mee{s} on both sides and invoking Lemma~\ref{Lemma:main bound for the gradient tracker in terms of other main terms-inexact}, we obtain $
{u}_{k+1}\le\mathbf{A}{u}_{k}+{b}_{k}$ for $k\geq 0$, where \us{for any $k \ge 0$}, we define $u_k$, $b_k$ and $\mathbf{A}$ as 
$${u}_{k}= 
\left[\mathbb{E}\left[\tfrac{1}{m}\|\mathbf{x}_{k}-\mathbf{1}\bar x_{k}\|^2\right], \mathbb{E}\left[\tfrac{\eta^2}{mL_0^2\mj{n}}\|\mathbf{y}_{k+1}- \mathbf{1} \bar{y}_{k+1}\|^2\right]\right]^\top, \qquad   \mathbf{A}=\begin{bmatrix}
\frac{1+\lambda_{\mathbf{W}}^2}{2}&\frac{2\gamma^2\lambda_{\mathbf{W}}^2L_0^2\mj{n}}{\eta^2(1-\lambda_{\mathbf{W}}^2)}\\
 \tfrac{45\lambda_{\mathbf{W}}^2}{1-\lambda_{\mathbf{W}}^2}  &\frac{1+\lambda_{\mathbf{W}}^2}{2}\\
\end{bmatrix},$$ 
{and $b_k = [0,b_{2,k}]^\top$, where $b_{2,k}$ is defined as} 
\mj{\begin{align*}
b_{2,k}&= \mj{\left(\tfrac{2\|\mathbf{W}\|^2\tilde L_0^2n\varepsilon_k}{L_0^2}\right)} +\mj{32\eta^2\sqrt{2\pi}\left(\|\mathbf{W}\|^2+\mee{\tfrac{11}{2}}\lambda_{\mathbf{W}}^2\right)\mj{}  }  +\left( \tfrac{30 \gamma^2\lambda_{\mathbf{W}}^2}{1-\lambda_{\mathbf{W}}^2} \right)\mathbb{E}[\|\overline{\nabla{f}^\eta}(\mathbf{x}_k)\|^2] \\
&+\mj{\tfrac{160\sqrt{2\pi}nL_0^2\gamma^2\lambda_{\mathbf{W}}^2(3+\lambda_{\mathbf{W}}^2)}{m(1-\lambda_{\mathbf{W}}^2)}
}+\mj{\RRme{\left(\tfrac{30\lambda_{\mathbf{W}}^2\gamma^2\tilde L_0^2 {n^2}      
\varepsilon_k}{\eta^2(1-\lambda_{\mathbf{W}}^2)}\right)}} +\mj{ \left(\tfrac{5n\lambda_{\mathbf{W}}^2\tilde L_0^2}{L_0^2}\right)\left(\tfrac{1+3\lambda_{\mathbf{W}}^2}{1-\lambda_{\mathbf{W}}^2}\right)(\varepsilon_k+\varepsilon_{k+1})  },\end{align*}}{and $b_{2,k}$ is obtained from Lemma~\ref{Lemma:main bound for the gradient tracker in terms of other main terms-inexact} by invoking the upper bound on $\mathbb{E}\left[\|\mathbf{y}_{k+2}- \mathbf{1} \bar{y}_{k+2}\|^2\right]$.} Next, we show that $\rho(\mathbf{A}) <1$. Invoking \cite[Corollary 8.1.29]{horn2012matrix}, it suffices to show that \us{when $\mathbf{A}$ is a nonnegative matrix}, there exists a vector $s >0$ such that $\mathbf{A}s < s$. By choosing $s= [(1-\lambda_{\mathbf{W}}^2)^2,100\lambda_{\mathbf{W}}^2]^\top$  and recalling that $\gamma\leq \tfrac{ (1-\lambda_{\mathbf{W}}^2)^2}{20\lambda_{\mathbf{W}}^2  }\left(\tfrac{\eta}{\mj{\sqrt{n}} L_0}\right)$, we have $\mathbf{A}s < s$ \us{implying that} $\rho(\mathbf{A}) <1$.  From Lemma~\ref{lem:recursive_matrix_ineq}, $\sum_{k=1}^{K}   {u}_{k} \leq (\mathbf{I}_\mee{2}-\mathbf{A})^{-1}\left({u}_{0}+\sum_{k=0}^{K-1}{b}_{k}\right)$. Next, we \us{analyze} $(\mathbf{I}_\mee{2}-\mathbf{A})^{-1}$. From $\gamma \leq \tfrac{ (1-\lambda_{\mathbf{W}}^2)^2}{20\lambda_{\mathbf{W}}^2  }\left(\tfrac{\eta}{\mj{\sqrt{n}} L_0}\right)$, we have $
\text{det}(\mathbf{I}_2-\mathbf{A})=\tfrac{(1-\lambda_{\mathbf{W}}^2)^2}{4}-\Rme{\tfrac{90\gamma^2\lambda_{\mathbf{W}}^4L_0^2n}{\eta^2(1-\lambda_{\mathbf{W}}^2)^2}}\ge \tfrac{(1-\lambda_{\mathbf{W}}^2)^2}{40}.$
This implies that {\Rme{$
\tfrac{1}{\text{det}(\mathbf{I}_2-\mathbf{A})}\le \tfrac{40}{(1-\lambda_{\mathbf{W}}^2)^2}$}.
Moreover, we have
\begin{align}(\mathbf{I}_2-\mathbf{A})^{-1}=\tfrac{1}{\text{det}(\mathbf{I}_2-\mathbf{A})}\begin{bmatrix}
\frac{1-\lambda_{\mathbf{W}}^2}{2}&\frac{2\gamma^2\lambda_{\mathbf{W}}^2L_0^2{n}}{\eta^2(1-\lambda_{\mathbf{W}}^2)}\\
 \tfrac{45\lambda_{\mathbf{W}}^2}{1-\lambda_{\mathbf{W}}^2}  &\frac{1-\lambda_{\mathbf{W}}^2}{2}\\
\end{bmatrix}.\label{eq:inverse of I_2 minus A}\end{align}
We define matrix ${\mathbf{D}}$ as 
\[{\mathbf{D}}\triangleq \tfrac{40}{(1-\lambda_{\mathbf{W}}^2)^2}\begin{bmatrix}
\frac{1-\lambda_{\mathbf{W}}^2}{2}&\frac{2\gamma^2\lambda_{\mathbf{W}}^2L_0^2{n}}{\eta^2(1-\lambda_{\mathbf{W}}^2)}\\
 \tfrac{45\lambda_{\mathbf{W}}^2}{1-\lambda_{\mathbf{W}}^2}  &\frac{1-\lambda_{\mathbf{W}}^2}{2}\\
\end{bmatrix}=\begin{bmatrix}
\frac{20}{1-\lambda_{\mathbf{W}}^2}&\frac{80\gamma^2\lambda_{\mathbf{W}}^2L_0^2{n}}{\eta^2(1-\lambda_{\mathbf{W}}^2)^3}\\
\frac{1800\lambda_{\mathbf{W}}^2}{(1-\lambda_{\mathbf{W}}^2)^3}&\frac{20}{1-\lambda_{\mathbf{W}}^2}\\
\end{bmatrix}.\]{In view of} \Rme{$
\tfrac{1}{\text{det}(\mathbf{I}_2-\mathbf{A})}\le \tfrac{40}{(1-\lambda_{\mathbf{W}}^2)^2}$} and~\eqref{eq:inverse of I_2 minus A}\fyr{,} $ {\mathbf{D}}-(\mathbf{I}_2-\mathbf{A})^{-1}$ is a nonnegative matrix. Since the vector $\left({u}_{0}+\sum_{k=0}^{K-1}{b}_{k}\right)$ is nonnegative, it follows that $\left({\mathbf{D}}-(\mathbf{I}_2-\mathbf{A})^{-1}\right)\left({u}_{0}+\sum_{k=0}^{K-1}{b}_{k}\right)$ is also a nonnegative vector.} {From Lemma~\ref{lem:recursive_matrix_ineq}, we have $\sum_{k=1}^{K}   {u}_{k} \leq (\mathbf{I}_d-\mathbf{A})^{-1}\left({u}_{0}+\textstyle\sum_{k=0}^{K-1}{b}_{k}\right)$, leading to
\Rme{\begin{align*}
\textstyle\sum_{k=1}^{K}   {u}_{k} &\!\leq (\mathbf{I}_d-\mathbf{A})^{-1}\left({u}_{0}\!+\textstyle\sum_{k=0}^{K-1}{b}_{k}\right)\!+\underbrace{\left({\mathbf{D}}-(\mathbf{I}_2-\mathbf{A})^{-1}\right)\left({u}_{0}\!+\textstyle\sum_{k=0}^{K-1}{b}_{k}\right)}_{\ge0}\!=\!{\mathbf{D}}\left({u}_{0}+\textstyle\sum_{k=0}^{K-1}{b}_{k}\right).
\end{align*}}
Invoking the definitions of $u_k, u_0,b_k$, and ${\mathbf{D}}$, we obtain}
\Rme{\begin{align*}
&\textstyle\sum_{k=0}^{K-1}\mathbb{E}\left[\tfrac{\|\mathbf{x}_{k}- \mathbf{1} \bar{x}_{k}\|^2}{m} \right]\le \left(1+\tfrac{20}{1-\lambda_{\mathbf{W}}^2}\right)\mathbb{E}\left[\tfrac{\|\mathbf{x}_{0}- \mathbf{1} \bar{x}_{0}\|^2}{m} \right]+
\tfrac{80\gamma^2\lambda_{\mathbf{W}}^2}{(1-\lambda_{\mathbf{W}}^2)^3} \mathbb{E}\left[\tfrac{1}{m }\|\mathbf{y}_{1}- \mathbf{1} \bar{y}_{1}\|^2\right]\notag \\
&+\left(\tfrac{80\gamma^2\lambda_{\mathbf{W}}^2L_0^2{n}}{\eta^2(1-\lambda_{\mathbf{W}}^2)^3}\right) 
\left( \tfrac{30  \gamma^2\lambda_{\mathbf{W}}^2}{1-\lambda_{\mathbf{W}}^2} \right)\textstyle\sum_{k=0}^{K-1}\mathbb{E}[\|\overline{\nabla{f}^\eta}(\mathbf{x}_k)\|^2]  + \left(\tfrac{80\gamma^2\lambda_{\mathbf{W}}^2L_0^2{n^2}}{\eta^2(1-\lambda_{\mathbf{W}}^2)^3}\right)  { \left(\tfrac{5\lambda_{\mathbf{W}}^2\tilde L_0^2}{L_0^2}\right)\left(\tfrac{1+3\lambda_{\mathbf{W}}^2}{1-\lambda_{\mathbf{W}}^2}\right)} \textstyle\sum_{k=0}^{K-1}\varepsilon_{k+1} \\
& + \left(\tfrac{80\gamma^2\lambda_{\mathbf{W}}^2L_0^2{nK}}{\eta^2(1-\lambda_{\mathbf{W}}^2)^3}\right){32\eta^2\sqrt{2\pi}\left(\|\mathbf{W}\|^2+{\tfrac{11}{2}}\lambda_{\mathbf{W}}^2\right){}  }+{\tfrac{12800K\sqrt{2\pi}n^2L_0^4\gamma^4\lambda_{\mathbf{W}}^4(3+\lambda_{\mathbf{W}}^2)}{m\eta^2\RRme{(1-\lambda_{\mathbf{W}}^2)^4}}
} \\
& + \left(\tfrac{80\gamma^2\lambda_{\mathbf{W}}^2L_0^2{n^2}}{\eta^2(1-\lambda_{\mathbf{W}}^2)^3}\right)\left( \mj{\Rme{ \tfrac{2\|\mathbf{W}\|^2\tilde L_0^2}{L_0^2}} }+{ \left(\tfrac{5\lambda_{\mathbf{W}}^2\tilde L_0^2}{L_0^2}\right)\left(\tfrac{1+3\lambda_{\mathbf{W}}^2}{1-\lambda_{\mathbf{W}}^2}\right)}+{\RRme{\left(\tfrac{30\lambda_{\mathbf{W}}^2\gamma^2\tilde L_0^2 {n}      
}{\eta^2(1-\lambda_{\mathbf{W}}^2)}\right)}}\right)\textstyle\sum_{k=0}^{K-1}\varepsilon_k.
\end{align*}}
Invoking the gradient tracking update \Rme{\eqref{eqn:R1}}, we obtain
\begin{align*}
\mathbb{E}\left[{\|\mathbf{y}_{1}-\tfrac{1}{m}\mathbf{1}\mathbf{1}^\top\mathbf{y}_{1}\|^2}\right] &= \mathbb{E}\left[\|(\mathbf{W}-\tfrac{1}{m}\mathbf{1}\mathbf{1}^\top)(\nabla \mathbf f^\eta(\mathbf{x}_0)+\boldsymbol{\delta}_{0}^\eta+\mathbf{e}_{0}^{\eta,\varepsilon_0})\|^2\right]\\
&\leq 2\mathbb{E}\left[\|(\mathbf{W}-\tfrac{1}{m}\mathbf{1}\mathbf{1}^\top)(\nabla \mathbf f^\eta(\mathbf{x}_0)+\boldsymbol{\delta}_{0}^\eta)\|^2\right]+2\mathbb{E}\left[\|(\mathbf{W}-\tfrac{1}{m}\mathbf{1}\mathbf{1}^\top)\mathbf{e}_{0}^{\eta,\varepsilon_0}\|^2\right]\\
& \leq 2\mathbb{E}\left[\|(\mathbf{W}-\tfrac{1}{m}\mathbf{1}\mathbf{1}^\top)\nabla \mathbf f^\eta(\mathbf{x}_0)\|^2\right]+2\mathbb{E}\left[\|(\mathbf{W}-\tfrac{1}{m}\mathbf{1}\mathbf{1}^\top)\boldsymbol{\delta}_{0}^\eta\|^2\right]+\mj{ \tfrac{2\lambda_{\mathbf{W}}^2m\tilde L_0^2n^2\varepsilon_0}{\eta^2} }\\
&\le 2\lambda_{\mathbf{W}}^2\|\nabla \mathbf f^\eta(\mathbf{x}_0)\|^2+ \mj{32\sqrt{2\pi}\lambda_{\mathbf{W}}^2mL_0^2n}+\mj{ \tfrac{2\lambda_{\mathbf{W}}^2m \tilde L_0^2n^2\varepsilon_0}{\eta^2} },
\end{align*}
where the last inequality is obtained using Lemma~\ref{lem:lambda_w} and Lemma~\ref{lemma:g_ik_eta_props}. The result is obtained from the preceding two inequalities.

\noindent\RRme{Next, consider the case} $\RRmath{\lambda_{\mathbf{W}}=0}$\RRme{. By the definition of} $\RRmath{\lambda_{\mathbf{W}}}$\RRme{, we have} $\RRmath{\mathbf{W}=\tfrac{1}{m}\mathbf{1}\mathbf{1}^{\top}}$\RRme{. Therefore, invoking the update rule~\eqref{eqn:R2}, for any} $\RRmath{k\geq 0}$\RRme{, we obtain} $$\RRmath{\mathbf{x}_{k+1}=\tfrac{1}{m}\mathbf{1}\mathbf{1}^{\top}\left(\mathbf{x}_{k}-\gamma\mathbf{y}_{k+1}\right)=\mathbf{1}\left(\tfrac{1}{m}\mathbf{1}^{\top}\left(\mathbf{x}_{k}-\gamma\mathbf{y}_{k+1}\right)\right)}.$$
\RRme{On the other hand, by the definition of } $\RRmath{\bar{x}_{k+1}}$\RRme{, we have} $\RRmath{\bar{x}_{k+1}=\tfrac{1}{m}\mathbf{1}^{\top}\mathbf{x}_{k+1}}$\RRme{. Substituting the preceding expression for} $\RRmath{\mathbf{x}_{k+1}}$ \RRme{and using} $\RRmath{\mathbf{1}^{\top}\mathbf{1}=m}$\RRme{, we obtain} $\RRmath{\bar{x}_{k+1}=\tfrac{1}{m}\mathbf{1}^{\top}\left(\mathbf{x}_{k}-\gamma\mathbf{y}_{k+1}\right)}$\RRme{. Hence,} $\RRmath{\mathbf{x}_{k+1}=\mathbf{1}\bar{x}_{k+1}}$ \RRme{ for all} $\RRmath{k\geq 0}$\RRme{. Equivalently,} $\RRmath{\|\mathbf{x}_{k}-\mathbf{1}\bar{x}_{k}\|^2=0}$ \RRme{ for all} $\RRmath{k\geq 1}$\RRme{. Thus, for any} $\RRmath{K\geq 1}$\RRme{, we have} $$\RRmath{\sum_{k=0}^{K-1}\mathbb{E}\left[\tfrac{\|\mathbf{x}_{k}-\mathbf{1}\bar{x}_{k}\|^2}{m}\right]=\mathbb{E}\left[\tfrac{\|\mathbf{x}_{0}-\mathbf{1}\bar{x}_{0}\|^2}{m}\right]+\sum_{k=1}^{K-1}\mathbb{E}\left[\tfrac{\|\mathbf{x}_{k}-\mathbf{1}\bar{x}_{k}\|^2}{m}\right].}$$\RRme{ Since} $\RRmath{\|\mathbf{x}_{k}-\mathbf{1}\bar{x}_{k}\|^2=0}$ \RRme{ for all} $\RRmath{k\geq 1}$\RRme{, every term in the second sum is zero. Therefore,} $$\RRmath{\sum_{k=0}^{K-1}\mathbb{E}\left[\tfrac{\|\mathbf{x}_{k}-\mathbf{1}\bar{x}_{k}\|^2}{m}\right]=\mathbb{E}\left[\tfrac{\|\mathbf{x}_{0}-\mathbf{1}\bar{x}_{0}\|^2}{m}\right].}$$
\end{proof}
Consider the inexact setting in Algorithm~\ref{alg:DZGT} where \us{approximate} solutions to the lower-level problems are \us{obtained} by Algorithm~\ref{alg:lowerlevel-1stage}. \us{Next}, we provide a bound relating the inexactness $\varepsilon_k$ to the number of iterations used in Algorithm~\ref{alg:lowerlevel-1stage}. This result characterizes the solution quality for the agent-wise lower-level problems and will be utilized to derive rate statements for Algorithm~\ref{alg:DZGT}. \fyy{The proof is presented in the Appendix.}} 
\fy{
\begin{lemma}[\textbf{\mee{Rate statement for Algorithm~\ref{alg:lowerlevel-1stage}}}]\label{lem:Alg2_conv}\em
Consider Algorithm~\ref{alg:lowerlevel-1stage}. Let Assumptions~\ref{assum:main2} and~\ref{assum:alg2} hold\RRme{, and define $\hat D:=\max_{j\in[m]}D_j$}. Let $k\geq 0$, $i \in [m]$, $\ell \in \{1,2\}$, input vector $\hat{x}_{i,k}$, and initial vector $z_{i,0}^{k,\ell} \in \mathcal{Z}_i(\hat x^\ell_{i,k})$ be given. Let the stepsize be given by $\hat{\gamma}_{t}:=\frac{\hat\gamma}{t+\hat{\Gamma}}$ where $\hat{\gamma} > \frac{1}{\mu_F} $ and $\hat{\Gamma} >\frac{\hat{\gamma}L_F^2}{\mu_F}$. Further, suppose the algorithm is terminated after $t_k$ iterations\mee{,} where $a>0$ is given. Then, 
$$\mathbb{E}[\|z_{i,\varepsilon_k}(\hat{x}_{i,k}^{\ell}) - z_i(\hat{x}_{i,k}^{\ell})\|^2\mid \tilde{\mathcal{F}}_{i,0}^{k,\ell}] \leq \varepsilon_k \triangleq \max\{\nu_F^2\Rme{\hat{\gamma}^2}(\mu_F\hat{\gamma}-1)^{-1}, \hat{\Gamma} \RRmath{\hat D}\}\tfrac{1}{t_k+\hat{\Gamma}},$$ 
where $D_i  $ is given by Assumption~\ref{assum:main2} and $\nu_F$ is given by Assumption~\ref{assum:alg2}. 
\end{lemma}
\RRme{\begin{remark}\em
In the present analysis, the inexactness level $\varepsilon_k$ is taken to be common across all agents rather than agent-dependent. Allowing agent-wise inexactness levels $\varepsilon_{i,k}$ can be viewed as an extension of the current framework and is left for future work.
\end{remark}}
We will make use of the following result in the rate analysis in this section. \fyy{The proof can be found in the Appendix.} 
\begin{lemma}[\mee{\textbf{Harmonic series bounds}}]\label{lem:Harmonic series bounds}\em
Let $\Gamma\geq 1$ be a given integer and $a \geq 0$ be a scalar. Then, for any $K\geq 1$ the following results hold. 

\noindent (i) If $a \in [0,1)$, then $\sum_{k=0}^{K-1}\tfrac{1}{(k+\Gamma)^a}  \leq \tfrac{K^{1-a}}{1-a}$.

\noindent (ii) If $a =1$, then $\sum_{k=0}^{K-1}\tfrac{1}{(k+\Gamma)^a}  \leq 1+\ln(K)$.

\noindent (iii) If $a >1$, then $\sum_{k=0}^{K-1}\tfrac{1}{(k+\Gamma)^a}  \leq \tfrac{a}{a-1}$.

In particular, when $\Gamma \geq 2$, the following results hold. 

\noindent (iv) If $a \in [0,1)$, then $\sum_{k=0}^{K-1}\tfrac{1}{(k+\Gamma)^a}  \leq \tfrac{(K+\Gamma -1)^{1-a}-(\Gamma-1)^{a-1}}{1-a} $.

\noindent (v) If $a =1$, then $\sum_{k=0}^{K-1}\tfrac{1}{(k+\Gamma)^a}   \leq \ln(\tfrac{K+\Gamma-1}{\Gamma-1})$.

\noindent (vi) If $a >1$, then $\sum_{k=0}^{K-1}\tfrac{1}{(k+\Gamma)^a}  \leq \tfrac{1}{(a-1)\Gamma^{a-1}}$.

\end{lemma}

}
\us{We now} present our main convergence result in addressing distributed \us{single-stage} SMPECs. 
 \begin{theorem}[\textbf{\Rme{Convergence guarantees for DiZS-GT$^{\text{1s}}$}}]\em \label{Theorem:thm 1} Consider Algorithm~\ref{alg:DZGT}. Let Assumptions~\ref{assum:mixxx}, \ref{assump:opt_f_bounded_below}, \ref{assum:main1}, \ref{assum:main2}, \RRme{\ref{assum:alg2}, and \ref{assum:samples}} hold. \RRme{Suppose} $\RRmath{\lambda_{\mathbf{W}}>0}$\RRme{.}
 
 \noindent (a) [\fyr{\textbf{Error bounds}}] \fyr{Let} $\gamma \leq \min \Big\{
 \tfrac{ \sqrt{1-\lambda_{\mathbf{W}}^2}}{10\sqrt{3}\lambda_{\mathbf{W}}^2 }
 ,\tfrac{(1-\lambda_{\mathbf{W}}^2) }{20\lambda_{\mathbf{W}}^3 },\tfrac{ (1-\lambda_{\mathbf{W}}^2)^2}{20\lambda_{\mathbf{W}}^2  },\frac{1}{6},\frac{(1-\lambda_{\mathbf{W}}^2)}{9\lambda_{\mathbf{W}}}\Big\} \left(\tfrac{\eta}{\mj{\sqrt{n}} L_0 }\right)$. \fyr{Suppose $K^*$ is uniformly distributed over $\{0,\ldots,K-1\}$}. Then, for any $K \geq 1$, \Rme{the following hold.}

\Rme{\noindent (a-i) [{\textbf{Stationarity error bound}}] }
 \mj{\begin{align}
 \mathbb{E}[\mbox{dist}^2(0,\partial_{\eta} f(\bar{{x}}_{K^*})) ]&\le  \tfrac{4(\mathbb{E}[f^\eta(\bar{{x}}_{0})]- \inf_{x} f^\eta(x))}{\gamma K}+\mj{\left(\tfrac{96\sqrt{2\pi}L_0^3n^{3/2}\gamma }{m\eta}\right)}+\tfrac{2 L_0^2n}{\eta^2K}\left(1+\tfrac{20}{1-\lambda_{\mathbf{W}}^2}\right)\mathbb{E}\left[\tfrac{\|\mathbf{x}_{0}- \mathbf{1} \bar{x}_{0}\|^2 }{m}\right]\notag\\
&+\tfrac{320\gamma^2\lambda_{\mathbf{W}}^4L_0^2n}{(1-\lambda_{\mathbf{W}}^2)^3\eta^2K} \left(\tfrac{1}{m}\|\nabla \mathbf f^\eta(\mathbf{x}_0)\|^2+ \mj{16\sqrt{2\pi}L_0^2n}+\mj{\left(\tfrac{\tilde L_0^2n^2\varepsilon_0}{\eta^2}\right)}\right)\notag\\
&+ \mj{\tfrac{25600\sqrt{2\pi}n^3L_0^6\gamma^4\lambda_{\mathbf{W}}^4(3+\lambda_{\mathbf{W}}^2)}{m\eta^4\RRme{(1-\lambda_{\mathbf{W}}^2)^4}}
} + \left(\tfrac{160\gamma^2\lambda_{\mathbf{W}}^2L_0^4\mj{n^2}}{\eta^2(1-\lambda_{\mathbf{W}}^2)^3}\right)\mj{32\sqrt{2\pi}\left(\|\mathbf{W}\|^2+\mee{\tfrac{11}{2}}\lambda_{\mathbf{W}}^2\right)  }\notag\\
& + \left(\tfrac{160\gamma^2\lambda_{\mathbf{W}}^2L_0^4\mj{n^3}}{\eta^4(1-\lambda_{\mathbf{W}}^2)^3}\right) \mj{ \left(\tfrac{5\lambda_{\mathbf{W}}^2 \tilde L_0^2}{L_0^2}\right)\left(\tfrac{1+3\lambda_{\mathbf{W}}^2}{1-\lambda_{\mathbf{W}}^2}\right)} \tfrac{\sum_{k=0}^{K-1}\varepsilon_{k+1}}{K}\notag\\
& + \left( \tfrac{160\gamma^2\lambda_{\mathbf{W}}^2L_0^4\mj{n^3}}{\eta^4(1-\lambda_{\mathbf{W}}^2)^3}\right) \left( \mj{ \tfrac{2\|\mathbf{W}\|^2\tilde L_0^2}{L_0^2} }+\mj{ \left(\tfrac{5\lambda_{\mathbf{W}}^2\tilde L_0^2}{L_0^2}\right)\left(\tfrac{1+3\lambda_{\mathbf{W}}^2}{1-\lambda_{\mathbf{W}}^2}\right)}+\mj{\RRme{ \tfrac{30\lambda_{\mathbf{W}}^2\gamma^2\tilde L_0^2 {n}      
}{\eta^2(1-\lambda_{\mathbf{W}}^2)} }}\right)  \tfrac{\sum_{k=0}^{K-1}\varepsilon_k}{K}\notag\\
& +   \left(1+\mj{\tfrac{3\gamma L_0\sqrt{n}}{2\eta}}\right) \mj{\left(\tfrac{\RRme{4}\tilde L_0^2n^2}{\eta^2}\right)} \tfrac{\sum_{k=0}^{K-1}\varepsilon_k}{K}.\label{eq:Non-asymptotic error bound in the single-stage case}
\end{align}}
\Rme{\noindent (a-ii) [{\textbf{Consensus error bound}}]} \fyr{Let $H_{n,\eta,\gamma}\triangleq \left(1+\tfrac{8 L_0^2n}{\eta^2}\left(\tfrac{80\gamma^2\lambda_{\mathbf{W}}^2L_0^2{n}}{\eta^2(1-\lambda_{\mathbf{W}}^2)^3}\right) 
\left( \tfrac{30  \gamma^2\lambda_{\mathbf{W}}^2}{1-\lambda_{\mathbf{W}}^2} \right)\right) $. Then, we have}
 \Rme{\begin{align}
\textstyle \mathbb{E}\left[\tfrac{\|\mathbf{x}_{K^*}- \mathbf{1} \bar{x}_{K^*}\|^2}{m} \right]& \le{\tfrac{160H_{n,\eta,\gamma}\gamma^2\lambda_{\mathbf{W}}^4}{(1-\lambda_{\mathbf{W}}^2)^3K}  \left(\tfrac{1}{m}\|\nabla \mathbf f^\eta(\mathbf{x}_0)\|^2+ {16\sqrt{2\pi}L_0^2n}+{\left(\tfrac{\tilde L_0^2n^2\varepsilon_0}{\eta^2}\right)}\right)}\notag  \notag\\
& + {\left(\tfrac{80\gamma^2\lambda_{\mathbf{W}}^2L_0^2{n}}{(1-\lambda_{\mathbf{W}}^2)^3}\right){32H_{n,\eta,\gamma}\sqrt{2\pi}\left(\|\mathbf{W}\|^2+{\tfrac{11}{2}}\lambda_{\mathbf{W}}^2\right)\mj{}  }+{\tfrac{12800H_{n,\eta,\gamma}\sqrt{2\pi}n^2L_0^4\gamma^4\lambda_{\mathbf{W}}^4(3+\lambda_{\mathbf{W}}^2)}{m\eta^2\RRme{(1-\lambda_{\mathbf{W}}^2)^4}}
} }\notag\\
& +{H_{n,\eta,\gamma} \left(\tfrac{80\gamma^2\lambda_{\mathbf{W}}^2L_0^2{n^2}}{\eta^2(1-\lambda_{\mathbf{W}}^2)^3}\right)\left( { \Rme{\tfrac{2\|\mathbf{W}\|^2\tilde L_0^2}{L_0^2}} }+{ \left(\tfrac{5\lambda_{\mathbf{W}}^2\tilde L_0^2}{L_0^2}\right)\left(\tfrac{1+3\lambda_{\mathbf{W}}^2}{1-\lambda_{\mathbf{W}}^2}\right)}\!+\!{ \RRme{\tfrac{30\lambda_{\mathbf{W}}^2\gamma^2\tilde L_0^2 {n}      
}{\eta^2(1-\lambda_{\mathbf{W}}^2)} }}\right) \tfrac{\textstyle\sum_{k=0}^{K-1}\varepsilon_k}{K} }\notag\\ &+{H_{n,\eta,\gamma}\left(\tfrac{80\gamma^2\lambda_{\mathbf{W}}^2L_0^2{n^2}}{\eta^2(1-\lambda_{\mathbf{W}}^2)^3}\right) { \left(\tfrac{5\lambda_{\mathbf{W}}^2 \tilde L_0^2}{L_0^2}\right)\left(\tfrac{1+3\lambda_{\mathbf{W}}^2}{1-\lambda_{\mathbf{W}}^2}\right)} \textstyle\tfrac{\sum_{k=0}^{K-1}\varepsilon_{k+1}}{K}}\notag\\
&+\RRme{\left(\tfrac{80\gamma^2\lambda_{\mathbf{W}}^2L_0^2{n}}{\eta^2(1-\lambda_{\mathbf{W}}^2)^3}\right) 
\left( \tfrac{30  \gamma^2\lambda_{\mathbf{W}}^2}{1-\lambda_{\mathbf{W}}^2} \right)\left(1+{\tfrac{3\gamma L_0\sqrt{n}}{2\eta}}\right) {\left(\tfrac{16\tilde L_0^2n^2}{\eta^2}\right)}\tfrac{\sum_{k=0}^{K-1}\varepsilon_k}{K}}\notag\\
&+{\frac{H_{n,\eta,\gamma}}{K}\left(1+\tfrac{20}{1-\lambda_{\mathbf{W}}^2}\right)\mathbb{E}\left[\tfrac{\|\mathbf{x}_{0}- \mathbf{1} \bar{x}_{0}\|^2}{m} \right]}+\RRme{\left(\tfrac{80\gamma^2\lambda_{\mathbf{W}}^2L_0^2{n}}{\eta^2(1-\lambda_{\mathbf{W}}^2)^3}\right) 
\left( \tfrac{30  \gamma^2\lambda_{\mathbf{W}}^2}{1-\lambda_{\mathbf{W}}^2} \right){\left(\tfrac{384\sqrt{2\pi}L_0^3n^{3/2}\gamma }{m\eta}\right)}}\notag\\
&+\RRme{\left(\tfrac{80\gamma^2\lambda_{\mathbf{W}}^2L_0^2{n}}{\eta^2(1-\lambda_{\mathbf{W}}^2)^3}\right) 
\left( \tfrac{30  \gamma^2\lambda_{\mathbf{W}}^2}{1-\lambda_{\mathbf{W}}^2} \right)\tfrac{16(\mathbb{E}[f^\eta(\bar{{x}}_{0})]- \inf_{x} f^\eta(x))}{\gamma K}}\fy{.}\label{equation: the bound for the consensus error}
\end{align}
}

 \noindent (b) [\mee{\textbf{Complexity bounds}}] Suppose at iteration $k$ in Algorithm~\ref{alg:DZGT}, in generating the inexact solutions, Algorithm~\ref{alg:lowerlevel-1stage} is terminated after \mj{$t_k=\left\lceil \left(n^{1/2}(k+\RRme{\hat{\Gamma}})^a\right)/\left(\eta^{2/3}\right)\right\rceil$} \RRme{iterations}\Rme{, where $a>0.5$}. Let $\epsilon>0$ be an arbitrary scalar such that $\mathbb{E}[\mbox{dist}^2(0,\partial_{\eta} f(\bar{{x}}_{K_\epsilon^*})) ] \leq \epsilon$. Then, the following results hold. 
 
 \noindent (b-i) [\mee{\textbf{Upper-level iteration/sample complexity}}] Suppose the stepsize in Algorithm~\ref{alg:DZGT} is given by \mj{$\gamma = \frac{\eta^{2/3}}{\sqrt{n^{3/2}\Rme{K_\epsilon}}L_0^{3/2}}$} such that \Rme{$K_\epsilon \geq  \left\{1,\left(\max \Big\{
 \tfrac{10\sqrt{3}\lambda_{\mathbf{W}}^2 }{ \sqrt{1-\lambda_{\mathbf{W}}^2}}
 ,\tfrac{20\lambda_{\mathbf{W}}^3 }{(1-\lambda_{\mathbf{W}}^2) },\tfrac{20\lambda_{\mathbf{W}}^2  }{ (1-\lambda_{\mathbf{W}}^2)^2},6,\frac{9\lambda_{\mathbf{W}}}{(1-\lambda_{\mathbf{W}}^2)}\Big\}  \right)^2\frac{1}{\eta^{2/3}L_0 n^{1/2}}\right\}$}. Then, we have
 \mje{
 \begin{align}
K_\epsilon &=  
\mathcal{O}\left(\left(\tfrac{n^{3/2}L_0^3}{\eta^{4/3}} \right)\epsilon^{-2}+\left(\tfrac{L_0^3n^{3/2}}{\eta^{2/3}m^2}\right)\epsilon^{-2}+\left(\tfrac{L_0^2n\mathbb{E}\left[\tfrac{1}{m}\|\mathbf{x}_{0}- \mathbf{1} \bar{x}_{0}\|^2 \right]}{\eta^2(1-\lambda_{\mathbf{W}}^2)}\right)\epsilon^{-1}+ \left(\tfrac{\lambda_{\mathbf{W}}^2L_0^{1/2}n^{1/4}}{\eta^{1/3}(1-\lambda_{\mathbf{W}}^2)^{3/2}}\right)\epsilon^{-1/2}\right.\notag\\
&\left. + \left(\tfrac{\lambda_{\mathbf{W}}^2\tilde L_0n^{3/4}\sqrt{\varepsilon_0}}{L_0^{1/2}\eta^{4/3}(1-\lambda_{\mathbf{W}}^2)^{3/2}}\right)\epsilon^{-1/2}+\left(\tfrac{\lambda_{\mathbf{W}}^2L_0n^{1/2}( \lambda_{\mathbf{W}}^2+     \|\mathbf{W}\|^2)}{\eta^{2/3}(1-\lambda_{\mathbf{W}}^2)^3}\right)\epsilon^{-1}+\left(\mj{\tfrac{\lambda_{\mathbf{W}}^2(3+\lambda_{\mathbf{W}}^2)^{1/2}}{\eta^{2/3}\sqrt{m}\RRmath{(1-\lambda_{\mathbf{W}}^2)^2}}}\right)\epsilon^{-1/2}\right)+\mee{J_\epsilon},\label{eq:complexity results for 1s}
\end{align}
where we have three cases for $\us{J_{\epsilon}}$ as follows:

\hspace{.2in} (1) If $a \in \left(\mee{\tfrac{1}{2}},1\right)$, then the following holds.
 \begin{align*}
  \mee{J_\epsilon}&= \mathcal{O}\left( \sqrt[1+a]{\left(\tfrac{ \RRmath{\kappa_F^2}\lambda_{\mathbf{W}}^2n}{L_0\eta^{2}(1-\lambda_{\mathbf{W}}^2)^3(1-a) }\right)\left( {\|\mathbf{W}\|^2\tilde L_0^2}{} +  \tfrac{\lambda_{\mathbf{W}}^4\tilde L_0^2}{ (1-\lambda_{\mathbf{W}}^2)}   \right)}\, \epsilon^{-1/(1+a)}+\left(\tfrac{ \RRmath{\kappa_F^2}\tilde L_0^2n^{3/2}}{\eta^{4/3}(1-a)}\right)^{1/a}\, \epsilon^{-1/a} \right.	\\
&+  \left. \sqrt[2+a]{\left(\tfrac{ \RRmath{\kappa_F^2}\RRmath{\lambda_{\mathbf{W}}^4}n^{1/2} \tilde L_0^2}{L_0^2\eta^{8/3}\RRmath{(1-\lambda_{\mathbf{W}}^2)^4}(1-a) }\right)}\, \epsilon^{-1/(2+a)}        +\left(\tfrac{ \RRmath{\kappa_F^2}\tilde L_0^2n^{5/4}}{L_0^{1/2}\eta^{5/3}(1-a)}\right)^{1/\left(a+\mee{\tfrac{1}{2}}\right)}\, \epsilon^{-1/\left(a+\mee{\tfrac{1}{2}}\right)}\right).
\end{align*}

 \hspace{.2in} (2)  If $a =1$, then the following holds.
\begin{align*}
  \mee{J_\epsilon}&=  \tilde{\mathcal{O}}\left(\sqrt{\left(\tfrac{ \RRmath{\kappa_F^2}\lambda_{\mathbf{W}}^2n}{L_0\eta^{2}(1-\lambda_{\mathbf{W}}^2)^3 }\right)\left( {\|\mathbf{W}\|^2\tilde L_0^2} +  \tfrac{\lambda_{\mathbf{W}}^4\tilde L_0^2}{ (1-\lambda_{\mathbf{W}}^2)}   \right)}\, \epsilon^{-1/2}+ \sqrt[3]{\left(\tfrac{ \RRmath{\kappa_F^2}\RRmath{\lambda_{\mathbf{W}}^4}n^{1/2} \tilde L_0^2}{L_0^2\eta^{8/3}\RRmath{(1-\lambda_{\mathbf{W}}^2)^4} }\right)}\, \epsilon^{-1/3} \right.	\\
&+   \left.\left(\tfrac{\RRmath{\kappa_F^2} \tilde L_0^2n^{5/4}}{L_0^{1/2}\eta^{5/3}}\right)^{2/3}\, \epsilon^{-2/3}+\left(\tfrac{\RRmath{\kappa_F^2} \tilde L_0^2n^{3/2}}{\eta^{4/3}}\right)\, \epsilon^{-1}\right).
\end{align*}

 \hspace{.2in} (3)  If $a >1$, then the following holds.
  \begin{align*}
  \mee{J_\epsilon}&= \mathcal{O}\left(\sqrt{\left(\tfrac{a\RRmath{\kappa_F^2} \lambda_{\mathbf{W}}^2n}{L_0\eta^{2}(1-\lambda_{\mathbf{W}}^2)^3(a-1) }\right)\left( {\|\mathbf{W}\|^2\tilde L_0^2}{} +  \tfrac{\lambda_{\mathbf{W}}^4\tilde L_0^2}{ (1-\lambda_{\mathbf{W}}^2)}   \right)}\, \epsilon^{-1/2}+ \sqrt[3]{\left(\tfrac{ a\RRmath{\kappa_F^2}\RRmath{\lambda_{\mathbf{W}}^4}n^{1/2}\tilde L_0^2}{L_0^2\eta^{8/3}\RRmath{(1-\lambda_{\mathbf{W}}^2)^4}(a-1) }\right)}\, \epsilon^{-1/3}\right. 	\\
&+   \left.\left(\tfrac{ a\RRmath{\kappa_F^2}\tilde L_0^2n^{5/4}}{L_0^{1/2}\eta^{5/3}(a-1)}\right)^{2/3}\, \epsilon^{-2/3}+\left(\tfrac{ a\RRmath{\kappa_F^2}\tilde L_0^2n^{3/2}}{\eta^{4/3}(a-1)}\right)\, \epsilon^{-1}\right).
\end{align*}}

  \noindent (b-ii) [\mee{\textbf{Overall iteration/sample complexity}}] To guarantee $\mathbb{E}[\mbox{dist}^2(0,\partial_{\eta} f(\bar{{x}}_{K^*})) ] \leq \epsilon$, the overall iteration/sample complexity is of  \mj{$\mathcal{O}\left(n^{1/2} K_\epsilon^{1+a}/\eta^{2/3}\right)$} where the magnitude of $K_\epsilon$ is given in (b-i).

\end{theorem}
\begin{proof}
\fy{\Rme{(a-i)} Invoking Lemma~\ref{lemma:descent lemma inexact}, for $\gamma\le \frac{\eta}{6L_0\mj{\sqrt{n}}}$ we have 
\begin{align*}
\mathbb{E}[f^\eta(\bar{{x}}_{k+1})\mid \mathcal{F}_k]&\le f^\eta(\bar{{x}}_{k})-\tfrac{\gamma}{4}\|\nabla f^\eta(\bar{{x}}_k)\|^2-\tfrac{\gamma}{4}\|\overline{\nabla f^\eta}(\mathbf{x}_k)\|^2+\mj{\left(\tfrac{L_0^2n \gamma}{2 m\eta^2}\right)}\|\mathbf{x}_k-\mathbf{1}\bar{x}_k\|^2\\
& +\mj{\left(\tfrac{24\sqrt{2\pi}L_0^3n^{3/2}}{m\eta}\right)}\gamma^2+\mj{\left(1+\mj{\tfrac{3\gamma L_0\sqrt{n}}{2\eta}}\right)\gamma\left(\tfrac{\tilde L_0^2n^2\varepsilon_k}{\eta^2}\right)}  .
\end{align*} 
Taking \uvs{unconditional} expectations on both sides and summing the inequality for \mee{$k=0,\ldots,K-1$}, we obtain
\begin{align*}
\mathbb{E}[f^\eta(\bar{{x}}_{K})]&\le \mathbb{E}[f^\eta(\bar{{x}}_{0})]-\tfrac{\gamma}{4}\sum_{k=0}^{K-1}\mathbb{E}[\|\nabla f^\eta(\bar{{x}}_{k})\|^2]-\tfrac{\gamma}{4}\sum_{k=0}^{K-1}\mathbb{E}[\|\overline{\nabla f^\eta}(\mathbf{x}_k)\|^2]
\\
&+\mj{\left(\tfrac{L_0^2n \gamma}{2 \eta^2}\right)}\sum_{k=0}^{K-1}\mathbb{E}\left[\tfrac{1}{m}\|\RRmath{{\mathbf x}_k}-\mathbf{1}\bar{{x}}_{k}\|^2\right]
+\mj{\left(\tfrac{24\sqrt{2\pi}L_0^3n^{3/2}}{m\eta}\right)}\gamma^2 K+\mj{ \left(1+\mj{\tfrac{3\gamma L_0\sqrt{n}}{2\eta}}\right)\gamma \left(\tfrac{\tilde L_0^2n^2}{\eta^2}\right)}\sum_{k=0}^{K-1}\varepsilon_k.
\end{align*} 
Noting that $\mathbb{E}[f^\eta(\bar{{x}}_{K})]\ge \inf_{x} f^\eta(x)$, we obtain
\begin{align*}
\sum_{k=0}^{K-1}\mathbb{E}[\|\nabla f^\eta(\bar{{x}}_{k})\|^2]&\le \tfrac{4(\mathbb{E}[f^\eta(\bar{{x}}_{0})]- \inf_{x} f^\eta(x))}{\gamma}-\sum_{k=0}^{K-1}\mathbb{E}[\|\overline{\nabla f^\eta}(\mathbf{x}_k)\|^2]+\mj{\tfrac{2 L_0^2n}{\eta^2}}\sum_{k=0}^{K-1}\mathbb{E}\left[\tfrac{1}{m}\|\RRmath{{\mathbf x}_k}-\mathbf{1}\bar{{x}}_{k}\|^2\right]\notag\\
&+\mj{\left(\tfrac{96\sqrt{2\pi}L_0^3n^{3/2}\gamma K}{m\eta}\right)} +\left(1+\mj{\tfrac{3\gamma L_0\sqrt{n}}{2\eta}}\right) \mj{\left(\tfrac{4 \tilde L_0^2n^2}{\eta^2}\right)}\sum_{k=0}^{K-1}\varepsilon_k.
\end{align*}}
\mje{
Invoking Lemma~\ref{lemma:bound for x-bar x inexact}, we obtain
 \begin{align}
\sum_{k=0}^{K-1}\mathbb{E}[\|\nabla f^\eta(\bar{{x}}_{k})\|^2]&\le \tfrac{4(\mathbb{E}[f^\eta(\bar{{x}}_{0})]- \inf_{x} f^\eta(x))}{\gamma}-\left(1-\left(\tfrac{4800\gamma^4\lambda_{\mathbf{W}}^4L_0^4\mj{n^2}}{\eta^4(1-\lambda_{\mathbf{W}}^2)^4}\right) 
\right)\sum_{k=0}^{K-1}\mathbb{E}[\|\overline{\nabla f^\eta}(\mathbf{x}_k)\|^2]\notag\\
&+\mj{\left(\tfrac{96\sqrt{2\pi}L_0^3n^{3/2}\gamma K}{m\eta}\right)}+\tfrac{2 L_0^2n}{\eta^2}\left(1+\tfrac{20}{1-\lambda_{\mathbf{W}}^2}\right)\mathbb{E}\left[\tfrac{1}{m}\|\mathbf{x}_{0}- \mathbf{1} \bar{x}_{0}\|^2 \right]\notag\\
&+\tfrac{320\gamma^2\lambda_{\mathbf{W}}^4L_0^2n}{(1-\lambda_{\mathbf{W}}^2)^3\eta^2} \left(\tfrac{1}{m}\|\nabla \mathbf f^\eta(\mathbf{x}_0)\|^2+ \mj{16\sqrt{2\pi}L_0^2n}+\mj{\left(\tfrac{\tilde L_0^2n^2\varepsilon_0}{\eta^2}\right)}\right)\notag\\
& + \left(\tfrac{160\gamma^2\lambda_{\mathbf{W}}^2L_0^4\mj{n^2K}}{\eta^2(1-\lambda_{\mathbf{W}}^2)^3}\right)\mj{32\sqrt{2\pi}\left(\|\mathbf{W}\|^2+\mee{\tfrac{11}{2}}\lambda_{\mathbf{W}}^2\right)\mj{}  }+\mj{\tfrac{25600K\sqrt{2\pi}n^3L_0^6\gamma^4\lambda_{\mathbf{W}}^4(3+\lambda_{\mathbf{W}}^2)}{m\eta^4\RRmath{(1-\lambda_{\mathbf{W}}^2)^4}}
}\notag \\
& + \left(\tfrac{160\gamma^2\lambda_{\mathbf{W}}^2L_0^4\mj{n^3}}{\eta^4(1-\lambda_{\mathbf{W}}^2)^3}\right)\left( \mj{\tfrac{2\|\mathbf{W}\|^2\tilde L_0^2}{L_0^2}}+\mj{ \left(\tfrac{5\lambda_{\mathbf{W}}^2\tilde L_0^2}{L_0^2}\right)\left(\tfrac{1+3\lambda_{\mathbf{W}}^2}{1-\lambda_{\mathbf{W}}^2}\right)}+\RRmath{\tfrac{30\lambda_{\mathbf{W}}^2\gamma^2\tilde L_0^2 {n}      
}{\eta^2(1-\lambda_{\mathbf{W}}^2)}}\right)\sum_{k=0}^{K-1}\varepsilon_k\notag\\
& + \left(\tfrac{160\gamma^2\lambda_{\mathbf{W}}^2L_0^4\mj{n^3}}{\eta^4(1-\lambda_{\mathbf{W}}^2)^3}\right)\left( \mj{ \left(\tfrac{5\lambda_{\mathbf{W}}^2\tilde L_0^2}{L_0^2}\right)\left(\tfrac{1+3\lambda_{\mathbf{W}}^2}{1-\lambda_{\mathbf{W}}^2}\right)}\right)\sum_{k=0}^{K-1}\varepsilon_{k+1}\notag\\
&+\left(1+\mj{\tfrac{3\gamma L_0\sqrt{n}}{2\eta}}\right) \mj{\left(\tfrac{4\tilde L_0^2n^2}{\eta^2}\right)}\sum_{k=0}^{K-1}\varepsilon_k.\label{equation: the last equation in the proof of (a-i)}
\end{align}
By setting $\gamma\le \frac{\eta(1-\lambda_{\mathbf{W}}^2)}{9L_0\sqrt{n}\lambda_{\mathbf{W}}}$, we have $\left(1-\left(\tfrac{4800\gamma^4\lambda_{\mathbf{W}}^4L_0^4\mj{n^2}}{\eta^4(1-\lambda_{\mathbf{W}}^2)^4}\right) \right)\ge\Rme{\tfrac{1}{4}>} 0$. Dropping the \RRme{nonpositive} term, dividing both sides by $K$, and invoking \Rme{Proposition~\ref{Prop:2eta}\noindent(ii)}, we obtain the result in \Rme{(a-i)}. 

\Rme{\noindent (a-ii) Rearranging the terms in~\eqref{equation: the last equation in the proof of (a-i)}
and choosing $\gamma \le \frac{\eta(1-\lambda_{\mathbf{W}}^2)}{9L_0\sqrt{n}\lambda_{\mathbf{W}}}$, it follows that
$1-\frac{4800\gamma^4\lambda_{\mathbf{W}}^4L_0^4 n^2}
{\eta^4(1-\lambda_{\mathbf{W}}^2)^4} \ge \frac{1}{4}$.
Moreover, since $\sum_{k=0}^{K-1}\mathbb{E}[\|\nabla f^\eta(\bar{x}_{k})\|^2]\ge 0$, we obtain
 \begin{align}
\sum_{k=0}^{K-1}\mathbb{E}[\|\overline{\nabla f^\eta}(\mathbf{x}_k)\|^2]&\le \tfrac{16(\mathbb{E}[f^\eta(\bar{{x}}_{0})]- \inf_{x} f^\eta(x))}{\gamma}\notag\\
&+{\left(\tfrac{384\sqrt{2\pi}L_0^3n^{3/2}\gamma K}{m\eta}\right)}+\tfrac{8 L_0^2n}{\eta^2}\left(1+\tfrac{20}{1-\lambda_{\mathbf{W}}^2}\right)\mathbb{E}\left[\tfrac{1}{m}\|\mathbf{x}_{0}- \mathbf{1} \bar{x}_{0}\|^2 \right]\notag\\
&+\tfrac{1280\gamma^2\lambda_{\mathbf{W}}^4L_0^2n}{(1-\lambda_{\mathbf{W}}^2)^3\eta^2} \left(\tfrac{1}{m}\|\nabla \mathbf f^\eta(\mathbf{x}_0)\|^2+ {16\sqrt{2\pi}L_0^2n}+{\left(\tfrac{\tilde L_0^2n^2\varepsilon_0}{\eta^2}\right)}\right)\notag\\
& + \left(\tfrac{160\gamma^2\lambda_{\mathbf{W}}^2L_0^4{n^2K}}{\eta^2(1-\lambda_{\mathbf{W}}^2)^3}\right){128\sqrt{2\pi}\left(\|\mathbf{W}\|^2+{\tfrac{11}{2}}\lambda_{\mathbf{W}}^2\right){}  }+{\tfrac{102400K\sqrt{2\pi}n^3L_0^6\gamma^4\lambda_{\mathbf{W}}^4(3+\lambda_{\mathbf{W}}^2)}{m\eta^4\RRmath{(1-\lambda_{\mathbf{W}}^2)^4}}
}\notag \\
& + \left(\tfrac{640\gamma^2\lambda_{\mathbf{W}}^2L_0^4{n^3}}{\eta^4(1-\lambda_{\mathbf{W}}^2)^3}\right)\left( {\tfrac{2\|\mathbf{W}\|^2\tilde L_0^2}{L_0^2}}+{ \left(\tfrac{5\lambda_{\mathbf{W}}^2\tilde L_0^2}{L_0^2}\right)\left(\tfrac{1+3\lambda_{\mathbf{W}}^2}{1-\lambda_{\mathbf{W}}^2}\right)}+{\RRmath{\tfrac{30\lambda_{\mathbf{W}}^2\gamma^2\tilde L_0^2 {n}      
}{\eta^2(1-\lambda_{\mathbf{W}}^2)}}}\right)\sum_{k=0}^{K-1}\varepsilon_k\notag\\
& + \left(\tfrac{640\gamma^2\lambda_{\mathbf{W}}^2L_0^4{n^3}}{\eta^4(1-\lambda_{\mathbf{W}}^2)^3}\right)\left( { \left(\tfrac{5\lambda_{\mathbf{W}}^2\tilde L_0^2}{L_0^2}\right)\left(\tfrac{1+3\lambda_{\mathbf{W}}^2}{1-\lambda_{\mathbf{W}}^2}\right)}\right)\sum_{k=0}^{K-1}\varepsilon_{k+1}\notag\\
&+\left(1+{\tfrac{3\gamma L_0\sqrt{n}}{2\eta}}\right) {\left(\tfrac{16\tilde L_0^2n^2}{\eta^2}\right)}\sum_{k=0}^{K-1}\varepsilon_k.\notag
\end{align}
Substituting the preceding bound into Lemma~\ref{lemma:bound for x-bar x inexact}, setting $H_{n,\eta,\gamma}\triangleq \left(1+\tfrac{8 L_0^2n}{\eta^2}\left(\tfrac{80\gamma^2\lambda_{\mathbf{W}}^2L_0^2{n}}{\eta^2(1-\lambda_{\mathbf{W}}^2)^3}\right) 
\left( \tfrac{30  \gamma^2\lambda_{\mathbf{W}}^2}{1-\lambda_{\mathbf{W}}^2} \right)\right) $ and dividing both sides by $K$, we obtain the result.}

\noindent (b-i)  Using the relation in (a) and that $\sum_{k=0}^{K-1}\varepsilon_{k+1}\le \sum_{k=0}^{K-1}\varepsilon_k$, for $\gamma = \frac{\eta^{2/3}}{\sqrt{n^{3/2}K}L_0^{3/2}}$,  we obtain
 \begin{align}
\mathbb{E}[\mbox{dist}^2(0,\partial_{\eta} f(\bar{{x}}_{K^*})) ] & =  
\mathcal{O}\left(\left(\tfrac{{n^{3/4}}L_0^{3/2}}{\eta^{2/3}}\right)\tfrac{1}{\sqrt{K}}+\left(\tfrac{L_0^{3/2}n^{3/4}}{\eta^{1/3}m}\right)\tfrac{1}{\sqrt{K}}+\left(\tfrac{L_0^2n\mathbb{E}\left[\tfrac{1}{m}\|\mathbf{x}_{0}- \mathbf{1} \bar{x}_{0}\|^2 \right]}{\eta^2(1-\lambda_{\mathbf{W}}^2)}\right)\tfrac{1}{K} \right.\notag\\
&\left. + \left(\tfrac{\lambda_{\mathbf{W}}^4L_0n^{1/2}}{\eta^{2/3}(1-\lambda_{\mathbf{W}}^2)^3}\right)\tfrac{1}{K^2}+ \left(\tfrac{\lambda_{\mathbf{W}}^4\tilde L_0^2n^{3/2}\varepsilon_0}{L_0\eta^{8/3}(1-\lambda_{\mathbf{W}}^2)^3}\right)\tfrac{1}{K^2}+\left(\tfrac{\lambda_{\mathbf{W}}^2L_0n^{1/2}( \lambda_{\mathbf{W}}^2+     \|\mathbf{W}\|^2)}{\eta^{2/3}(1-\lambda_{\mathbf{W}}^2)^3}\right)\tfrac{1}{K} \right.\notag\\
&\left.+\left(\mj{\tfrac{\lambda_{\mathbf{W}}^4(3+\lambda_{\mathbf{W}}^2)}{\eta^{4/3}m\RRmath{(1-\lambda_{\mathbf{W}}^2)^4}}}\right) \tfrac{1}{K^2}+ \left(\tfrac{ \lambda_{\mathbf{W}}^2n^{3/2}}{L_0\eta^{8/3}(1-\lambda_{\mathbf{W}}^2)^3 }\right)\left( {\|\mathbf{W}\|^2\tilde L_0^2} +  \tfrac{\lambda_{\mathbf{W}}^4\tilde L_0^2}{(1-\lambda_{\mathbf{W}}^2)} \right) \tfrac{\sum_{k=0}^{K-1}\varepsilon_k}{K^{2}}\right.\notag\\
&+  \left. \left(\tfrac{ \tilde L_0^2  n \RRmath{\lambda_{\mathbf{W}}^4}}{L_0^2\eta^{10/3}\RRmath{(1-\lambda_{\mathbf{W}}^2) ^4}}\right)
 \tfrac{\sum_{k=0}^{K-1}\varepsilon_k}{K^{3}}+ \left(\tfrac{ 4\tilde L_0^2n^{2}}{\eta^2}\right) \tfrac{\sum_{k=0}^{K-1}\varepsilon_k}{K}+\left(\tfrac{ \tilde L_0^2n^{7/4}}{L_0^{1/2}\eta^{7/3}}\right) \tfrac{\sum_{k=0}^{K-1}\varepsilon_k}{K^{3/2}} \right).\label{eq:main one after replacing}
\end{align}
Invoking Lemma~\ref{lem:Alg2_conv} with \fyr{$t_k:=\left\lceil \left(n^{1/2}(k+\RRme{\hat{\Gamma}})\right)/\left(\eta^{2/3}\right)\right\rceil$} and using the definition of $\kappa_F$, we obtain 
\Rme{ \begin{align}
&\mathbb{E}[\mbox{dist}^2(0,\partial_{\eta} f(\bar{{x}}_{K^*})) ]  =  
\mathcal{O}\left(\left(\tfrac{{n^{3/4}}L_0^{3/2}}{\eta^{2/3}}\right)\tfrac{1}{\sqrt{K}}+\left(\tfrac{L_0^{3/2}n^{3/4}}{\eta^{1/3}m}\right)\tfrac{1}{\sqrt{K}}+\left(\tfrac{L_0^2n\mathbb{E}\left[\tfrac{1}{m}\|\mathbf{x}_{0}- \mathbf{1} \bar{x}_{0}\|^2 \right]}{\eta^2(1-\lambda_{\mathbf{W}}^2)}\right)\tfrac{1}{K} \right.\notag\\
&\left. + \left(\tfrac{\lambda_{\mathbf{W}}^4L_0n^{1/2}}{\eta^{2/3}(1-\lambda_{\mathbf{W}}^2)^3}\right)\tfrac{1}{K^2}+ \left(\tfrac{\lambda_{\mathbf{W}}^4\tilde L_0^2n^{3/2}\varepsilon_0}{L_0\eta^{8/3}(1-\lambda_{\mathbf{W}}^2)^3}\right)\tfrac{1}{K^2}+\left(\tfrac{\lambda_{\mathbf{W}}^2L_0n^{1/2}( \lambda_{\mathbf{W}}^2+     \|\mathbf{W}\|^2)}{\eta^{2/3}(1-\lambda_{\mathbf{W}}^2)^3}\right)\tfrac{1}{K} \right.\notag\\
&\left.+\left({\tfrac{\lambda_{\mathbf{W}}^4(3+\lambda_{\mathbf{W}}^2)}{\eta^{4/3}m\RRme{(1-\lambda_{\mathbf{W}}^2)^4}}}\right) \tfrac{1}{K^2}+ \left(\tfrac{ \lambda_{\mathbf{W}}^2n}{L_0\eta^{2}(1-\lambda_{\mathbf{W}}^2)^3 }\right)\left( {\|\mathbf{W}\|^2\tilde L_0^2} +  \tfrac{\lambda_{\mathbf{W}}^4\tilde L_0^2}{(1-\lambda_{\mathbf{W}}^2)} \right) \tfrac{\sum_{k=0}^{K-1}\tfrac{\RRme{\kappa_F^2}}{(k+\RRme{\hat{\Gamma}})^a}}{K^{2}}\right.\notag\\
& \left.  +\left(\tfrac{ \tilde L_0^2n^{3/2}}{\eta^{4/3}}\right) \tfrac{\sum_{k=0}^{K-1}\tfrac{\RRme{\kappa_F^2}}{(k+\RRme{\hat{\Gamma}})^a}}{K}+\RRme{\left(\tfrac{ \tilde L_0^2  n^{1/2} \lambda_{\mathbf{W}}^4}{L_0^2\eta^{8/3}(1-\lambda_{\mathbf{W}}^2) ^4 }\right)}
 \tfrac{\sum_{k=0}^{K-1}\tfrac{\RRme{\kappa_F^2}}{(k+\RRme{\hat{\Gamma}})^a}}{K^{3}}+\left(\tfrac{ \tilde L_0^2n^{5/4}}{L_0^{1/2}\eta^{5/3}}\right) \tfrac{\sum_{k=0}^{K-1}\tfrac{\RRme{\kappa_F^2}}{(k+\RRme{\hat{\Gamma}})^a}}{K^{3/2}} \right).\label{eq:proof of complexity bound}
\end{align}}\fy{
By invoking \Rme{Lemma~\ref{lem:Harmonic series bounds}\noindent(i)} for the case where $a\in (0.5,1)$, we obtain the result in (1). To show (2), by substituting $a$ in~\eqref{eq:proof of complexity bound} by 1 and applying \Rme{Lemma~\ref{lem:Harmonic series bounds}\noindent(ii)}, we obtain the result in (2). The result in part (3) can be done in a similar fashion.

\noindent (b-ii) In the inexact setting, at iteration $k$ of Algorithm~\ref{alg:DZGT}, $2 t_k$ iterations of Algorithm~\ref{alg:lowerlevel-1stage} \mee{are} needed by each agent to approximate the lower-level solutions. Thus, the overall iteration complexity is given by $\sum_{k=0}^{\RRmath{K_\epsilon-1}} 2m t_k$. Substituting $K_\epsilon$ by \eqref{eq:complexity results for 1s} and $t_k$ by $\left\lceil \left(n^{1/2}(k+\RRme{\hat{\Gamma}})^a\right)/\left(\eta^{2/3}\right)\right\rceil$, we obtain the overall complexity in (b-ii). Note that the overall sample complexity matches that of the overall iteration complexity. This is because, at each iteration of both the lower-level and upper-level schemes, one sample per agent is taken.   }
}
\end{proof}
\mje{\begin{remark}\label{rem:1s-a}\em We consider three cases for the parameter $a$ in the previous theorem. Setting $a=1$ yields the best choice for minimizing the overall iteration complexity. We elaborate on this in the following. From part (b-ii) of Theorem~\ref{Theorem:thm 1} we know that the overall iteration/sample complexity is \mj{$\mathcal{O}\left(n^{1/2} K_\epsilon^{1+a}/\eta^{2/3}\right)$}. Theorem~\ref{Theorem:thm 1} provides three different cases for $a$ with the only distinction being the term $\mee{J_\epsilon}$. Let us denote the relevant part of $K_\epsilon$ as $K_{\epsilon,\mee{J_\epsilon}}$. The complexity $\mathcal{O}\left(K_{\epsilon,\mee{J_\epsilon}}^{1+a}\right)$ now varies across the different cases for $a$ as follows. (1) If $a \in \left(\mee{\tfrac{1}{2}},1\right)$, then $ \mathcal{O}\left(K_{\epsilon,\mee{J_\epsilon}}^{1+a}\right)= \mathcal{O}\left(\varepsilon^{-1}+\varepsilon^{{-(1+a)}/{(2+a)}}+\varepsilon^{{-(1+a)}/{\left(\mee{\tfrac{1}{2}}+a\right)}}\right).$ (2)  If $a =1$, then $
  \mathcal{O}\left(K_{\epsilon,\mee{J_\epsilon}}^{1+a}\right)= \mathcal{O}\left(\varepsilon^{-1}+\varepsilon^{{-2}/{3}}+\varepsilon^{{-4}/{3}}\right). $  (3)  If $a >1$, then $
  \mathcal{O}\left(K_{\epsilon,\mee{J_\epsilon}}^{1+a}\right)= \mathcal{O}\left(\varepsilon^{-(1+a)/2}+\varepsilon^{{-(1+a)}/{3}}+\varepsilon^{{-2(1+a)}/{3}}\right)$.  Among the three cases above, the minimum value of $\mathcal{O}\left(K_\epsilon^{1+a}\right)$ occurs when $a=1$. \mee{$\hfill \Box$}
\end{remark}}

\begin{remark}\em
\RRme{In Theorem~\ref{Theorem:thm 1}, we consider the case } $\RRmath{\lambda_{\mathbf{W}}>0}$ \RRme{ because the stepsize condition and the intermediate network-dependent bounds contain inverse powers of} $\RRmath{\lambda_{\mathbf{W}}}$\RRme{. The case} $\RRmath{\lambda_{\mathbf{W}}=0}$ \RRme{ is a special case allowed by Assumption~\ref{assum:mixxx}. In this case,} $\RRmath{\mathbf{W}=\tfrac{1}{m}\mathbf{1}\mathbf{1}^{\top}}$\RRme{, and, by Lemma~\ref{lemma:bound for x-bar x inexact}, the iterates reach exact consensus after the first update. Consequently, the convergence analysis can be obtained directly by using the bound in Lemma~\ref{lemma:bound for x-bar x inexact} for} $\RRmath{\lambda_{\mathbf{W}}=0}$ \RRme{ in the proof of Theorem~\ref{Theorem:thm 1}. The resulting error and complexity bounds are stated in the following result.} $\hfill\Box$
\end{remark}

\begin{corollary}[\RRme{\textbf{Convergence guarantees for DiZS-GT$^{\text{1s}}$ when } $\RRmath{\lambda_{\mathbf{W}}=0}$}]\em
\label{cor:thm1-lambda-zero}
\RRme{Consider Algorithm~\ref{alg:DZGT}. Let Assumptions~\ref{assum:mixxx}, \ref{assump:opt_f_bounded_below}, \ref{assum:main1}, \ref{assum:main2}, \RRme{\ref{assum:alg2}, and \ref{assum:samples}} hold, and suppose} $\RRmath{\lambda_{\mathbf{W}}=0}$\RRme{.}

\noindent\RRme{(a) [\textbf{Non-asymptotic error bounds}] Suppose the stepsize} $\RRmath{\gamma}$ \RRme{is constant and satisfies} $\RRmath{\gamma\leq\tfrac{\eta}{6L_0\sqrt{n}}}$\RRme{. Suppose} $\RRmath{K^*}$ \RRme{is a discrete uniform random variable such that} $\RRmath{\mathbb{P}[K^*=\ell]=\tfrac{1}{K}}$ \RRme{for} $\RRmath{\ell=0,\ldots,K-1}$\RRme{. Then, for any} $\RRmath{K\geq 1}$\RRme{, the following hold.}

\noindent\RRme{(a-i) [\textbf{Stationarity error bound}] We have}
\begin{align}
\RRmath{
\mathbb{E}[\mbox{dist}^2(0,\partial_{\eta} f(\bar{x}_{K^*}))]
&\leq
\tfrac{4(\mathbb{E}[f^\eta(\bar{x}_0)]-\inf_x f^\eta(x))}{\gamma K}
+\tfrac{96\sqrt{2\pi}L_0^3n^{3/2}\gamma}{m\eta}
+\tfrac{2L_0^2n}{\eta^2K}
\mathbb{E}\left[
\tfrac{\|\mathbf{x}_0-\mathbf{1}\bar{x}_0\|^2}{m}
\right] \notag\\
&
+\left(1+\tfrac{3\gamma L_0\sqrt{n}}{2\eta}\right)
\left(\tfrac{4\tilde L_0^2n^2}{\eta^2}\right)
\tfrac{\sum_{k=0}^{K-1}\varepsilon_k}{K}.
}
\label{eq:lambda-zero-stationarity-1s}
\end{align}

\noindent\RRme{(a-ii) [\textbf{Consensus error bound}] We have}
\begin{equation}
\RRmath{
\mathbb{E}\left[
\tfrac{\|\mathbf{x}_{K^*}-\mathbf{1}\bar{x}_{K^*}\|^2}{m}
\right]
=
\tfrac{1}{K}
\mathbb{E}\left[
\tfrac{\|\mathbf{x}_0-\mathbf{1}\bar{x}_0\|^2}{m}
\right].
}
\label{eq:lambda-zero-consensus-1s}
\end{equation}

\noindent\RRme{(b) [\textbf{Complexity bounds}] Suppose that, at iteration} $\RRmath{k}$ \RRme{of Algorithm~\ref{alg:DZGT}, in generating the inexact solutions, Algorithm~\ref{alg:lowerlevel-1stage} is terminated after} $\RRmath{t_k=\left\lceil \left(n^{1/2}(k+\hat\Gamma)^a\right)/\eta^{2/3}\right\rceil}$ \RRme{iterations, where} $\RRmath{a>0.5}$\RRme{. Let} $\RRmath{\epsilon>0}$ \RRme{be an arbitrary scalar such that} $\RRmath{\mathbb{E}[\mbox{dist}^2(0,\partial_{\eta} f(\bar{x}_{K_\epsilon^*}))]\leq\epsilon}$\RRme{. Then, the following results hold.}

\noindent\RRme{(b-i) [\textbf{Upper-level iteration/sample complexity}] Suppose the stepsize is given by} $\RRmath{\gamma=\tfrac{\eta^{2/3}}{\sqrt{n^{3/2}K_\epsilon}L_0^{3/2}}}$ \RRme{and} $\RRmath{K_\epsilon\geq\max\left\{1,\tfrac{36}{\eta^{2/3}L_0n^{1/2}}\right\}}$\RRme{. Then,}
\begin{align}
\RRmath{
K_\epsilon
=
\mathcal{O}\left(
\left(\tfrac{n^{3/2}L_0^3}{\eta^{4/3}}\right)\epsilon^{-2}
+
\left(\tfrac{L_0^3n^{3/2}}{\eta^{2/3}m^2}\right)\epsilon^{-2}
+
\left(
\tfrac{L_0^2n\mathbb{E}\left[\tfrac{1}{m}\|\mathbf{x}_0-\mathbf{1}\bar{x}_0\|^2\right]}
{\eta^2}
\right)\epsilon^{-1}
\right)
+J_{\epsilon}',
}
\label{eq:lambda-zero-complexity-1s}
\end{align}
\RRme{where} $\RRmath{J_{\epsilon}'}$ \RRme{is characterized as follows.}

\noindent\RRme{(1) If} $\RRmath{a\in(\tfrac{1}{2},1)}$\RRme{, then}
\begin{align*}
\RRmath{
J_{\epsilon}'
=
\mathcal{O}\left(
\left(
\tfrac{\kappa_F^2\tilde L_0^2n^{3/2}}
{\eta^{4/3}(1-a)}
\right)^{1/a}
\epsilon^{-1/a}
+
\left(
\tfrac{\kappa_F^2\tilde L_0^2n^{5/4}}
{L_0^{1/2}\eta^{5/3}(1-a)}
\right)^{1/(a+\tfrac{1}{2})}
\epsilon^{-1/(a+\tfrac{1}{2})}
\right).
}
\end{align*}

\noindent\RRme{(2) If} $\RRmath{a=1}$\RRme{, then}
\begin{align*}
\RRmath{
J_{\epsilon}'
=
\tilde{\mathcal{O}}\left(
\left(
\tfrac{\kappa_F^2\tilde L_0^2n^{5/4}}
{L_0^{1/2}\eta^{5/3}}
\right)^{2/3}
\epsilon^{-2/3}
+
\left(
\tfrac{\kappa_F^2\tilde L_0^2n^{3/2}}
{\eta^{4/3}}
\right)
\epsilon^{-1}
\right).
}
\end{align*}

\noindent\RRme{(3) If} $\RRmath{a>1}$\RRme{, then}
\begin{align*}
\RRmath{
J_{\epsilon}'
=
\mathcal{O}\left(
\left(
\tfrac{a\kappa_F^2\tilde L_0^2n^{5/4}}
{L_0^{1/2}\eta^{5/3}(a-1)}
\right)^{2/3}
\epsilon^{-2/3}
+
\left(
\tfrac{a\kappa_F^2\tilde L_0^2n^{3/2}}
{\eta^{4/3}(a-1)}
\right)
\epsilon^{-1}
\right).
}
\end{align*}

\noindent\RRme{(b-ii) [\textbf{Overall iteration/sample complexity}] To guarantee} $\RRmath{\mathbb{E}[\mbox{dist}^2(0,\partial_{\eta} f(\bar{x}_{K_\epsilon^*}))]\leq\epsilon}$\RRme{, the overall iteration/sample complexity is} $\RRmath{\mathcal{O}\left(n^{1/2}K_\epsilon^{1+a}/\eta^{2/3}\right)}$\RRme{, where the magnitude of} $\RRmath{K_\epsilon}$ \RRme{is given in (b-i).}
\end{corollary}

\begin{proof}
\RRme{For part (a-i), invoking Lemma~\ref{lemma:descent lemma inexact}, for} $\RRmath{\gamma\leq\tfrac{\eta}{6L_0\sqrt{n}}}$\RRme{, and following the proof of Theorem~\ref{Theorem:thm 1}, we obtain}
\begin{align*}
\RRmath{
\sum_{k=0}^{K-1}\mathbb{E}[\|\nabla f^\eta(\bar{x}_k)\|^2]
&\leq
\tfrac{4(\mathbb{E}[f^\eta(\bar{x}_0)]-\inf_x f^\eta(x))}{\gamma}
-\sum_{k=0}^{K-1}\mathbb{E}[\|\overline{\nabla f^\eta}(\mathbf{x}_k)\|^2]
+\tfrac{96\sqrt{2\pi}L_0^3n^{3/2}\gamma K}{m\eta} \\
&
+\tfrac{2L_0^2n}{\eta^2}
\sum_{k=0}^{K-1}
\mathbb{E}\left[
\tfrac{\|\mathbf{x}_k-\mathbf{1}\bar{x}_k\|^2}{m}
\right]
+\left(1+\tfrac{3\gamma L_0\sqrt{n}}{2\eta}\right)
\left(\tfrac{4\tilde L_0^2n^2}{\eta^2}\right)
\sum_{k=0}^{K-1}\varepsilon_k.
}
\end{align*}
\RRme{Since} $\RRmath{\lambda_{\mathbf{W}}=0}$\RRme{, Lemma~\ref{lemma:bound for x-bar x inexact} gives} $\RRmath{\sum_{k=0}^{K-1}\mathbb{E}\left[\tfrac{\|\mathbf{x}_k-\mathbf{1}\bar{x}_k\|^2}{m}\right]=\mathbb{E}\left[\tfrac{\|\mathbf{x}_0-\mathbf{1}\bar{x}_0\|^2}{m}\right]}$\RRme{. Substituting this equality into the preceding relation, dropping the \RRme{nonpositive} term} $\RRmath{-\sum_{k=0}^{K-1}\mathbb{E}[\|\overline{\nabla f^\eta}(\mathbf{x}_k)\|^2]}$\RRme{, dividing both sides by} $\RRmath{K}$\RRme{, and invoking Proposition~\ref{Prop:2eta}(ii), we obtain~\eqref{eq:lambda-zero-stationarity-1s}.}

\RRme{For part (a-ii), since} $\RRmath{K^*}$ \RRme{is uniformly distributed over} $\RRmath{\{0,\ldots,K-1\}}$\RRme{, we have} $\RRmath{\mathbb{E}\left[\tfrac{\|\mathbf{x}_{K^*}-\mathbf{1}\bar{x}_{K^*}\|^2}{m}\right]=\tfrac{1}{K}\sum_{k=0}^{K-1}\mathbb{E}\left[\tfrac{\|\mathbf{x}_k-\mathbf{1}\bar{x}_k\|^2}{m}\right]}$\RRme{. Applying Lemma~\ref{lemma:bound for x-bar x inexact} for} $\RRmath{\lambda_{\mathbf{W}}=0}$ \RRme{yields} $\RRmath{\mathbb{E}\left[\tfrac{\|\mathbf{x}_{K^*}-\mathbf{1}\bar{x}_{K^*}\|^2}{m}\right]=\tfrac{1}{K}\mathbb{E}\left[\tfrac{\|\mathbf{x}_0-\mathbf{1}\bar{x}_0\|^2}{m}\right]}$\RRme{, proving~\eqref{eq:lambda-zero-consensus-1s}.}

\RRme{For part (b-i), setting} $\RRmath{\gamma=\tfrac{\eta^{2/3}}{\sqrt{n^{3/2}K}L_0^{3/2}}}$ \RRme{in~\eqref{eq:lambda-zero-stationarity-1s} gives}
\begin{align*}
\RRmath{
\mathbb{E}[\mbox{dist}^2(0,\partial_{\eta} f(\bar{x}_{K^*}))]
&=
\mathcal{O}\Bigg(
\left(\tfrac{n^{3/4}L_0^{3/2}}{\eta^{2/3}}\right)\tfrac{1}{\sqrt{K}}
+
\left(\tfrac{L_0^{3/2}n^{3/4}}{\eta^{1/3}m}\right)\tfrac{1}{\sqrt{K}}
+
\left(
\tfrac{L_0^2n\mathbb{E}\left[\tfrac{1}{m}\|\mathbf{x}_0-\mathbf{1}\bar{x}_0\|^2\right]}
{\eta^2}
\right)\tfrac{1}{K} \\
&+
\left(\tfrac{\tilde L_0^2n^2}{\eta^2}\right)
\tfrac{\sum_{k=0}^{K-1}\varepsilon_k}{K}
+
\left(\tfrac{\tilde L_0^2n^{7/4}}{L_0^{1/2}\eta^{7/3}}\right)
\tfrac{\sum_{k=0}^{K-1}\varepsilon_k}{K^{3/2}}
\Bigg).
}
\end{align*}
\RRme{Invoking Lemma~\ref{lem:Alg2_conv} with} $\RRmath{t_k=\left\lceil n^{1/2}(k+\Gamma)^a/\eta^{2/3}\right\rceil}$ \RRme{and using the definition of} $\RRmath{\kappa_F}$\RRme{, we obtain}
\begin{align*}
\RRmath{
\mathbb{E}[\mbox{dist}^2(0,\partial_{\eta} f(\bar{x}_{K^*}))]
&=
\mathcal{O}\Bigg(
\left(\tfrac{n^{3/4}L_0^{3/2}}{\eta^{2/3}}\right)\tfrac{1}{\sqrt{K}}
+
\left(\tfrac{L_0^{3/2}n^{3/4}}{\eta^{1/3}m}\right)\tfrac{1}{\sqrt{K}}
+
\left(
\tfrac{L_0^2n\mathbb{E}\left[\tfrac{1}{m}\|\mathbf{x}_0-\mathbf{1}\bar{x}_0\|^2\right]}
{\eta^2}
\right)\tfrac{1}{K} \\
&+
\left(\tfrac{\tilde L_0^2n^{3/2}}{\eta^{4/3}}\right)
\tfrac{\sum_{k=0}^{K-1}\tfrac{\kappa_F}{(k+\Gamma)^a}}{K}
+
\left(\tfrac{\tilde L_0^2n^{5/4}}{L_0^{1/2}\eta^{5/3}}\right)
\tfrac{\sum_{k=0}^{K-1}\tfrac{\kappa_F}{(k+\Gamma)^a}}{K^{3/2}}
\Bigg).
}
\end{align*}
\RRme{If} $\RRmath{a\in(\tfrac{1}{2},1)}$\RRme{, invoking Lemma~\ref{lem:Harmonic series bounds}(i) and requiring each term on the right-hand side to be of order} $\RRmath{\epsilon}$ \RRme{gives the expression for} $\RRmath{J_\epsilon'}$ \RRme{in case (1). If} $\RRmath{a=1}$\RRme{, Lemma~\ref{lem:Harmonic series bounds}(ii) gives case (2). If} $\RRmath{a>1}$\RRme{, the corresponding bound on the harmonic sum gives case (3). Combining these bounds yields~\eqref{eq:lambda-zero-complexity-1s}.}

\noindent\RRme{Finally, for part (b-ii), at iteration} $\RRmath{k}$ \RRme{of Algorithm~\ref{alg:DZGT},} $\RRmath{2t_k}$ \RRme{iterations of Algorithm~\ref{alg:lowerlevel-1stage} are required by each agent to generate the two inexact lower-level solutions. Therefore, the overall iteration complexity is proportional to} $\RRmath{\sum_{k=0}^{K_\epsilon}2m t_k}$\RRme{. Substituting} $\RRmath{t_k=\left\lceil n^{1/2}(k+\Gamma)^a/\eta^{2/3}\right\rceil}$ \RRme{yields the overall iteration/sample complexity} $\RRmath{\mathcal{O}\left(n^{1/2}K_\epsilon^{1+a}/\eta^{2/3}\right)}$\RRme{. As in Theorem~\ref{Theorem:thm 1}, the overall sample complexity matches the overall iteration complexity because one sample per agent is generated at each iteration of both the lower-level and upper-level schemes.}
\end{proof}

\fy{In the following corollary, we derive the complexity guarantees for the exact setting, where every agent has access to the exact solution to its local lower-level problem.
\begin{corollary}[\mee{\textbf{Convergence guarantees for DiZS-GT — Exact case}}]\em \label{corollary:corollary 1} Consider Algorithm~\ref{alg:DZGT}. Let Assumptions~\ref{assum:mixxx}, \ref{assump:opt_f_bounded_below}, \ref{assum:main1}, \ref{assum:main2}\RRme{, and~\ref{assum:samples}} hold. Suppose an exact solution to the lower-level problem is available \RRme{and} $\RRmath{\lambda_{\mathbf{W}}>0}$.

 \noindent (a) [\mee{\textbf{Non-asymptotic error bound}}] Suppose the stepsize is constant such that $\gamma \leq \min \Big\{
 \tfrac{ \sqrt{1-\lambda_{\mathbf{W}}^2}}{6\sqrt{3}\lambda_{\mathbf{W}}^2 }
 ,\tfrac{(1-\lambda_{\mathbf{W}}^2) }{12\lambda_{\mathbf{W}}^3 },\tfrac{ (1-\lambda_{\mathbf{W}}^2)^2}{16\lambda_{\mathbf{W}}^2  },\frac{1}{6},\frac{(1-\lambda_{\mathbf{W}}^2)}{6\lambda_{\mathbf{W}}}\Big\} \left(\tfrac{\eta}{\sqrt{n} L_0 }\right)$. Suppose $K^*$ is a discrete uniform random variable where $\mathbb{P}[K^*=\ell] = \tfrac{1}{K}$ for $\ell=0,\ldots,K-1$. Then, for any $K \geq 1$
 \begin{align}
\mathbb{E}[\mbox{dist}^2(0,\partial_{\eta} f(\bar{{x}}_{K^*})) ]&\le\tfrac{2(\mathbb{E}[f^\eta(\bar{{x}}_{0})]- \inf_{x} f^\eta(x))}{\gamma K}+\tfrac{ L_0^2n}{\eta^2K}\left(1+\tfrac{13}{1-\lambda_{\mathbf{W}}^2}\right)\mathbb{E}\left[\tfrac{1}{m}\|\mathbf{x}_{0}- \mathbf{1} \bar{x}_{0}\|^2 \right]\notag\\
&+\tfrac{52\gamma^2\lambda_{\mathbf{W}}^4L_0^2n}{\eta^2(1-\lambda_{\mathbf{W}}^2)^3K} \left(\tfrac{1}{m}\|\nabla \mathbf f^\eta(\mathbf{x}_0)\|^2+ 16\sqrt{2\pi}L_0^2n\right)+\tfrac{832\sqrt{2\pi}\gamma^2\lambda_{\mathbf{W}}^2L_0^4n^2\left(\lambda_{\mathbf{W}}^2 +   \|\mathbf{W}\|^2   \right)}{\eta^2(1-\lambda_{\mathbf{W}}^2)^3}\notag \\
&+\left(\tfrac{32\sqrt{2\pi}n^{3/2}L_0^3\gamma }{m\eta}\right)+\tfrac{2496\sqrt{2\pi}\lambda_{\mathbf{W}}^4\gamma^4 n^3L_0^6(1+3\lambda_{\mathbf{W}}^2)}{m\eta^4(1-\lambda_{\mathbf{W}}^2)^4} +\tfrac{4992\sqrt{2\pi}\gamma^2\lambda_{\mathbf{W}}^4L_0^4n^2}{\eta^2(1-\lambda_{\mathbf{W}}^2)^3}.\label{cor:1s_err_bound}
\end{align}

 \noindent (b) [\mee{\textbf{Iteration/sample complexity}}]  Let $\epsilon>0$ be an arbitrary scalar such that $\mathbb{E}[\mbox{dist}^2(0,\partial_{\eta} f(\bar{{x}}_{K_\epsilon^*})) ] \leq \epsilon$. Then, the following results hold. Suppose the stepsize in Algorithm~\ref{alg:DZGT} is given by \Rme{$\gamma = \frac{\eta^{1/2}}{\sqrt{n^{3/2}K}L_0^{3/2}}$} such that   \Rme{$K_\epsilon \geq  \left\{1,\left(\max \Big\{
 \tfrac{6\sqrt{3}\lambda_{\mathbf{W}}^2 }{ \sqrt{1-\lambda_{\mathbf{W}}^2}}
 ,\tfrac{12\lambda_{\mathbf{W}}^3 }{(1-\lambda_{\mathbf{W}}^2) },\tfrac{16\lambda_{\mathbf{W}}^2  }{ (1-\lambda_{\mathbf{W}}^2)^2},6,\frac{6\lambda_{\mathbf{W}}}{(1-\lambda_{\mathbf{W}}^2)}\Big\} \right)^2\frac{1}{\eta L_0n^{1/2}}\right\}$}. Then, we have 
 \Rme{\begin{align*}
&K_\epsilon =  
\mathcal{O}\left( \left(\tfrac{n^{3/2}L_0^3}{\eta}\right)\epsilon^{-2}+\left(\tfrac{L_0^3n^{3/2}}{\eta m^2}\right)\epsilon^{-2}+\left(\tfrac{L_0^2n\mathbb{E}\left[\tfrac{1}{m}\|\mathbf{x}_{0}- \mathbf{1} \bar{x}_{0}\|^2 \right]}{\eta^2(1-\lambda_{\mathbf{W}}^2)}\right)\epsilon^{-1}+ \left(\tfrac{\lambda_{\mathbf{W}}^2L_0^{1/2}n^{1/4}}{\eta^{1/2} (1-\lambda_{\mathbf{W}}^2)^{3/2}}\right)\epsilon^{-1/2}+\right.\\
&\left.  \left(\tfrac{\lambda_{\mathbf{W}}^2L_0^{-1/2}n^{-1/4}}{\eta^{1/2} (1-\lambda_{\mathbf{W}}^2)^{3/2}}\right)\epsilon^{-1/2}+\left(\tfrac{\lambda_{\mathbf{W}}^2L_0n^{1/2}( \lambda_{\mathbf{W}}^2+     \|\mathbf{W}\|^2)}{\eta(1-\lambda_{\mathbf{W}}^2)^{3}}\right)\epsilon^{-1}+\left(\tfrac{\lambda_{\mathbf{W}}^2(1+3\lambda_{\mathbf{W}}^2)^{1/2}}{\sqrt{m}\eta(1-\lambda_{\mathbf{W}}^2)^2} \right)\epsilon^{-1/2}+\left(\tfrac{\lambda_{\mathbf{W}}^4L_0n^{1/2}}{\eta(1-\lambda_{\mathbf{W}}^2)^3}\right)\epsilon^{-1}\right).
\end{align*}}
\end{corollary}}
\fy{\begin{proof}
(a) Recall that in the exact setting $\varepsilon_k =0 $ for all $k$. Consequently, the bound in \Rme{Theorem~\ref{Theorem:thm 1}(a)} holds  for the exact setting by letting $\varepsilon_k =0 $ for all $k$. Note that in deriving the bound in \eqref{cor:1s_err_bound}, some of the scalars on the right-hand side are slightly improved in comparison with the bound in \eqref{eq:Non-asymptotic error bound in the single-stage case}. For example, the term $\tfrac{2(\mathbb{E}[f^\eta(\bar{{x}}_{0})]- \inf_{x} f^\eta(x))}{\gamma K}$ in \eqref{cor:1s_err_bound} is smaller than $\tfrac{4(\mathbb{E}[f^\eta(\bar{{x}}_{0})]- \inf_{x} f^\eta(x))}{\gamma K}$ in \eqref{eq:Non-asymptotic error bound in the single-stage case}. This is because, by following the same main steps taken in the analysis of the inexact setting, in some steps we may obtain slightly tighter bounds in the absence of inexactness. Because of the major overlap with the steps taken in the convergence and rate analysis of the inexact setting, we have omitted the detailed proof for this result.

\noindent (b) This result can be obtained from part (a), in a similar fashion to the proof of  \Rme{Theorem~\ref{Theorem:thm 1}(b)}.
\end{proof}}
\fy{\begin{remark}\em\label{remark:remark 12}The complexity guarantees \mee{of} Theorem~\ref{Theorem:thm 1} and Corollary~\ref{corollary:corollary 1} are summarized in \mee{Table~\ref{table:tbl1-smpec_complexity_1s}}. We have highlighted the dependence on $\epsilon$, $n$, $\eta$, $L_0$, and $\tilde{L}_0$. Note that when all the agents initialize from the same initial point, the term $\mathbb{E}\left[\tfrac{1}{m}\|\mathbf{x}_{0}- \mathbf{1} \bar{x}_{0}\|^2 \right]$ in Theorem~\ref{Theorem:thm 1} evaluates to zero. The complexity bounds presented in Table~\ref{table:tbl1-smpec_complexity_1s} are derived under this assumption. \mee{$\hfill \Box$}
\end{remark}

\begin{table}[]
\mee{
\centering \footnotesize{
\caption{Complexity guarantees for solving distributed single-stage SMPECs~\eqref{eqn:prob-1stage}}
\label{table:tbl1-smpec_complexity_1s}
\begin{tabular}{@{}lc@{}}
\toprule
\multicolumn{1}{c}{Single-stage SMPEC}                                     &Iteration complexity bound           \\ \midrule
{\bf{\underline{Inexact setting}}}                                          & \multicolumn{1}{l}{} \\
\begin{tabular}[c]{@{}l@{}}Upper level\\ problem\end{tabular} & $\mathcal{O}\left(\tfrac{n^{1/6}\tilde L_0^{2/3}\epsilon^{-1/3}}{ L_0^{2/3}\eta^{8/9}}
+\tfrac{\tilde L_0^2n^{3/2}\epsilon^{-1}}{\eta^{4/3}}+\tfrac{L_0^3n^{3/2}\epsilon^{-2}}{\eta^{4/3}m^2}+\tfrac{  n^{3/4}\tilde L_0\epsilon^{\mee{-1/2}} }{\eta^{4/3}L_0^{1/2}  }+\tfrac{  L_0^{1/2}n^{1/4}\epsilon^{\mee{-1/2}} }{\eta^{1/3}}+\tfrac{  n^{5/6}\tilde L_0^{4/3}\epsilon^{-2/3} }{L_0^{1/3}\eta^{10/9}  }\right)$  ~\\    \cmidrule(l){2-2}              
\begin{tabular}[c]{@{}l@{}}Lower level\\ problem\end{tabular} & $\mathcal{O}\left(\kappa_F\epsilon^{-1}\right)$                  \\\cmidrule(l){2-2}
\begin{tabular}[c]{@{}l@{}}Overall\\ complexity\end{tabular}     & $  \mathcal{O}\left(\tfrac{n^{5/6}\tilde L_0^{4/3}\epsilon^{-2/3}}{ L_0^{4/3}\eta^{22/9}}
+\tfrac{\tilde L_0^4n^{7/2}\epsilon^{-2}}{\eta^{10/3}}+\tfrac{L_0^6n^{7/2}\epsilon^{-4}}{\eta^{10/3}m^4}+\tfrac{  n^{2}\tilde L_0^2\epsilon^{-1} }{\eta^{10/3}L_0  }+\tfrac{  L_0n\epsilon^{-1} }{\eta^{4/3}}+\tfrac{  n^{13/6}\tilde L_0^{8/3}\epsilon^{-4/3} }{L_0^{2/3}\eta^{26/9}  }\right)$                  \\ \midrule
\begin{tabular}[c]{@{}l@{}}{\bf{{Exact setting   }}} \\ \end{tabular}     & \Rme{$\mathcal{O}\left(\tfrac{L_0n^{1/2}\epsilon^{-1}}{\eta}+\tfrac{L_0^3n^{3/2}\epsilon^{-2}}{\eta m^2}+\tfrac{L_0^3n^{3/2}\epsilon^{-2}}{\eta}+\tfrac{  L_0^{1/2}n^{1/4}\epsilon^{\mee{\Rme{-1/2}}} }{\eta^{1/2} }\right)$}                      \\ \midrule
\end{tabular}}}
\end{table}

\begin{remark}\em 
In addressing the exact setting of the single-stage problem~\eqref{eqn:prob-1stage}, we achieve the best-known complexity bound for centralized nonsmooth nonconvex stochastic optimization~\cite{lin2022gradient}, that is $\Rme{\mathcal{O}\left(n^{3/2}\epsilon^{-2}\eta^{-1}\right)}$ where $\epsilon>0$ is an arbitrary scalar such that $\mathbb{E}[\mbox{dist}^2(0,\partial_{\eta} f(\bullet)) ] \leq \epsilon$ where $f$ denotes the implicit objective function. In addressing the inexact setting, the overall complexity of \mj{$\mathcal{O}\left(n^{7/2}\epsilon^{-4}\right)$} is achieved. This implies an improvement in the dependence on $n$ compared to the centralized stochastic MPECs in~\cite{cui2023complexity}, where a dependence of $\mathcal{O}\left(n^{4}\right)$ \mee{is} established. \mee{$\hfill \Box$}
\end{remark}
}

\Rme{\begin{remark}\em[Exact setting and related work]
Corollary~\ref{corollary:corollary 1} addresses the exact setting, where each agent can access an exact solution of the lower-level VI and hence can evaluate the induced single-level objective directly.
In this case, the problem reduces to decentralized zeroth-order stochastic optimization for (possibly) nonsmooth and nonconvex objectives, which has been studied in, e.g., \cite{lin2024aaai,sahinoglu2024online}. Moreover, in the exact setting, our gradient-tracking scheme specializes to the DGFM-type method studied in \cite{lin2024aaai}.
Accordingly, Corollary~\ref{corollary:corollary 1} is included primarily for completeness and for aligning notation; it is not the main technical contribution of this work.
Our main contribution concerns the inexact setting, where lower-level problems are solved only approximately, and the resulting error in the zeroth-order gradient-tracking recursion can be biased, requiring new analysis beyond the exact literature.
\end{remark}}

\begin{remark}\em
\RRme{In Corollary~\ref{corollary:corollary 1}, we consider the case} $\RRmath{\lambda_{\mathbf{W}}>0}$\RRme{. The case} $\RRmath{\lambda_{\mathbf{W}}=0}$ \RRme{can also be addressed in the exact setting by following the analysis used for Corollary~\ref{cor:thm1-lambda-zero}. In particular, Lemmas~\ref{Lemma:main bound for the gradient tracker in terms of other main terms-inexact} and~\ref{lemma:bound for x-bar x inexact} show that the gradient-tracking error is zero and the iterates reach exact consensus after the first update. Since} $\RRmath{\varepsilon_k=0}$ \RRme{in the exact setting, the corresponding convergence and complexity bounds can be obtained directly from the} $\RRmath{\lambda_{\mathbf{W}}=0}$ \RRme{analysis by removing the inexactness terms.} $\hfill\Box$
\end{remark}


\section{Distributed two-stage stochastic MPECs}\label{sec:2s}
In this section, we extend the zeroth-order scheme discussed earlier to address the two-stage SMPEC problem, defined in~\eqref{eqn:prob-2stage}. We assume that the agents can communicate their updated iterates with neighboring agents over an undirected connected network, as described in Assumption~\ref{assum:mixxx}. \us{Next}, we present some assumptions and the outline of the algorithm. 
\begin{assumption}\label{assum:main6}\em Consider ~\eqref{eqn:prob-2stage}. Let the following conditions hold. (i) For any agent $i \in [m]$, $\tilde h_i(x,\bullet,\xi_i)$ is $\tilde L_0(\xi_i)$-Lipschitz continuous for any $\xi_i$, where $\tilde L_0 \triangleq  {\displaystyle \max_{i\in [m]}}\sqrt{\mathbb{E}[\tilde L_0^2({\bxi}_i)]}$ is finite. (ii)  $\tilde h_i(\bullet,z_i(\bullet,\xi_i),\xi_i)$ is $ L_0(\xi_i)$-Lipschitz continuous for any $\xi_i$ and $ L_0 \triangleq {\displaystyle \max_{i\in [m]}}\sqrt{\mathbb{E}[ L_0^2(\mee{\bxi}_i)]}$ is finite.
\end{assumption}
\begin{remark}\em
\us{Akin} to the single-stage case,  the implicit objective is \us{not assumed to be} differentiable. The Lipschitz continuity of the implicit function in Assumption~\ref{assum:main6} has been studied in \cite{patriksson1999stochastic}. \mee{$\hfill \Box$}
\end{remark}
\begin{remark}\label{rem:smoothness_of_f_i-2stage}\em
In view of \Rme{Assumption~\ref{assum:main6}(ii)} and the definition of the local implicit functions $f_i$ introduced \us{earlier}, $f_i$ is $\mathbb{E}[L_0({\bxi}_i)]$-Lipschitz. Invoking Jensen's inequality,  $\mathbb{E}[L_0({\bxi}_i)] \leq \sqrt{\mathbb{E}[L_0^2({\bxi}_i)]} \leq {\displaystyle \max_{i\in [m]}}\sqrt{\mathbb{E}[ L_0^2({\bxi}_i)]} = L_0$. Thus $f_i$ is $L_0$-Lipschitz for all $i\in [m]$. Invoking \RRme{Lemma~\ref{lem:smoothing_props}(iii)}, $\nabla f^{\eta}_i$ is $\tfrac{L_0 \sqrt{n}}{\eta}$-Lipschitz for all $i\in [m]$ and the smoothed global implicit function $f^\eta$ is $\tfrac{L_0 \sqrt{n}}{\eta}$-smooth. $\hfill \Box$
\end{remark}
\begin{assumption}\label{assum:main2-2stage}\em Consider problem (\ref{eqn:prob-2stage}). Let $F_i(x,\bullet) \triangleq \tilde{F}_i(x,\bullet,\xi_i)$  be a $\mu_F$-strongly monotone and $L_F$-Lipschitz continuous mapping uniformly in $x$. For each $i \in [m]$ and any \us{$\xi_i$ and} $x \in \mathbb{R}^n$, assume that the set $\mathcal{Z}_i(x,\xi_i)\subseteq \mathbb{R}^p$ is nonempty, closed, and convex. Further, $\sup_{x\in\mathbb{R}^n}\sup_{z_1,z_2 \in \mathcal{Z}_i(x,\xi_i)}\|z_1-z_2\|^2 \leq D_i$ for some $D_i>0$, for all $i \in [m]$, and for all $\xi_i\in \mathbb{R}^d$.
\end{assumption}
{
\begin{assumption}\label{assumption:measuribility assumption}\em Consider problem (\ref{eqn:prob-2stage}). For each $i \in [m]$, the implicit stochastic local function \us{$\tilde{f}_i(x,\bxi_i)$, defined as} $\tilde{f}_i(x,\xi_i) \triangleq \tilde{h}_i(x,z_i(x,\xi_i),\xi_i)$, is measurable for all $x \in \mathbb{R}^n$. 
\end{assumption}
\begin{remark}\em\RRme{Sufficient} conditions \us{for} measurability of the implicit stochastic function in two-stage SMPECs have been studied in the literature~\cite{xu2010necessary, evgrafov2004existence}. In particular, \fyy{\cite{xu2010necessary}} considers a centralized variant of \eqref{eqn:prob-2stage}, analyzing settings where the set $\mathcal{Z}_i$ is independent of both the first-stage decision variable $x$ and the random variable \us{$\bxi_i$}. {In view of~\cite[Thm.~2.16]{xu2010necessary},} the measurability of the implicit stochastic local function $\tilde{f}_i(x,\us{\bxi_i})$ is guaranteed under the following assumption:
{\em For each $i \in [m]$, there exists a constant $\tau > 0$ such that, for every constant $\phi$, the set $$\left\{z_i \in\mathbb{R}^p:\RRme{\exists\,} r\in \tilde F_i(x,z_i,\xi_i)+\mathcal{N}_{\mathcal{Z}_i}(z_i), \RRme{\hbox{ such that }}\tilde{f}_i(x,\xi_i)\le \phi,\| r\|\le \tau\right\}$$ is bounded, where $\mathcal{N}_{\mathcal{Z}_i}(z_i)$ denotes the normal cone \us{of} the set $\mathcal{Z}_i$ at point $z_i$.} $\hfill \Box$
\end{remark}

\noindent \textbf{History of the method.} In both exact and inexact settings, we define the history of Algorithm~\ref{alg:DZGT-2stage} as $\mathcal{F}_k \triangleq \cup_{i=1}^m \mathcal{F}_{i,k} $ for $k\geq 0$, where $\mathcal{F}_{i,k}\triangleq  \mathcal{F}_{i,k-1}\cup \{\xi_{i,k-1},v_{i,k-1}\} $ for any $k\geq 1$, and $\mathcal{F}_{i,0}\triangleq  \{x_{i,0}\}$\RRme{, for all $i \in [m]$}.
\begin{remark}\em
From the above definitions, in both inexact and exact settings, we have that for any $i \in [m]$, $x_{i,k}$ is $\RRmath{\mathcal{F}_k}$-measurable, for all $k \geq 0$. \mee{$\hfill \Box$}
\end{remark}}
Leveraging \us{a} smoothing-enabled zeroth-order framework, we propose a distributed implicit zeroth-order gradient-tracking method~ (DiZS-GT$^{\text{2s}}$) for addressing ~\eqref{eqn:prob-2stage}. This method is outlined in Algorithm~\ref{alg:DZGT-2stage}, capturing both exact and inexact settings. At iteration $k$ of the algorithm, agent $i$ generates $x_{i,k}$ and $y_{i,k}$, which represent a local copy of the decision variable $x$ and a zeroth-order gradient tracker of the smoothed implicit objective function, respectively. In this setting, there is no stochasticity in the lower-level problem, \us{which is adapted to $\xi_{i,k}$}. Each agent calls Algorithm~\ref{alg:lowerlevel-2stage} twice to evaluate $z_i({x}_{i,{k}}{-\eta v_{i,{k}}},\xi_{i,k})$ and $z_i({x}_{i,{k}}+\eta v_{i,{k}},\xi_{i,k})$ \us{in} the exact setting, {while} the $\varepsilon_{k}$-inexact \us{analogs} of these terms \us{are denoted by} $z_{i,\varepsilon_{k}}({x}_{i,{k}}+\eta v_{i,{k}},\xi_{i,k})$ and $z_{i,\varepsilon_{k}}({x}_{i,{k}}-\eta v_{i,{k}},\xi_{i,k})$ in the inexact case. We assume that these approximations satisfy the condition $\|z_{i,\varepsilon_k}(x,\xi_{i,k}) -z_i(x,\xi_{i,k})\|^2 \leq \varepsilon_k$ for any random variable $x \in \mathbb{R}^n$ and $\xi_{i,k} \in \mathbb{R}^d$, almost surely, where $\varepsilon_k$ is a deterministic scalar independent of $x$ and $\xi_{i,k}$. In Lemma~\ref{lem:Alg2_conv_2stage}, we verify that this assumption is met and derive a bound for $\varepsilon_k$.  Utilizing the randomized smoothing scheme and \Rme{Lemma~\ref{lem:smoothing_props}(i)}, the gradients of the $\eta$-smoothed implicit local objective function $\tilde{f}_i^\eta(\bullet,\xi_{i,k})$, denoted by $g_{i,{k}}^{\eta}$ for the exact case and $g_{i,{k}}^{\eta,\varepsilon_{k}}$ for the inexact case, are approximated as follows. 
 \begin{align}
  \notag g_{i,{k}}^{\eta}&\triangleq \tfrac{n}{2\eta}\left(\tilde h_i({x}_{i,{k}}+\eta v_{i,{k}},z_{i}({x}_{i,{k}}+\eta v_{i,{k}},\xi_{i,k}),\xi_{i,{k}})\right. \\
    & \left.-\tilde h_i({x}_{i,{k}}{-\eta v_{i,{k}}},z_{i}({x}_{i,{k}}{-\eta v_{i,{k}}},\xi_{i,k}),\xi_{i,{k}})\right)\, v_{i,{k}},\label{eqn:g_eta_two stage} \\
  g_{i,{k}}^{\eta,\varepsilon_{k}}&\triangleq \tfrac{n}{2\eta}\left(\tilde h_i({x}_{i,{k}}+\eta v_{i,{k}},z_{i,\varepsilon_{k}}({x}_{i,{k}}+\eta v_{i,{k}},\xi_{i,k}),\xi_{i,{k}})\right. \notag \\
        & \left.-\tilde h_i({x}_{i,{k}}{-\eta v_{i,{k}}},z_{i,\varepsilon_{k}}({x}_{i,{k}}{-\eta v_{i,{k}}},\xi_{i,k}),\xi_{i,{k}})\right)\, v_{i,{k}}.\label{eqn:g_eta_eps_two stage}
\end{align}
\begin{algorithm}[htb]
\caption{\fy{DiZS-GT$^{\text{2s}}$} (by agent $i$)}\label{alg:DZGT-2stage}
\begin{algorithmic}[1]
\State {\bf input}  weights $w_{ij}$ for all $j\in[m]$, stepsize $\gamma$ and smoothing parameter $\eta$, local random initial vector $x_{i,0} \in \mathbb{R}^n$, $y_{i,0}:=0_{n}$, and $g_{i,-1}^{\eta,\varepsilon_{-1}}=g_{i,-1}^{\eta}:=0_n$. (\colorbox{blue!10}{Inexact} and \colorbox{yellow!22}{Exact} schemes highlighted)
%
 \FOR {$k = 0,1,2, \ldots$} 
\State Generate random samples $\xi_{i,k}$ and $v_{i,k} \in   \mathbb{S}$ 
 \State \colorbox{blue!10}{Call \Rme{Alg.~\ref{alg:lowerlevel-2stage}} twice to get $z_{i,\varepsilon_{k}}({x}_{i,{k}}{-\eta v_{i,{k}}},\xi_{i,k})$ and $z_{i,\varepsilon_{k}}({x}_{i,{k}}+\eta v_{i,{k}},\xi_{i,k})$} 
\hspace{-0.19 in} \colorbox{yellow!22}{Evaluate $z_i({x}_{i,{k}}{-\eta v_{i,{k}}},\xi_{i,k})$ and $z_i({x}_{i,{k}}+\eta v_{i,{k}},\xi_{i,k})$}
\State \colorbox{blue!10}{%
$\begin{aligned}
g_{i,{k}}^{\eta,\varepsilon_{k}} &:= \left(\tfrac{n}{2\eta}\right)(\tilde h_i({x}_{i,{k}}+\eta v_{i,{k}},z_{i,\varepsilon_{k}}({x}_{i,{k}}+\eta v_{i,{k}},\xi_{i,k}),\xi_{i,{k}}) \\
&- \tilde h_i({x}_{i,{k}}{-\eta v_{i,{k}}},z_{i,\varepsilon_{k}}({x}_{i,{k}}{-\eta v_{i,{k}}},\xi_{i,k}),\xi_{i,{k}}))\, v_{i,{k}}
\end{aligned}$}
\State \colorbox{yellow!22}{%
$\begin{aligned}
g_{i,{k}}^{\eta} &:= \left(\tfrac{n}{2\eta}\right)(\tilde h_i({x}_{i,{k}}+\eta v_{i,{k}},z_{i}({x}_{i,{k}}+\eta v_{i,{k}},\xi_{i,k}),\xi_{i,{k}}) \\
& - \tilde h_i({x}_{i,{k}}{-\eta v_{i,{k}}},z_{i}({x}_{i,{k}}{-\eta v_{i,{k}}},\xi_{i,k}),\xi_{i,{k}}))\, v_{i,{k}}
\end{aligned}$}
\State  \colorbox{blue!10}{$y_{i,k+1}:=\sum_{j=1}^mw_{ij}\left(\RRmath{y_{j,k}}+g_{j,k}^{\eta,\varepsilon_{k}}-g_{j,k-1}^{\eta,\varepsilon_{k-1}}\right)$}  \colorbox{yellow!22}{$y_{i,k+1}:=\sum_{j=1}^mw_{ij}\left(\RRmath{y_{j,k}}+g_{j,k}^{\eta}-g_{j,k-1}^{\eta}\right)$}
\State $
{x}_{i,k+1}:=\sum_{j=1}^mw_{ij}\left(x_{j,k}-\gamma y_{j,k+1} \right)$
\ENDFOR
\end{algorithmic}
\end{algorithm}
\begin{algorithm}[htb]
\caption{{Deterministic projection method (by agent $i$)}}
\label{alg:lowerlevel-2stage}
\begin{algorithmic}[1]
\State {\bf input} upper-level iteration index $k$, input vectors $\hat{x}_{i,k}^{\ell}$ and ${\xi_{i,k}}$, set $\ell:=1$ if  $\hat{x}_{i,k}^{\ell}=x_{i,k}-\eta v_{i,k}$  and $\ell:=2$ if  $\hat{x}_{i,k}^{\ell}=x_{i,k}+\eta v_{i,k}$, \RRme{an initial vector $z_{i,0}^{k,\ell}\in\mathcal{Z}_i(\hat{x}_{i,k}^{\ell},\xi_{i,k})$,} $\RRme{0<\hat{\gamma}<\tfrac{\mu_F}{L_F^2}}$, \RRme{and number of lower-level iterations $t_k$.}
 \FOR {$t = 0,1, \ldots, t_k-1 $} 
\State  $z_{i,t+1}^{k,\ell}:= \Pi_{\mathcal{Z}_i(\hat x_{i,k}^{\ell}\RRme{,}{\xi_{i,k}})}\left[z_{i,t}^{k,\ell}-\hat{\gamma}\tilde{F}_i(\hat{x}_{i,k}^{\ell},z_{i,t}^{k,\ell},{\xi_{i,k}})\right]$
\ENDFOR
\State \RRme{{\bf output} $z_{i,t_k}^{k,\ell}$ for $\ell\in\{1,2\}$}
\end{algorithmic}
\end{algorithm}
As in the single-stage case, a key research question in the inexact setting is to \us{prescribe a} termination criterion for Algorithm~\ref{alg:lowerlevel-2stage} that ensures convergence of the iterates generated by DiZS-GT$^{\text{2s}}$. The following result shows a linear convergence rate for the lower-level problem \us{via a} gradient method. \fyy{The proof is presented in the Appendix.}

\begin{lemma}[\mee{\textbf{Rate statement for Algorithm~\ref{alg:lowerlevel-2stage}}}]\label{lem:Alg2_conv_2stage}\em
Consider Algorithm~\ref{alg:lowerlevel-2stage}. Let Assumption~\ref{assum:main2-2stage} hold \RRme{and define $\hat D:=\max_{j\in[m]}D_j$, where $D_j  $ is given by Assumption~\ref{assum:main2-2stage}}. Let $k\geq 0$, $i \in [m]$, $\ell \in \{1,2\}$, \us{and} $\hat{x}_{i,k}$, $\xi_{i,k}$, and $z_{i,0}^{k,\ell} \in \mathcal{Z}_i(\hat x^\ell_{i,k},{\xi_{i,k}})$ be given. \us{Further, \RRme{let}} $\tilde d\triangleq1-\hat{\gamma}\mu_F$, \RRme{where} $\RRme{0<\hat{\gamma}<\tfrac{\mu_F}{L_F^2}}$, and \RRme{assume that} the algorithm is terminated after $t_k$ iterations. Then, 
 $\mathbb{E}[ \|z_{i,\varepsilon_k}(\hat{x}_{i,k}^{\ell},{\bxi_{i,k}}) - z_i(\hat{x}_{i,k}^{\ell},\bxi_{i,k}) \|^2\mid \mathcal{F}_k] \leq \varepsilon_k \triangleq \tilde d^{t_k}\RRme{\hat D}$. 
\end{lemma}

\fy{Similar to the single-stage setting, we define $g_{i,-1}^{\eta,\varepsilon_{-1}}=g_{i,-1}^{\eta}=\nabla f_i^\eta(x_{i,-1})=\delta_{i,-1}^\eta={e}_{i,-1}^{\eta,\varepsilon_{-1}}=y_{i,0}=0_n$ for all $i \in [m]$. We define the error terms $\boldsymbol{\delta}_{k}^\eta\triangleq [\delta_{1,k}^\eta,\ldots,\delta_{m,k}^\eta]^\top$ and $\mathbf{e}_{k}^{\eta,\varepsilon_k}\triangleq[{e}_{1,k}^{\eta,\varepsilon_k},\ldots,{e}_{m,k}^{\eta,\varepsilon_k}]^\top$, where for $i \in [m]$ and $k\ge -1$, we define $\delta_{i,k}^\eta\triangleq g_{i,k}^\eta-\nabla f_i^\eta(x_{i,k})$ and ${e}_{i,k}^{\eta,\varepsilon_k}\triangleq{g}_{i,k}^{\eta,\varepsilon_k}-{g}_{i,k}^\eta.$  We define the averaged terms $\bar \delta_{k}^\eta=\tfrac{1}{m}\mathbf{1}^\top \boldsymbol{\delta}_{k}^\eta$ and $\bar{e}_{k}^{\eta,\varepsilon_k}=\tfrac{1}{m}\mathbf{1}^\top\mathbf{e}_{k}^{\eta,\varepsilon_k}.$ Then, for $k \geq 0$, {DiZS-GT$^{\text{2s}}$} can be compactly represented by \eqref{eqn:R1} and \eqref{eqn:R2}. In the following result, we provide a bound on the second moment of the inexact error $e_{i,k}^{\eta,\varepsilon_k}$ for each agent $i$. \fyy{The proof of this result is presented in the Appendix.}
\begin{lemma}\em\label{lemm:omega_vareps} 
Consider Algorithm~\ref{alg:lowerlevel-2stage} for an agent $i$. \RRme{Let Assumptions~\RRme{\ref{assum:samples}} and \ref{assum:main6} hold. Suppose that} $ \|z_{i,\varepsilon_k}(x,\xi_i) -z_i(x,\xi_i)\|^2 \leq \varepsilon_k$ hold\RRme{s} for \RRme{all} \fyr{realizations of random variables ${\bf x} \in \mathbb{R}^n$ and ${\bf \xi}_i \in \mathbb{R}^d$} \us{a.s.}\RRme{, where} \RRme{$\varepsilon_k$} is a deterministic scalar (independent \RRme{of} $x$ and $\xi_i$). Then, for all $i\in [m]$ and $k\geq 0$, $\mathbb{E}[\|e_{i,k}^{\eta,\varepsilon_k}\|^2 |\mathcal{F}_k]\le \left(\tfrac{\tilde L_0^2n^2\varepsilon_k}{\eta^2}\right)$ a.s.\fyr{.} 
\end{lemma}
}
\fy{Similar to the single-stage case, we analyze \fyy{the} convergence  by considering \us{the following} three error metrics. (i) {$\mathbb{E}[\|{\nabla{f}^\eta}(\bar{x}_k)\|^2]$}, (ii) $\mathbb{E}[\|{ \mathbf{x}}_{k}-\mathbf{1}{\bar{ {x}}}_{k}\|^2]$, and (iii) $\mathbb{E}[\|{ \mathbf{y}}_{k}-\mathbf{1}{\bar{ {y}}}_{k}\|^2]$. Notably, the results of Lemmas~\ref{lemma:descent lemma inexact},~\ref{lemma: multiple inequalities for the metrics}, and~\ref{Lemma:main bound for the gradient tracker in terms of other main terms-inexact} from the single-stage analysis remain applicable. These results are utilized in the next theorem to establish convergence and provide rate statements for the two-stage setting.
\begin{theorem}[{\textbf{\Rme{Convergence guarantees for DiZS-GT$^{\text{2s}}$}}}]\em \label{Theorem:thm 2} Consider Algorithm~\ref{alg:DZGT-2stage}. Let Assumptions~\ref{assum:mixxx}, \ref{assump:opt_f_bounded_below}, \RRme{\ref{assum:samples}}, \ref{assum:main6}, \ref{assum:main2-2stage} \mee{, and \ref{assumption:measuribility assumption}} hold. \RRme{Suppose} $\RRmath{\lambda_{\mathbf{W}}>0}$\RRme{.}

 \noindent (a) [\fyr{\textbf{Error bounds}}] \fyr{Let} $\gamma \leq \min \Big\{
 \tfrac{ \sqrt{1-\lambda_{\mathbf{W}}^2}}{10\sqrt{3}\lambda_{\mathbf{W}}^2 }
 ,\tfrac{(1-\lambda_{\mathbf{W}}^2) }{20\lambda_{\mathbf{W}}^3 },\tfrac{ (1-\lambda_{\mathbf{W}}^2)^2}{20\lambda_{\mathbf{W}}^2  },\frac{1}{6},\frac{(1-\lambda_{\mathbf{W}}^2)}{9\lambda_{\mathbf{W}}}\Big\} \left(\tfrac{\eta}{{\sqrt{n}} L_0 }\right)$. \fyr{Suppose $K^*$ is uniformly distributed over $\{0,\ldots,K-1\}$}. Then, for any $K \geq 1$, the \Rme{bounds in~\eqref{eq:Non-asymptotic error bound in the single-stage case} and~\eqref{equation: the bound for the consensus error}} \fyr{hold.}

 \noindent (b) [\mee{\textbf{Complexity bounds}}] Suppose at iteration $k$ in Algorithm~\ref{alg:DZGT-2stage}, in generating the inexact solutions, Algorithm~\ref{alg:lowerlevel-2stage} is terminated after $t_k:=\left\lceil \tfrac{-a}{\ln(1-\mu_F\hat{\gamma})}\ln\left(n^{1/2a}(k+1)\eta^{-2/3a}\right)\right\rceil$ iterations\Rme{, where $a>0.5$ \RRme{and $\hat{\gamma}:=\tfrac{\mu_F}{2L_F^2}$}}. Let $\epsilon>0$ be an arbitrary scalar \fyr{where} $\mathbb{E}[\mbox{dist}^2(0,\partial_{\eta} f(\bar{{x}}_{K_\epsilon^*})) ] \leq \epsilon$. Then, the following \fyr{hold}. 
 
 \noindent (b-i) [\mee{\textbf{Upper-level iteration/sample complexity}}] Suppose the stepsize in Algorithm~\ref{alg:DZGT-2stage} is given by $\gamma = \frac{\eta^{2/3}}{\sqrt{n^{3/2}\fyr{K_\epsilon}}L_0^{3/2}}$ such that   \Rme{$K_\epsilon \geq  \left\{1,\left(\max \Big\{
 \tfrac{10\sqrt{3}\lambda_{\mathbf{W}}^2 }{ \sqrt{1-\lambda_{\mathbf{W}}^2}}
 ,\tfrac{20\lambda_{\mathbf{W}}^3 }{(1-\lambda_{\mathbf{W}}^2) },\tfrac{20\lambda_{\mathbf{W}}^2  }{ (1-\lambda_{\mathbf{W}}^2)^2},6,\frac{9\lambda_{\mathbf{W}}}{(1-\lambda_{\mathbf{W}}^2)}\Big\}  \right)^2\frac{1}{\eta^{2/3}L_0n^{1/2}}\right\}$}. Then, we have 
 \begin{align}
K_\epsilon &=  
\mathcal{O}\left(\left(\tfrac{n^{3/2}L_0^3}{\eta^{4/3}} \right)\epsilon^{-2}+\left(\tfrac{L_0^3n^{3/2}}{\eta^{2/3}m^2}\right)\epsilon^{-2}+\left(\tfrac{L_0^2n\mathbb{E}\left[\tfrac{1}{m}\|\mathbf{x}_{0}- \mathbf{1} \bar{x}_{0}\|^2 \right]}{\eta^2(1-\lambda_{\mathbf{W}}^2)}\right)\epsilon^{-1}+ \left(\tfrac{\lambda_{\mathbf{W}}^2L_0^{1/2}n^{1/4}}{\eta^{1/3}(1-\lambda_{\mathbf{W}}^2)^{3/2}}\right)\epsilon^{-1/2}\right.\notag\\
&\left. + \left(\tfrac{\lambda_{\mathbf{W}}^2\tilde L_0n^{3/4}\sqrt{\varepsilon_0}}{L_0^{1/2}\eta^{4/3}(1-\lambda_{\mathbf{W}}^2)^{3/2}}\right)\epsilon^{-1/2}+\left(\tfrac{\lambda_{\mathbf{W}}^2L_0n^{1/2}( \lambda_{\mathbf{W}}^2+     \|\mathbf{W}\|^2)}{\eta^{2/3}(1-\lambda_{\mathbf{W}}^2)^3}\right)\epsilon^{-1}+\left({\tfrac{\lambda_{\mathbf{W}}^2(3+\lambda_{\mathbf{W}}^2)^{1/2}}{\eta^{2/3}\sqrt{m}\RRme{(1-\lambda_{\mathbf{W}}^2)^2}}}\right)\epsilon^{-1/2}\right)+\mee{J_\epsilon},\label{eqn:K_eps_2s}
\end{align}
where we have three cases for $J$ as follows:

\hspace{.2in} (1) If $a \in \left(\mee{\tfrac{1}{2}},1\right)$, then the following holds.
 \begin{align*}
  \mee{J_\epsilon}&=\mathcal{O}\left( \sqrt[1+a]{\left(\tfrac{ \RRmath{\hat D}\lambda_{\mathbf{W}}^2n}{L_0\eta^{2}(1-\lambda_{\mathbf{W}}^2)^3(1-a) }\right)\left( {\|\mathbf{W}\|^2\tilde L_0^2}{} +  \tfrac{\lambda_{\mathbf{W}}^4\tilde L_0^2}{ (1-\lambda_{\mathbf{W}}^2)}   \right)}\, \epsilon^{-1/(1+a)}+\left(\tfrac{ \RRmath{\hat D}\tilde L_0^2n^{3/2}}{\eta^{4/3}(1-a)}\right)^{1/a}\, \epsilon^{-1/a} \right.	\\
&+  \left. \sqrt[2+a]{\left(\tfrac{\RRmath{\hat D}\RRmath{\lambda_{\mathbf{W}}^4}n^{1/2} \tilde L_0^2}{L_0^2\eta^{8/3}\RRmath{(1-\lambda_{\mathbf{W}}^2)^4}(1-a) }\right)}\, \epsilon^{-1/(2+a)}        +\left(\tfrac{\RRmath{\hat D}\tilde L_0^2n^{5/4}}{L_0^{1/2}\eta^{5/3}(1-a)}\right)^{1/\left(a+\mee{\tfrac{1}{2}}\right)}\, \epsilon^{-1/\left(a+\mee{\tfrac{1}{2}}\right)}\right).
\end{align*}

 \hspace{.2in} (2)  If $a =1$, then the following holds.
\begin{align*}
  \mee{J_\epsilon}&= \tilde{\mathcal{O}}\left(\sqrt{\left(\tfrac{ \RRmath{\hat D}\lambda_{\mathbf{W}}^2n}{L_0\eta^{2}(1-\lambda_{\mathbf{W}}^2)^3 }\right)\left( {\|\mathbf{W}\|^2\tilde L_0^2} +  \tfrac{\lambda_{\mathbf{W}}^4\tilde L_0^2}{ (1-\lambda_{\mathbf{W}}^2)}   \right)}\, \epsilon^{-1/2}+ \sqrt[3]{\left(\tfrac{ \RRmath{\hat D}\RRmath{\lambda_{\mathbf{W}}^4}n^{1/2} \tilde L_0^2}{L_0^2\eta^{8/3}\RRmath{(1-\lambda_{\mathbf{W}}^2)^4} }\right)}\, \epsilon^{-1/3} \right.	\\
&+   \left.\left(\tfrac{\RRmath{\hat D}\tilde L_0^2n^{5/4}}{L_0^{1/2}\eta^{5/3}}\right)^{2/3}\, \epsilon^{-2/3}+\left(\tfrac{\RRmath{\hat D} \tilde L_0^2n^{3/2}}{\eta^{4/3}}\right)\, \epsilon^{-1}\right).
\end{align*}

 \hspace{.2in} (3)  If $a >1$, then the following holds.
  \begin{align*}
  \mee{J_\epsilon}&= \mathcal{O}\left(\sqrt{\left(\tfrac{a\RRmath{\hat D} \lambda_{\mathbf{W}}^2n}{L_0\eta^{2}(1-\lambda_{\mathbf{W}}^2)^3(a-1) }\right)\left( {\|\mathbf{W}\|^2\tilde L_0^2}{} +  \tfrac{\lambda_{\mathbf{W}}^4\tilde L_0^2}{ (1-\lambda_{\mathbf{W}}^2)}   \right)}\, \epsilon^{-1/2}+ \sqrt[3]{\left(\tfrac{ a\RRmath{\hat D}\RRmath{\lambda_{\mathbf{W}}^4}n^{1/2}\tilde L_0^2}{L_0^2\eta^{8/3}\RRmath{(1-\lambda_{\mathbf{W}}^2)^4}(a-1) }\right)}\, \epsilon^{-1/3}\right. 	\\
&+   \left.\left(\tfrac{ a\RRmath{\hat D}\tilde L_0^2n^{5/4}}{L_0^{1/2}\eta^{5/3}(a-1)}\right)^{2/3}\, \epsilon^{-2/3}+\left(\tfrac{ a\RRmath{\hat D}\tilde L_0^2n^{3/2}}{\eta^{4/3}(a-1)}\right)\, \epsilon^{-1}\right).
\end{align*}
\noindent (b-ii) [\mee{\textbf{Overall iteration complexity}}] To guarantee $\mathbb{E}[\mbox{dist}^2(0,\partial_{\eta} f(\bar{{x}}_{K^*_{\RRmath{\epsilon}}})) ] \leq \epsilon$, the overall iteration complexity is  $\tfrac{- a}{\ln(1-\RRmath{\kappa_F^{-2}/2})}{\tilde{\mathcal{O}}}\left(K_\epsilon\right),  $ where $K_\epsilon$ is given by \eqref{eqn:K_eps_2s}. 
\end{theorem}}
\begin{proof}
 \noindent \fy{(a) \mee{According to Assumption~\ref{assumption:measuribility assumption}, the implicit stochastic local function, $\tilde{f}_i(x,\xi_i)$, is measurable for all $i \in [m]$ and every $x \in \mathbb{R}^n$.} Notably, although $z_{i,\varepsilon_k}(x,\xi_i)$ and $z_{i}(x,\xi_i)$ depend on both $x$ and $\xi_i$ in the two-stage setting, while $z_{i,\varepsilon_k}(x)$ and $z_{i}(x)$ are independent of $\xi_i$ in the single-stage case, the compact representations defined in \eqref{eqn:R1} and \eqref{eqn:R2} remain identical for both single-stage and two-stage settings. Further, while the definition of $\varepsilon_k$ differs between the two cases, given by $\mathbb{E}[\|z_{i,\varepsilon_k}(x) -z_i(x)\|^2 \mid x]\leq \varepsilon_k$ for the single-stage case and $ \|z_{i,\varepsilon_k}(x,\xi_i) -z_i(x,\xi_i)\|^2 \leq \varepsilon_k$ for the two-stage case, and also, the history of the method and the proofs of Lemma~\ref{lemma:g_ik_eta_props}(iv) and Lemma~\ref{lemm:omega_vareps} differ, the results of these two lemmas are the same, i.e., in both lemmas we have $\mathbb{E}[\|e_{i,k}^{\eta,\varepsilon_k}\|^2 |\mathcal{F}_k]\le \left(\tfrac{\tilde L_0^2n^2\varepsilon_k}{\eta^2}\right)$. Consequently, the results of Lemmas~\ref{lemma:descent lemma inexact},~\ref{lemma: multiple inequalities for the metrics}, and~\ref{Lemma:main bound for the gradient tracker in terms of other main terms-inexact} from the single-stage analysis remain applicable in the two-stage setting. Therefore, we conclude that the error bound of part \noindent (a) in Theorem~\ref{Theorem:thm 1} holds true for the two-stage setting, in that $\mathcal{F}_k$ and $\varepsilon_k$ are distinctly defined for the two-stage setting.}

 \fy{
\noindent (b-i) First\mee{,} we show the result in (1). Consider the bound in \eqref{eq:Non-asymptotic error bound in the single-stage case}, which also holds for the two-stage setting, as discussed in (a). Noting that $\sum_{k=0}^{K-1}\varepsilon_{k+1}\le \sum_{k=0}^{K-1}\varepsilon_k$, for $\gamma = \frac{\eta^{2/3}}{\sqrt{n^{3/2}K}L_0^{3/2}}$, we obtain
 \begin{align}
\mathbb{E}[\mbox{dist}^2(0,\partial_{\eta} f(\bar{{x}}_{K^*})) ] & =  
\mathcal{O}\left(\left(\tfrac{{n^{3/4}}L_0^{3/2}}{\eta^{2/3}}\right)\tfrac{1}{\sqrt{K}}+\left(\tfrac{L_0^{3/2}n^{3/4}}{\eta^{1/3}m}\right)\tfrac{1}{\sqrt{K}}+\left(\tfrac{L_0^2n\mathbb{E}\left[\tfrac{1}{m}\|\mathbf{x}_{0}- \mathbf{1} \bar{x}_{0}\|^2 \right]}{\eta^2(1-\lambda_{\mathbf{W}}^2)}\right)\tfrac{1}{K} \right.\notag\\
&\left. + \left(\tfrac{\lambda_{\mathbf{W}}^4L_0n^{1/2}}{\eta^{2/3}(1-\lambda_{\mathbf{W}}^2)^3}\right)\tfrac{1}{K^2}+ \left(\tfrac{\lambda_{\mathbf{W}}^4\tilde L_0^2n^{3/2}\varepsilon_0}{L_0\eta^{8/3}(1-\lambda_{\mathbf{W}}^2)^3}\right)\tfrac{1}{K^2}+\left(\tfrac{\lambda_{\mathbf{W}}^2L_0n^{1/2}( \lambda_{\mathbf{W}}^2+     \|\mathbf{W}\|^2)}{\eta^{2/3}(1-\lambda_{\mathbf{W}}^2)^3}\right)\tfrac{1}{K} \right.\notag\\
&\left.+\left({\tfrac{\lambda_{\mathbf{W}}^4(3+\lambda_{\mathbf{W}}^2)}{\eta^{4/3}m\RRmath{(1-\lambda_{\mathbf{W}}^2)^4}}}\right) \tfrac{1}{K^2}+ \left(\tfrac{ \lambda_{\mathbf{W}}^2n^{3/2}}{L_0\eta^{8/3}(1-\lambda_{\mathbf{W}}^2)^3 }\right)\left( {\|\mathbf{W}\|^2\tilde L_0^2} +  \tfrac{\lambda_{\mathbf{W}}^4\tilde L_0^2}{(1-\lambda_{\mathbf{W}}^2)} \right) \tfrac{\sum_{k=0}^{K-1}\varepsilon_k}{K^{2}}\right.\notag\\
&+  \left. \left(\tfrac{ \tilde L_0^2  n \RRmath{\lambda_{\mathbf{W}}^4}}{L_0^2\eta^{10/3}\RRmath{(1-\lambda_{\mathbf{W}}^2) ^4} }\right)
 \tfrac{\sum_{k=0}^{K-1}\varepsilon_k}{K^{3}}+ \left(\tfrac{ \tilde L_0^2n^{2}}{\eta^2}\right) \tfrac{\sum_{k=0}^{K-1}\varepsilon_k}{K}+\left(\tfrac{ \tilde L_0^2n^{7/4}}{L_0^{1/2}\eta^{7/3}}\right) \tfrac{\sum_{k=0}^{K-1}\varepsilon_k}{K^{3/2}} \right).\label{eq:main one after replacing-2s}
\end{align}
Invoking Lemma~\ref{lem:Alg2_conv_2stage} with $t_k$ replaced by $\left\lceil \tfrac{-a}{\ln(1-\mu_F\hat{\gamma})}\ln\left(n^{1/2a}(k+1)\eta^{-2/3a}\right)\right\rceil$,  $\tilde d = 1-\mu_F\hat{\gamma}$, and that $(k+1)^a=e^{a\ln(k+1)}$, we obtain 
\begin{align*}
\left(\tilde d\right)^{t_k}\left(n^{1/2a}(k+1)\eta^{-2/3a}\right)^a&\le \left(\tilde d\right)^{\tfrac{-a}{\ln(1-\mu_F\hat{\gamma})}\ln\left(n^{1/2a}(k+1)\eta^{-2/3a}\right)}e^{a\ln\left(n^{1/2a}(k+1)\eta^{-2/3a}\right)}\\
&=\left(\tilde d^{\tfrac{-a}{\ln(\tilde d)}}e^{a}\right)^{\ln\left(n^{1/2a}(k+1)\eta^{-2/3a}\right)} =\left(e^{{\tfrac{-a}{\ln(\tilde d)}}\ln(\tilde d)}e^{a}\right)^{\ln\left(n^{1/2a}(k+1)\eta^{-2/3a}\right)}\\
&=1.
\end{align*}Therefore, invoking Lemma~\ref{lem:Alg2_conv_2stage}, we obtain $
\varepsilon_k\le \tfrac{\eta^{2/3}\RRmath{\hat D}}{n^{1/2}(k+1)^a}.$ By applying the preceding bound in~\eqref{eq:main one after replacing-2s}, we obtain
\begin{align}
&\mathbb{E}[\mbox{dist}^2(0,\partial_{\eta} f(\bar{{x}}_{K^*})) ]  =  
\mathcal{O}\left(\left(\tfrac{{n^{3/4}}L_0^{3/2}}{\eta^{2/3}}\right)\tfrac{1}{\sqrt{K}}+\left(\tfrac{L_0^{3/2}n^{3/4}}{\eta^{1/3}m}\right)\tfrac{1}{\sqrt{K}}+\left(\tfrac{L_0^2n\mathbb{E}\left[\tfrac{1}{m}\|\mathbf{x}_{0}- \mathbf{1} \bar{x}_{0}\|^2 \right]}{\eta^2(1-\lambda_{\mathbf{W}}^2)}\right)\tfrac{1}{K} \right.\notag\\
&\left. + \left(\tfrac{\lambda_{\mathbf{W}}^4L_0n^{1/2}}{\eta^{2/3}(1-\lambda_{\mathbf{W}}^2)^3}\right)\tfrac{1}{K^2}+ \left(\tfrac{\lambda_{\mathbf{W}}^4\tilde L_0^2n^{3/2}\varepsilon_0}{L_0\eta^{8/3}(1-\lambda_{\mathbf{W}}^2)^3}\right)\tfrac{1}{K^2}+\left(\tfrac{\lambda_{\mathbf{W}}^2L_0n^{1/2}( \lambda_{\mathbf{W}}^2+     \|\mathbf{W}\|^2)}{\eta^{2/3}(1-\lambda_{\mathbf{W}}^2)^3}\right)\tfrac{1}{K} \right.\notag\\
&\left.+\RRme{\left({\tfrac{\lambda_{\mathbf{W}}^4(3+\lambda_{\mathbf{W}}^2)}{\eta^{4/3}m(1-\lambda_{\mathbf{W}}^2)^4}}\right)} \tfrac{1}{K^2}+ \left(\tfrac{ \lambda_{\mathbf{W}}^2n}{L_0\eta^{2}(1-\lambda_{\mathbf{W}}^2)^3 }\right)\left( {\|\mathbf{W}\|^2\tilde L_0^2} +  \tfrac{\lambda_{\mathbf{W}}^4\tilde L_0^2}{(1-\lambda_{\mathbf{W}}^2)} \right) \tfrac{\sum_{k=0}^{K-1}\tfrac{\RRme{\hat D}}{\RRme{(k+1)}^a}}{K^{2}}\right.\notag\\
& \left.  +\left(\tfrac{ \tilde L_0^2n^{3/2}}{\eta^{4/3}}\right) \tfrac{\sum_{k=0}^{K-1}\tfrac{\RRme{\hat D}}{\RRme{(k+1)}^a}}{K}+\RRme{\left(\tfrac{ \tilde L_0^2  n^{1/2} \lambda_{\mathbf{W}}^4}{L_0^2\eta^{8/3}(1-\lambda_{\mathbf{W}}^2) ^4 }\right)}
 \tfrac{\sum_{k=0}^{K-1}\tfrac{\RRme{\hat D}}{\RRme{(k+1)}^a}}{K^{3}}\Rme{+\left(\tfrac{ \tilde L_0^2n^{5/4}}{L_0^{1/2}\eta^{5/3}}\right) \tfrac{\sum_{k=0}^{K-1}\tfrac{\RRme{\hat D}}{\RRme{(k+1)}^a}}{K^{3/2}} }\right).\label{eq:proof of complexity bound-2s}
\end{align}
By invoking \Rme{Lemma~\ref{lem:Harmonic series bounds}\noindent(i)} for the case where $a\in\left (\mee{\tfrac{1}{2}},1\right)$, we obtain the result in cases (1). To show (2), letting $a=1$ in~\eqref{eq:proof of complexity bound-2s} and invoking \Rme{Lemma~\ref{lem:Harmonic series bounds}\noindent(ii)} for the case where $a=1$, we obtain the result. A similar approach can be applied to show the result in (3).}

  \noindent (b-ii) \fy{The overall iteration complexity per agent is equal to $2\sum_{k=0}^{\RRmath{K_\epsilon-1}} t_k$ that is  $$\tfrac{- a}{\ln(1-\mu_F\hat{\gamma})}\mathcal{O}\left(\sum_{k=0}^{K_\epsilon} \ln\left(n^{1/2a}(k+1)\eta^{-2/3a}\right)\right) = \tfrac{- a}{\ln(1-\mu_F\hat{\gamma})}\mathcal{O}\left(\tfrac{1}{2a}\ln(n)K_\epsilon + K_\epsilon\ln(K_\epsilon) +\tfrac{2}{3a}\ln(1/\eta)K_\epsilon\right),$$ 
where we used  $
\int_1^{K_\epsilon+1} \ln(x) \, dx = (K_\epsilon+1) \ln(K_\epsilon+1) - (K_\epsilon+1) + 1.
$ Noting that $\hat{\gamma}\RRmath{:=\frac{\mu_F}{2L_F^2}} \leq \frac{\mu_F}{L_F^2}$, $a > \mee{\tfrac{1}{2}}$. and that $\ln(K_\epsilon) \geq \ln(n)$ and $\ln(K_\epsilon) \geq -\ln(\eta)$ (in view of \eqref{eqn:K_eps_2s}), we obtain  
$2\sum_{k=0}^{K_\epsilon} t_k =\tfrac{- a}{\ln(1-\RRmath{\kappa_F^{-2}/2})}{\tilde{\mathcal{O}}}\left(K_\epsilon\right),  $ where $K_\epsilon$ is given by \eqref{eqn:K_eps_2s}. 
  } 
\end{proof}
 \fy{
\begin{remark}\em In a similar approach discussed in Remark~\ref{rem:1s-a}, we may consider three cases for the parameter $a$ in Theorem~\ref{Theorem:thm 2}. Choosing $a=1$ yields the best choice for minimizing the overall iteration complexity.  From part \noindent (b-ii) of Theorem~\ref{Theorem:thm 2} we know that the overall iteration is $\tilde{\mathcal{O}}\left(K_\epsilon\right)$. Theorem~\ref{Theorem:thm 2} provides three different cases for $a$ with the only distinction being the term $\mee{J_\epsilon}$. Let us denote the corresponding term of $K_\epsilon$ as $K_{\epsilon,\mee{J_\epsilon}}$. The complexity $\tilde{\mathcal{O}}\left(K_{\epsilon,\mee{J_\epsilon}}\right)$ now varies across the different cases for $a$.  If $a \in \left(\mee{\tfrac{1}{2}},1\right)$, then $\tilde{\mathcal{O}}\left(K_{\epsilon,\mee{J_\epsilon}}\right)= \tilde{ \mathcal{O}}\left(\epsilon^{-1/(1+a)}+\epsilon^{-1/{(2+a)}}+\epsilon^{-1/{\left(\mee{\tfrac{1}{2}}+a\right)}}\right)$. If $a =1$, then $\tilde{\mathcal{O}}\left(K_{\epsilon,\mee{J_\epsilon}}\right)=\tilde{ \mathcal{O}}\left(\epsilon^{-1/2}+\epsilon^{{-1}/{3}}+\epsilon^{{-2}/{3}}\right)$. If $a >1$, then $\tilde{\mathcal{O}}\left(K_{\epsilon,\mee{J_\epsilon}}\right)= \tilde{\mathcal{O}}\left(\epsilon^{-1/2}+\epsilon^{{-1}/{3}}+\epsilon^{{-2}/{3}}\right)$. Among these three cases, the minimum value of $\tilde{\mathcal{O}}\left(K_\epsilon\right)$ occurs when $a=1$. \mee{$\hfill \Box$}
\end{remark}

\begin{remark}\em As in the single-stage case, when all agents initialize at the same point, the term $\mathbb{E}\left[\tfrac{1}{m}\|\mathbf{x}_{0}- \mathbf{1} \bar{x}_{0}\|^2 \right]$ in Theorem~\ref{Theorem:thm 2} evaluates to zero. The \mee{summarized} complexity results presented in Table~\ref{table:tbl2-smpec_complexity_2s} are derived under this assumption. \mee{$\hfill \Box$}
\end{remark}

\begin{table}[]
\mee{
\centering\footnotesize{
\caption{Complexity guarantees for solving distributed two-stage SMPECs~\eqref{eqn:prob-2stage}}
\label{table:tbl2-smpec_complexity_2s}
\begin{tabular}{@{}lc@{}}
\toprule
\multicolumn{1}{c}{Two-stage SMPEC}                                     &Iteration complexity bound          \\ \midrule
{\bf{\underline{Inexact setting}}}                                         & \multicolumn{1}{l}{} \\
\begin{tabular}[c]{@{}l@{}}Upper level\\ problem\end{tabular} &$\mathcal{O}\left(\tfrac{n^{1/6}\tilde L_0^{2/3}\epsilon^{-1/3}}{ L_0^{2/3}\eta^{8/9}}
+\tfrac{\tilde L_0^2n^{3/2}\epsilon^{-1}}{\eta^{4/3}}+\tfrac{L_0^3n^{3/2}\epsilon^{-2}}{\eta^{4/3}m^2}+\tfrac{  n^{3/4}\tilde L_0\epsilon^{\mee{-1/2}} }{\eta^{4/3}L_0^{1/2}  }+\tfrac{  L_0^{1/2}n^{1/4}\epsilon^{\mee{-1/2}} }{\eta^{1/3}}+\tfrac{  n^{5/6}\tilde L_0^{4/3}\epsilon^{-2/3} }{L_0^{1/3}\eta^{10/9}  }\right)$                    \\\cmidrule(l){2-2}
\begin{tabular}[c]{@{}l@{}}Lower level\\ problem\end{tabular} & $\mathcal{O}\left(\ln({\epsilon})/\ln(1- \RRmath{\kappa_F^{-2}/2})\right)$                   \\\cmidrule(l){2-2}
\begin{tabular}[c]{@{}l@{}}Overall\\ complexity\end{tabular}     & ${\tfrac{1}{\ln(1- \RRmath{\kappa_F^{-2}/2})}\,\tilde{\mathcal{O}}(K_\epsilon)},$ {such that $K_\epsilon$ is the upper level iteration complexity  }                  \\ \midrule
\begin{tabular}[c]{@{}l@{}}{\bf{{Exact setting  }}} \\ \end{tabular}         &\Rme{$ \mathcal{O}\left(\tfrac{L_0n^{1/2}\epsilon^{-1}}{\eta}+\tfrac{L_0^3n^{3/2}\epsilon^{-2}}{\eta m^2}+\tfrac{L_0^3n^{3/2}\epsilon^{-2}}{\eta}+\tfrac{  L_0^{1/2}n^{1/4}\epsilon^{\mee{\Rme{-1/2}}} }{\eta^{1/2} }\right)$ }                \\ \bottomrule
\end{tabular}}}
\end{table}

\begin{remark}\em Similar to the single-stage settings, the \Rme{exponents} of $\eta$, $n$, and $L_0$ in the definition of $\gamma$ in Theorem~\ref{Theorem:thm 2} are carefully selected to minimize the dependence of the complexity guarantees on these terms. \mee{$\hfill \Box$}
\end{remark}
\begin{remark}\em The choice of the number of iterations for the lower-level algorithm plays a crucial role in the overall performance. To ensure that the error bound in part (\noindent {a}) of Theorems~\ref{Theorem:thm 1} and~\ref{Theorem:thm 2} converges to zero, it is sufficient for the term $\tfrac{ \sum_{k=0}^{K-1}\varepsilon_k}{K}$ to approach zero. This implies that the lower-level algorithm must run for a sufficient number of iterations to achieve convergence. However, running the lower-level algorithm for too many iterations increases the overall iteration complexity. To minimize the total number of iterations required to solve the problem, we prescribe \fy{an efficient} termination rule for the lower-level algorithm. \Rme{Alternative choices for the number of lower-level iterations can be expressed in terms of the total number of upper-level iterations, i.e., $K$. For instance, one may set
$t_k=\left\lceil \frac{n^{1/2}(K+\Gamma)^a}{\eta^{2/3}}\right\rceil$ for the single-stage case and
$t_k=\left\lceil \tfrac{-a}{\ln(1-\mu_F\hat{\gamma})}\ln\left(n^{1/2a}(K+1)\eta^{-2/3a}\right)\right\rceil$ for the two-stage case.
However, since $K$ is typically determined  \Rme{in accordance with the required error tolerance $\epsilon$ by leveraging the rate guarantee at the upper level}, such choices are less practical and \Rme{may not necessarily} improve the asymptotic complexity guarantees. They may also lead to a larger total number of iterations when constant factors are taken into account.} Furthermore, although the upper bounds in part (\noindent {a}) of Theorems~\ref{Theorem:thm 1} and~\ref{Theorem:thm 2} are identical, \Rme{obtaining an inexact lower-level solution} requires significantly fewer iterations in the two-stage setting \Rme{than in} the single-stage setting, \Rme{since the lower-level problem is deterministic {and strongly monotone} in the two-stage case (\Rme{which in turn allows for leveraging geometrically convergent schemes}).} Consequently, the overall iteration complexity is \Rme{lower} in the two-stage setting. \mee{$\hfill \Box$}
\end{remark}}

\begin{remark}\em
\RRme{As discussed for the single-stage setting, the case} $\RRmath{\lambda_{\mathbf{W}}=0}$ \RRme{is a special case allowed by Assumption~\ref{assum:mixxx}. In this case,} $\RRmath{\mathbf{W}=\tfrac{1}{m}\mathbf{1}\mathbf{1}^{\top}}$\RRme{, and the iterates reach exact consensus after the first update. Since the compact representations in~\eqref{eqn:R1} and~\eqref{eqn:R2} also apply to the two-stage setting, the case} $\RRmath{\lambda_{\mathbf{W}}=0}$ \RRme{can be treated by using the corresponding bound in Lemma~\ref{lemma:bound for x-bar x inexact}. The resulting error and complexity bounds are provided in the following corollary.} $\hfill\Box$
\end{remark}

\begin{corollary}[\RRme{\textbf{Convergence guarantees for DiZS-GT$^{\text{2s}}$ when } $\RRmath{\lambda_{\mathbf{W}}=0}$}]\em
\label{cor:thm2-lambda-zero}
\RRme{Consider Algorithm~\ref{alg:DZGT-2stage}. Let Assumptions~\ref{assum:mixxx}, \ref{assump:opt_f_bounded_below}, \RRme{\ref{assum:samples}}, \ref{assum:main6}, \ref{assum:main2-2stage}, and \ref{assumption:measuribility assumption} hold, and suppose} $\RRmath{\lambda_{\mathbf{W}}=0}$\RRme{.}

\noindent\RRme{(a) [\textbf{Non-asymptotic error bounds}] Suppose the stepsize} $\RRmath{\gamma}$ \RRme{is constant and satisfies} $\RRmath{\gamma\leq\tfrac{\eta}{6L_0\sqrt{n}}}$\RRme{. Suppose} $\RRmath{K^*}$ \RRme{is a discrete uniform random variable such that} $\RRmath{\mathbb{P}[K^*=\ell]=\tfrac{1}{K}}$ \RRme{for} $\RRmath{\ell=0,\ldots,K-1}$\RRme{. Then, for any} $\RRmath{K\geq 1}$\RRme{, the bounds in~\eqref{eq:lambda-zero-stationarity-1s} and~\eqref{eq:lambda-zero-consensus-1s} also hold in this setting.}

\noindent\RRme{(b) [\textbf{Complexity bounds}] Suppose at iteration} $\RRmath{k}$ \RRme{in Algorithm~\ref{alg:DZGT-2stage}, in generating the inexact solutions, Algorithm~\ref{alg:lowerlevel-2stage} is terminated after} $\RRmath{t_k:=\left\lceil \tfrac{-a}{\ln(1-\mu_F\hat{\gamma})}\ln\left(n^{1/2a}(k+1)\eta^{-2/3a}\right)\right\rceil}$ \RRme{iterations, where} $\RRmath{a>0.5}$\RRme{. Let} $\RRmath{\epsilon>0}$ \RRme{be an arbitrary scalar such that} $\RRmath{\mathbb{E}[\mbox{dist}^2(0,\partial_{\eta} f(\bar{x}_{K_\epsilon^*}))]\leq\epsilon}$\RRme{. Then, the following results hold.}

\noindent\RRme{(b-i) [\textbf{Upper-level iteration/sample complexity}] Suppose the stepsize in Algorithm~\ref{alg:DZGT-2stage} is given by} $\RRmath{\gamma=\tfrac{\eta^{2/3}}{\sqrt{n^{3/2}K_\epsilon}L_0^{3/2}}}$ \RRme{such that} $\RRmath{K_\epsilon\geq\max\left\{1,\tfrac{36}{\eta^{2/3}L_0n^{1/2}}\right\}}$\RRme{. Then, we have}
\begin{align}
\RRmath{
K_\epsilon
=
\mathcal{O}\left(
\left(\tfrac{n^{3/2}L_0^3}{\eta^{4/3}}\right)\epsilon^{-2}
+
\left(\tfrac{L_0^3n^{3/2}}{\eta^{2/3}m^2}\right)\epsilon^{-2}
+
\left(
\tfrac{L_0^2n\mathbb{E}\left[\tfrac{1}{m}\|\mathbf{x}_0-\mathbf{1}\bar{x}_0\|^2\right]}
{\eta^2}
\right)\epsilon^{-1}
\right)
+J_{\epsilon}',
}
\label{eq:lambda-zero-complexity-2s}
\end{align}
\RRme{where} $\RRmath{J_{\epsilon}'}$ \RRme{is characterized as follows.}

\noindent\RRme{(1) If} $\RRmath{a\in(\tfrac{1}{2},1)}$\RRme{, then}
\begin{align*}
\RRmath{
J_{\epsilon}'
=
\mathcal{O}\left(
\left(
\tfrac{\hat D\tilde L_0^2n^{3/2}}
{\eta^{4/3}(1-a)}
\right)^{1/a}
\epsilon^{-1/a}
+
\left(
\tfrac{\hat D\tilde L_0^2n^{5/4}}
{L_0^{1/2}\eta^{5/3}(1-a)}
\right)^{1/(a+\tfrac{1}{2})}
\epsilon^{-1/(a+\tfrac{1}{2})}
\right).
}
\end{align*}

\noindent\RRme{(2) If} $\RRmath{a=1}$\RRme{, then}
\begin{align*}
\RRmath{
J_{\epsilon}'
=
\tilde{\mathcal{O}}\left(
\left(
\tfrac{\hat D\tilde L_0^2n^{5/4}}
{L_0^{1/2}\eta^{5/3}}
\right)^{2/3}
\epsilon^{-2/3}
+
\left(
\tfrac{\hat D\tilde L_0^2n^{3/2}}
{\eta^{4/3}}
\right)
\epsilon^{-1}
\right).
}
\end{align*}

\noindent\RRme{(3) If} $\RRmath{a>1}$\RRme{, then}
\begin{align*}
\RRmath{
J_{\epsilon}'
=
\mathcal{O}\left(
\left(
\tfrac{a\hat D\tilde L_0^2n^{5/4}}
{L_0^{1/2}\eta^{5/3}(a-1)}
\right)^{2/3}
\epsilon^{-2/3}
+
\left(
\tfrac{a\hat D\tilde L_0^2n^{3/2}}
{\eta^{4/3}(a-1)}
\right)
\epsilon^{-1}
\right).
}
\end{align*}

\noindent\RRme{(b-ii) [\textbf{Overall iteration complexity}] To guarantee} $\RRmath{\mathbb{E}[\mbox{dist}^2(0,\partial_{\eta} f(\bar{x}_{K^*}))]\leq\epsilon}$\RRme{, the overall iteration complexity is} $\RRmath{\tfrac{-a}{\ln(1-\kappa_F^{-2})}\tilde{\mathcal{O}}\left(K_\epsilon\right)}$\RRme{, where} $\RRmath{K_\epsilon}$ \RRme{is given by~\eqref{eq:lambda-zero-complexity-2s}.}
\end{corollary}
\begin{proof}
\noindent\RRme{(a) As established in the proof of Theorem~\ref{Theorem:thm 2}, the compact representations in~\eqref{eqn:R1} and~\eqref{eqn:R2} apply to the two-stage setting, and the error term satisfies} $\RRmath{\mathbb{E}[\|e_{i,k}^{\eta,\varepsilon_k}\|^2\mid\mathcal{F}_k]\leq\tfrac{\tilde L_0^2n^2\varepsilon_k}{\eta^2}}$\RRme{. Therefore, the upper-level descent analysis used for the single-stage setting can also be applied here. In particular, for} $\RRmath{\gamma\leq\tfrac{\eta}{6L_0\sqrt{n}}}$\RRme{, we obtain}
\begin{align*}
\RRmath{
\sum_{k=0}^{K-1}\mathbb{E}[\|\nabla f^\eta(\bar{x}_k)\|^2]
&\leq
\tfrac{4(\mathbb{E}[f^\eta(\bar{x}_0)]-\inf_x f^\eta(x))}{\gamma}
-\sum_{k=0}^{K-1}\mathbb{E}[\|\overline{\nabla f^\eta}(\mathbf{x}_k)\|^2]
+\tfrac{96\sqrt{2\pi}L_0^3n^{3/2}\gamma K}{m\eta}\\
&
+\tfrac{2L_0^2n}{\eta^2}
\sum_{k=0}^{K-1}
\mathbb{E}\left[
\tfrac{\|\mathbf{x}_k-\mathbf{1}\bar{x}_k\|^2}{m}
\right]
+\left(1+\tfrac{3\gamma L_0\sqrt{n}}{2\eta}\right)
\left(\tfrac{4\tilde L_0^2n^2}{\eta^2}\right)
\sum_{k=0}^{K-1}\varepsilon_k.
}
\end{align*}
\RRme{Since} $\RRmath{\lambda_{\mathbf{W}}=0}$\RRme{, Lemma~\ref{lemma:bound for x-bar x inexact} gives} $\RRmath{\sum_{k=0}^{K-1}\mathbb{E}\left[\tfrac{\|\mathbf{x}_k-\mathbf{1}\bar{x}_k\|^2}{m}\right]=\mathbb{E}\left[\tfrac{\|\mathbf{x}_0-\mathbf{1}\bar{x}_0\|^2}{m}\right]}$\RRme{. Substituting this relation into the preceding bound, dropping the \RRme{nonpositive} term involving} $\RRmath{\overline{\nabla f^\eta}(\mathbf{x}_k)}$\RRme{, dividing by} $\RRmath{K}$\RRme{, and invoking Proposition~\ref{Prop:2eta}(ii), we obtain the stationarity bound in~\eqref{eq:lambda-zero-stationarity-1s}. Moreover, since} $\RRmath{K^*}$ \RRme{is uniformly distributed over} $\RRmath{\{0,\ldots,K-1\}}$\RRme{,}
$\RRmath{\mathbb{E}\left[\tfrac{\|\mathbf{x}_{K^*}-\mathbf{1}\bar{x}_{K^*}\|^2}{m}\right]=\tfrac{1}{K}\sum_{k=0}^{K-1}\mathbb{E}\left[\tfrac{\|\mathbf{x}_k-\mathbf{1}\bar{x}_k\|^2}{m}\right]=\tfrac{1}{K}\mathbb{E}\left[\tfrac{\|\mathbf{x}_0-\mathbf{1}\bar{x}_0\|^2}{m}\right]}$\RRme{. Hence, the consensus bound in~\eqref{eq:lambda-zero-consensus-1s} also holds.}

\noindent\RRme{(b-i) We first substitute} $\RRmath{\gamma=\tfrac{\eta^{2/3}}{\sqrt{n^{3/2}K}L_0^{3/2}}}$ \RRme{into the stationarity bound from part (a). This gives}
\begin{align}
\RRmath{
\mathbb{E}[\mbox{dist}^2(0,\partial_{\eta} f(\bar{x}_{K^*}))]
&=
\mathcal{O}\Bigg(
\left(\tfrac{n^{3/4}L_0^{3/2}}{\eta^{2/3}}\right)\tfrac{1}{\sqrt{K}}
+\left(\tfrac{L_0^{3/2}n^{3/4}}{\eta^{1/3}m}\right)\tfrac{1}{\sqrt{K}}
+\left(
\tfrac{L_0^2n\mathbb{E}\left[\tfrac{1}{m}\|\mathbf{x}_0-\mathbf{1}\bar{x}_0\|^2\right]}
{\eta^2}
\right)\tfrac{1}{K}\notag\\
&+
\left(\tfrac{\tilde L_0^2n^2}{\eta^2}\right)
\tfrac{\sum_{k=0}^{K-1}\varepsilon_k}{K}
+
\left(\tfrac{\tilde L_0^2n^{7/4}}{L_0^{1/2}\eta^{7/3}}\right)
\tfrac{\sum_{k=0}^{K-1}\varepsilon_k}{K^{3/2}}
\Bigg).
}
\label{eq:lambda-zero-main-2s}
\end{align}\RRme{Next, invoking Lemma~\ref{lem:Alg2_conv_2stage} with} $\RRmath{t_k=\left\lceil \tfrac{-a}{\ln(1-\mu_F\hat{\gamma})}\ln\left(n^{1/2a}(k+1)\eta^{-2/3a}\right)\right\rceil}$\RRme{, and letting} $\RRmath{\tilde d=1-\mu_F\hat{\gamma}}$\RRme{, we have}
\begin{align*}
\RRmath{
\left(\tilde d\right)^{t_k}
\left(n^{1/2a}(k+1)\eta^{-2/3a}\right)^a
&\leq
\left(\tilde d\right)^{
\tfrac{-a}{\ln(1-\mu_F\hat{\gamma})}
\ln\left(n^{1/2a}(k+1)\eta^{-2/3a}\right)}
e^{a\ln\left(n^{1/2a}(k+1)\eta^{-2/3a}\right)}
=1.
}
\end{align*}
\RRme{Consequently, Lemma~\ref{lem:Alg2_conv_2stage} yields} $\RRmath{\varepsilon_k\leq\tfrac{\eta^{2/3}\hat D}{n^{1/2}(k+1)^a}}$\RRme{. Substituting this bound into~\eqref{eq:lambda-zero-main-2s}, we obtain}
\begin{align}
\RRmath{
\mathbb{E}[\mbox{dist}^2(0,\partial_{\eta} f(\bar{x}_{K^*}))]
&=
\mathcal{O}\Bigg(
\left(\tfrac{n^{3/4}L_0^{3/2}}{\eta^{2/3}}\right)\tfrac{1}{\sqrt{K}}
+\left(\tfrac{L_0^{3/2}n^{3/4}}{\eta^{1/3}m}\right)\tfrac{1}{\sqrt{K}}
+\left(
\tfrac{L_0^2n\mathbb{E}\left[\tfrac{1}{m}\|\mathbf{x}_0-\mathbf{1}\bar{x}_0\|^2\right]}
{\eta^2}
\right)\tfrac{1}{K}\notag\\
&+
\left(\tfrac{\tilde L_0^2n^{3/2}}{\eta^{4/3}}\right)
\tfrac{\sum_{k=0}^{K-1}\tfrac{\hat D}{(k+1)^a}}{K}
+
\left(\tfrac{\tilde L_0^2n^{5/4}}{L_0^{1/2}\eta^{5/3}}\right)
\tfrac{\sum_{k=0}^{K-1}\tfrac{\hat D}{(k+1)^a}}{K^{3/2}}
\Bigg).
}
\label{eq:lambda-zero-complexity-proof-2s}
\end{align}

\noindent\RRme{For} $\RRmath{a\in(\tfrac{1}{2},1)}$\RRme{, applying Lemma~\ref{lem:Harmonic series bounds}(i) to~\eqref{eq:lambda-zero-complexity-proof-2s} and requiring each term to be bounded by a constant multiple of} $\RRmath{\epsilon}$ \RRme{gives case (1) for} $\RRmath{J_\epsilon'}$\RRme{. For} $\RRmath{a=1}$\RRme{, Lemma~\ref{lem:Harmonic series bounds}(ii) gives case (2). The argument for} $\RRmath{a>1}$ \RRme{proceeds by applying the corresponding bound on the harmonic sum and gives case (3). Combining these terms yields the bound on} $\RRmath{K_\epsilon}$ \RRme{in~\eqref{eq:lambda-zero-complexity-2s}.}

\noindent\RRme{(b-ii) At iteration} $\RRmath{k}$ \RRme{of Algorithm~\ref{alg:DZGT-2stage},} $\RRmath{2t_k}$ \RRme{iterations of Algorithm~\ref{alg:lowerlevel-2stage} are required. Thus, the overall iteration complexity per agent is proportional to} $\RRmath{2\sum_{k=0}^{K_\epsilon}t_k}$\RRme{. Using the definition of} $\RRmath{t_k}$\RRme{, we obtain}
\begin{align*}
\RRmath{
2\sum_{k=0}^{K_\epsilon}t_k
=
\tfrac{-a}{\ln(1-\mu_F\hat{\gamma})}
\mathcal{O}\left(
\tfrac{1}{2a}\ln(n)K_\epsilon
+
K_\epsilon\ln(K_\epsilon)
+
\tfrac{2}{3a}\ln(1/\eta)K_\epsilon
\right).
}
\end{align*}
\RRme{Here, we use} $\RRmath{\sum_{k=0}^{K_\epsilon}\ln(k+1)=\mathcal{O}(K_\epsilon\ln K_\epsilon)}$\RRme{. Since} $\RRmath{\hat{\gamma}\leq\tfrac{\mu_F}{L_F^2}}$\RRme{, we have} $\RRmath{\mu_F\hat{\gamma}\leq\kappa_F^{-2}}$\RRme{. Together with the bounds on} $\RRmath{K_\epsilon}$\RRme{, this gives}
$\RRmath{2\sum_{k=0}^{K_\epsilon}t_k=\tfrac{-a}{\ln(1-\kappa_F^{-2})}\tilde{\mathcal{O}}(K_\epsilon)}$\RRme{, where} $\RRmath{K_\epsilon}$ \RRme{is given by~\eqref{eq:lambda-zero-complexity-2s}.}
\end{proof}

\section{Numerical experiments}\label{sec:num}
In this section, we present preliminary experiments to empirically validate the convergence \us{claims for} the proposed schemes for distributed single-stage and two-stage SMPECs {and provide comparisons with \us{centralized counterparts} for} centralized SMPECs~\cite{cui2023complexity} in Sections~\ref{subsection:single-stage numerics} and~\ref{subsection:two-stage numerics}, respectively.  
\subsection{{Distributed single-stage SMPECs}}\label{subsection:single-stage numerics}
{In this subsection}, we compare the performance of the {DiZS-GT$^{\text{1s}}$ method presented in Algorithm~\ref{alg:DZGT} with ZSOL$_{\text{ncvx}}^{\text{1s}}$ (Algorithm 3 in \cite{cui2023complexity}), which can be viewed as a centralized counterpart of our scheme. We consider a bilevel optimization problem of the form
\begin{align*}
\min_x \ \tfrac{1}{m}\textstyle \sum_{i=1}^m\mathbb{E}[-x_1^2-3x_2-{\bxi_iy_{i,1}(x)}+\left({y_{i,2}}(x)\right)^2],
\end{align*}
where $y_i(x)\triangleq(y_{i,1},y_{i,2})$ is the unique solution to the following \us{parametrized} optimization problem. 
\begin{align*}
\min_{y \, \us{ \, \ge \, 0}}   \ \ &\mathbb{E}[\, 2x_1^2+{y_{i,1}}^2+{y^2_{i,2}}-{\boldsymbol\zeta}_i y_{i,2}\, ]\\
   \hbox{s.t.} \ \    &x_1^2-2x_1+x_2^2-2{y_{i,1}}+{y_{i,2}}\ge -3, \notag \\
  &x_2+3{y_{i,1}}-{y_{i,2}}\ge 4.
\end{align*}
Notably, the constraints of the lower-level problem depend on the upper-level decision $x$. 

\noindent \textbf{Problem and algorithm parameters.} \RRme{We consider $m=20$ agents and} assume that ${\bxi}_i$ and ${\boldsymbol\zeta}_i$ are \us{uniform random variables} and iid for all the agents, given as ${\bxi}_i\sim\RRme{\mathcal{U}(5,8)}$ and ${\boldsymbol\zeta}_i \sim\mathcal{U}(3,6)$. \RRme{For the distributed method, each agent uses a mini-batch of size five for the stochastic computations in both the upper- and lower-level algorithms. For the centralized counterpart, we use a mini-batch of size $5m=100$ in the corresponding upper- and lower-level computations. Thus, the total sampling and oracle budgets are matched between the distributed and centralized implementations. In particular, since the zeroth-order gradient estimator evaluates the implicit objective at two perturbed points for each sampled realization, both implementations use $2\times m\times 5=200$ upper-level objective function evaluations per upper-level iteration. Similarly, at each lower-level iteration and for each perturbed point, the distributed implementation uses $m\times5=100$ stochastic mapping evaluations across the agents, while the centralized implementation uses $100$ such evaluations. Therefore, since the lower-level algorithm is run for $t_k$ iterations at each of the two perturbed points, both implementations use $2\times100\times t_k=200t_k$ lower-level stochastic mapping evaluations at upper-level iteration $k$.} We run both upper-level algorithms for \RRme{$100$} iterations and choose $\gamma=10^{-4}$ and $\eta=10^{-1}$. In addition, the lower-level algorithm (Algorithm~\ref{alg:lowerlevel-1stage} in our work and Algorithm 4 in~\cite{cui2023complexity}) at each iteration of the upper-level algorithm is terminated after $t_k=\left\lceil \sqrt{n}(k+1)/\eta^{2/3}\right\rceil$ iterations, where $n$ in this example is $2$ and $k$ denotes the iteration index of the upper-level algorithm (Algorithm~\ref{alg:DZGT} in our work and Algorithm 3 in~\cite{cui2023complexity}). Furthermore, we consider five different network settings, including the ring graph, a sparse graph, a sparse tree graph, the Erd\H{o}s--R\'enyi graph, and the complete graph.

\noindent \textbf{Evaluation of the global implicit objective function.}
For each method and setting, we run the scheme five times and report the sample mean of the global implicit objective function. We consider five epochs for plotting the global implicit objective function. To evaluate \RRme{and plot} the \RRme{global} objective function at each epoch, we use approximations of \RRme{$y_{i,1}(x)$} and \RRme{$y_{i,2}(x)$} obtained by implementing the projected stochastic gradient method, i.e., Algorithm~\ref{alg:lowerlevel-1stage}, \RRme{for $150$ lower-level iterations. These iterations are used only for evaluating the global implicit objective function and are distinct from the $100$ upper-level iterations used to generate the trajectories reported in the figures. The horizontal axis in the figures represents the iteration index of the upper-level algorithms.}

\noindent \textbf{Insights.} The implementation results are presented in Figure~\ref{fig:comparison-single stage}. We observe the following. (i) {DiZS-GT$^{\text{1s}}$ appears to display robustness to the connectivity of the network among the agents. By increasing the connectivity of the network among the agents, our method performs better. This {becomes particularly} clear from the consensus error plots. (ii) The performance of our method is similar to that of ZSOL$_{\text{ncvx}}^{\text{1s}}$, especially for the complete graph.

\begin{figure}[!htbp]
\centering{
\setlength{\tabcolsep}{0pt}
\centering
 \begin{tabular}{c || c  c  c }
  {\footnotesize {Setting}\ \ }& {\footnotesize  Communication network} & {\footnotesize $\mathbb{E}\left[f(\bar{x}_{k})\right]$} & {\footnotesize $\mathbb{E}\left[{\|\mathbf{x}_{k}- \mathbf{1} \bar{x}_{k}\|^2 }\right]$ } \\
 \hline\\
\rotatebox[origin=c]{90}{{\footnotesize {complete graph}}}
&
\begin{minipage}{.30\textwidth}
\includegraphics[scale=.329, angle=0]{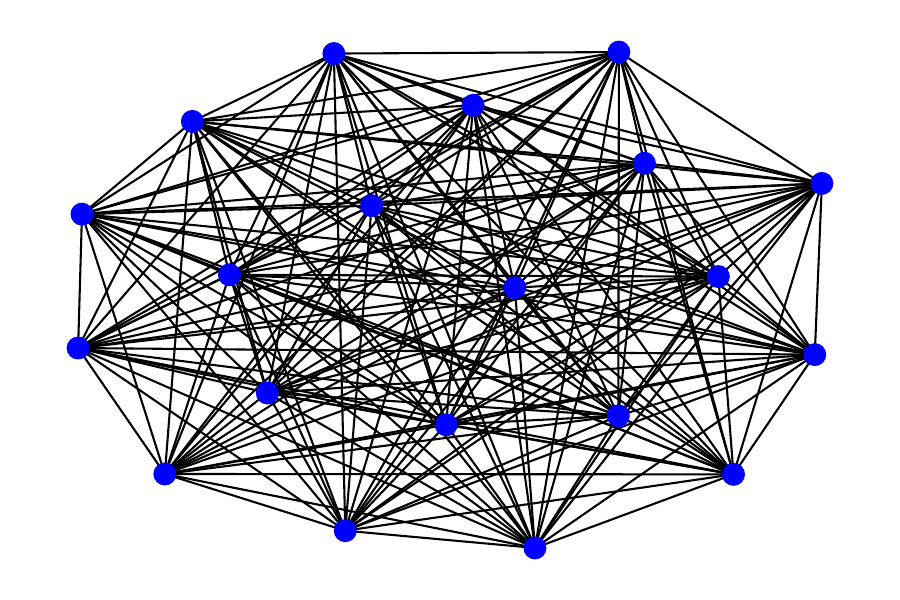}
\end{minipage}
&
\begin{minipage}{.30\textwidth}
\includegraphics[scale=.199, angle=0]{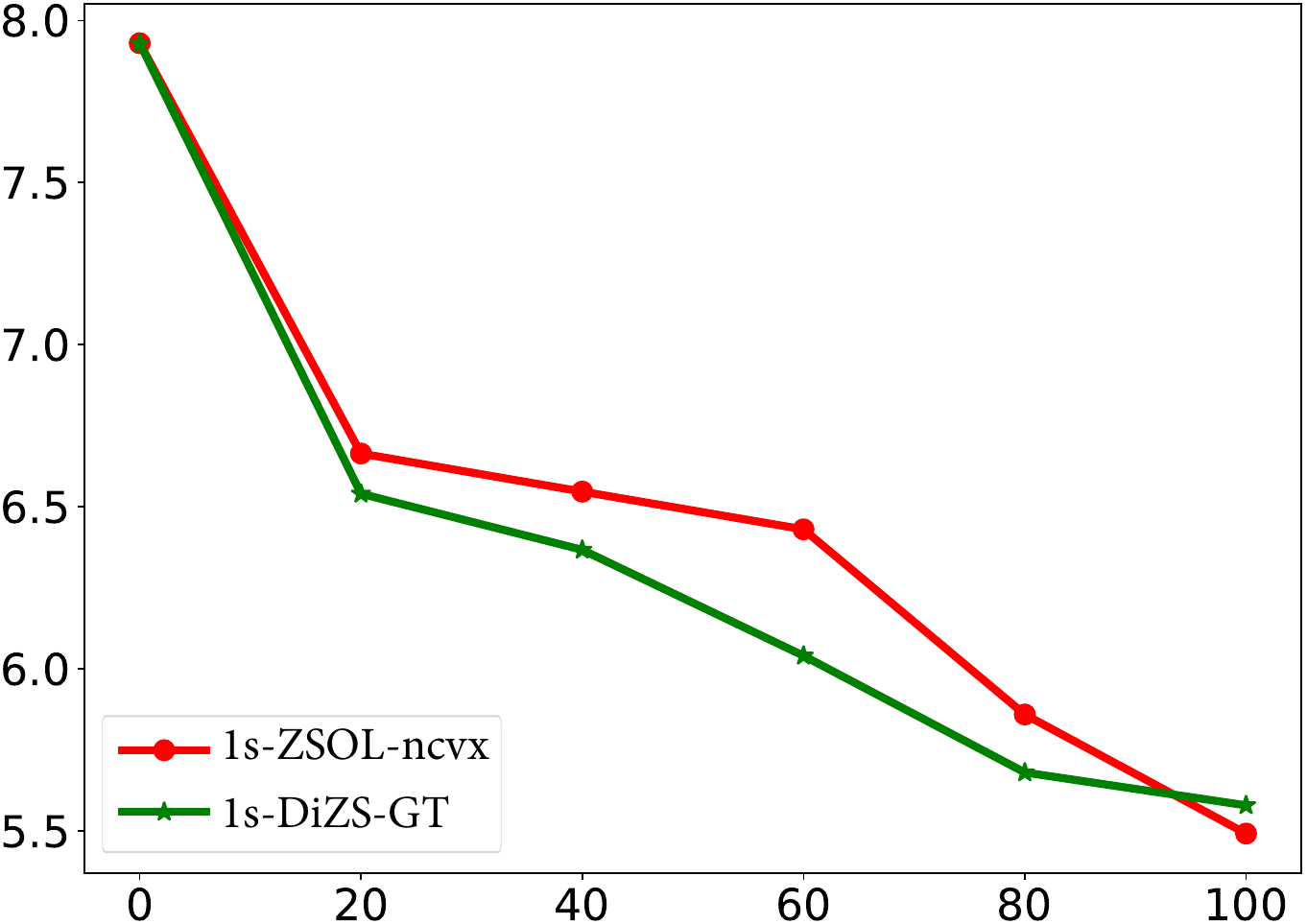}
\end{minipage}
	&
\begin{minipage}{.30\textwidth}
\includegraphics[scale=.199, angle=0]{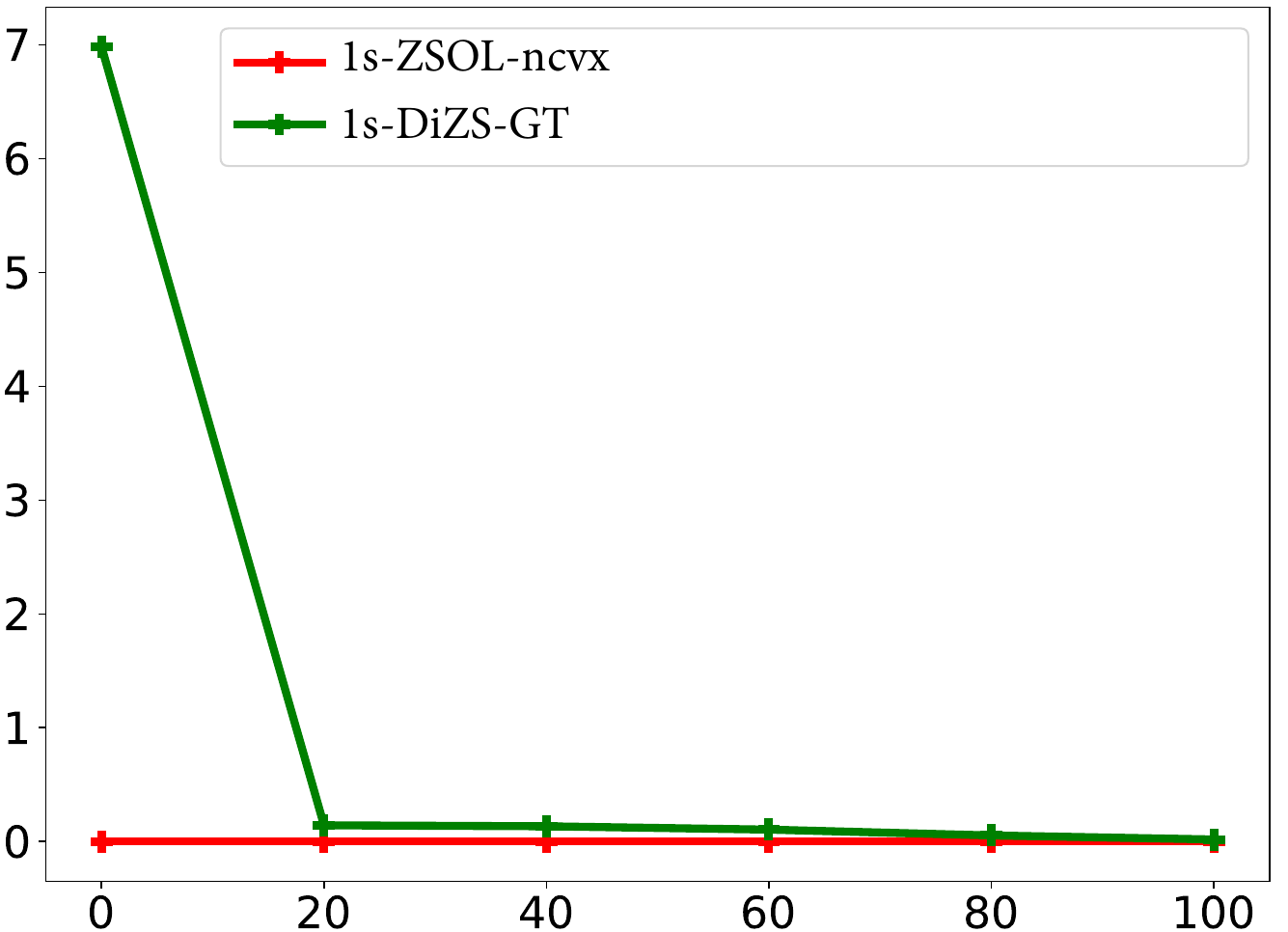}
\end{minipage}
\\ 
\hline\\
\rotatebox[origin=c]{90}{{\footnotesize {ring graph}}}
&
\begin{minipage}{.30\textwidth}
\includegraphics[scale=.329, angle=0]{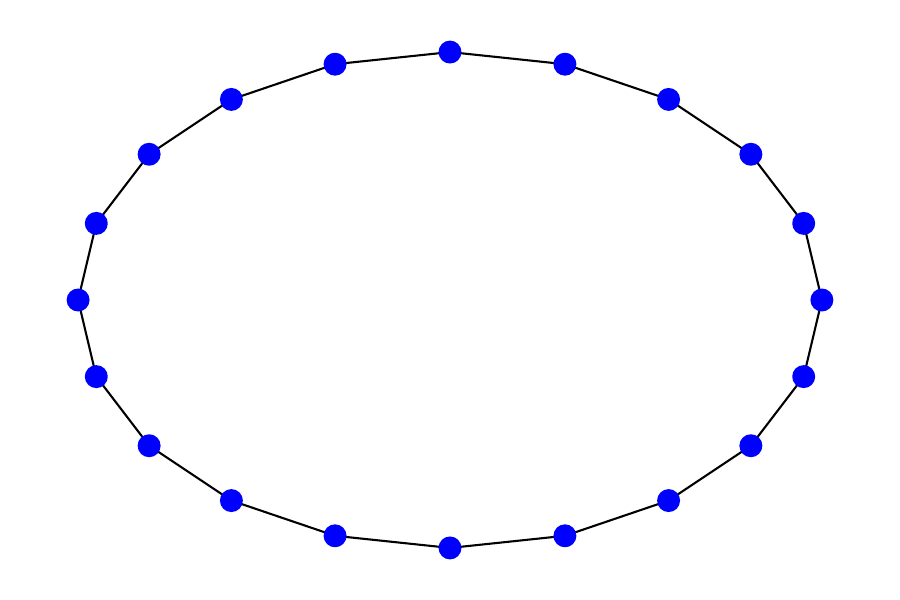}
\end{minipage}
&
\begin{minipage}{.30\textwidth}
\includegraphics[scale=.199, angle=0]{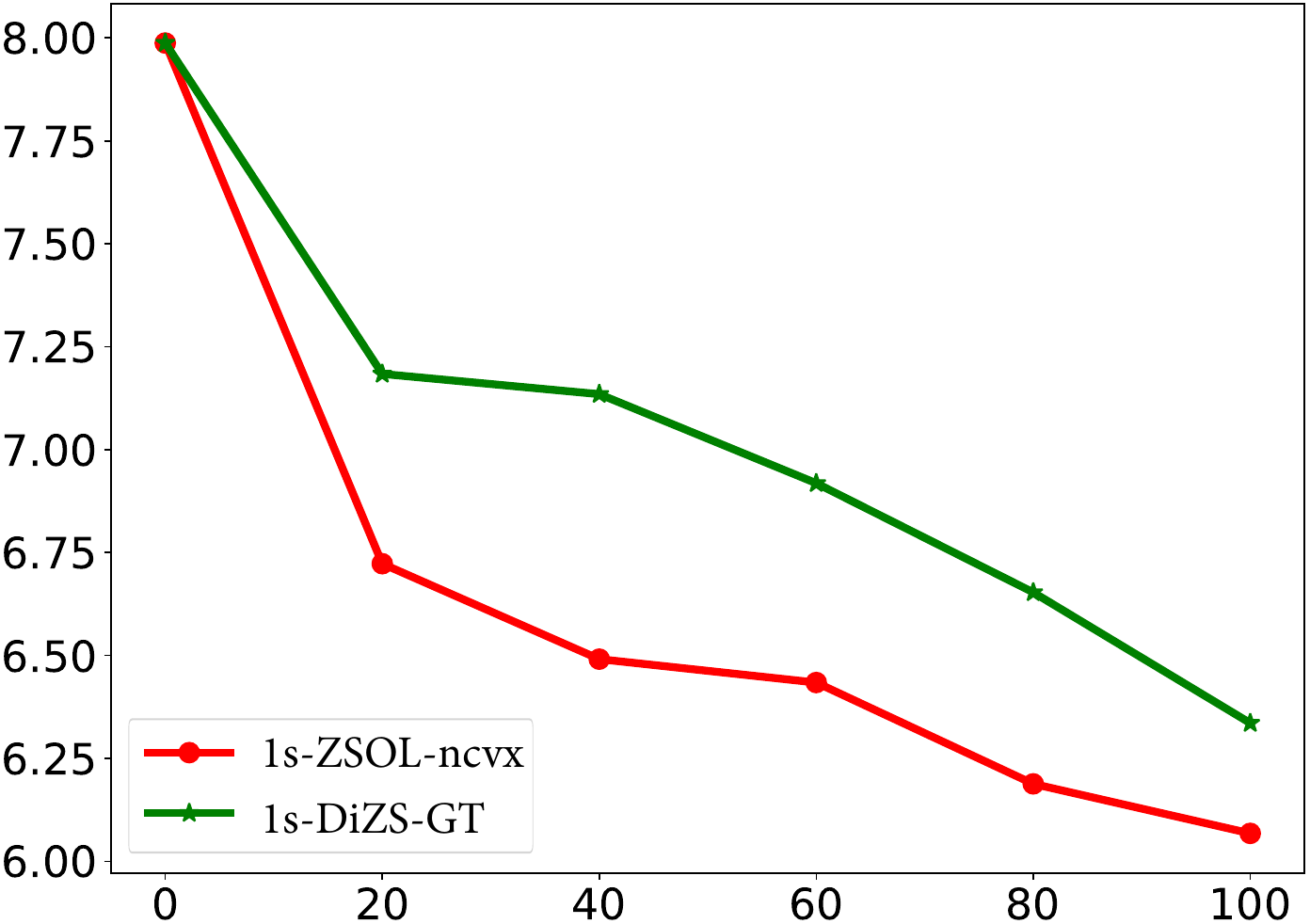}
\end{minipage}
	&
\begin{minipage}{.30\textwidth}
\includegraphics[scale=.199, angle=0]{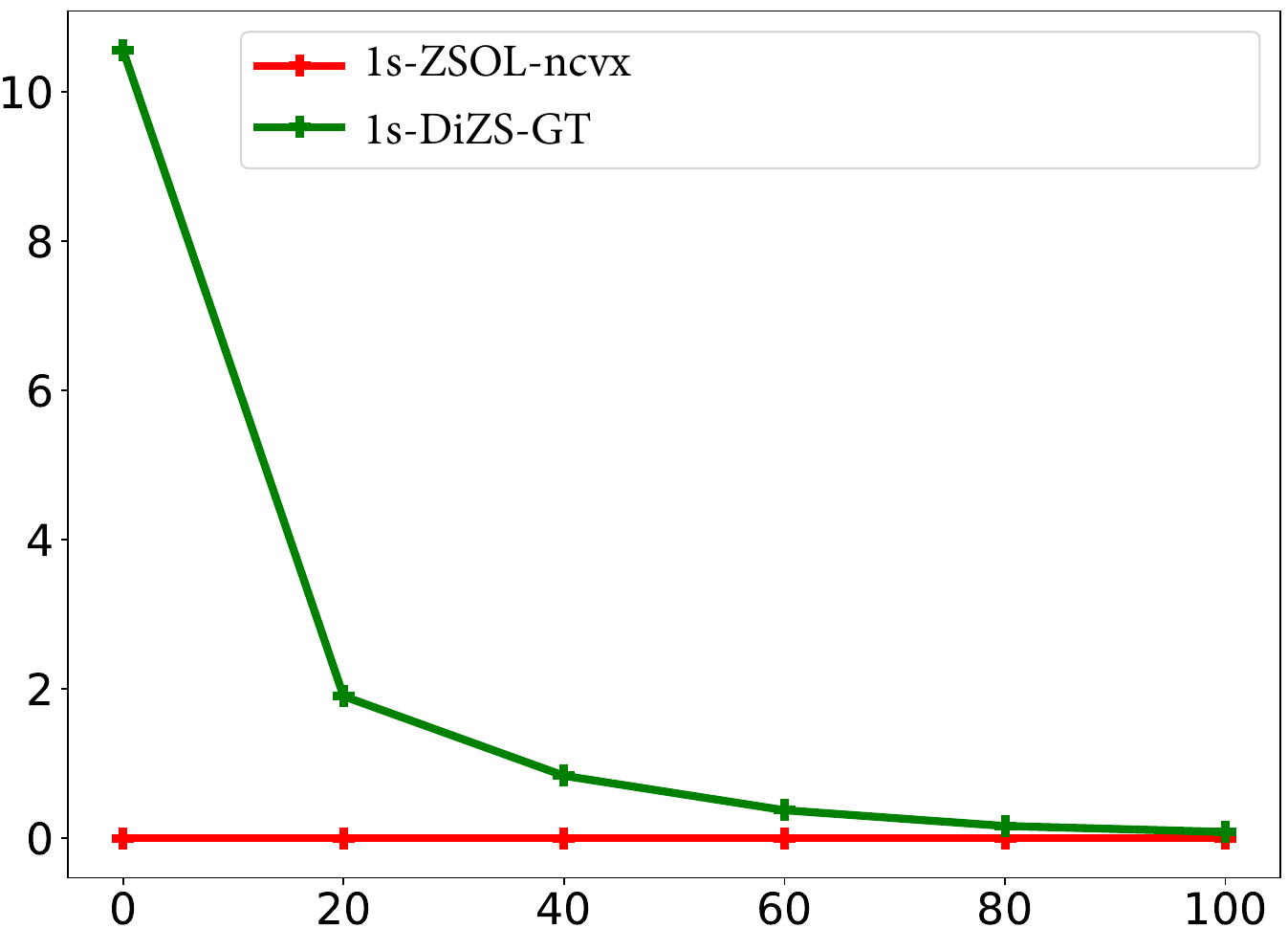}
\end{minipage}
\\ 
\hline\\
\rotatebox[origin=c]{90}{{\footnotesize {sparse graph}}}
&
\begin{minipage}{.30\textwidth}
\includegraphics[scale=.29, angle=0]{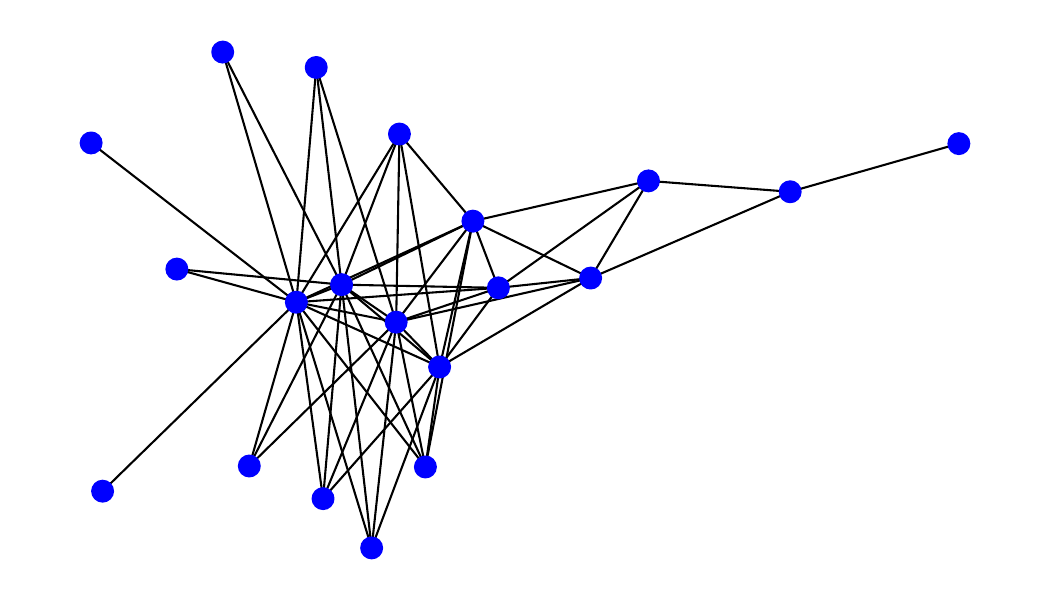}
\end{minipage}
&
\begin{minipage}{.30\textwidth}
\includegraphics[scale=.199, angle=0]{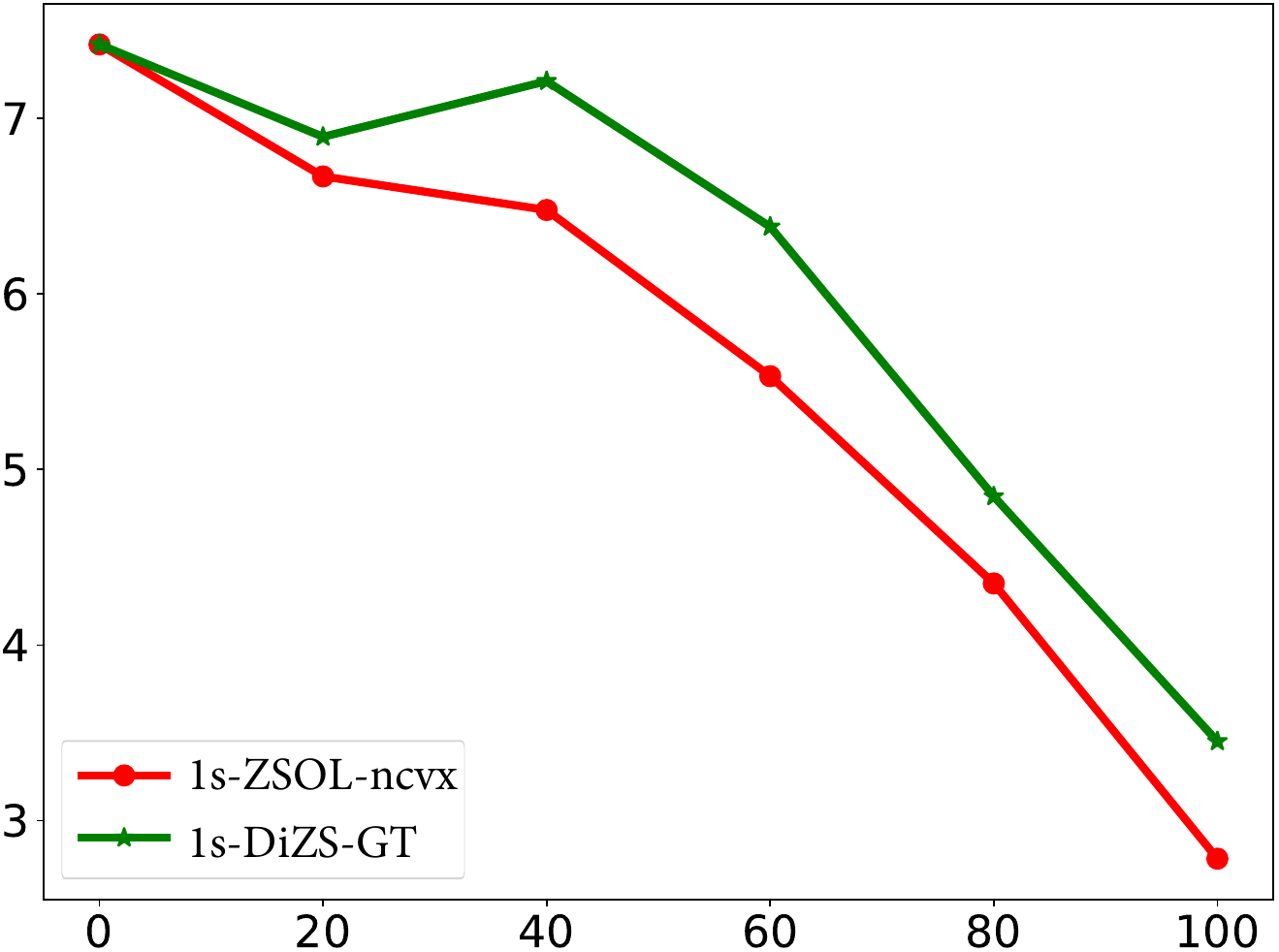}
\end{minipage}
	&
\begin{minipage}{.30\textwidth}
\includegraphics[scale=.199, angle=0]{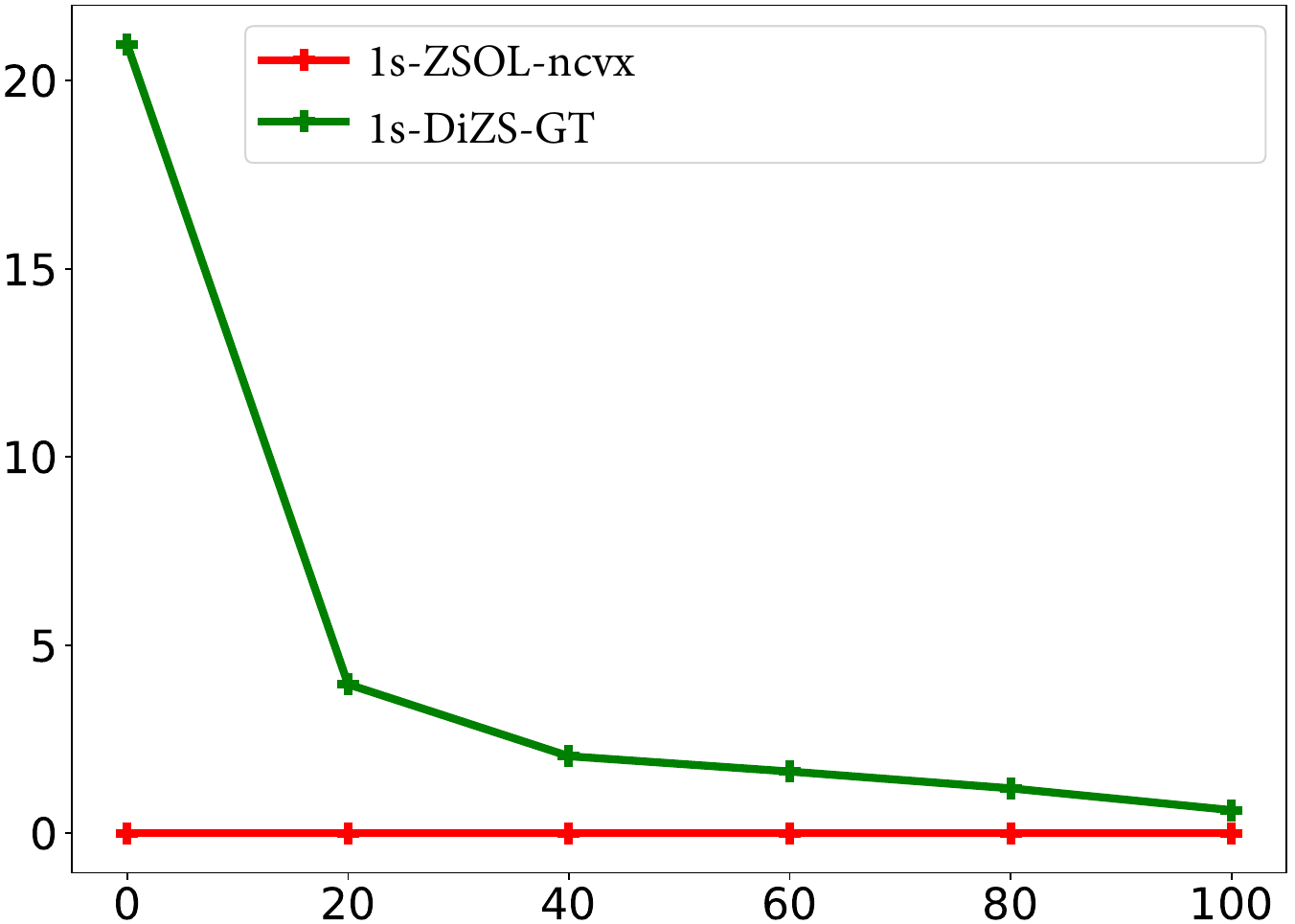}
\end{minipage}
\\ 
\hline\\
\rotatebox[origin=c]{90}{{\footnotesize {tree graph}}}
&
\begin{minipage}{.30\textwidth}
\includegraphics[scale=.329, angle=0]{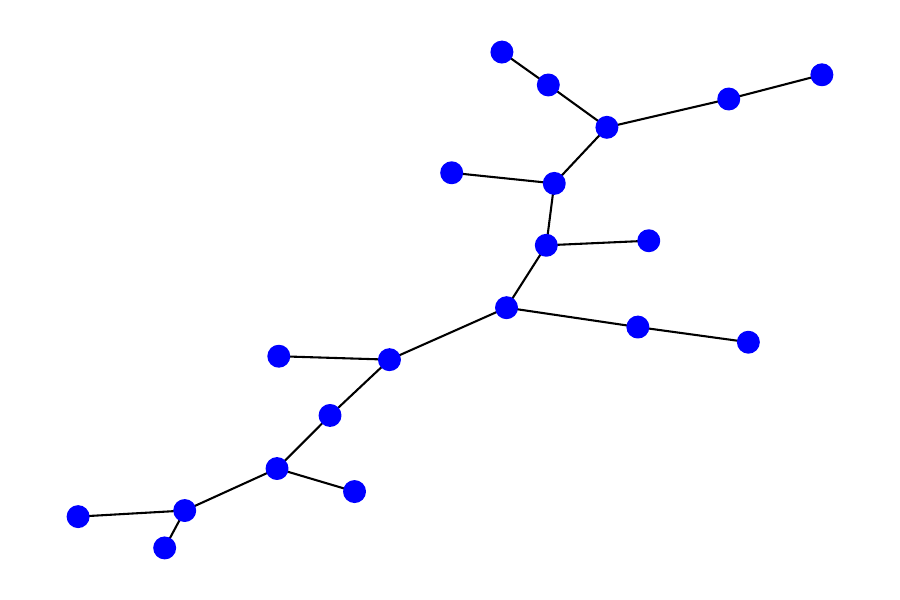}
\end{minipage}
&
\begin{minipage}{.30\textwidth}
\includegraphics[scale=.199, angle=0]{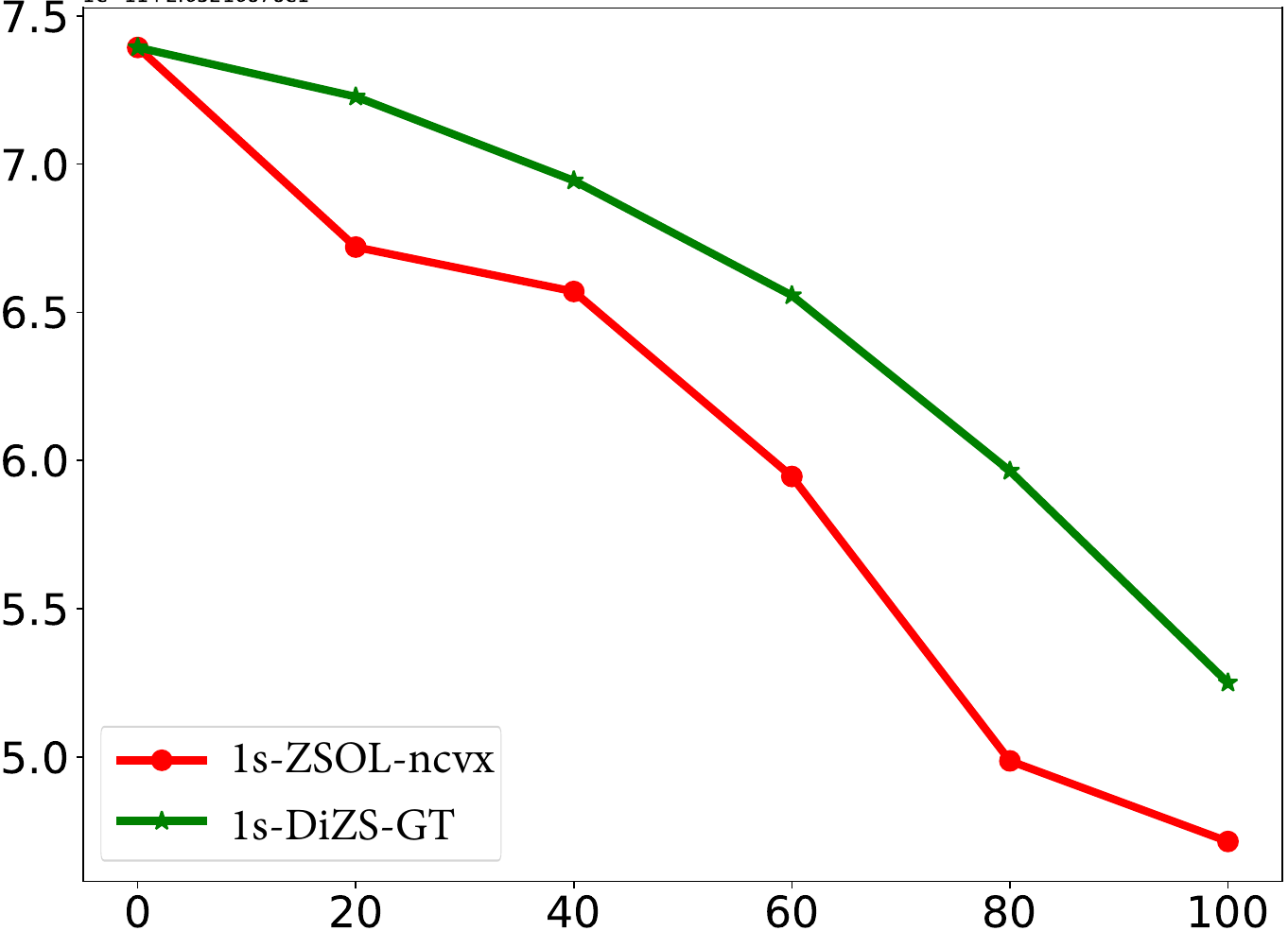}
\end{minipage}
	&
\begin{minipage}{.30\textwidth}
\includegraphics[scale=.199, angle=0]{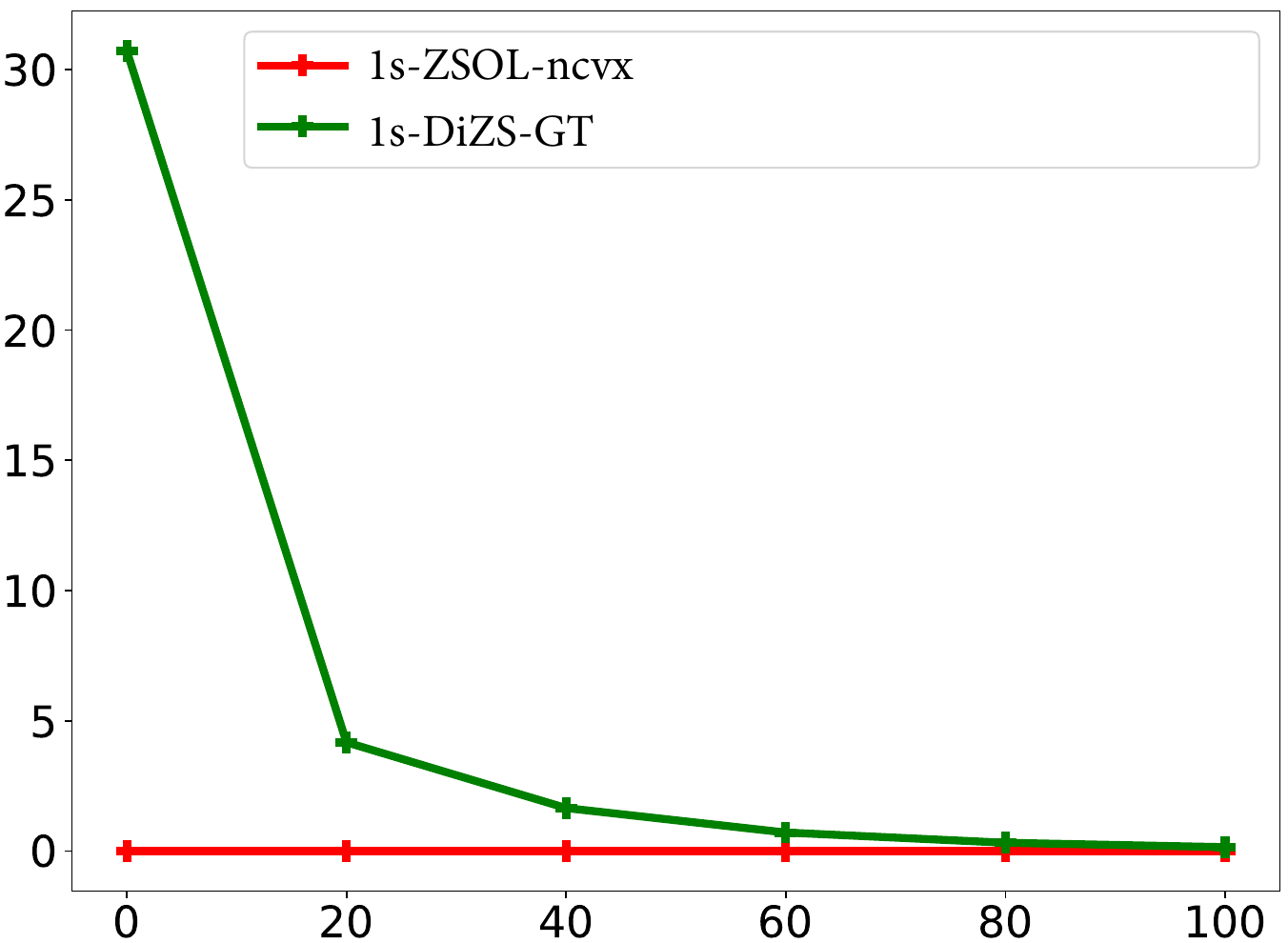}
\end{minipage}
\\ 
\hline\\
\rotatebox[origin=c]{90}{{\footnotesize {Erd\H{o}s--R\'enyi graph}}}
&
\begin{minipage}{.30\textwidth}
\includegraphics[scale=.329, angle=0]{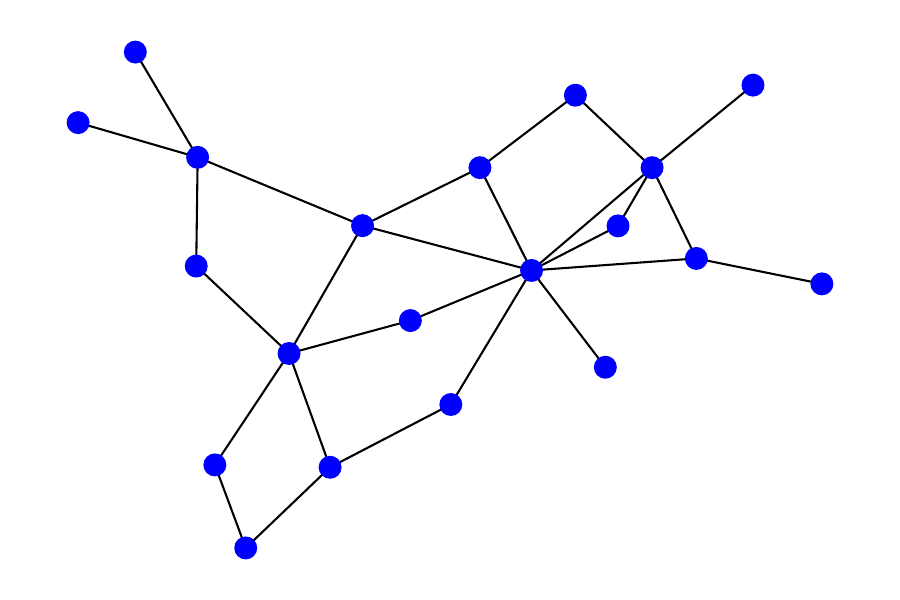}
\end{minipage}
&
\begin{minipage}{.30\textwidth}
\includegraphics[scale=.199, angle=0]{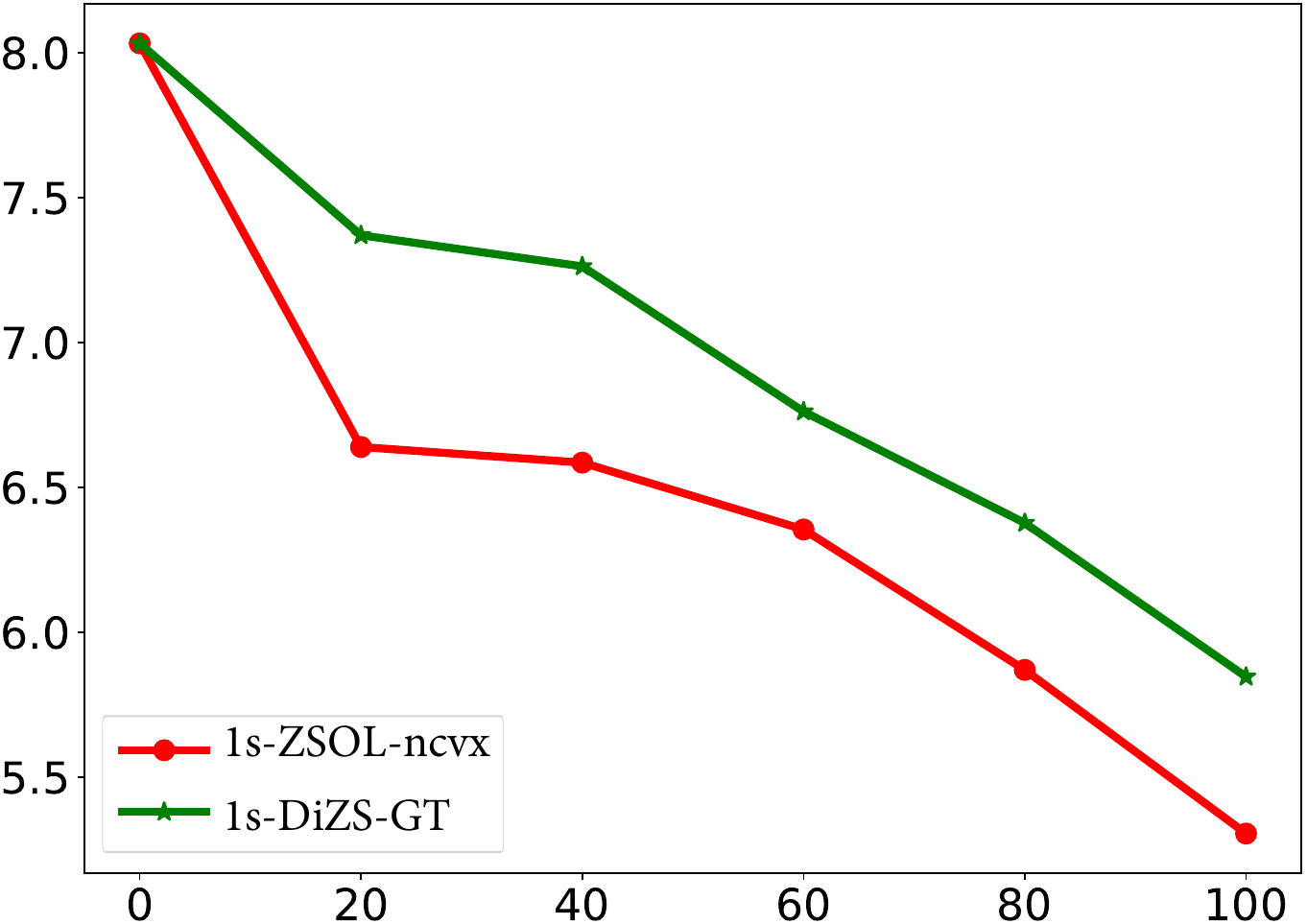}
\end{minipage}
	&
\begin{minipage}{.30\textwidth}
\includegraphics[scale=.199, angle=0]{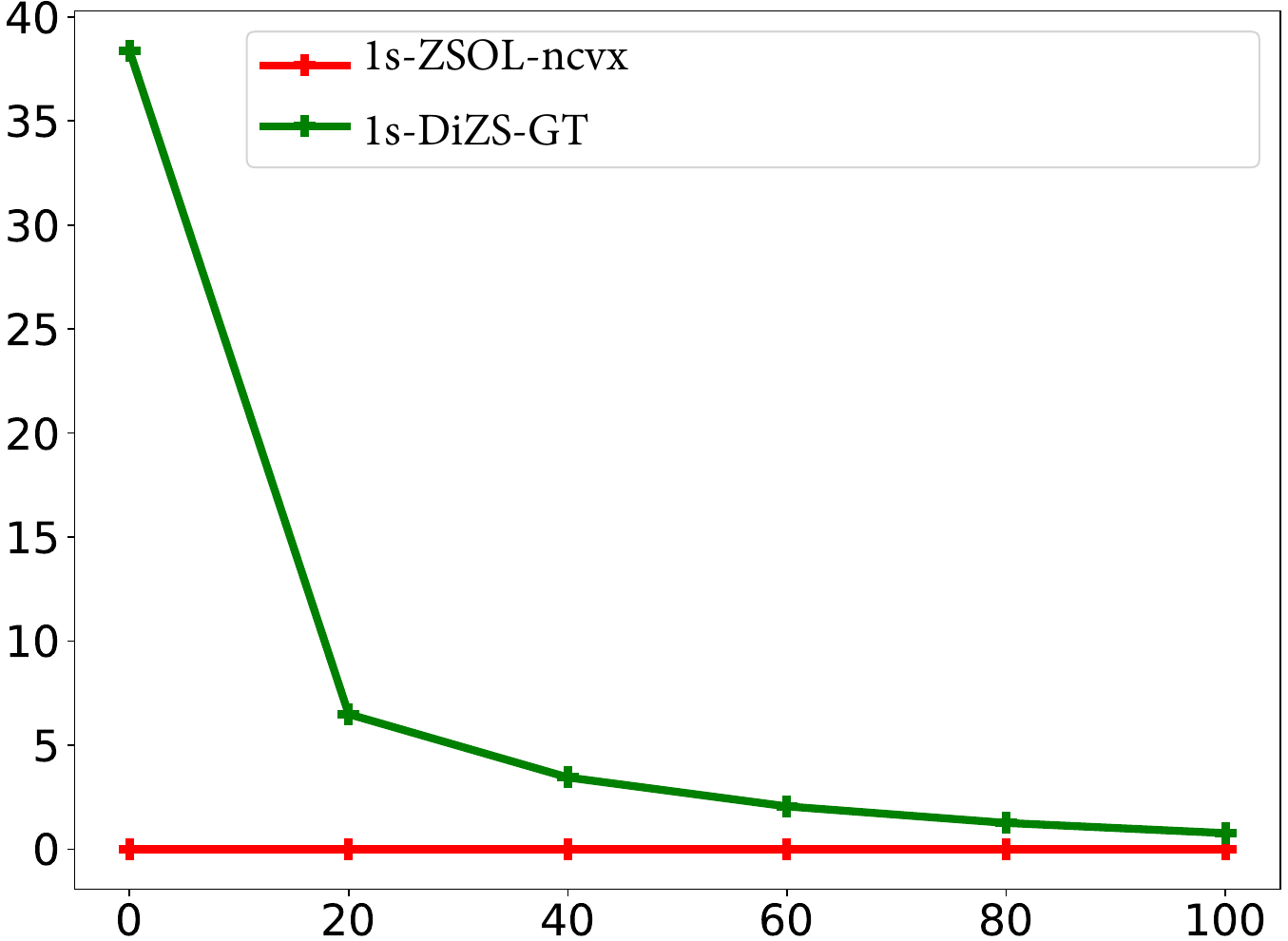}
\end{minipage}
\end{tabular}
\captionof{figure}{Comparison of Algorithm~\ref{alg:DZGT} with ZSOL$_{\text{ncvx}}^{\text{1s}}$ in terms of the sample averaged global objective function and consensus error five different networks, each with 20 nodes.}
\label{fig:comparison-single stage}}
\end{figure}

\subsection{{Distributed two-stage SMPECs}}\label{subsection:two-stage numerics}
In this section, we compare the performance of the {DiZS-GT$^{\text{2s}}$} method, presented in Algorithm~\ref{alg:DZGT-2stage}, with its centralized counterpart ZSOL$_{\text{ncvx}}^{\text{2s}}$ (Algorithm 5 in \cite{cui2023complexity}). The comparison is conducted to address the stochastic Stackelberg–Nash–Cournot equilibrium problem studied in~\cite{sherali1983stackelberg}, which can be captured by a two-stage SMPEC. Consider a commodity market with one leader, referred to as the Stackelberg firm, and $m$ followers, referred to as Cournot firms. The followers supply a product in a noncooperative manner.
We define the total cost of supplying $z_j$ units of the product for firm (follower) $j$ when the supplying units of other firms are fixed and defined as $z_{(-j)}$ as follows
$${h}_j\left(x,(z_j,z_{(-j)}),{\bxi}\right)\triangleq c_j(z_j)-z_jp(z_j+x+\textstyle\sum_{l\neq j}z_l(x,{\bxi}),{\bxi}),$$
where $z_j$ is the amount of products supplied by the $j$th firm (follower), $x$ is the production level by the leader, $c_j(z_j)$ denotes the cost of producing $z_j$ units, and $p(\bullet,{\bxi})$ represents the random inverse demand curve. The $j$th follower's optimization problem is given as
$$\min_{z_j\ge 0}  \quad {h}_j\left(x,(z_j,z_{(-j)}),{\bxi}\right) \triangleq c_j(z_j)-z_jp(z_j+x+\textstyle\sum_{l \neq j}\RRme{z_l}(x,{\bxi}),{\bxi}).$$
The leader's optimization problem is given as
$$\RRmath{\min_{ x }} \quad c(x)-\mathbb{E}[xp(x+\us{\bar{z}}(x,{\bxi}),{\bxi})],$$
where $c(x)$ denotes the production cost of the leader \RRme{and} $\us{\bar{z}}(x,{\bxi})\triangleq \sum_{j=1}^m z_j(x,{\bxi})$. The optimality conditions for the lower-level problem are as follows
$$0\le z_j(x,{\bxi}) \perp \nabla_{z_j}{h}_j\left(x,(z_j,z_{(-j)}),{\bxi}\right)\ge 0, \quad \text{for } j=1,2,\ldots,m.$$

The Cournot game among the followers can be captured by the equilibrium constraint  $\us{\mathcal{Z}}(x,{\bxi}) \in \mbox{SOL}\left({\mathbb{R}_+^m},F(x,\bullet,{\bxi})\right)$, where we define
\[
\mathcal{Z}(x,\mee{\bxi})=\begin{bmatrix}
z_1(x,{\bxi}) \\
z_2(x,{\bxi})  \\
\vdots \\
z_m(x,{\bxi}) \\
\end{bmatrix}\quad \hbox{and} \quad
F(x,z,{\bxi})=\begin{bmatrix}
\nabla_{z_1}{h}_1\left(x,(z_1,z_{(-1)}),{\bxi}\right)\\
\nabla_{z_2}{h}_2\left(x,(z_2,z_{(-2)}),{\bxi}\right)  \\
\vdots \\
\nabla_{z_m}{h}_m\left(x,(z_m,z_{(-m)}),{\bxi}\right)\\
\end{bmatrix}.\]
 The Stackelberg game is succinctly cast as  
 \begin{align*}
\RRmath{\min_{ x }}&\quad c(x) -   \mathbb{E}_{{\bxi}}[xp(x+{\us{\bar{z}}(x,{\bxi})},{\bxi})]\\
 \hbox{subject to} &\quad \us{\mathcal{Z}}(x,{\xi}) \in \mbox{SOL}\left({\mathbb{R}_+^m},F(x,\bullet,{\xi})\right) \mbox{ for every $\xi$.}
\end{align*} 
 We assume that $p(u,{\xi})=a({\xi})-bu$, for some function $a(\bullet)$ and scalar $b$. To compute an approximate solution to this problem, we consider
a distributed computational setting with $m$ clients, associated with iid random variables ${\bxi}_i$. This leads to the following distributed 2s-SMPEC problem.
 \begin{align*}
\RRmath{\min_{ x }}&\quad c(x)\ - \tfrac{1}{m}\textstyle\sum_{i=1}^m \mathbb{E}_{{\bxi}_i}[xp(x+{\us{\bar{z}}_i(x,{\bxi}_i)},{\bxi}_i)]\\
 \hbox{subject to} &\quad { \us{\mathcal{Z}}_i(x,\fyy{\xi}_i)} \in \mbox{SOL}\left({\mathbb{R}_+^m},F(x,\bullet,\fyy{\xi}_i)\right).
\end{align*} 
Let $p(u,{\bxi_i})=a({\bxi_i})-bu$. We assume that  $c_j(z)=\tfrac{1}{2}c_jz^2$ for $j=1,\ldots,m$, and $c(x)=\frac{1}{2}dx^2$, where $c_j>0$ for all $j$, and $d >0$. Therefore, we have $
\nabla_{z_j}{h}_j\left(x,(z_j,z_{(-j)}),\mee{\bxi}_i\right)=(c_j+b)z_j\RRmath{+bx}-a(\mee{\bxi}_i)+b\textstyle\sum_{l=1}^m z_l.$

\noindent \textbf{Problem and algorithm parameters.} We assume that $\mee{\bxi}_i\sim \mathcal{U}(7.5,12.5)$, for all computing agents, and run both algorithms for \RRme{$100$} iterations, setting the parameters as $\gamma=10^{-4}$, $\eta=10^{-1}$, and $b=0.1$. The values of $c_j$ are generated from a uniform distribution over the range $[0.05,0.5]$ for $j=1,...,m$. Additionally, at each iteration of the upper-level algorithm (Algorithm~\ref{alg:DZGT-2stage} in our work and Algorithm~5 in~\cite{cui2023complexity}), the lower-level algorithm\RRme{s,} Algorithm~\ref{alg:lowerlevel-2stage} in our work and Algorithm~6 in~\cite{cui2023complexity}\RRme{, are}  terminated after \RRme{$t_k$ and $mt_k$ iterations, respectively, where} $\RRme{t_k:= }\left\lceil \ln\left(n^{1/2}(k+1)\eta^{-2/3}\right)\right\rceil$, $n=1$ in this example, and $k$ represents the iteration index of the upper-level algorithm. Similar to the previous experiment, we consider $m=20$ computing agents with five different network settings. \RRme{As in the single-stage experiment, we match the upper-level sampling budget between the distributed and centralized implementations by using a mini-batch of size five at each of the $m$ agents and a mini-batch of size $5m=100$ for the centralized counterpart; hence, both implementations use $200$ upper-level objective function evaluations per upper-level iteration. The distinction in the two-stage setting is that the lower-level problem is deterministic, so no additional lower-level sampling is required. The distributed implementation performs $t_k$ lower-level iterations at each of the $m$ agents, while the centralized implementation is run for $mt_k$ iterations at each perturbed point. Therefore, since the zeroth-order estimator uses two perturbed points, both implementations require $2mt_k$ lower-level mapping evaluations per upper-level iteration.}

\noindent \textbf{Evaluation of the \RRme{global} implicit objective function.}
In line with the single-stage experiment, we run each scheme for each network setting over five different sample paths and report the sample mean of the global implicit objective function. The global implicit objective function is plotted over $5$ epochs. At each epoch, we approximate $z_i(x,\mee{\bxi}_i)$ by applying the gradient method (Algorithm~\ref{alg:lowerlevel-2stage}) \RRme{for $150$ lower-level iterations. These iterations are used only for evaluating the global implicit objective function and are distinct from the $100$ upper-level iterations used to generate the trajectories reported in the figures. The horizontal axis in the figures indicates the iteration index of the two-stage upper-level algorithms.}
\begin{remark}\em Although the ZSOL$_{\text{ncvx}}^{\text{1s}}$ and ZSOL$_{\text{ncvx}}^{\text{2s}}$ methods incorporate a variance-reduction scheme, our implementation of these approach\vvs{es} does not utilize this technique. This is mainly because, unlike in ZSOL, here the upper-level problem is assumed to be unconstrained. \mee{$\hfill \Box$}
\end{remark}

\begin{figure}[!htbp]
\centering{
\setlength{\tabcolsep}{0pt}
\centering
 \begin{tabular}{c || c  c  c }
  {\footnotesize {Setting}\ \ }& {\footnotesize  Communication network} & {\footnotesize $\mathbb{E}\left[f(\bar{x}_{k})\right]$} & {\footnotesize $\mathbb{E}\left[{\|\mathbf{x}_{k}- \mathbf{1} \bar{x}_{k}\|^2 }\right]$ } \\
 \hline\\
\rotatebox[origin=c]{90}{{\footnotesize {complete graph}}}
&
\begin{minipage}{.30\textwidth}
\includegraphics[scale=.329, angle=0]{complete}
\end{minipage}
&
\begin{minipage}{.30\textwidth}
\includegraphics[scale=.199, angle=0]{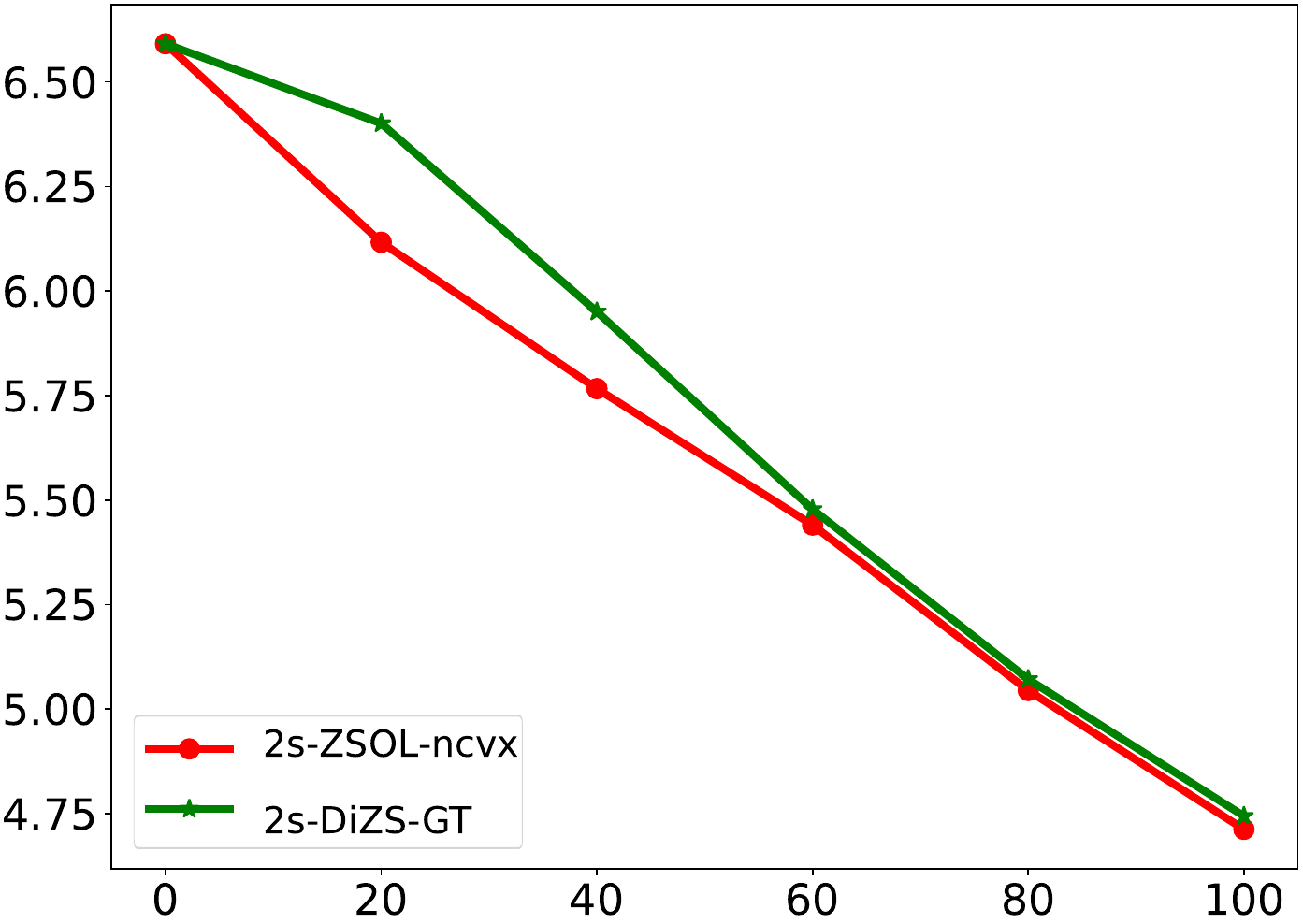}
\end{minipage}
	&
\begin{minipage}{.30\textwidth}
\includegraphics[scale=.199, angle=0]{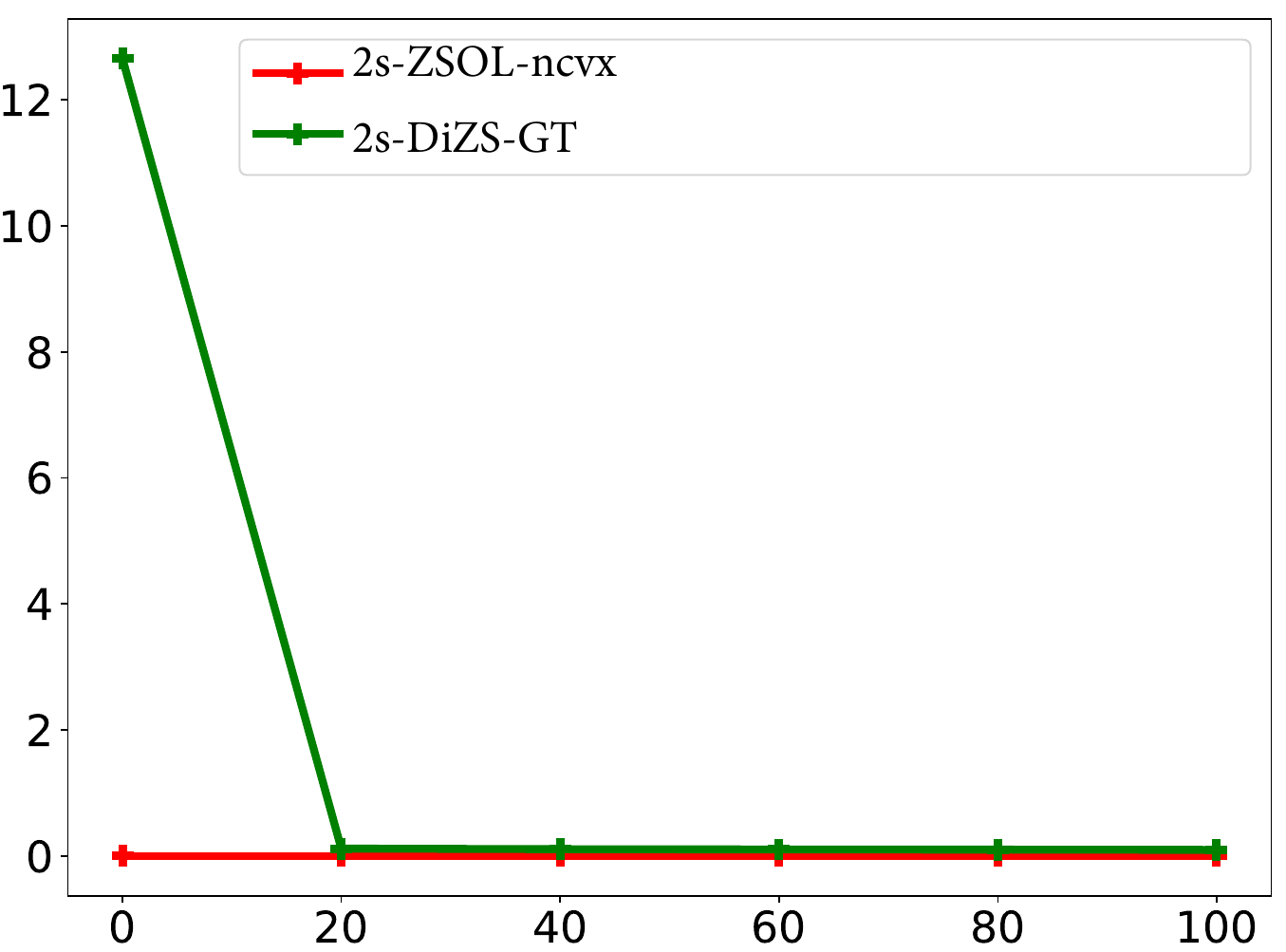}
\end{minipage}
\\ 
\hline\\
\rotatebox[origin=c]{90}{{\footnotesize {ring graph}}}
&
\begin{minipage}{.30\textwidth}
\includegraphics[scale=.329, angle=0]{ring}
\end{minipage}
&
\begin{minipage}{.30\textwidth}
\includegraphics[scale=.199, angle=0]{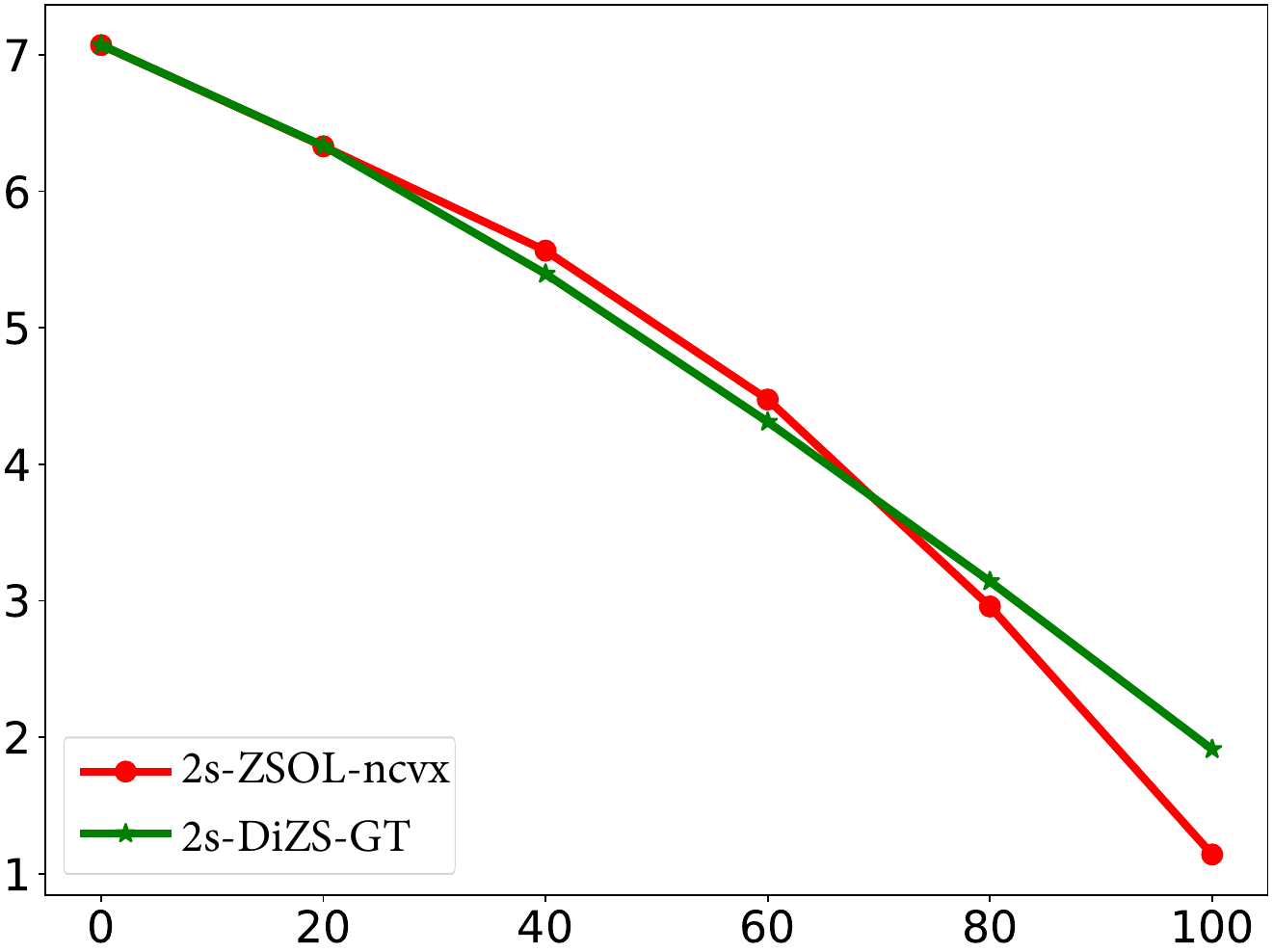}
\end{minipage}
	&
\begin{minipage}{.30\textwidth}
\includegraphics[scale=.199, angle=0]{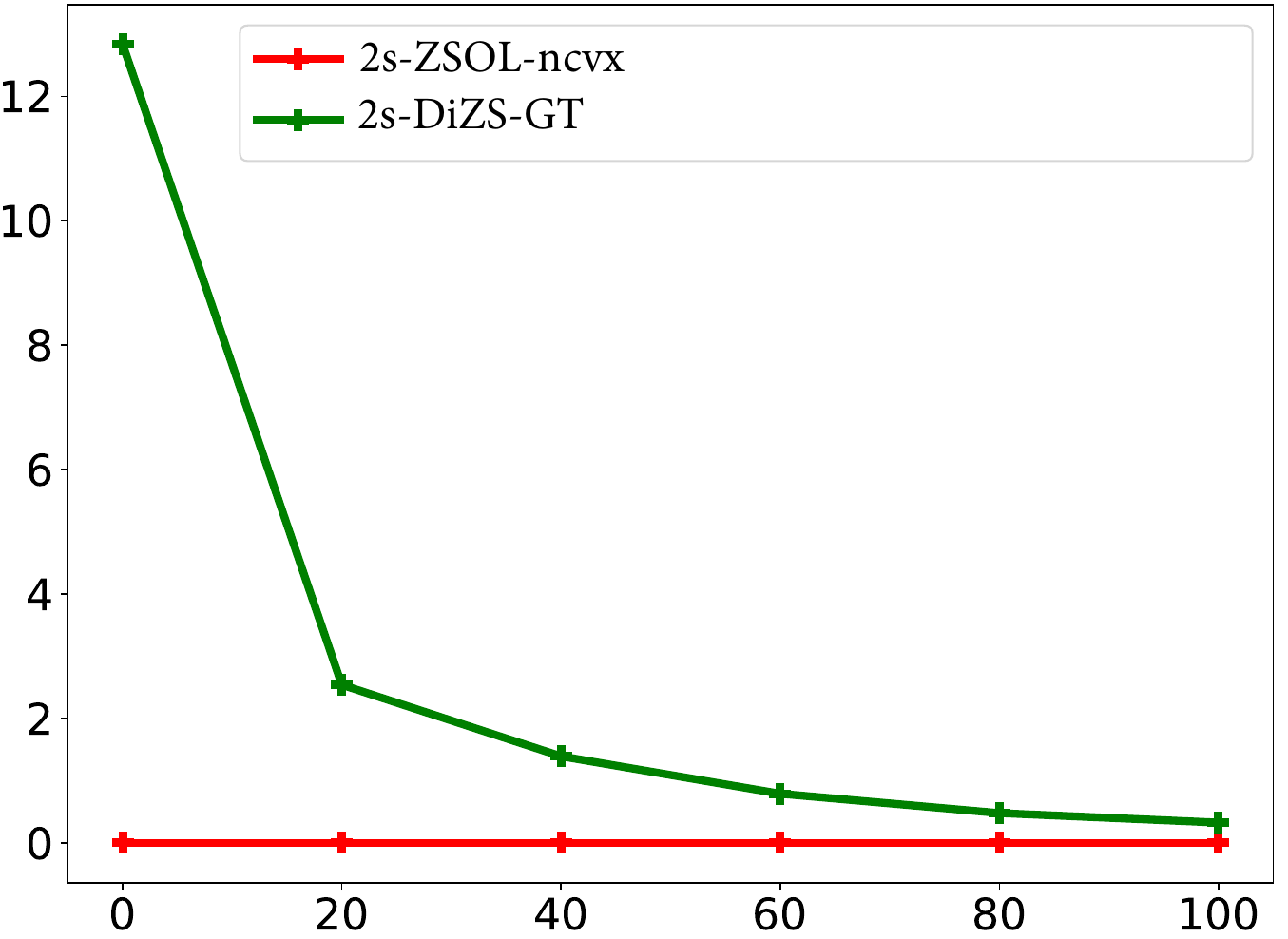}
\end{minipage}
\\ 
\hline\\
\rotatebox[origin=c]{90}{{\footnotesize {sparse graph}}}
&
\begin{minipage}{.30\textwidth}
\includegraphics[scale=.29, angle=0]{sparse}
\end{minipage}
&
\begin{minipage}{.30\textwidth}
\includegraphics[scale=.199, angle=0]{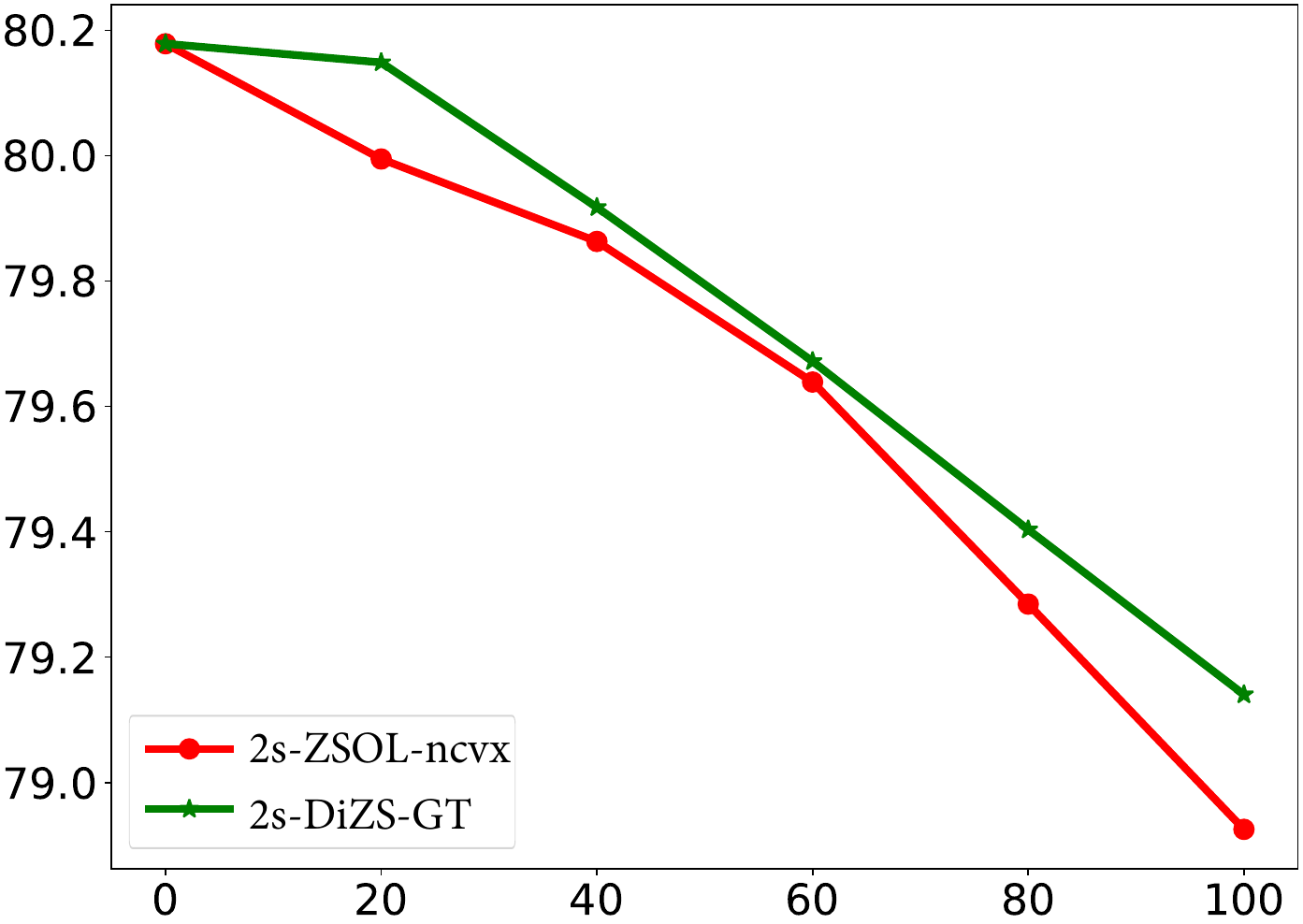}
\end{minipage}
	&
\begin{minipage}{.30\textwidth}
\includegraphics[scale=.199, angle=0]{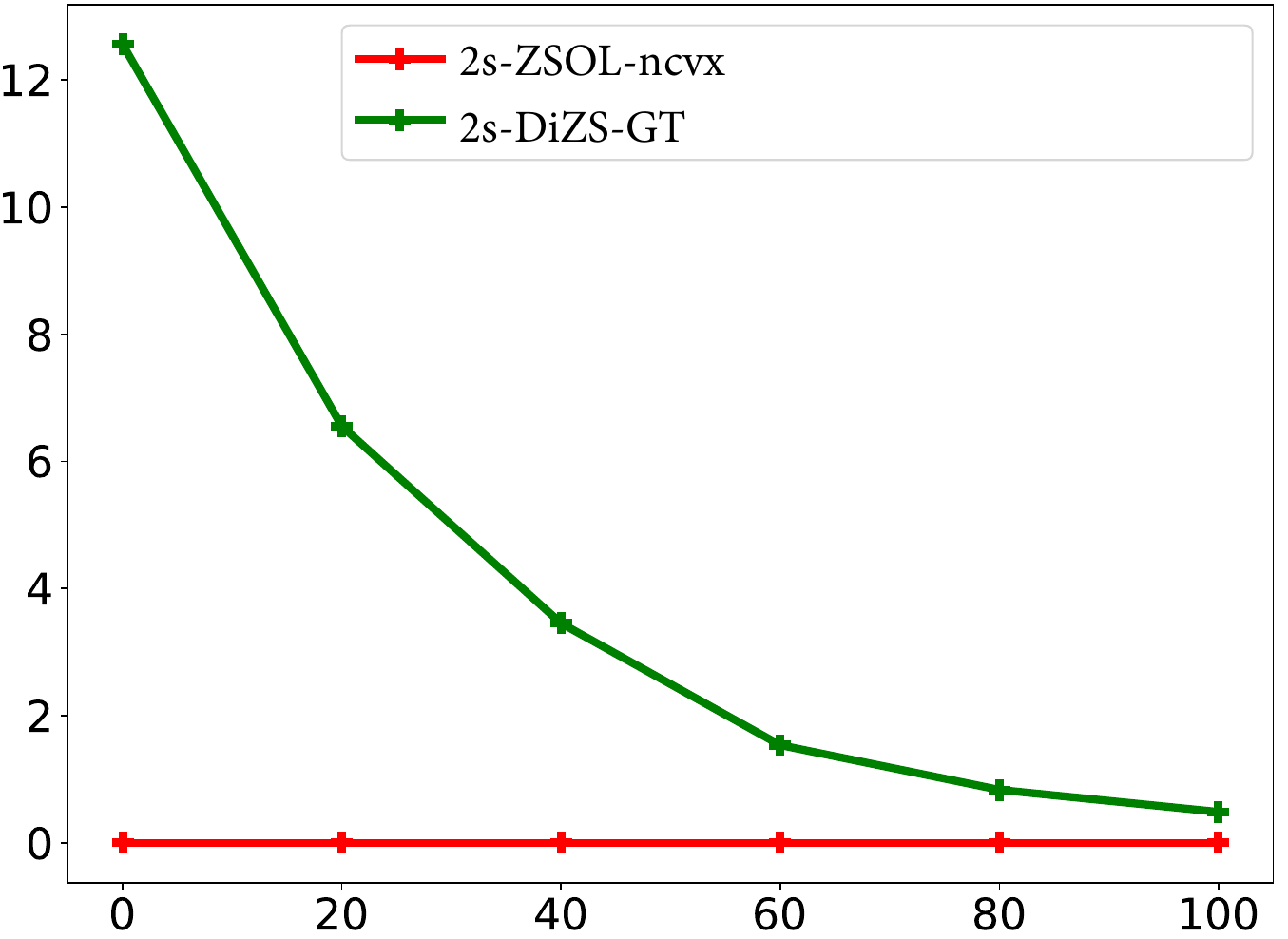}
\end{minipage}
\\ 
\hline\\
\rotatebox[origin=c]{90}{{\footnotesize {tree graph}}}
&
\begin{minipage}{.30\textwidth}
\includegraphics[scale=.329, angle=0]{tree}
\end{minipage}
&
\begin{minipage}{.30\textwidth}
\includegraphics[scale=.199, angle=0]{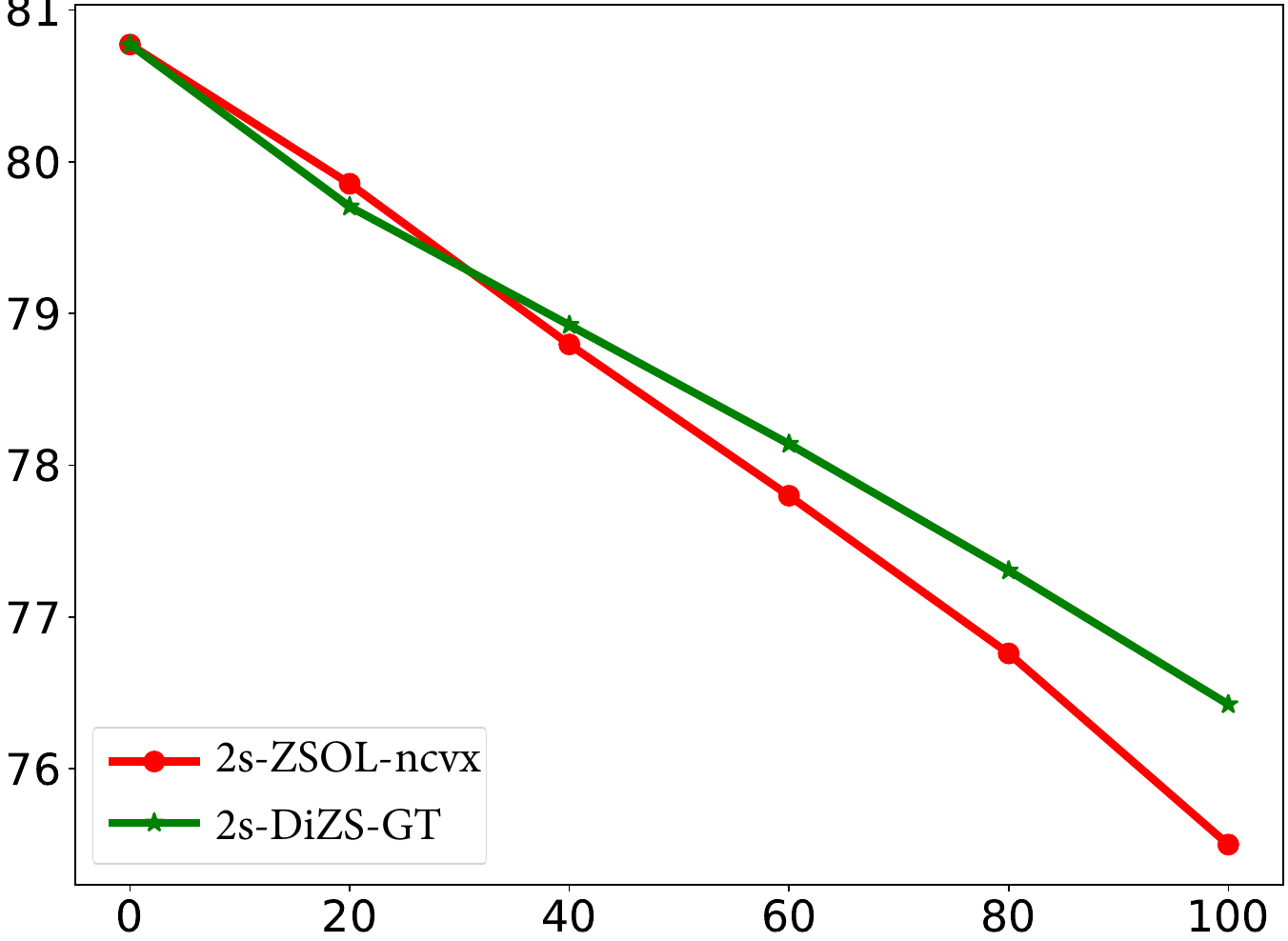}
\end{minipage}
	&
\begin{minipage}{.30\textwidth}
\includegraphics[scale=.199, angle=0]{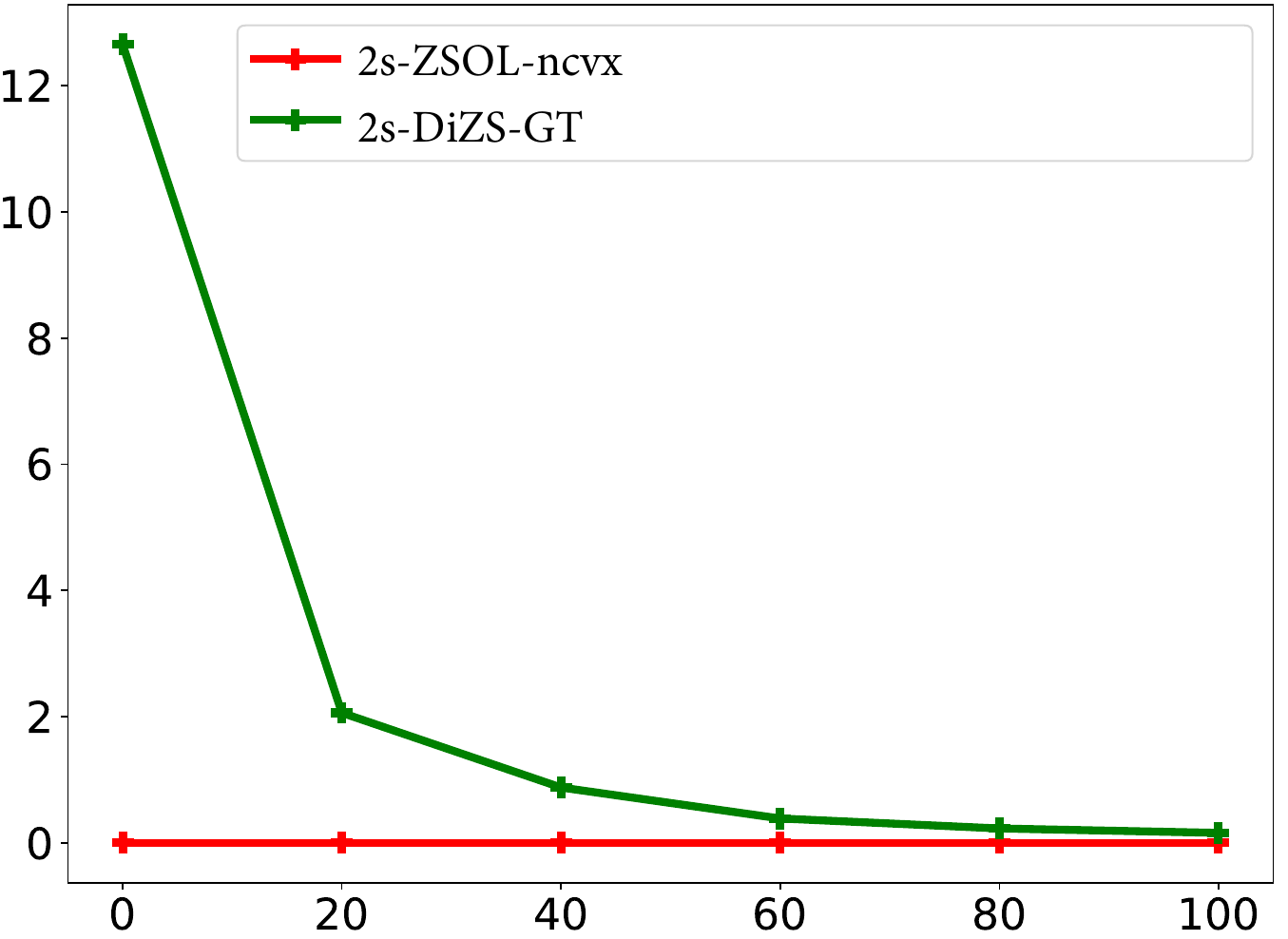}
\end{minipage}
\\ 
\hline\\
\rotatebox[origin=c]{90}{{\footnotesize {Erd\H{o}s--R\'enyi graph}}}
&
\begin{minipage}{.30\textwidth}
\includegraphics[scale=.329, angle=0]{erdos}
\end{minipage}
&
\begin{minipage}{.30\textwidth}
\includegraphics[scale=.199, angle=0]{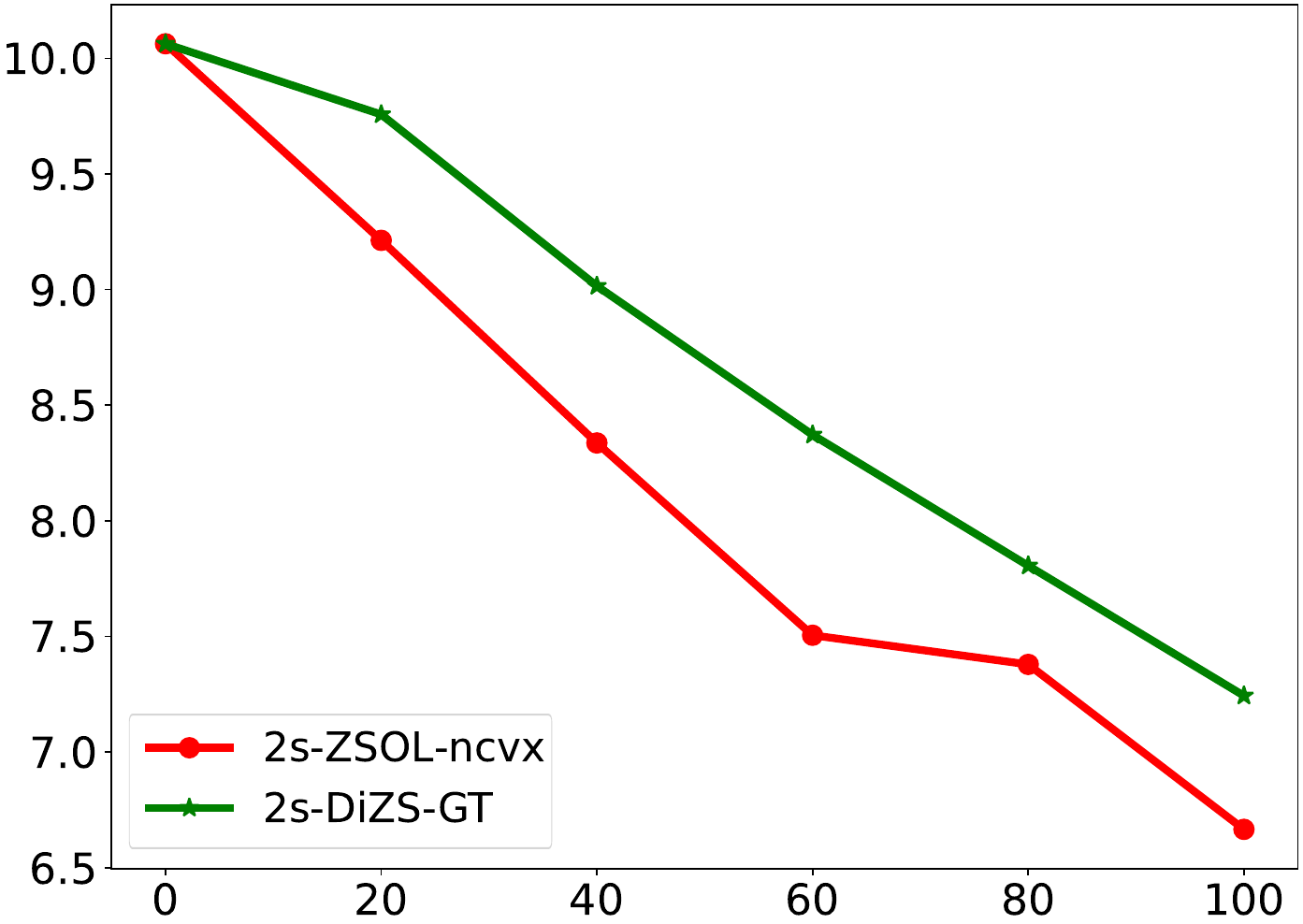}
\end{minipage}
	&
\begin{minipage}{.30\textwidth}
\includegraphics[scale=.199, angle=0]{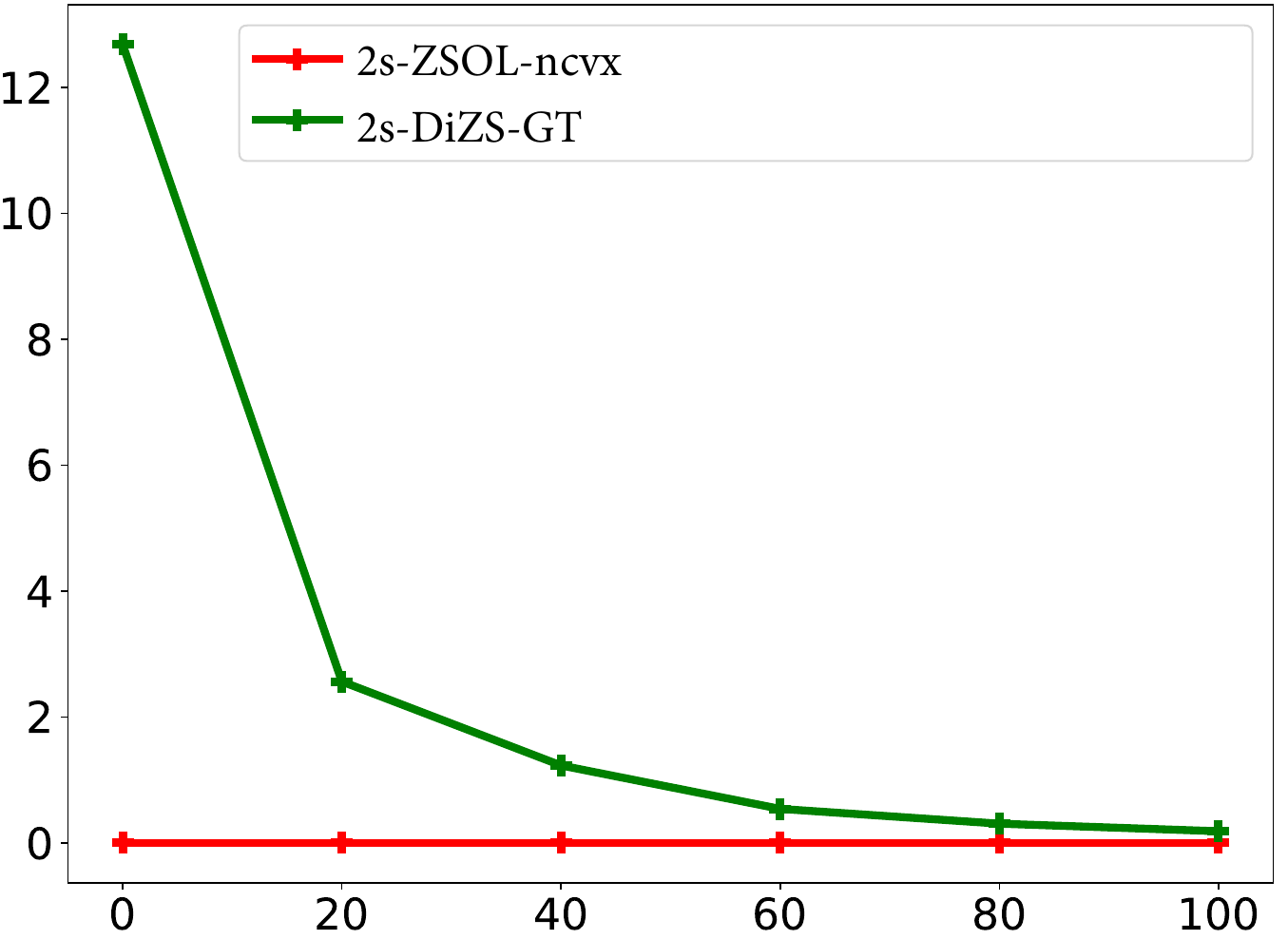}
\end{minipage}
\end{tabular}
\captionof{figure}{Comparison of Algorithm~\ref{alg:DZGT-2stage} with ZSOL$_{\text{ncvx}}^{\text{2s}}$ in terms of the sample averaged global objective function and consensus error five different networks, each with 20 nodes.}
\label{fig:comparison-2s}}
\end{figure}

\noindent \textbf{Insights.} Figure~\ref{fig:comparison-2s} presents the implementation results in addressing the two-stage setting, illustrating the impact of network connectivity. Notably, {DiZS-GT$^{\text{2s}}$} demonstrates robustness to changes in network connectivity.  As the network connectivity increases, our method exhibits a clear improvement in performance. This trend is particularly evident in the consensus error plots, where higher connectivity leads to faster convergence. Notably, in fully connected networks, {DiZS-GT$^{\text{2s}}$} performs at a level close to {ZSOL$_{\text{ncvx}}^{\text{2s}}$}. These observations highlight the advantages of our method in the decentralized setting.

\section{Concluding remarks}\label{sec:conc}
The mathematical program with equilibrium constraints (MPEC) is \Rme{an expansive} framework that \Rme{can capture settings arising in the modeling of Stackelberg games and bilevel optimization, amongst others}. In this work, we focus on \Rme{single and two-stage} stochastic variants of \Rme{the} MPEC. Motivated by the absence of distributed computational methods to address \Rme{such a} challenging \Rme{class of} mathematical models, we propose distributed gradient tracking methods designed for \Rme{the distributed generalization of the two aforementioned} classes of problems over networks, \Rme{i.e.,} single-stage and two-stage distributed stochastic MPECs. By leveraging a randomized smoothing technique, we develop fully iterative distributed zeroth-order gradient tracking methods. We establish complexity guarantees for computing a stationary point to the implicit optimization problem in each setting. Notably, our methods for the exact single-stage and two-stage settings achieve the best-known complexity bound for the centralized nonsmooth nonconvex stochastic optimization. Further, this is the first time that complexity guarantees for resolving the inexact distributed single-stage and two-stage stochastic MPECs over networks are obtained. We numerically compare the proposed methods with their centralized counterparts across networks of varying sizes and connectivity levels. \Rme{A key direction for future work is to weaken strong monotonicity at the lower level and establish guarantees for monotone problems, where equilibria may be nonunique, and the response is set-valued. In such settings, iterative penalization or regularization can be used to obtain a well-defined equilibrium for the upper level~\cite{samadi2023improved,jalilzadeh2024stochastic,kaushik2021method,ebrahimi2025regularized}}.
 
\bibliographystyle{siam}
\bibliography{ref_imp_DSGT}
\section{Appendix}\label{sec:app}

\noindent {\bf Proof of Lemma~\ref{lem:smoothing_props}.}
\noindent (i) The first equation can be shown in a similar vein to {\cite[Lemma 1]{cui2023complexity}}. To show the second equation, we have
\begin{align*}
\left(\tfrac{n}{2\eta}\right)\mathbb{E}_{{\bf v}\in \mathbb{S}}\left[\left(h(x+\eta {\bf v})-h(x-\eta {\bf v})\right){\bf v}\right]&=\left(\tfrac{n}{2\eta}\right)\mathbb{E}_{{\bf v}\in \mathbb{S}}\left[\left(h(x+\eta {\bf v})\right){\bf v}\right]+\left(\tfrac{n}{2\eta}\right)\mathbb{E}_{{\bf v}\in \mathbb{S}}\left[\left(h(x-\eta {\bf v})\right)(-{\bf v})\right].
\end{align*}
Due to the symmetric distribution of $v$ around the origin, we have $\mathbb{E}_{{\bf v}\in \mathbb{S}}\left[\left(h(x-\eta {\bf v})\right)(-{\bf v})\right]=\mathbb{E}_{{\bf v}\in \mathbb{S}}\left[\left(h(x+\eta {\bf v})\right){\bf v}\right]$. 
we obtain
$
\left(\tfrac{n}{2\eta}\right)\mathbb{E}_{{\bf v}\in \mathbb{S}}\left[\left(h(x+\eta {\bf v})-h(x-\eta {\bf v})\right){\bf v}\right]=\left(\tfrac{n}{\eta}\right)\mathbb{E}_{{\bf v}\in \mathbb{S}}\left[\left(h(x+\eta {\bf v})\right){\bf v}\right].
$
 \noindent (ii) The proof can be done in a similar vein to {\cite[Lemma 1]{cui2023complexity}}.

\noindent (iii) In view of {\cite[Prop. 2.2]{lin2022gradient}}, we have
$\|\nabla h^{\eta}(x)-\nabla h^{\eta}(y)\|\le \tfrac{2c_{n-1}}{c_{n}}\tfrac{L_0}{\eta}\|x-y\| \ \text{for any} \ x , y \in \mathbb{R}^n,$ where $\tfrac{2c_{n-1}}{c_{n}} = \tfrac{n!!}{(n-1)!!}$ if $n$ is odd, and  $\tfrac{2c_{n-1}}{c_{n}}=\tfrac{2}{\pi}\tfrac{n!!}{(n-1)!!}$ if $n$ is even. We use mathematical induction to show $\tfrac{2c_{n-1}}{c_{n}}\le\sqrt{n}$. Consider the case when $n$ is odd. For  $n = 1$  we have $\frac{n!!}{(n-1)!!} = \frac{1}{1} \leq \sqrt{1}$. Suppose the inequality holds for some odd $k$, i.e.,
$\frac{k!!}{(k-1)!!} \leq \sqrt{k}.$ For $k+2$, by the definition of the double factorial, we obtain
$\frac{(k+2)!!}{(k+1)!!} = \frac{k+2}{k+1} \cdot \frac{k!!}{(k-1)!!}.$ Invoking the inductive hypothesis, we obtain $\frac{(k+2)!!}{(k+1)!!} \leq \frac{k+2}{k+1} \cdot \sqrt{k}.$ We now need to prove that $\frac{k+2}{k+1} \cdot \sqrt{k} \leq \sqrt{k+2}$. Squaring both sides and rearranging the terms, we obtain $k^3 + 4k^2 + 4k \leq k^3 + 4k^2 + 5k + 2$, which is hold for all $k\ge1$. Therefore, the inequality holds for some odd $k+2$. Next, consider the case when $n$ is even. For $n = 2$, we have $\frac{2}{\pi} \cdot \frac{2!!}{(2-1)!!} =\frac{4}{\pi}\leq \sqrt{2}$. Suppose the inequality holds for some even $k \geq 2$, i.e., $\frac{2}{\pi} \cdot \frac{k!!}{(k-1)!!} \leq \sqrt{k}.$ For $k+2$, we have
$\frac{2}{\pi} \cdot \frac{(k+2)!!}{(k+1)!!} = \frac{2}{\pi} \cdot \frac{k+2}{k+1} \cdot \frac{k!!}{(k-1)!!}.$  Invoking the inductive hypothesis, we obtain $\frac{2}{\pi} \cdot \frac{(k+2)!!}{(k+1)!!} \leq \frac{k+2}{k+1} \cdot \sqrt{k}.$ Recall that we showed $\frac{k+2}{k+1} \cdot \sqrt{k} \leq \sqrt{k+2}$, completing the proof. 
\mee{$\hfill \Box$}

\medskip 
\noindent {\bf Proof of Lemma~\ref{Lemma:main bound for the gradient tracker in terms of other main terms-inexact}.} \RRme{We first consider the case } $\RRmath{\lambda_{\mathbf{W}}>0}$\RRme{.} \fy{Invoking the gradient tracking update \eqref{eqn:R1} and \Rme{Lemma~\ref{lemma:multiple parts}(i)}, we have
\begin{align}
\|{ \mathbf{y}}_{k+2}-\mathbf{1}{\bar{ {y}}}_{k+2}\|^2&=\|\mathbf{W}({ \mathbf{y}}_{k+1}+\nabla \mathbf f^\eta(\mathbf{x}_{k+1})+\boldsymbol{\delta}_{k+1}^\eta+\mathbf{e}_{k+1}^{\eta,\varepsilon_{k+1}}-(\nabla \mathbf f^\eta(\mathbf{x}_{k})+\boldsymbol{\delta}_{k}^\eta+\mathbf{e}_k^{\eta,\varepsilon_k}))\notag\\
&-\mathbf{1}({ \bar{y}}_{k+1}+\overline{\nabla f^\eta}(\mathbf{x}_{k+1})+\bar{\delta}_{k+1}^\eta+\bar{\us{\bf e}}_{k+1}^{\eta,\varepsilon_{k+1}}-(\overline{\nabla f^\eta}(\mathbf{x}_k)+\bar{\delta}_{k}^\eta+\bar{e}_k^{\eta,\varepsilon_k}))\|^2\notag\\
&=\|(\mathbf{W}{ \mathbf{y}}_{k+1}-\mathbf{1}{ \bar{y}}_{k+1})\notag\\
&+(\mathbf{W}-\tfrac{1}{m}\mathbf{1}\mathbf{1}^\top)(\nabla \mathbf f^\eta(\mathbf{x}_{k+1})+\boldsymbol{\delta}_{k+1}^\eta+\mathbf{e}_{k+1}^{\eta,\varepsilon_{k+1}}-(\nabla \mathbf f^\eta(\mathbf{x}_{k})+\boldsymbol{\delta}_{k}^\eta+\mathbf{e}_k^{\eta,\varepsilon_k}))\|^2\notag.
\end{align}
Invoking Lemma~\ref{lem:lambda_w}, we obtain
\begin{align}
&\|{ \mathbf{y}}_{k+2}-\mathbf{1}{\bar{ {y}}}_{k+2}\|^2\le \lambda_{\mathbf{W}}^2\|{ \mathbf{y}}_{k+1}-\mathbf{1}{ \bar{y}}_{k+1}\|^2 \notag\\ 
&+\lambda_{\mathbf{W}}^2\|\nabla \mathbf f^\eta(\mathbf{x}_{k+1})+\boldsymbol{\delta}_{k+1}^\eta+\mathbf{e}_{k+1}^{\eta,\varepsilon_{k+1}}-(\nabla \mathbf f^\eta(\mathbf{x}_{k})+\boldsymbol{\delta}_{k}^\eta+\mathbf{e}_k^{\eta,\varepsilon_k}))\|^2\notag\\
&+ 2\langle (\mathbf{W}-\tfrac{1}{m}\mathbf{1}\mathbf{1}^\top) \mathbf{y}_{k+1},(\mathbf{W}-\tfrac{1}{m}\mathbf{1}\mathbf{1}^\top)(\nabla \mathbf f^\eta(\mathbf{x}_{k+1})+\boldsymbol{\delta}_{k+1}^\eta+\mathbf{e}_{k+1}^{\eta,\varepsilon_{k+1}}-(\nabla \mathbf f^\eta(\mathbf{x}_{k})+\boldsymbol{\delta}_{k}^\eta+\mathbf{e}_k^{\eta,\varepsilon_k})) \rangle \notag\\
&\le \lambda_{\mathbf{W}}^2\|{ \mathbf{y}}_{k+1}-\mathbf{1}{ \bar{y}}_{k+1}\|^2 \notag\\
&+5\lambda_{\mathbf{W}}^2(\underbrace{\|\nabla \mathbf f^\eta(\mathbf{x}_{k+1})-\nabla \mathbf f^\eta(\mathbf{x}_{k})\|^2}_{\us{\scriptsize \mbox{Term } T_1}} +\underbrace{\|\boldsymbol{\delta}_{k+1}^\eta\|^2+\|\boldsymbol{\delta}_{k}^\eta\|^2 +\|\mathbf{e}_{k+1}^{\eta,\varepsilon_{k+1}}\|^2+\|\mathbf{e}_k^{\eta,\varepsilon_k}\|^2)}_{\us{\scriptsize \mbox{Term } T_2}}\notag\\
&+ 2\underbrace{\langle (\mathbf{W}-\tfrac{1}{m}\mathbf{1}\mathbf{1}^\top) \mathbf{y}_{k+1},(\mathbf{W}-\tfrac{1}{m}\mathbf{1}\mathbf{1}^\top)(\nabla \mathbf f^\eta(\mathbf{x}_{k+1})-\nabla \mathbf f^\eta(\mathbf{x}_{k}))\rangle}_{\us{\scriptsize \mbox{Term } T_3}}
\notag\\
&+ 2\underbrace{\langle (\mathbf{W}-\tfrac{1}{m}\mathbf{1}\mathbf{1}^\top) \mathbf{y}_{k+1},(\mathbf{W}-\tfrac{1}{m}\mathbf{1}\mathbf{1}^\top)(\boldsymbol{\delta}_{k+1}^\eta-\boldsymbol{\delta}_{k}^\eta) \rangle}_{\us{\scriptsize \mbox{Term } T_4}}\notag\\
&+ 2\underbrace{\langle (\mathbf{W}-\tfrac{1}{m}\mathbf{1}\mathbf{1}^\top) \mathbf{y}_{k+1},(\mathbf{W}-\tfrac{1}{m}\mathbf{1}\mathbf{1}^\top)(\mathbf{e}_{k+1}^{\eta,\varepsilon_{k+1}}-\mathbf{e}_k^{\eta,\varepsilon_k}) \rangle}_{\us{\scriptsize \mbox{Term } T_5}}. \label{eq:bound for y-bar y-inexact}
\end{align}
To analyze the \us{expressions} on the right-hand side of the preceding inequality, {denoted by \us{Terms} $T_1,T_2,T_3,T_4$, and $T_5$}, in the following, we prove some intermediate inequalities labeled as {\it claims}. 
\smallskip 

\noindent {\bf Claim 1 {(\textit{Analysis of \us{Term} $T_1$})}. } We have $\|\nabla \mathbf f^\eta(\mathbf{x}_{k+1})-\nabla \mathbf f^\eta(\mathbf{x}_{k})\|^2 \leq \mj{\left(\tfrac{\sqrt{n}L_0}{\eta}\right)^2}\underbrace{\|\mathbf{x}_{k+1}-\mathbf{x}_{k}\|^2}_{\us{\scriptsize \mbox{Term } T_6}}$.

\noindent {\it Proof of Claim 1.} Recall from Remark~\ref{rem:smoothness_of_f_i} that that $f_i$ is \mj{$\tfrac{\sqrt{n}L_0}{\eta}$}-smooth. We have 
\begin{align*}
\|\nabla \mathbf f^\eta(\mathbf{x}_{k+1})-\nabla \mathbf f^\eta(\mathbf{x}_{k})\|^2& =\textstyle\sum_{i=1}^m \|\nabla f_i^\eta(x_{i,k+1})-\nabla f_i^\eta(x_{i,k})\|^2 \\
&\leq \mj{\left(\tfrac{\sqrt{n}L_0}{\eta}\right)^2}\textstyle\sum_{i=1}^m\|   x_{i,k+1}  -  x_{i,k}\|^2 =\mj{\left(\tfrac{\sqrt{n}L_0}{\eta}\right)^2}\|\mathbf{x}_{k+1}-\mathbf{x}_{k}\|^2. 
\end{align*}

\smallskip 

\noindent {\bf Claim 2 {(\textit{Analysis of \us{Term} $T_6$})}.} We have 
\begin{align*}
\mathbb{E}[\|\mathbf{x}_{k+1}-\mathbf{x}_{k}\|^2]&\le  6\RRmath{\gamma}^2\lambda_{\mathbf{W}}^2\mathbb{E}[\|\mathbf{y}_{k+1}-\mathbf{1}\bar{y}_{k+1}\|^2]+9\mathbb{E}[\|\mathbf{x}_{k}-\mathbf{1}\bar{x}_{k}\|^2]\\
&+6m\RRmath{\gamma}^2\mathbb{E}[\|\overline{\nabla{f}^\eta}(\mathbf{x}_k)\|^2]+\mj{96\sqrt{2\pi}\RRmath{\gamma}^2nL_0^2}+\mj{\left(\tfrac{6m\tilde L_0^2n^2\varepsilon_k\RRmath{\gamma}^2}{\eta^2}\right)}.
\end{align*}

\noindent {\it Proof of Claim 2.}  Invoking  \Rme{Lemma~\ref{lemma:multiple parts}(i)}, we have
\begin{align*}
 \|\mathbf{x}_{k+1}-\mathbf{x}_{k}\|^2&= \|\mathbf{x}_{k+1}-\mathbf{1}\bar{x}_{k+1}+\mathbf{1}\bar{x}_{k+1}-\mathbf{1}\bar{x}_{k}+\mathbf{1}\bar{x}_{k}-\mathbf{x}_{k}\|^2\\
&\le 3 \|\mathbf{x}_{k+1}-\mathbf{1}\bar{x}_{k+1}\|^2+3m\RRmath{\gamma}^2 \|\overline{\nabla{f}^\eta}(\mathbf{x}_k)+\bar \delta_{k}^\eta+\bar{e}_k^{\eta,\varepsilon_k}\|^2+3\|\mathbf{x}_{k}-\mathbf{1}\bar{x}_{k}\|^2.
\end{align*}
Invoking \Rme{Lemma~\ref{lemma: multiple inequalities for the metrics}(i)} with $\theta=1$ to bound the first term on the right, we obtain
\begin{align*}
 \|\mathbf{x}_{k+1}-\mathbf{x}_{k}\|^2&\le 6\RRmath{\gamma}^2\lambda_{\mathbf{W}}^2 \|\mathbf{y}_{k+1}-\mathbf{1}\bar{y}_{k+1}\|^2+9 \|\mathbf{x}_{k}-\mathbf{1}\bar{x}_{k}\|^2+6m\RRmath{\gamma}^2 \|\overline{\nabla{f}^\eta}(\mathbf{x}_k)+\bar \delta_{k}^\eta\|^2+6m\RRmath{\gamma}^2 \|\bar{e}_k^{\eta,\varepsilon_k}\|^2.
\end{align*}
Taking conditional expectations \us{over} the \us{prior} inequality and invoking Lemma~\ref{lemma:multiple parts}, we obtain
\begin{align*}
\mathbb{E}[\|\mathbf{x}_{k+1}-\mathbf{x}_{k}\|^2\mid \mathcal{F}_k]&\le 6\RRmath{\gamma}^2\lambda_{\mathbf{W}}^2\mathbb{E}[\|\mathbf{y}_{k+1}-\mathbf{1}\bar{y}_{k+1}\|^2\mid \mathcal{F}_k]+9\|\mathbf{x}_{k}-\mathbf{1}\bar{x}_{k}\|^2\\
&+6m\RRmath{\gamma}^2\left(\|\overline{\nabla{f}^\eta}(\mathbf{x}_k)\|^2+\mathbb{E}[\|\bar \delta_{k}^\eta\|^2\mid \mathcal{F}_k] + 2(\overline{\nabla{f}^\eta}(\mathbf{x}_k))^\top\mathbb{E}[\bar \delta_{k}^\eta\mid \mathcal{F}_k]\right)+ { \tfrac{  6m \tilde{L}_0^2n^2\varepsilon_k  \RRmath{\gamma}^2}{\eta^2} }\\
& \leq 6\RRmath{\gamma}^2\lambda_{\mathbf{W}}^2\mathbb{E}[\|\mathbf{y}_{k+1}-\mathbf{1}\bar{y}_{k+1}\|^2\mid \mathcal{F}_k]+9\|\mathbf{x}_{k}-\mathbf{1}\bar{x}_{k}\|^2\\
&+6m\RRmath{\gamma}^2\|\overline{\nabla{f}^\eta}(\mathbf{x}_k)\|^2+\mj{96\sqrt{2\pi}\RRmath{\gamma}^2nL_0^2}+{ \tfrac{ 6m \tilde{L}_0^2n^2\varepsilon_k  \RRmath{\gamma}^2}{\eta^2} }.
\end{align*}
Taking expectations on both sides, we obtain the result.

\smallskip 
 
\noindent {\bf Claim 3 {(\textit{Analysis of \us{Term } $T_2$})}.} We have 
$$\mathbb{E}[\|\boldsymbol{\delta}_{k+1}^\eta\|^2+\|\boldsymbol{\delta}_{k}^\eta\|^2 +\|\mathbf{e}_{k+1}^{\eta,\varepsilon_{k+1}}\|^2+\|\mathbf{e}_k^{\eta,\varepsilon_k}\|^2\mj{\mid \mathcal{F}_k}] \leq \mj{32m\sqrt{2\pi}nL_0^2}+\mj{\left(\tfrac{m \tilde L_0^2n^2(\varepsilon_{k+1}+\varepsilon_k)}{\eta^2}\right)}.$$

\noindent {\it Proof of Claim 3.} This result follows immediately from Lemma~\ref{lemma:g_ik_eta_props}.\\

\medskip 

\noindent {\bf Claim 4 \us{(\textit{Analysis of \us{Term}  $T_3$})}.} For any arbitrary scalars $\RRmath{\beta_1,\beta_2>0}$, we have 
\begin{align*}
&\mathbb{E}[\langle (\mathbf{W}-\tfrac{1}{m}\mathbf{1}\mathbf{1}^\top) \mathbf{y}_{k+1},(\mathbf{W}-\tfrac{1}{m}\mathbf{1}\mathbf{1}^\top)(\nabla \mathbf f^\eta(\mathbf{x}_{k+1})-\nabla \mathbf f^\eta(\mathbf{x}_{k}))\rangle]  \\
&\le \left(\tfrac{\lambda_{\mathbf{W}}\RRmath{\gamma}L_0\mj{\sqrt{n}}}{\eta}+\mee{\tfrac{1}{2}}\beta_1+\beta_2\right)\lambda_{\mathbf{W}}^2\mathbb{E}[\|\mathbf{y}_{k+1}-\mathbf{1}\bar y_{k+1}\|^2] +\mj{\left(\tfrac{\lambda_{\mathbf{W}}^2L_0^2n}{\beta_2\eta^2}\right)}\mathbb{E}[\|\mathbf{x}_{k}-\mathbf{1}\bar x_{k}\|^2]\\
&+\left(\tfrac{\mee{3}\lambda_{\mathbf{W}}^2\RRmath{\gamma}^2L_0^2\mj{n}m}{\mee{2}\beta_1\eta^2}\right)\mathbb{E}[\|\overline{\nabla{f}^\eta}(\mathbf{x}_k)\|^2]+\mj{\left(\tfrac{32\sqrt{2\pi}\lambda_{\mathbf{W}}^2\RRmath{\gamma}^2L_0^4n^2}{\beta_1\eta^2}\right) }
+\mj{\left(\tfrac{\mee{3}\lambda_{\mathbf{W}}^2\RRmath{\gamma}^2L_0^2\tilde L_0^2 \mj{n^3} m      
\varepsilon_k}{\mee{2}\beta_1\eta^4}\right)}   .
\end{align*}

\noindent {\it Proof of Claim 4.} From the column
stochasticity of $\mathbf{W}$ we have
$(\mathbf{W}-\tfrac{1}{m}\mathbf{1}\mathbf{1}^\top)\tfrac{1}{m}\mathbf{1}\mathbf{1}^\top=\mathbf{0}_{m\times
n}$. Using this identity, \us{recalling that  
$\RRmath{\lambda_{\mathbf{W}} = \| {\bf W} - \tfrac{1}{m} {\bf 1}{\bf 1}^\top \|_2}$}, \us{and by invoking the} 
$\left(\frac{L_0\sqrt{n}}{\eta}\right)$-smoothness of
$\nabla {\bf f}^{\eta}$ and Claim 1, we have
\begin{align*}
&\langle (\mathbf{W}-\tfrac{1}{m}\mathbf{1}\mathbf{1}^\top) \mathbf{y}_{k+1},(\mathbf{W}-\tfrac{1}{m}\mathbf{1}\mathbf{1}^\top)(\nabla \mathbf f^\eta(\mathbf{x}_{k+1})-\nabla \mathbf f^\eta(\mathbf{x}_{k}))\rangle\\
&=\langle(\mathbf{W}-\tfrac{1}{m}\mathbf{1}\mathbf{1}^\top) (\mathbf{y}_{k+1}-\tfrac{1}{m}\mathbf{1}\mathbf{1}^\top\mathbf{y}_{k+1}) , (\mathbf{W}-\tfrac{1}{m}\mathbf{1}\mathbf{1}^\top)(\nabla \mathbf f^\eta(\mathbf{x}_{k+1})-\nabla \mathbf f^\eta(\mathbf{x}_{k}))\rangle \\
&\le \mj{\tfrac{\lambda_{\mathbf{W}}^2L_0\sqrt{n}}{\eta}}\|\mathbf{y}_{k+1}- \mathbf{1}\bar{y}_{k+1}\|\|\mathbf{x}_{k+1}-\mathbf{x}_{k}\|.
\end{align*}
Invoking Lemma~\ref{lemma: multiple inequalities for the metrics}, we have
\begin{align*}
&\|\mathbf{x}_{k+1}-\mathbf{x}_{k}\|=\|\mathbf{x}_{k+1}-\tfrac{1}{m}\mathbf{1}\mathbf{1}^\top\mathbf{x}_{k+1}+\tfrac{1}{m}\mathbf{1}\mathbf{1}^\top\mathbf{x}_{k+1}-\tfrac{1}{m}\mathbf{1}\mathbf{1}^\top\mathbf{x}_{k}+\tfrac{1}{m}\mathbf{1}\mathbf{1}^\top\mathbf{x}_{k}-\mathbf{x}_{k}\|\\
&\le \|\mathbf{x}_{k+1}- \mathbf{1} \bar{x}_{k+1}\|+\RRmath{\gamma}\sqrt{m}\|\overline{\nabla{f}^\eta}(\mathbf{x}_k)+\bar \delta_{k}^\eta+\bar{e}_k^{\eta,\varepsilon_k}\|+\|\mathbf{x}_{k}-\tfrac{1}{m}\mathbf{1}\mathbf{1}^\top\mathbf{x}_{k}\|\\
&\le \us{\RRmath{\gamma} \lambda_{\bf W} \|\mathbf{y}_{k+1}- \mathbf{1} \bar{y}_{k+1}\| + \| {\bf W} \mathbf{x}_k - {\bf 1} \bar{x}_k\|}+\RRmath{\gamma}\sqrt{m}\|\overline{\nabla{f}^\eta}(\mathbf{x}_k)+\bar \delta_{k}^\eta+\bar{e}_k^{\eta,\varepsilon_k}\|+\|\mathbf{x}_{k}-\tfrac{1}{m}\mathbf{1}\mathbf{1}^\top\mathbf{x}_{k}\|\\
&\overset{\tiny \us{\mbox{Lemma}~\ref{lem:lambda_w}}}{\le} \us{\RRmath{\gamma} \lambda_{\bf W} \|\mathbf{y}_{k+1}- \mathbf{1} \bar{y}_{k+1}\| + \lambda_{\bf W}\|  \mathbf{x}_k - {\bf 1} \bar{x}_k\|}+\RRmath{\gamma}\sqrt{m}\|\overline{\nabla{f}^\eta}(\mathbf{x}_k)+\bar \delta_{k}^\eta+\bar{e}_k^{\eta,\varepsilon_k}\|+\|\mathbf{x}_{k}-\mathbf{1}\us{\bar{x}_{k}}\|\\
&\overset{\us{\lambda_{\bf W}\le 1}}{\le} \RRmath{\gamma}\lambda_{\mathbf{W}}\|\mathbf{y}_{k+1}- \mathbf{1} \bar{y}_{k+1}\|+\RRmath{\gamma}\sqrt{m}\|\overline{\nabla{f}^\eta}(\mathbf{x}_k)+\bar \delta_{k}^\eta+\bar{e}_k^{\eta,\varepsilon_k}\|+2\|\mathbf{x}_{k}-\mathbf{1}\bar{x}_{k}\|.
\end{align*}
From the two preceding inequalities, we obtain
\begin{align}
&\langle (\mathbf{W}-\tfrac{1}{m}\mathbf{1}\mathbf{1}^\top) \mathbf{y}_{k+1},(\mathbf{W}-\tfrac{1}{m}\mathbf{1}\mathbf{1}^\top)(\nabla \mathbf f^\eta(\mathbf{x}_{k+1})-\nabla \mathbf f^\eta(\mathbf{x}_{k}))\rangle\notag\\
&\le \tfrac{\RRmath{\gamma}\lambda_{\mathbf{W}}^3L_0\mj{\sqrt{n}}}{\eta}\|\mathbf{y}_{k+1}- \mathbf{1} \bar{y}_{k+1}\|^\mj{2}+\left(\lambda_{\mathbf{W}}\|\mathbf{y}_{k+1}- \mathbf{1} \bar{y}_{k+1}\|\right)\left(\tfrac{\sqrt{m}\RRmath{\gamma}\lambda_{\mathbf{W}} L_0\mj{\sqrt{n}}}{\eta}\|\overline{\nabla{f}^\eta}(\mathbf{x}_k)+\bar \delta_{k}^\eta+\bar{e}^{\eta,\varepsilon_k}_k\|\right)\notag\\
&+2\left(\lambda_{\mathbf{W}}\|\mathbf{y}_{k+1}- \mathbf{1} \bar{y}_{k+1}\|\right)\left(\tfrac{\lambda_{\mathbf{W}} L_0\mj{\sqrt{n}}}{\eta}\|\mathbf{x}_{k}- \mathbf{1} \bar{x}_{k}\|\right).\label{eq:bound for lemma 7}
\end{align}
{Via} Young’s inequality, {the second term on the right can be bounded as }
\begin{align*}
&\left(\lambda_{\mathbf{W}}\|\mathbf{y}_{k+1}-\mathbf{1}\bar{y}_{k+1}\|\right)\left(\tfrac{\sqrt{m}\RRmath{\gamma}\lambda_{\mathbf{W}} L_0\mj{\sqrt{n}}}{\eta}\|\overline{\nabla{f}^\eta}(\mathbf{x}_k)+\bar \delta_{k}^\eta+\bar{e}_k^{\eta,\varepsilon_k}\|\right)\\
&\le \mee{\tfrac{1}{2}}\beta_1\lambda_{\mathbf{W}}^2\|\mathbf{y}_{k+1}-\mathbf{1} \bar{y}_{k+1}\|^2+\frac{ m\RRmath{\gamma}^2\lambda_{\mathbf{W}}^2 L_0^2\mj{{n}}}{\mee{2}\beta_1\eta^2}\|\overline{\nabla{f}^\eta}(\mathbf{x}_k)+\bar \delta_{k}^\eta+\bar{e}_k^{\eta,\varepsilon_k}\|^2,
\end{align*}
and $
 2\left(\lambda_{\mathbf{W}}\|\mathbf{y}_{k+1}- \mathbf{1}\bar{y}_{k+1}\|\right)\left(\tfrac{\lambda_{\mathbf{W}} L_0\mj{\sqrt{n}}}{\eta}\|\mathbf{x}_{k}- \mathbf{1} \bar{x}_{k}\|\right) \le \beta_2\lambda_{\mathbf{W}}^2\|\mathbf{y}_{k+1}- \mathbf{1} \bar{y}_{k+1}\|^2+\tfrac{\lambda_{\mathbf{W}}^2 L_0^2\mj{{n}}}{\beta_2\eta^2}\|\mathbf{x}_{k}- \mathbf{1} \bar{x}_{k}\|^2.$
Using the \us{prior} two inequalities in \eqref{eq:bound for lemma 7}, taking expectations, and invoking Lemma~\ref{lemma:multiple parts}, the result follows.\\

\smallskip 

\noindent {\bf Claim 5 \us{(\textit{Analysis of \us{Term} $T_4$})}.} {For any $k\ge 0$,} \begin{align*}
\mathbb{E}[\langle (\mathbf{W}-\tfrac{1}{m}\mathbf{1}\mathbf{1}^\top) \mathbf{y}_{k+1},(\mathbf{W}-\tfrac{1}{m}\mathbf{1}\mathbf{1}^\top)(\boldsymbol{\delta}_{k+1}^\eta-\boldsymbol{\delta}_{k}^\eta) \rangle ] &\leq \mj{16\sqrt{2\pi}m\left(\|\mathbf{W}\|^2+\mee{\tfrac{1}{2}}\lambda_{\mathbf{W}}^2\right)\mj{nL_0^2}} +\tfrac{m\|\mathbf{W}\|^2\tilde L_0^2n^2\varepsilon_k}{\eta^2} .
\end{align*}

\noindent {\it Proof of Claim 5.} Recall from Lemma~\ref{lemma:g_ik_eta_props} that  $\mathbb{E}[\boldsymbol{\delta}_{k+1}^\eta\mid \mathcal{F}_{k+1}]=0$. Utilizing this and that $ \mathbf{y}_{k+1}$ is $\mathcal{F}_{k+1}$-measurable, we \us{may express $\mathbb{E}[\langle (\mathbf{W}-\tfrac{1}{m}\mathbf{1}\mathbf{1}^\top) \mathbf{y}_{k+1},(\mathbf{W}-\tfrac{1}{m}\mathbf{1}\mathbf{1}^\top)\boldsymbol{\delta}_{k+1}^\eta \rangle \mid \mathcal{F}_k]$ as} 
\begin{align*}
&\mathbb{E}[\langle (\mathbf{W}-\tfrac{1}{m}\mathbf{1}\mathbf{1}^\top) \mathbf{y}_{k+1},(\mathbf{W}-\tfrac{1}{m}\mathbf{1}\mathbf{1}^\top)\boldsymbol{\delta}_{k+1}^\eta \rangle \mid \mathcal{F}_k] \\
&=\mathbb{E}_{\boldsymbol{\xi}_k,\mathbf{v}_k}[\mathbb{E}[\langle (\mathbf{W}-\tfrac{1}{m}\mathbf{1}\mathbf{1}^\top) \mathbf{y}_{k+1},(\mathbf{W}-\tfrac{1}{m}\mathbf{1}\mathbf{1}^\top)\boldsymbol{\delta}_{k+1}^\eta \rangle\mid \mathcal{F}_{k+1}]\, \mid \mathcal{F}_k \, ] \\
&=\mathbb{E}_{\boldsymbol{\xi}_k,\mathbf{v}_k}[\langle (\mathbf{W}-\tfrac{1}{m}\mathbf{1}\mathbf{1}^\top) \mathbf{y}_{k+1},(\mathbf{W}-\tfrac{1}{m}\mathbf{1}\mathbf{1}^\top)\mathbb{E}[\boldsymbol{\delta}_{k+1}^\eta \mid \mathcal{F}_{k+1}]\rangle\, \mid \mathcal{F}_k \, ] =
0.
\end{align*}
Consider $\langle (\mathbf{W}-\tfrac{1}{m}\mathbf{1}\mathbf{1}^\top) \mathbf{y}_{k+1},(\mathbf{W}-\tfrac{1}{m}\mathbf{1}\mathbf{1}^\top) \boldsymbol{\delta}_{k}^\eta  \rangle$. Invoking $(\mathbf{W}-\tfrac{1}{m}\mathbf{1}\mathbf{1}^\top)\tfrac{1}{m}\mathbf{1}\mathbf{1}^\top=\mathbf{0}_{m\times n}$, we have
\begin{align*}
&\langle (\mathbf{W}-\tfrac{1}{m}\mathbf{1}\mathbf{1}^\top) \mathbf{y}_{k+1},(\mathbf{W}-\tfrac{1}{m}\mathbf{1}\mathbf{1}^\top) \boldsymbol{\delta}_{k}^\eta  \rangle = \mbox{Trace}\left(\mathbf{y}_{k+1}^\top(\mathbf{W}-\tfrac{1}{m}\mathbf{1}\mathbf{1}^\top)^\top(\mathbf{W}-\tfrac{1}{m}\mathbf{1}\mathbf{1}^\top) \boldsymbol{\delta}_{k}^\eta \right)\\
& = \mbox{Trace}\left(\mathbf{y}_{k+1}^\top \mathbf{W}^\top(\mathbf{W}-\tfrac{1}{m}\mathbf{1}\mathbf{1}^\top) \boldsymbol{\delta}_{k}^\eta \right) = \langle \mathbf{W} \mathbf{y}_{k+1},(\mathbf{W}-\tfrac{1}{m}\mathbf{1}\mathbf{1}^\top) \boldsymbol{\delta}_{k}^\eta  \rangle.
\end{align*}
Invoking \eqref{eqn:R1} and Lemma~\ref{lemma:g_ik_eta_props}, from the  preceding relations, we obtain 
\begin{align*}
&\mathbb{E}[\langle (\mathbf{W}-\tfrac{1}{m}\mathbf{1}\mathbf{1}^\top) \mathbf{y}_{k+1},(\mathbf{W}-\tfrac{1}{m}\mathbf{1}\mathbf{1}^\top) (\boldsymbol{\delta}_{k+1}^\eta-\boldsymbol{\delta}_{k}^\eta)   \rangle \mid \mathcal{F}_k] = -\mathbb{E}[\langle \mathbf{W} \mathbf{y}_{k+1},(\mathbf{W}-\tfrac{1}{m}\mathbf{1}\mathbf{1}^\top) \boldsymbol{\delta}_{k}^\eta  \rangle\mid \mathcal{F}_k]\\
&=-\mathbb{E}[\langle \mathbf{W}^2(\mathbf{y}_{k}+ \nabla \mathbf{f}^\eta(\mathbf{x}_k) - \nabla \mathbf{f}^\eta(\mathbf{x}_{k-1})+\boldsymbol{\delta}_{k}^\eta-\boldsymbol{\delta}_{k-1}^\eta+\mathbf{e}_{k}^{\eta,\varepsilon_k}-\mathbf{e}_{k-1}^{\eta,\varepsilon_{k-1}}),(\mathbf{W}-\tfrac{1}{m}\mathbf{1}\mathbf{1}^\top) \boldsymbol{\delta}_{k}^\eta  \rangle\mid \mathcal{F}_k].
\end{align*}
\mee{Recall from Lemma~\ref{lemma:g_ik_eta_props} that  $\mathbb{E}[\boldsymbol{\delta}_{k}^\eta\mid \mathcal{F}_{k}]=0$. \fyy{Further, invoking} the fact that the terms $\mathbf{y}_{k}, \nabla \mathbf{f}^\eta(\mathbf{x}_k) , \nabla \mathbf{f}^\eta(\mathbf{x}_{k-1}),\boldsymbol{\delta}_{k-1}^\eta$ and $\mathbf{e}_{k-1}^{\eta,\varepsilon_{k-1}}$ are $\mathcal{F}_{k}$-measurable, we have $$\mathbb{E}[\langle \mathbf{W}^2(\mathbf{y}_{k}+ \nabla \mathbf{f}^\eta(\mathbf{x}_k) - \nabla \mathbf{f}^\eta(\mathbf{x}_{k-1})-\boldsymbol{\delta}_{k-1}^\eta-\mathbf{e}_{k-1}^{\eta,\varepsilon_{k-1}}),(\mathbf{W}-\tfrac{1}{m}\mathbf{1}\mathbf{1}^\top) \boldsymbol{\delta}_{k}^\eta  \rangle\mid \mathcal{F}_k]=0.$$ Therefore, \fyy{from the two preceding 
 relations} we obtain}
\begin{align*}
&\mathbb{E}[\langle (\mathbf{W}-\tfrac{1}{m}\mathbf{1}\mathbf{1}^\top) \mathbf{y}_{k+1},(\mathbf{W}-\tfrac{1}{m}\mathbf{1}\mathbf{1}^\top) (\boldsymbol{\delta}_{k+1}^\eta-\boldsymbol{\delta}_{k}^\eta)   \rangle \mid \mathcal{F}_k]\\
&=-\mathbb{E}[\langle \mathbf{W}^2(\boldsymbol{\delta}_{k}^\eta+\mathbf{e}_{k}^{\eta,\varepsilon_k}),(\mathbf{W}-\tfrac{1}{m}\mathbf{1}\mathbf{1}^\top) \boldsymbol{\delta}_{k}^\eta  \rangle\mid \mathcal{F}_k].
\end{align*}
\mee{
Invoking the identity $-\mathbb{E}[\langle a ,b\rangle]\le \tfrac{1}{2}\left(\mathbb{E}[\|a\|^2]+ \mathbb{E}[\|b\|^2]\right)$, and applying Lemma~\ref{lem:lambda_w} where we define $\RRmath{\lambda_{\mathbf{W}}\triangleq \|\mathbf{W}-\tfrac{1}{m}\mathbf{1}\mathbf{1}^{\top}\|_2}$, we obtain
\begin{align*}
&\mathbb{E}[\langle (\mathbf{W}-\tfrac{1}{m}\mathbf{1}\mathbf{1}^\top) \mathbf{y}_{k+1},(\mathbf{W}-\tfrac{1}{m}\mathbf{1}\mathbf{1}^\top) (\boldsymbol{\delta}_{k+1}^\eta-\boldsymbol{\delta}_{k}^\eta)   \rangle \mid \mathcal{F}_k]\\
&\leq \mee{\tfrac{1}{2}}\left(\|\mathbf{W}\|^2\mathbb{E}[\|\boldsymbol{\delta}_{k}^\eta+\mathbf{e}_{k}^{\eta,\varepsilon_k}\|^2\mid \mathcal{F}_k] + \lambda_{\mathbf{W}}^2\mathbb{E}[\|\boldsymbol{\delta}_{k}^\eta\|^2\mid \mathcal{F}_k]\right)\\
&\leq  \left(\|\mathbf{W}\|^2+{\tfrac{1}{2}}\lambda_{\mathbf{W}}^2\right)\mathbb{E}[\|\boldsymbol{\delta}_{k}^\eta\|^2\mid \mathcal{F}_k] + \|\mathbf{W}\|^2\mathbb{E}[\|\mathbf{e}_{k}^{\eta,\us{\varepsilon_k}}\|^2\mid \mathcal{F}_k].
\end{align*}}The result follows by invoking Lemma~\ref{lemma:g_ik_eta_props} and taking the expectation on both sides.

\smallskip 

\noindent {\bf Claim 6 {(\textit{Analysis of \us{Term} $T_5$})}.} For any arbitrary $\beta_3>0$ and for any $k\ge 0$, we have  \begin{align*}
\mathbb{E}[\langle (\mathbf{W}-\tfrac{1}{m}\mathbf{1}\mathbf{1}^\top) \mathbf{y}_{k+1},(\mathbf{W}-\tfrac{1}{m}\mathbf{1}\mathbf{1}^\top)(\mathbf{e}_{k+1}^{\eta,\varepsilon_{k+1}}-\mathbf{e}_k^{\eta,\varepsilon_k})\rangle ] &\leq  \mee{\tfrac{1}{2}}\lambda_{\mathbf{W}}^2\beta_3 \mathbb{E}[\|\mathbf{y}_{k+1}-\mathbf{1}\bar{y}_{k+1}\|^2] \\
&+  \tfrac{m\lambda_{\mathbf{W}}^2\tilde L_0^2n^2(\varepsilon_k+\varepsilon_{k+1})}{\eta^2\beta_3} .
\end{align*}

\noindent {\it Proof of Claim 6.} By recalling $\bar{y}_{k+1} = \tfrac{1}{m}\mathbf{1}^{\top} \mathbf{y}_{k+1}$ and invoking the identity $a^\top b\le \tfrac{\beta_3}{2}\|a\|^2+ \RRmath{\tfrac{1}{2\beta_3}}\|b\|^2$, 
\begin{align*}
 &\langle (\mathbf{W}-\tfrac{1}{m}\mathbf{1}\mathbf{1}^\top) \mathbf{y}_{k+1},(\mathbf{W}-\tfrac{1}{m}\mathbf{1}\mathbf{1}^\top)(\mathbf{e}_{k+1}^{\eta,\varepsilon_{k+1}}-\mathbf{e}_k^{\eta,\varepsilon_k})\rangle  & \\
 & = \langle (\mathbf{W}-\tfrac{1}{m}\mathbf{1}\mathbf{1}^\top) \us{(\mathbf{y}_{k+1} - {\bf 1} \bar{y}_{k+1})},(\mathbf{W}-\tfrac{1}{m}\mathbf{1}\mathbf{1}^\top)(\mathbf{e}_{k+1}^{\eta,\varepsilon_{k+1}}-\mathbf{e}_k^{\eta,\varepsilon_k})\rangle  & \\
 &\leq {\tfrac{1}{2}}\lambda_{\mathbf{W}}^2\beta_3 \|\mathbf{y}_{k+1}-\mathbf{1}\bar{y}_{k+1}\|^2 +{\tfrac{1}{2}}\lambda_{\mathbf{W}}^2\beta_3^{-1}\|\mathbf{e}_{k+1}^{\eta,\varepsilon_{k+1}}-\mathbf{e}_k^{\eta,\varepsilon_k}\|^2\\
 & \leq  {\tfrac{1}{2}}\lambda_{\mathbf{W}}^2\beta_3 \|\mathbf{y}_{k+1}-\mathbf{1}\bar{y}_{k+1}\|^2 + \lambda_{\mathbf{W}}^2\beta_3^{-1}\left(\|\mathbf{e}_{k+1}^{\eta,\varepsilon_{k+1}}\|^2+\|\mathbf{e}_k^{\eta,\varepsilon_k}\|^2\right).
\end{align*}
Taking expectation\mee{s} on both sides of the preceding inequality and invoking Lemma~\ref{lemma:g_ik_eta_props}, we obtain the result in Claim 6. 

\medskip 

We are now ready to complete the proof of the main result. Consider \eqref{eq:bound for y-bar y-inexact}. Taking expectation\mee{s} on both sides and invoking the bounds in Claims 1 to 6\mj{,} we obtain
\begin{align*}
\mathbb{E}[\|{ \mathbf{y}}_{k+2}-\mathbf{1}{\bar{ {y}}}_{k+2}\|^2]
&\le \lambda_{\mathbf{W}}^2\mathbb{E}[\|{ \mathbf{y}}_{k+1}-\mathbf{1}{ \bar{y}}_{k+1}\|^2]  \\
&+5\lambda_{\mathbf{W}}^2
\mj{\left(\tfrac{\sqrt{n}L_0}{\eta}\right)^2}
\left(6\RRmath{\gamma}^2\lambda_{\mathbf{W}}^2\mathbb{E}[\|\mathbf{y}_{k+1}-\mathbf{1}\bar{y}_{k+1}\|^2]+9\mathbb{E}[\|\mathbf{x}_{k}-\mathbf{1}\bar{x}_{k}\|^2]\right.\\
&\left.+6m\RRmath{\gamma}^2\mathbb{E}[\|\overline{\nabla{f}^\eta}(\mathbf{x}_k)\|^2]+\mj{96\sqrt{2\pi}\RRmath{\gamma}^2nL_0^2}+\mj{\left(\tfrac{6m\tilde L_0^2n^2\varepsilon_k\RRmath{\gamma}^2}{\eta^2}\right)}\right)\\
 & +5\lambda_{\mathbf{W}}^2\left(\mj{32m\sqrt{2\pi}nL_0^2}+\mj{\left(\tfrac{m \tilde L_0^2n^2(\varepsilon_{k+1}+\varepsilon_k)}{\eta^2}\right)}\right)\notag\\
&+  2\left(\tfrac{\lambda_{\mathbf{W}}\RRmath{\gamma}L_0\mj{\sqrt{n}}}{\eta}+\mee{\tfrac{1}{2}}\beta_1+\beta_2\right)\lambda_{\mathbf{W}}^2\mathbb{E}[\|\mathbf{y}_{k+1}-\mathbf{1}\bar y_{k+1}\|^2] \notag\\
&+\left(\tfrac{2\lambda_{\mathbf{W}}^2L_0^2\mj{n}}{\beta_2\eta^2}\right)\mathbb{E}[\|\mathbf{x}_{k}-\mathbf{1}\bar x_{k}\|^2]+\left(\tfrac{3\lambda_{\mathbf{W}}^2\RRmath{\gamma}^2L_0^2\mj{n}m}{\beta_1\eta^2}\right)\mathbb{E}[\|\overline{\nabla{f}^\eta}(\mathbf{x}_k)\|^2]+\mj{ \tfrac{64\sqrt{2\pi}\lambda_{\mathbf{W}}^2\RRmath{\gamma}^2L_0^4n^2}{\beta_1\eta^2}  }
\notag\\
&+\mj{ \tfrac{3\lambda_{\mathbf{W}}^2\RRmath{\gamma}^2L_0^2\tilde L_0^2 \mj{n^3} m      
\varepsilon_k}{\beta_1\eta^4}  }+  \mj{32\sqrt{2\pi}m\left(\|\mathbf{W}\|^2+\mee{\tfrac{1}{2}}\lambda_{\mathbf{W}}^2\right)\mj{nL_0^2}} + \mj{2m\|\mathbf{W}\|^2\left(\tfrac{\tilde L_0^2n^2\varepsilon_k}{\eta^2}\right)}\notag\\
&+\lambda_{\mathbf{W}}^2\beta_3 \mathbb{E}[\|\mathbf{y}_{k+1}-\mathbf{1}\bar{y}_{k+1}\|^2] +  \mj{ \tfrac{2m\lambda_{\mathbf{W}}^2\tilde L_0^2n^2(\varepsilon_k+\varepsilon_{k+1})}{\eta^2\beta_3} }.
\end{align*}
Rearranging the terms, we obtain 
\begin{align*}
\mathbb{E}[\|{ \mathbf{y}}_{k+2}-\mathbf{1}{\bar{ {y}}}_{k+2}\|^2]
&\le \lambda_{\mathbf{W}}^2\left(1+30\lambda_{\mathbf{W}}^2
\mj{\left(\tfrac{\sqrt{n}L_0}{\eta}\right)^2}\RRmath{\gamma}^2 \right. \\
 &\left. +2\left(\tfrac{\lambda_{\mathbf{W}}\RRmath{\gamma}L_0\mj{\sqrt{n}}}{\eta}+\mee{\tfrac{1}{2}}\beta_1+\beta_2\right) + \beta_3 \right)\mathbb{E}[\|{ \mathbf{y}}_{k+1}-\mathbf{1}{ \bar{y}}_{k+1}\|^2]  \\
&  +\lambda_{\mathbf{W}}^2
\left(\tfrac{\mj{\sqrt{n}}L_0}{\eta}\right)^2\left( 2 \beta_2^{-1}+45\right)\mathbb{E}[\|\mathbf{x}_{k}-\mathbf{1}\bar x_{k}\|^2]+\mj{32m\sqrt{2\pi}\left(\|\mathbf{W}\|^2+\mee{\tfrac{11}{2}}\lambda_{\mathbf{W}}^2\right)\mj{nL_0^2}  } \\
&+3m\RRmath{\gamma}^2\lambda_{\mathbf{W}}^2
\left(\tfrac{\mj{\sqrt{n}}L_0}{\eta}\right)^2\left( \beta_1^{-1} +10 \right)\mathbb{E}[\|\overline{\nabla{f}^\eta}(\mathbf{x}_k)\|^2]\\
&+\mj{32\sqrt{2\pi}nL_0^2\RRmath{\gamma}^2\lambda_{\mathbf{W}}^2
\left(\tfrac{\mj{\sqrt{n}}L_0}{\eta}\right)^2
(15+2\beta_1^{-1})}+\mj{\Rme{\left(\tfrac{3\lambda_{\mathbf{W}}^2\RRmath{\gamma}^2L_0^2\tilde L_0^2 \mj{\Rme{n^3}} m      
\varepsilon_k}{\eta^4}\right)} (10+\beta_1^{-1})}\\
&+\mj{2m\|\mathbf{W}\|^2\left(\tfrac{\tilde L_0^2n^2\varepsilon_k}{\eta^2}\right)  } +\mj{m\lambda_{\mathbf{W}}^2 \left(\tfrac{n^2\tilde L_0^2}{\eta^2}\right)\left(5 + 2\beta_3^{-1}\right)(\varepsilon_k+\varepsilon_{k+1})}.
\end{align*}
Choosing $ \beta_1 =2\beta_2=\beta_3=\tfrac{1-\lambda_{\mathbf{W}}^2}{10\lambda_{\mathbf{W}}^2}$ and $\RRmath{\gamma} \leq \min \Big\{
 \tfrac{ \sqrt{1-\lambda_{\mathbf{W}}^2}}{10\sqrt{3}\lambda_{\mathbf{W}}^2 }
 ,\tfrac{(1-\lambda_{\mathbf{W}}^2) }{20\lambda_{\mathbf{W}}^3 }\Big\} \left(\tfrac{\eta}{\mj{\sqrt{n}} L_0 }\right)$, we have that 
 \begin{align*}
 \lambda_{\mathbf{W}}^2\left(1+30\lambda_{\mathbf{W}}^2
\left(\tfrac{\mj{\sqrt{n}}L_0}{\eta}\right)^2\RRmath{\gamma}^2 +2\left(\tfrac{\lambda_{\mathbf{W}}\RRmath{\gamma}L_0\mj{\sqrt{n}}}{\eta}+\mee{\tfrac{1}{2}}\beta_1+\beta_2\right) + \beta_3 \right) \leq \tfrac{1+\lambda_{\mathbf{W}}^2}{2}.
 \end{align*}
From the preceding inequalities and choices of $\beta_1, \beta_2$, and $\beta_3$, we obtain the result. 

\noindent\RRme{Next, consider the case} $\RRmath{\lambda_{\mathbf{W}}=0}$\RRme{. By the definition of} $\RRmath{\lambda_{\mathbf{W}}}$\RRme{, we have} $\RRmath{\mathbf{W}=\tfrac{1}{m}\mathbf{1}\mathbf{1}^{\top}}$\RRme{. Therefore, invoking the gradient-tracking update~\eqref{eqn:R1}, for any} $\RRmath{k\geq 0}$\RRme{, we obtain} $$\RRmath{\mathbf{y}_{k+1}=\tfrac{1}{m}\mathbf{1}\mathbf{1}^{\top}\left(\mathbf{y}_{k}+\nabla\mathbf{f}^{\eta}(\mathbf{x}_{k})-\nabla\mathbf{f}^{\eta}(\mathbf{x}_{k-1})+\boldsymbol{\delta}_{k}^{\eta}-\boldsymbol{\delta}_{k-1}^{\eta}+\mathbf{e}_{k}^{\eta,\varepsilon_{k}}-\mathbf{e}_{k-1}^{\eta,\varepsilon_{k-1}}\right).}$$\RRme{Equivalently,} $$\RRmath{\mathbf{y}_{k+1}=\mathbf{1}\left(\tfrac{1}{m}\mathbf{1}^{\top}\left(\mathbf{y}_{k}+\nabla\mathbf{f}^{\eta}(\mathbf{x}_{k})-\nabla\mathbf{f}^{\eta}(\mathbf{x}_{k-1})+\boldsymbol{\delta}_{k}^{\eta}-\boldsymbol{\delta}_{k-1}^{\eta}+\mathbf{e}_{k}^{\eta,\varepsilon_{k}}-\mathbf{e}_{k-1}^{\eta,\varepsilon_{k-1}}\right)\right).}$$\RRme{On the other hand, by the definition of} $\RRmath{\bar{y}_{k+1}}$\RRme{, we have} $\RRmath{\bar{y}_{k+1}=\tfrac{1}{m}\mathbf{1}^{\top}\mathbf{y}_{k+1}}$\RRme{. Substituting the preceding expression for} $\RRmath{\mathbf{y}_{k+1}}$ \RRme{and using} $\RRmath{\mathbf{1}^{\top}\mathbf{1}=m}$\RRme{, we obtain} $$\RRmath{\bar{y}_{k+1}=\tfrac{1}{m}\mathbf{1}^{\top}\left(\mathbf{y}_{k}+\nabla\mathbf{f}^{\eta}(\mathbf{x}_{k})-\nabla\mathbf{f}^{\eta}(\mathbf{x}_{k-1})+\boldsymbol{\delta}_{k}^{\eta}-\boldsymbol{\delta}_{k-1}^{\eta}+\mathbf{e}_{k}^{\eta,\varepsilon_{k}}-\mathbf{e}_{k-1}^{\eta,\varepsilon_{k-1}}\right).}$$\RRme{Hence, from the gradient-tracking update~\eqref{eqn:R1} and the preceding expression, we obtain} $\RRmath{\mathbf{y}_{k+1}=\mathbf{1}\bar{y}_{k+1}}$ \RRme{for all} $\RRmath{k\geq 0}$\RRme{. Consequently,} $\RRmath{\mathbb{E}\left[\|\mathbf{y}_{k+1}-\mathbf{1}\bar{y}_{k+1}\|^2\right]=0}$ \RRme{for all} $\RRmath{k\geq 0}$\RRme{.}

\mee{$\hfill \Box$}
}

\medskip

\noindent {\bf Proof of Lemma~\ref{lem:Alg2_conv}.} 
Consider the update rule in Algorithm~\ref{alg:lowerlevel-1stage}. Let us define $\delta_{i,t}^{k,\ell} = \tilde{F}_i(\hat{x}_{i,k}^{\ell},z_{i,t}^{k,\ell},\xi_{i,t}^{k,\ell}) -  {F}_i(\hat{x}_{i,k}^{\ell},z_{i,t}^{k,\ell}) $. Invoking the nonexpansivity of the Euclidean projection, for any $\ell \in \{1,2\}$ and $t \geq 0$, we obtain 
\begin{align*}
& \|z_{i,t+1}^{k,\ell} - z_i(\hat{x}_{i,k}^{\ell}) \|^2   = \left\| \Pi_{\mathcal{Z}_i(\hat x_{i,k}^{\ell})}\left[z_{i,t}^{k,\ell}-\hat{\gamma}_t\tilde{F}_i(\hat{x}_{i,k}^{\ell},z_{i,t}^{k,\ell},\xi_{i,t}^{k,\ell})\right]  - z_i(\hat{x}_{i,k}^{\ell})\right\|^2 \\
 & = \left\| \Pi_{\mathcal{Z}_i(\hat x_{i,k}^{\ell})}\left[z_{i,t}^{k,\ell}-\hat{\gamma}_t\left({F}_i(\hat{x}_{i,k}^{\ell},z_{i,t}^{k,\ell})+\delta_{i,t}^{k,\ell}\right)\right]  - \Pi_{\mathcal{Z}_i(\hat x_{i,k}^{\ell})}\left[z_i(\hat{x}_{i,k}^{\ell}) -\hat{\gamma}_t {F}_i(\hat{x}_{i,k}^{\ell},z_i(\hat{x}_{i,k}^{\ell}) )\right]  \right\|^2\\
 & \leq \left\|  z_{i,t}^{k,\ell}-\hat{\gamma}_t\left({F}_i(\hat{x}_{i,k}^{\ell},z_{i,t}^{k,\ell})+\delta_{i,t}^{k,\ell}\right)   -  z_i(\hat{x}_{i,k}^{\ell}) +\hat{\gamma}_t {F}_i(\hat{x}_{i,k}^{\ell},z_i(\hat{x}_{i,k}^{\ell}) )  \right\|^2.
\end{align*}
 Note that from Assumption~\ref{assum:alg2}, $\mathbb{E}[\delta_{i,t}^{k,\ell} \mid \tilde{\mathcal{F}}_{i,t}^{k,\ell}] =0  $ and $\mathbb{E}[\|\delta_{i,t}^{k,\ell}\|^2 \mid \tilde{\mathcal{F}}_{i,t}^{k,\ell}] \leq \nu_F^2$. Taking the conditional expectation on both sides of the preceding inequality, we obtain 
\begin{align*}
\mathbb{E}[\|z_{i,t+1}^{k,\ell} - z_i(\hat{x}_{i,k}^{\ell}) \|^2 \mid \tilde{\mathcal{F}}_{i,t}^{k,\ell}]  
& \leq \|z_{i,t}^{k,\ell} -  z_i(\hat{x}_{i,k}^{\ell})\|^2 + \hat{\gamma}^2_t\|{F}_i(\hat{x}_{i,k}^{\ell},z_{i,t}^{k,\ell})-{F}_i(\hat{x}_{i,k}^{\ell},z_i(\hat{x}_{i,k}^{\ell}) ) \|^2\\
& -2\hat{\gamma}_t(z_{i,t}^{k,\ell} -  z_i(\hat{x}_{i,k}^{\ell}))^\top(F_i(\hat{x}_{i,k}^{\ell},z_{i,t}^{k,\ell})-{F}_i(\hat{x}_{i,k}^{\ell},z_i(\hat{x}_{i,k}^{\ell}) )) + \hat{\gamma}_t^2\nu_F^2.
\end{align*}
Invoking the strong monotonicity and Lipschitzian property of $F_i$, we obtain
\begin{align*}
\mathbb{E}[\|z_{i,t+1}^{k,\ell} - z_i(\hat{x}_{i,k}^{\ell}) \|^2 \mid \tilde{\mathcal{F}}_{i,t}^{k,\ell}]  
& \leq (1-2\mu_F\hat\gamma_t +\Rme{\hat\gamma_t^2}L_F^2)\|z_{i,t}^{k,\ell} -  z_i(\hat{x}_{i,k}^{\ell})\|^2  + \hat{\gamma}^2_t\nu_F^2.
\end{align*} 
Note that from $\hat{\Gamma} >\frac{\hat{\gamma}_tL_F^2}{\mu_F}$, we have $\hat \gamma_t \leq \frac{\mu_F}{L_F^2}$, implying that 
\begin{align*}
\mathbb{E}[\|z_{i,t+1}^{k,\ell} - z_i(\hat{x}_{i,k}^{\ell}) \|^2 \mid \tilde{\mathcal{F}}_{i,t}^{k,\ell}]  
& \leq (1- \mu_F\hat\gamma_t)\|z_{i,t}^{k,\ell} -  z_i(\hat{x}_{i,k}^{\ell})\|^2  + \hat{\gamma}^2_t\nu_F^2.
\end{align*}
Invoking \cite[Lemma 8]{cui2023complexity} \us{and choosing} $\hat{\gamma} > \tfrac{1}{\mu_F}$,  the result \us{follows,} 
where $D_i$ is given by \Rme{Assumption~\ref{assum:main2}}.
\begin{align*}
\mathbb{E}[\|z_{i,\varepsilon_k}(\hat{x}_{i,k}^{\ell}) - z_i(\hat{x}_{i,k}^{\ell})\|^2  \mid \tilde{\mathcal{F}}_{i,0}^{k,\ell}] \leq \max\{\nu_F^2\Rme{\hat{\gamma}^2}(\mu_F\hat{\gamma}-1)^{-1}, \hat{\Gamma} D_i\}\tfrac{1}{t_k+\hat{\Gamma}}. 
\end{align*}\mee{$\hfill \Box$}

\medskip

\noindent{\bf Proof of Lemma~\ref{lem:Harmonic series bounds}.} 
 First, assume that $\Gamma \geq 1$. We have 
\begin{align}\label{eqn:harmonic_ineq_1}
\textstyle\sum_{k=0}^{K-1}\tfrac{1}{(k+\Gamma)^a} \leq \textstyle\sum_{k=0}^{K-1}\tfrac{1}{(k+1)^a} = 1+\textstyle\sum_{k=2}^{K}\tfrac{1}{k^a} \leq 1+ \int_{1}^{K}x^{-a}dx,
\end{align}
 where the last inequality is implied by noting that $x^{-a}$ is a nonincreasing function in $x$, for $a \geq 0$.

\noindent (i) From \eqref{eqn:harmonic_ineq_1} and that $a \in [0,1)$, we obtain  
\begin{align*} 
\textstyle\sum_{k=0}^{K-1}\tfrac{1}{(k+\Gamma)^a} \leq  1+ \int_{1}^{K}x^{-a}dx = 1+\tfrac{K^{(1-a)}-1}{1-a} \leq 	\int_{1}^{K}x^{-a}dx =  \tfrac{K^{(1-a)}}{1-a}.
\end{align*}

\noindent (ii) From \eqref{eqn:harmonic_ineq_1} and that $a =1$, we obtain $
\textstyle\sum_{k=0}^{K-1}\tfrac{1}{(k+\Gamma)^a} \leq  1+ \int_{1}^{K}x^{-1}dx =  1+\ln(K).$

\noindent (iii) From \eqref{eqn:harmonic_ineq_1} and that $a >1$, we obtain  
\begin{align*} 
\textstyle\sum_{k=0}^{K-1}\tfrac{1}{(k+\Gamma)^a} \leq  1+ \int_{1}^{K}x^{-a}dx = 1+\tfrac{K^{(1-a)}-1}{1-a} = 1+\tfrac{1-1/k^{a-1}}{a-1} \leq 1+\tfrac{1}{a-1} = \tfrac{a}{a-1}. 
\end{align*}
Next, we assume that  $\Gamma \geq 2$. We have 
\begin{align}\label{eqn:harmonic_ineq_2}
\textstyle\sum_{k=0}^{K-1}\tfrac{1}{(k+\Gamma)^a} =\textstyle\sum_{k=\Gamma}^{K+\Gamma-1}\tfrac{1}{k^a} \leq   \int_{\Gamma-1}^{K+\Gamma-1}x^{-a}dx.
\end{align}

\noindent (iv) From \eqref{eqn:harmonic_ineq_2} and that $a \in [0,1)$, we obtain $ 
\textstyle\sum_{k=0}^{K-1}\tfrac{1}{(k+\Gamma)^a} \leq  \int_{\Gamma-1}^{K+\Gamma-1}x^{-a}dx \leq \tfrac{(K+\Gamma -1)^{1-a}-(\Gamma-1)^{1-a}}{1-a}.$

\noindent (v) From \eqref{eqn:harmonic_ineq_2} and that $a =1$, we obtain $
\textstyle\sum_{k=0}^{K-1}\tfrac{1}{(k+\Gamma)^a} \leq  \int_{\Gamma-1}^{K+\Gamma-1}x^{-1}dx= \ln(\tfrac{K+\Gamma-1}{\Gamma-1}).$

\noindent (vi) From \eqref{eqn:harmonic_ineq_2} and that $a >1$, we obtain  
\begin{align*} 
\textstyle\sum_{k=0}^{K-1}\tfrac{1}{(k+\Gamma)^a} \leq  \int_{\Gamma-1}^{K+\Gamma-1}x^{-a}dx =  \tfrac{(K+\Gamma -1)^{1-a}-(\Gamma-1)^{1-a}}{1-a} =  \tfrac{(\Gamma-1)^{1-a}-(K+\Gamma -1)^{1-a}}{a-1} \leq \tfrac{1}{(a-1)\Gamma^{a-1}}.
\end{align*}\mee{$\hfill \Box$}


\medskip

\noindent{\bf Proof of Lemma~\ref{lem:Alg2_conv_2stage}.} Consider the update rule in Algorithm~\ref{alg:lowerlevel-2stage}. Invoking the nonexpansivity of the Euclidean projection, for any $\ell \in \{1,2\}$ and $t \geq 0$, we obtain 
\begin{align*}
 \|z_{i,t+1}^{k,\ell} - z_i(\hat{x}_{i,k}^{\ell},{\xi_{i,k}}) \|^2 &  = \left\| \Pi_{\mathcal{Z}_i(\hat x_{i,k}^{\ell},{\xi_{i,k}})}\left[z_{i,t}^{k,\ell}-\hat{\gamma}\tilde{F}_i(\hat{x}_{i,k}^{\ell},z_{i,t}^{k,\ell},{\xi_{i,k}})\right]  - z_i(\hat{x}_{i,k}^{\ell},{\xi_{i,k}})\right\|^2 \\
 & = \left\| \Pi_{\mathcal{Z}_i(\hat x_{i,k}^{\ell},{\xi_{i,k}})}\left[z_{i,t}^{k,\ell}-\hat{\gamma}\tilde{F}_i(\hat{x}_{i,k}^{\ell},z_{i,t}^{k,\ell},{\xi_{i,k}})\right]\right. \\
& \left.- \Pi_{\mathcal{Z}_i(\hat x_{i,k}^{\ell},{\xi_{i,k}})}\left[z_i(\hat{x}_{i,k}^{\ell},{\xi_{i,k}}) -\hat{\gamma}{F}_i(\hat{x}_{i,k}^{\ell},z_i(\hat{x}_{i,k}^{\ell},{\xi_{i,k}}),{\xi_{i,k}} )\right]  \right\|^2\\
 & \leq \left\|  z_{i,t}^{k,\ell}-\hat{\gamma}\tilde{F}_i(\hat{x}_{i,k}^{\ell},z_{i,t}^{k,\ell},{\xi_{i,k}})  -  z_i(\hat{x}_{i,k}^{\ell},{\xi_{i,k}}) +\hat{\gamma} {F}_i(\hat{x}_{i,k}^{\ell},z_i(\hat{x}_{i,k}^{\ell},{\xi_{i,k}}),{\xi_{i,k}} )  \right\|^2\\
&=\|  z_{i,t}^{k,\ell} -  z_i(\hat{x}_{i,k}^{\ell},{\xi_{i,k}})  \|^2+\left\|  \hat{\gamma}\tilde{F}_i(\hat{x}_{i,k}^{\ell},z_{i,t}^{k,\ell},{\xi_{i,k}})  -\hat{\gamma} {F}_i(\hat{x}_{i,k}^{\ell},z_i(\hat{x}_{i,k}^{\ell},{\xi_{i,k}}),{\xi_{i,k}} )  \right\|^2\\
&-2\hat{\gamma}\left( z_{i,t}^{k,\ell} -  z_i(\hat{x}_{i,k}^{\ell},{\xi_{i,k}}) \right)^\top\left(\tilde{F}_i(\hat{x}_{i,k}^{\ell},z_{i,t}^{k,\ell},{\xi_{i,k}})  - {F}_i(\hat{x}_{i,k}^{\ell},z_i(\hat{x}_{i,k}^{\ell},{\xi_{i,k}}),{\xi_{i,k}} )\right).
\end{align*}
Invoking Assumption~\ref{assum:main2-2stage} and for any $0\le t\le t_k-1$, we obtain
\begin{align*}
 \|z_{i,t+1}^{k,\ell} - z_i(\hat{x}_{i,k}^{\ell},\xi_{i,k}) \|^2 &  \le \|  z_{i,t}^{k,\ell} -  z_i(\hat{x}_{i,k}^{\ell},\xi_{i,k})  \|^2+ \hat{\gamma}^2L_F^2\|  z_{i,t}^{k,\ell} -  z_i(\hat{x}_{i,k}^{\ell},{\xi_{i,k}})  \|^2\\
 &-2\hat{\gamma}\mu_F\|  z_{i,t}^{k,\ell} -  z_i(\hat{x}_{i,k}^{\ell},{\xi_{i,k}})  \|^2
=(1+ \hat{\gamma}^2L_F^2-2\hat{\gamma}\mu_F)\|  z_{i,t}^{k,\ell} -  z_i(\hat{x}_{i,k}^{\ell},{\xi_{i,k}})  \|^2.
\end{align*}Unrolling the preceding recursive inequality, we obtain
 $\|z_{i,{t_k}}^{k,\ell} - z_i(\hat{x}_{i,k}^{\ell},{\xi_{i,k}}) \|^2  \le (1+ \hat{\gamma}^2L_F^2-2\hat{\gamma}\mu_F)^{t_k}\|  z_{i,0}^{k,\ell} -  z_i(\hat{x}_{i,k}^{\ell},{\xi_{i,k}})  \|^2.$
Note that $\hat{\gamma}\le\tfrac{\mu_F}{L_F^2}$ implies that $1+ \hat{\gamma}^2L_F^2-2\hat{\gamma}\mu_F\le 1-\hat{\gamma}\mu_F\triangleq \tilde d$. Based on the definition of $\tilde d$ and invoking Assumption~\ref{assum:main2-2stage}, where $D_i$ is such that $\sup_{x\in\mathbb{R}^n}\sup_{z_1,z_2 \in \mathcal{Z}_i(x,\xi_i)}\|z_1-z_2\|^2 \leq D_i$ for all $i \in [m]$, and for all $\xi_i\in \mathbb{R}^d$, $
 \|z_{i,{t_k}}^{k,\ell} - z_i(\hat{x}_{i,k}^{\ell},{\xi_{i,k}}) \|^2    \le{\tilde d}^{t_k} D_i .$ Taking conditional expectations on both sides of the preceding inequality, we obtain the result.\mee{$\hfill \Box$}

\medskip

\noindent{\bf Proof of Lemma~\ref{lemm:omega_vareps}.} 
(i) Using the definition of $e_{i,k}^{\eta,\varepsilon_k} $,  equations~\eqref{eqn:g_eta_eps_two stage} and \eqref{eqn:g_eta_two stage}, and the triangle inequality, we have
\begin{align*}
\|e_{i,k}^{\eta,\varepsilon_k}\| &= \|g_{i,k}^{\eta,\varepsilon_k}-g_{i,k}^{\eta}\|\\
& \leq (\tfrac{n}{2\eta})|\tilde h_i(x_{i,k}+\eta v_{i,k},z_{{i,\varepsilon}_k}({x}_{i,k}+\eta v_{i,k},\xi_{i,{k}}),\xi_{i,k})-\tilde h_i(x_{i,k}+\eta v_{i,k},z_i({x}_{i,k}+\eta v_{i,k},\xi_{i,{k}}),\xi_{i,k})| \\
&+(\tfrac{n}{2\eta})|\tilde h_i(x_{i,k}-\eta v_{i,k},z_{i,\varepsilon_k}({x}_{i,k}-\eta v_{i,k},\xi_{i,{k}}),\xi_{i,k})-\tilde h_i(x_{i,k}-\eta v_{i,k},z_i({x}_{i,k}-\eta v_{i,k},\xi_{i,{k}}),\xi_{i,k})|.
\end{align*}
Invoking the Lipschitz continuity of $\tilde h_i(x,\bullet,\xi_i)$ in \Rme{Assumption~\ref{assum:main6}(i)}, we obtain  
\begin{align*}
\|e_{i,k}^{\eta,\varepsilon_k}\|  &\le  (\tfrac{n}{2\eta} )\tilde{L}_0(\xi_{i,k})\left\|z_{{i,\varepsilon_k}}({x}_{i,k}+\eta v_{i,k},\xi_{i,k})-z_i({x}_{i,k}+\eta v_{i,k},\xi_{i,k})\right\|\\
&+  (\tfrac{n}{2\eta} )\tilde{L}_0(\xi_{i,k})\left\|z_{{i,\varepsilon_k}}({x}_{i,k}-\eta v_{i,k},\xi_{i,k})-z_i({x}_{i,k}-\eta v_{i,k},\xi_{i,k})\right\|.
\end{align*}
From the preceding inequality, we obtain
\begin{align*}
\|e_{i,k}^{\eta,\varepsilon_k}\|^2  &\le  2(\tfrac{n}{2\eta} )^2\tilde{L}_0(\xi_{i,k})^2\left\|z_{{i,\varepsilon_k}}({x}_{i,k}+\eta v_{i,k},\xi_{i,k})-z_i({x}_{i,k}+\eta v_{i,k},\xi_{i,k})\right\|^2\\
&+ 2 (\tfrac{n}{2\eta} )^2\tilde{L}_0(\xi_{i,k})^2\left\|z_{{i,\varepsilon_k}}({x}_{i,k}-\eta v_{i,k},\xi_{i,k})-z_i({x}_{i,k}-\eta v_{i,k},\xi_{i,k})\right\|^2.
\end{align*}
By invoking the inexactness bound and subsequently taking the conditional expectation on both sides of the preceding inequality, while leveraging the independence of $\xi_{i,k}$ and $v_{i,k}$, we obtain 
\begin{align*}
&\mathbb{E}[\|e_{i,k}^{\eta,\varepsilon_k}\| ^2 \mid\mathcal{F}_k] \le 2(\tfrac{n}{2\eta} )^2\mathbb{E}[\tilde{L}_0(\bxi_{i,k})^2]\varepsilon_k+ 2(\tfrac{n}{2\eta} )^2\mathbb{E}[\tilde{L}_0(\bxi_{i,k})^2]\varepsilon_k.
\end{align*}
Invoking the definition of $\tilde{L}_0$ in \Rme{Assumption~\ref{assum:main6}(i)}, we obtain the result. 
\mee{$\hfill \Box$}

\end{document}